\documentclass[10pt]{amsart}

\usepackage{latexsym,exscale,enumerate,amsfonts,amssymb,xparse, mathtools}
\usepackage[normalem]{ulem}
\usepackage{amsmath,amsthm,amsfonts,amssymb,amscd, stmaryrd,textcomp}
\usepackage[bookmarks=false]{hyperref}
\usepackage[usenames,dvipsnames]{xcolor}
\usepackage{enumitem}
\usepackage{verbatim}
\usepackage{ytableau}
\usepackage{euscript}

\usepackage{bbm}

\colorlet{green}{black!30!green}

\definecolor{myblue}{RGB}{109, 156, 179}

\usepackage{xcolor}
\definecolor{myorange}{rgb}{1,0.647,0}

\definecolor{mypurple}{cmyk}{.6,.9,0, .11}

\newcommand{\bl}[1]{{\color{blue}#1}}
\newcommand{\gray}[1]{{\color{gray}#1}}

\numberwithin{equation}{section}

\newcommand{\arxiv}[1]{\href{https://arxiv.org/abs/#1}{\small  arXiv:#1}}

\theoremstyle{definition}
\newtheorem{thm}{Theorem}[section]
\newtheorem{cor}[thm]{Corollary}

\newtheorem{conj}[thm]{Conjecture}
\newtheorem{lem}[thm]{Lemma}
\newtheorem{rem}[thm]{Remark}
\newtheorem{conv}[thm]{Convention}

\newtheorem{prop}[thm]{Proposition}
\newtheorem{defn}[thm]{Definition}
\newtheorem{example}[thm]{Example}

\newtheorem{problem}[thm]{Problem}

\newtheorem*{ack}{Acknowledgements}
\newtheorem*{fund}{Funding}

\usepackage{tikz}
\usetikzlibrary{cd}
\usetikzlibrary{snakes}
\usetikzlibrary{decorations.markings}
\usetikzlibrary{decorations.pathreplacing}
\usetikzlibrary{arrows,shapes,positioning}
\usetikzlibrary{decorations.pathmorphing}
\tikzset{anchorbase/.style={baseline={([yshift=-0.5ex]current bounding box.center)}},
tinynodes/.style={font=\tiny,text height=0.75ex,text depth=0.15ex},
smallnodes/.style={font=\scriptsize,text height=0.75ex,text depth=0.15ex},
>={Latex[length=1mm, width=1.5mm]},
overcross/.style={line width=4pt,color=white},
overcros/.style={line width=3pt,color=white},
overcro/.style={line width=2pt,color=white},
}
\tikzstyle directed=[postaction={decorate,decoration={markings,
	mark=at position #1 with {\arrow{>}}}}]
\tikzstyle rdirected=[postaction={decorate,decoration={markings,
	mark=at position #1 with {\arrow{<}}}}]

\newcommand{\Hom}{{\rm Hom}}

\newcommand{\End}{{\rm End}}

\renewcommand{\to}{\rightarrow}

\newcommand{\id}{{\rm id}}

\newcommand{\TL}{{\mathbf{TL}}}

\let\hat=\widehat
\let\tilde=\widetilde
\newcommand{\scs}{\scriptsize}
\def\C{{\mathbb C}}

\def\R{{\mathbb R}}
\def\Z{{\mathbb Z}}

\def\K{{\mathbb K}}

\def\N{\mathbb{Z}_{\ge 0}}

\newcommand{\sln}[1][n]{\mathfrak{sl}_{#1}}
\newcommand{\gln}[1][n]{\mathfrak{gl}_{#1}}

\newcommand{\slm}{\mathfrak{sl}_{m}}
\newcommand{\glm}{\mathfrak{gl}_{m}}
\newcommand{\spn}[1][2n]{\mathfrak{sp}_{#1}}
\newcommand{\son}[1][2n+1]{\mathfrak{so}_{#1}}
\newcommand{\som}[1][m]{\mathfrak{so}_{#1}}

\newcommand{\Rep}{\mathbf{Rep}}

\newcommand{\FRep}{\mathbf{FundRep}}

\newcommand{\Web}{\mathbf{Web}}

\newcommand{\Foam}{\mathbf{Foam}}

\newcommand\mydots{\makebox[1em][c]{$\cdot$\hfil$\cdot$\hfil$\cdot$}}
\newcommand\mysdots{\makebox[.5em][c]{$\cdot$\hfil$\cdot$\hfil$\cdot$}}
\newcommand{\ot}{\otimes}

\DeclareMathOperator{\Kar}{Kar}

\newcommand{\Br}{\mathrm{Br}}

\newcommand{\iqUsom}[1][m]{{U}^{\iota}_{-q^2}(\som[#1])}

\newcommand{\wt}{\mathbf{a}}

\newcommand{\iQW}{\prescript{\iota}{}{T}}

\newcommand{\Ad}{\prescript{\iota\!}{}{\mathrm{Ad}}}

\newcommand{\ee}{\alpha}

\newcommand{\FEk}{\mathbf{X}^{(k)}}

\newcommand{\FE}{\mathbf{X}}
\newcommand{\FEalt}{\mathbf{Y}}

\newcommand{\xx}{\gray{\mathsf{X}}}
\newcommand{\xxb}{\mathsf{x}}
\newcommand{\hh}{\mathsf{H}}

\newcommand{\gS}{\gray{\mathsf{S}}}
\newcommand{\gM}{\gray{M}}
\newcommand{\Lad}{\mathbf{Lad}}
\newcommand{\rung}{r}
\newcommand{\FreeWeb}{\mathbf{FreeWeb}}
\newcommand{\Free}{\mathbf{Free}}

\newcommand{\ab}{\mathbf{a}}
\newcommand{\abb}{\mathbf{b}}
\DeclareMathOperator{\lead}{lead}
\newcommand{\ep}{\varepsilon}
\newcommand{\SP}{\mathcal{SP}}

\newcommand{\dd}{d}

\newcommand{\nctriv}{\mathcal{N}}

\begin{document}
%

\title[]{Type $B$ Webs}

\author{Elijah Bodish}
\address{Department of Mathematics, Indiana University Bloomington, 
Rawles Hall, Bloomington, IN 47405-7106, USA}
\email{ebodish@iu.edu}

\author{Ben Elias}
\address{Department of Mathematics, University of Oregon,
Fenton Hall, Eugene, OR 97403-1222, USA}
\email{belias@uoregon.edu}

\author{David E. V. Rose}
\address{Department of Mathematics, University of North Carolina, 
Phillips Hall CB \#3250, UNC-CH, Chapel Hill, NC 27599-3250, USA}
\email{davidrose@unc.edu}

\begin{abstract}
We solve the type $B$ case of the main open problem from Kuperberg's 1996 paper ``Spiders for rank 2 Lie algebras.''
That is, we define a $\C(q)$-linear pivotal category $\Web(\son)$ and prove that it 
is equivalent to the full subcategory of finite-dimensional representations of 
$U_q(\son)$ tensor-generated by the fundamental representations. 
A consequence of our main result is an explicit construction 
of braid group symmetries for nonclassical finite-dimensional representations of the 
$\iota$quantum group $\iqUsom$, 
which may be of independent interest.
\end{abstract}

\maketitle

\setcounter{tocdepth}{1}
\tableofcontents

%
\section{Introduction}
%

Given a finite-dimensional simple complex Lie algebra $\mathfrak{g}$, 
let $U_q(\mathfrak{g})$ denote its quantized enveloping algebra,
and let $\FRep(U_q(\mathfrak{g}))$ denote the category whose 
objects\footnote{The dual of a fundamental representation is again fundamental,
so $\FRep(U_q(\mathfrak{g}))$ is equivalent to the category generated by the fundamentals 
and their duals. In some places in the literature, the duals are also included as generating objects in 
web categories, but this is not necessary (or desirable) in Type $B$.} 
are tensor products of \emph{fundamental representations}, i.e.~ tensor products of highest weight irreducible type I  
$U_q(\mathfrak{g})$-representations $V_{\varpi}$, where $\varpi$ is a fundamental weight of $\mathfrak{g}$. 
This is a full subcategory of the category $\Rep(U_q(\mathfrak{g}))$ of type I finite-dimensional 
representations over $\C(q)$. We call $\FRep(U_q(\mathfrak{g}))$ the \emph{fundamental subcategory}.

In \cite{Kup}, Kuperberg poses the problem of giving a presentation of 
$\FRep(U_q(\mathfrak{g}))$, as a pivotal category, 
and he solves this problem in rank $\leq 2$.
Subsequent work by Cautis--Kamnitzer--Morrison \cite{CKM} 
and the present authors and Tatham \cite{BERT,BERT2} solve this problem 
for $\mathfrak{g}=\sln$ (type $A$) and $\spn$ (type $C$), respectively.
See ibid.~for some history on this problem and references to related works.
In the present paper, we solve this problem for $\mathfrak{g} = \son$ (type $B$).

We now define our presentation, and then state the main result of this paper.
Throughout, we will use the following notation for quantum numbers:
\begin{itemize}
\item $[k]_v:= \tfrac{v^k - v^{-k}}{v - v^{-1}}$,
\item $[k] := [k]_q = \tfrac{q^k - q^{-k}}{q - q^{-1}}$ (the \emph{quantum integer} $k$),
\item $[2]_{m} := [2]_{q^m} = q^m + q^{-m}$,
\item $``[k][\ell]" := [\ell+k-1] - [\ell+k-3] + \cdots + (-1)^{k-1} [\ell-k+1]$, 
when $k\le \ell$ (the \emph{devil's product}), 
\item $``{2k \brack k}" := \prod_{i=1}^k [2]_{2i-1}$ (the \emph{devil's central binomial coefficient}).
\end{itemize}

\begin{defn}
	\label{def:webs}
Let $\Web(\son)$ be the $\C(q)$-linear pivotal category defined by 
the following graphical presentation. 
The objects are generated monoidally by self-dual objects $\{1,\ldots,n{-}1,\gS \}$. 
In addition to the cap/cup (co)unit morphisms implicit in $\Web(\son)$ being pivotal, 
the morphisms are generated by
\begin{equation}\label{eq:WebGen}
\Bigg\{
\begin{tikzpicture}[scale =.5, tinynodes,anchorbase]
	\draw[very thick, gray] (0,0) node[below=-1pt]{$\gS$} to [out=90,in=210] (.5,.75);
	\draw[very thick, gray] (1,0) node[below=-1pt]{$\gS$} to [out=90,in=330] (.5,.75);
	\draw[very thick] (.5,.75) to (.5,1.5) node[above=-2pt]{$k$};
\end{tikzpicture}
\Bigg\}_{k=1}^{n-1}
\cup
\Bigg\{
\begin{tikzpicture}[scale=.5, tinynodes,anchorbase]
	\draw[very thick] (0,0) node[below=-1pt]{$1$} to [out=90,in=210] (.5,.75);
	\draw[very thick] (1,0) node[below=-1pt]{$k$} to [out=90,in=330] (.5,.75);
	\draw[very thick] (.5,.75) to (.5,1.5) node[above=-2pt]{$k{+}1$};
\end{tikzpicture}
\, , \, 
\begin{tikzpicture}[scale =.5, tinynodes,anchorbase]
	\draw[very thick] (0,0) node[below=-1pt]{$k$} to [out=90,in=210] (.5,.75);
	\draw[very thick] (1,0) node[below=-1pt]{$1$} to [out=90,in=330] (.5,.75);
	\draw[very thick] (.5,.75) to (.5,1.5) node[above=-2pt]{$k{+}1$};
\end{tikzpicture}
\Bigg\}_{k=1}^{n-2} \, .
\end{equation}
One then takes the quotient by the tensor ideal generated by the following (local) relations.
\begin{subequations}
	\label{eq:webrel}

\begin{minipage}{.55\textwidth} 
\begin{gather} \label{eq:digonS}
\begin{tikzpicture}[scale=.175,tinynodes, anchorbase]
	\draw [very thick] (0,.75) to (0,2.5) node[above,yshift=-3pt]{$k$};
	\draw [very thick,gray] (0,-2.75) to [out=30,in=330] node[right,xshift=-2pt]{$\gS$} (0,.75);
	\draw [very thick,gray] (0,-2.75) to [out=150,in=210] node[left,xshift=2pt]{$\gS$} (0,.75);
	\draw [very thick] (0,-4.5) node[below,yshift=2pt]{$k$} to (0,-2.75);
\end{tikzpicture}
= (-1)^{\binom{n-k+1}{2}} 
``{\textstyle {2(n-k) \brack n-k}}"
\begin{tikzpicture}[scale=.175, tinynodes, anchorbase]
	\draw [very thick] (0,-4.5) node[below,yshift=2pt]{$k$} 
		to (0,2.5) node[above,yshift=-3pt]{$k$};
\end{tikzpicture}
\end{gather}
\end{minipage}
\begin{minipage}{.45\textwidth} 
\begin{gather} \label{eq:lolliS}
\begin{tikzpicture}[scale=.175,tinynodes,anchorbase,rotate=180]
	\draw [very thick, gray] (0,-2.75) to [out=30,in=0] (0,.75) node[below=-1pt]{$\gS$};
	\draw [very thick, gray] (0,-2.75) to [out=150,in=180] (0,.75);
	\draw [very thick] (0,-4.5) node[above=-2pt]{$k$} to (0,-2.75);
\end{tikzpicture}
= 0
\text{ if } 1 \leq k \leq n-1
\end{gather}
\end{minipage}
\begin{gather}
	\label{eq:spinH=I}
\begin{tikzpicture}[scale=.4, rotate=90, tinynodes, anchorbase]
	\draw[very thick, gray] (-1,.5) node[below,yshift=2pt,xshift=2pt]{$\gS$} to (0,1);
	\draw[very thick, gray] (1,.5) node[above,yshift=-4pt,xshift=2pt]{$\gS$} to (0,1);
	\draw[very thick, gray] (0,2.5) to (-1,3) node[below,yshift=2pt,xshift=-2pt]{$\gS$};
	\draw[very thick, gray] (0,2.5) to (1,3) node[above,yshift=-4pt,xshift=-2pt]{$\gS$};
	\draw[very thick] (0,1) to node[below,yshift=2pt]{$1$} (0,2.5);
\end{tikzpicture}
=\frac{1}{[2]}
\begin{tikzpicture}[scale=.5, tinynodes, anchorbase]
	\draw[very thick,gray] (1,-1) to (1,1);
	\draw[very thick,gray] (0,-1) to (0,1);
\end{tikzpicture} \
+ \sum_{k=1}^{n}(-1)^{\binom{k}{2}}\frac{``[k][k{+}1]"}{``{2k \brack k}"}
\begin{tikzpicture}[scale=.375, smallnodes, anchorbase]
	\draw[very thick,gray] (-1,0) to (0,1);
	\draw[very thick,gray] (1,0) to (0,1);
	\draw[very thick,gray] (0,2.5) to (-1,3.5);
	\draw[very thick,gray] (0,2.5) to (1,3.5);
	\draw[very thick] (0,1) to node[right=-2pt]{$n{-}k$} (0,2.5);
\end{tikzpicture} \\
	\label{eq:blackspinH=I}
\begin{tikzpicture}[scale=.4, rotate=90, tinynodes, anchorbase]
	\draw[very thick] (-1,.5) node[below,yshift=2pt,xshift=2pt]{$1$} to (0,1);
	\draw[very thick, gray] (1,.5) node[above,yshift=-4pt,xshift=2pt]{$\gS$} to (0,1);
	\draw[very thick] (0,2.5) to (-1,3) node[below,yshift=2pt,xshift=-2pt]{$k$};
	\draw[very thick, gray] (0,2.5) to (1,3) node[above,yshift=-4pt,xshift=-2pt]{$\gS$};
	\draw[very thick, gray] (0,1) to node[below,yshift=1pt]{$\gS$} (0,2.5);
\end{tikzpicture}
= 
\begin{tikzpicture}[scale=.375, tinynodes, anchorbase]
	\draw[very thick] (-1,0) node[below=-2pt]{$k$} to (0,1);
	\draw[very thick] (1,0) node[below=-2pt]{$1$} to (0,1);
	\draw[very thick,gray] (0,2.5) to (-1,3.5);
	\draw[very thick,gray] (0,2.5) to (1,3.5);
	\draw[very thick] (0,1) to node[right=-2pt]{$k{+}1$} (0,2.5);
\end{tikzpicture}
+ (-1)^{n-k+1} \frac{1}{[2]_{2n{-}2k{+}1}}
\begin{tikzpicture}[scale=.375, tinynodes, anchorbase]
	\draw[very thick] (-1,0) node[below=-2pt]{$k$} to (0,1);
	\draw[very thick] (1,0) node[below=-2pt]{$1$} to (0,1);
	\draw[very thick,gray] (0,2.5) to (-1,3.5);
	\draw[very thick,gray] (0,2.5) to (1,3.5);
	\draw[very thick] (0,1) to node[right=-2pt]{$k{-}1$} (0,2.5);
\end{tikzpicture}
\quad \text{if } 1 \leq k \leq n-2
\end{gather}
\begin{gather}
	\label{eq:n-1triangle}
\begin{tikzpicture}[scale=.25,tinynodes,anchorbase]
	\draw[very thick, gray] (-1,0) to (1,0);
	\draw[very thick,gray] (-1,0) to (0,1.732);
	\draw[very thick,gray] (1,0) to (0,1.732);
	\draw[very thick] (0,1.732) to (0,3.232) node[above=-3pt]{$n{-}1$};
	\draw[very thick] (-2.3,-.75) node[below=-2pt]{$k$} to (-1,0);
	\draw[very thick] (2.3,-.75) node[below=-2pt]{$1$} to (1,0);
\end{tikzpicture} = 0 
\quad \text{if } k \neq n-2
	\end{gather}
\begin{minipage}{.55\textwidth} 
\begin{gather}
	\label{eq:digon1}
\begin{tikzpicture}[scale=.175,tinynodes, anchorbase]
	\draw [very thick] (0,.75) to (0,2.5) node[above,yshift=-3pt]{$k$};
	\draw [very thick] (0,-2.75) to [out=30,in=330] node[right,xshift=-2pt]{$k{-}1$} (0,.75);
	\draw [very thick] (0,-2.75) to [out=150,in=210] node[left,xshift=2pt]{$1$} (0,.75);
	\draw [very thick] (0,-4.5) node[below,yshift=2pt]{$k$} to (0,-2.75);
\end{tikzpicture}
= (-1)^{k-1} ``[k]^2"
\begin{tikzpicture}[scale=.175,tinynodes, anchorbase]
	\draw [very thick] (0,-4.5) node[below,yshift=2pt]{$k$} to (0,2.5);
\end{tikzpicture} \qquad
\end{gather}
\end{minipage}
\begin{minipage}{.45\textwidth} 
\begin{gather}
	\label{eq:assoc}
	\begin{tikzpicture}[scale=.2, xscale=-1,tinynodes, anchorbase]
		\draw [very thick] (-1,-1) node[below,yshift=2pt]{$1$} to [out=90,in=210] (0,.75);
		\draw [very thick] (1,-1) node[below,yshift=2pt]{$k$} to [out=90,in=330] (0,.75);
		\draw [very thick] (3,-1) node[below,yshift=2pt]{$1$} to [out=90,in=330] (1,2.5);
		\draw [very thick] (0,.75) to [out=90,in=210] (1,2.5);
		\draw [very thick] (1,2.5) to (1,4.25) node[above,yshift=-3pt]{$k{+}2$};
	\end{tikzpicture}
	=
	\begin{tikzpicture}[scale=.2, tinynodes, anchorbase]
		\draw [very thick] (-1,-1) node[below,yshift=2pt]{$1$} to [out=90,in=210] (0,.75);
		\draw [very thick] (1,-1) node[below,yshift=2pt]{$k$} to [out=90,in=330] (0,.75);
		\draw [very thick] (3,-1) node[below,yshift=2pt]{$1$} to [out=90,in=330] (1,2.5);
		\draw [very thick] (0,.75) to [out=90,in=210] (1,2.5);
		\draw [very thick] (1,2.5) to (1,4.25) node[above,yshift=-3pt]{$k{+}2$};
	\end{tikzpicture} \qquad
\end{gather}
\end{minipage}
\begin{gather}
	\label{eq:blackH=I}
\begin{tikzpicture}[scale=.4, rotate=90, tinynodes, anchorbase]
	\draw[very thick] (-1,.5) node[below,yshift=2pt,xshift=2pt]{$k$} to (0,1);
	\draw[very thick] (1,.5) node[above,yshift=-4pt,xshift=2pt]{$1$} to (0,1);
	\draw[very thick] (0,2.5) to (-1,3) node[below,yshift=2pt,xshift=-2pt]{$k$};
	\draw[very thick] (0,2.5) to (1,3) node[above,yshift=-4pt,xshift=-2pt]{$1$};
	\draw[very thick] (0,1) to node[below,yshift=2pt]{$k{+}1$} (0,2.5);
\end{tikzpicture}
=
\begin{tikzpicture}[scale=.4, tinynodes, anchorbase]
	\draw[very thick] (-1,0) node[below,yshift=2pt,xshift=-2pt]{$k$} to (0,1.5);
	\draw[very thick] (1,0) node[below,yshift=2pt,xshift=2pt]{$k$} to (0,1.5);
	\draw[very thick] (-.7,.5) to node[below,yshift=4pt]{$_{k{-}1}$} (.7,.5);
	\draw[very thick] (0,2.5) to (-1,3.5) node[above,yshift=-4pt,xshift=-2pt]{$1$};
	\draw[very thick] (0,2.5) to (1,3.5) node[above,yshift=-4pt,xshift=2pt]{$1$};
	\draw[very thick] (0,1.5) to node[right, xshift=-2pt]{$2$} (0,2.5);
\end{tikzpicture} \!\!
+\frac{[2]_{2n{-}2k{-}1}}{[2]_{2n{-}2k{+}1}}
\begin{tikzpicture}[scale=.4, rotate=90, tinynodes, anchorbase]
	\draw[very thick] (-1,0.5) node[below,yshift=2pt,xshift=2pt]{$k$} to (0,1);
	\draw[very thick] (1,0.5) node[above,yshift=-4pt,xshift=2pt]{$1$} to (0,1);
	\draw[very thick] (0,2.5) to (-1,3) node[below,yshift=2pt,xshift=-2pt]{$k$};
	\draw[very thick] (0,2.5) to (1,3) node[above,yshift=-4pt,xshift=-2pt]{$1$};
	\draw[very thick] (0,1) to node[below,yshift=2pt]{$k{-}1$} (0,2.5);
\end{tikzpicture}
+ (-1)^k
\frac{[2]_{2n{-}2k{-}1}}{[2]_{2n{-}1}}
\begin{tikzpicture}[scale=.4, tinynodes, anchorbase]
	\draw[very thick] (-1,0) node[below,yshift=2pt]{$k$} to [out=90,in=180] (0,1) 
		to [out=0,in=90] (1,0);
	\draw[very thick] (-1,3) node[above,yshift=-3pt]{$1$} to [out=270,in=180] (0,2)
		to [out=0,in=270] (1,3);
\end{tikzpicture}
\quad \text{if } 1 \leq k \leq n-2
\end{gather}
\end{subequations}
\end{defn}

\begin{conv}\label{conv:webs}
Iterative compositions of tensor products of the generating morphisms in \eqref{eq:WebGen} 
and the cap/cup morphisms are called \emph{webs}; 
thus, morphisms in $\Web(\son)$ are $\C(q)$-linear combinations of webs.
When a strand in a web is labeled zero one can simply erase it. 
For example, the $k=1$ version of \eqref{eq:digon1} holds tautologically, 
and the $k=0$ version of \eqref{eq:digonS} gives the circle relation
\begin{gather} \label{eq:circleS}
\begin{tikzpicture}[scale =.625, anchorbase]
	\draw[very thick, gray] (0,0) node[left=7pt]{\scs$\gS$} circle (.5);
\end{tikzpicture}
= (-1)^{\binom{n+1}{2}} {\textstyle ``{2n \brack n}"}
= (-1)^{\binom{n+1}{2}}\prod_{i=1}^n [2]_{2i-1}  \, .
	\end{gather}
On the other hand, 
whenever a strand has negative label, the web is interpreted as the zero morphism. 
With this convention, 
relations \eqref{eq:blackspinH=I} and \eqref{eq:blackH=I} hold tautologically when $k=0$ as well.
Lastly, we emphasize that 
\textbf{there is no version} of \eqref{eq:blackspinH=I} or \eqref{eq:blackH=I} 
when $k=n-1$, as \textbf{we do not permit} $n$-labeled strands;
see Remark \ref{rem:morelabels}.
\end{conv}

\begin{rem}
The \emph{devil's arithmetic}
describes scalars appearing in quantum type $B$ representation theory 
as signed versions of scalars appearing in type $A$ representation theory. 
Normally, for $k \leq \ell$, the product of quantum integers satisfies
\[
[k][\ell] = [\ell+k-1] + [\ell+k-3] + [\ell+k-5] + \cdots + [\ell-k+1] \, , 
\]
but in the devil's arithmetic, we instead set 
\begin{equation}\label{eq:devils}
``[k][\ell]" = [\ell+k-1] - [\ell+k-3] + [\ell+k-5] - \cdots + (-1)^{k-1} [\ell-k+1] \, . 
\end{equation}
One can check the useful identity $``[k]^2"=[k]_{q^{2}}$.

Similarly, $``{2k \brack k}" := \prod_{i=1}^k [2]_{2i-1}$ is a signed version of ${2k \brack k}$.
Indeed, in \cite[Equation 10.35]{BER} we showed that
$``{2k \brack k}"$ is a graded count of the ``transpose-invariant'' 
partitions fitting inside a $k \times k$ box.
(Recall that ${2k \brack k}$ is a graded count of all partitions fitting inside a $k \times k$ box.)
Using
\begin{equation} [2]_{2k+1} = \frac{``{2(k+1) \brack (k+1)}"}{``{2k \brack k}"}, \end{equation}
one could express all the coefficients in \eqref{eq:webrel} with the devil's arithmetic.

Devil's arithmetic was introduced in our previous work \cite{BER}, 
where these signs were explained via equivariant categorification; 
this connection is discussed in further detail in \S \ref{ss:categorification} below. 
\end{rem}

\begin{rem}\label{rem:rels}
The relations in \eqref{eq:webrel} that involve the spin object/representation appear in our 
previous work \cite{BER} on spin link homology,
save for \eqref{eq:blackspinH=I} and \eqref{eq:n-1triangle}, which are new. 
Our relations that do not involve the spin representation are
signed versions of the orthogonal web relations from \cite{BodWu},
which themselves are analogues of the type $C$ web relations from \cite{BERT},
but with different coefficients.
	\end{rem}

\begin{rem} 
Our relations above are not obviously invariant under 
the horizontal or vertical reflection of diagrams.
However, the ideal they generate (and hence the category itself) 
does possess these symmetries. 
By construction the category possesses rotational symmetry, 
thus horizontal symmetry implies vertical symmetry. 
The horizontal flip of \eqref{eq:digon1} is just a rotation of itself. 
The horizontal flips of \eqref{eq:blackspinH=I} and \eqref{eq:n-1triangle} 
are proven in Proposition \ref{prop:hardrel}. 
\end{rem}

The objects of the category $\FRep(U_q(\son))$ 
are tensor products of fundamental representations 
 of $U_q(\son)$. 
 For $1 \leq k \leq n$, let $\varpi_k$ denote the $k$-th fundamental $\son$ weight 
 and write $V_{\varpi_k}$ to denote the corresponding highest weight irreducible type I $U_q(\son)$-representation. 
For $k=1, \dots, n-1$, these are the quantized analogues of the $\son$-representations $\Lambda^k(\C^{2n+1})$. 
Let $S$ denote the (quantized) spin representation, 
which is the fundamental representation $V_{\varpi_n}$ of $U_q(\son)$. 
(The representation $S$ is distinguished from the formal object $\gS$ in $\Web(\son)$ 
by a change of color and font.) 
The main result of this paper is the following.

\begin{thm} \label{thm:main}
There is an equivalence of $\C(q)$-linear pivotal categories 
\[
\varphi \colon \Web(\son) \to \FRep(U_q(\son))
\]
sending $k \mapsto V_{\varpi_k}$ (for $1 \le k \le n-1$) and $\gS \mapsto S$.
\end{thm}

Every irreducible representation in $\Rep(U_q(\son))$ is a direct summand of some tensor product 
of fundamental representations. 
Hence, $\Rep(U_q(\son))$ is equivalent to the additive 
Karoubi completion of $\FRep(U_q(\son)) \cong \Web(\son)$, 
so we view Theorem \ref{thm:main} as giving a diagrammatic description for the 
(type I, finite-dimensional) representation theory of $U_q(\son)$.
Furthermore, $\FRep(U_q(\son))$ is braided, and we obtain formulae for the 
braiding isomorphisms in $\Web(\son)$ which extend Theorem \ref{thm:main}
to an equivalence of ribbon (so, in particular, of braided pivotal) categories;
see Remark \ref{rem:braiding}.

\subsection{Spin relations from categorification and folding}
	\label{ss:categorification}
Given a finite-dimensional simple complex Lie algebra $\mathfrak{g}$, 
the first step in establishing an equivalence $\FRep(U_q(\mathfrak{g})) \cong \Web(\mathfrak{g})$
is finding an appropriate candidate for the web category.
The general form of the presentation in Definition \ref{def:webs} is suggested by 
the semisimplicity of $\Rep(U_q(\son))$.
For example, the existence of trivalent vertices $\gS \otimes \gS \to k$ 
and of a relation like \eqref{eq:spinH=I} follow (roughly) 
from the need to decompose $S \otimes S$ into a direct sum of irreducible representations.
Similarly, \eqref{eq:lolliS} follows from Schur's Lemma, 
as do the forms of relations \eqref{eq:digonS} and \eqref{eq:digon1}.
However, the coefficients appearing in these relations are not obvious.
We presently focus on those involving $\gS$ since, up to sign,
the remainder appeared in \cite{BodWu} (see Remark \ref{rem:rels}).

Here, our presentation is informed by our previous work on 
\emph{spin link homology} \cite{BER}.
Therein, we equip $\Lambda_q^n \C^{2n}$-colored $\sln[2n]$ Khovanov--Rozansky link homology 
with a novel involution, and show that this refined invariant gives a 
categorification of the spin-colored $\son$ quantum link polynomials 
(when $n \leq 3$, and for all $n$ assuming some conjectures of a technical nature).
This is an implementation, in the realm of link homology, 
of the principle of folding along Dynkin diagram automorphisms;
specifically, this is the version\footnote{The other version of folding, 
which is Langlands dual to this one, instead considers the invariant 
subalgebra under the induced action on the Lie algebra.} 
of folding that constructs one root system by summing over orbits of simple roots 
under the action induced by the diagram automorphism.
When we fold from type $A_{2n-1}$ to type $B_n$, 
the fundamental $U_q(\sln[2n])$-representation $\Lambda_q^n\C^{2n}$ 
corresponds to the spin representation of $U_q(\son)$.

Let us explain how this folding motivates our coefficients, 
beginning with a recollection of type $A$ categorification. 
There is a bicategory $2n\mathbf{Foam_+}$ of 
\emph{progressive}\footnote{This is new terminology from \cite{Recio}, which we like.} 
(or \emph{upward}) $\gln[2n]$ foams \cite{QR1}, 
wherein the $\Hom$-categories categorify certain morphism spaces 
in $\FRep(U_q(\sln[2n])) \cong \Web(\sln[2n])$.
In particular, there is an object $n^{\otimes m} \in 2n\mathbf{Foam_+}$ 
so that
\begin{equation} \label{grothofbnm}
\C(q) \otimes_{\Z[q^{\pm}]} K_0(\End_{2n\mathbf{Foam_+}}(n^{\otimes m})) 
	\cong \End_{U_q(\sln[2n])}\big((\Lambda_q^n\C^{2n})^{\otimes m} \big) \, .
	\end{equation}
(Here $K_0$ denotes the Grothendieck ring.)
Set $\mathbf{B}^n_m := \End_{2n\mathbf{Foam_+}}(n^{\otimes m})$. 

Let us examine the case $m=2$. For $0 \leq k \leq n$, the $\sln[2n]$-webs
\begin{equation}\label{eq:FEwebs}
\chi^{(k)}:=
\begin{tikzpicture}[anchorbase,tinynodes,scale=.625]
	\draw[very thick,->] (0,.25) to [out=150,in=270] (-.25,1) node[above=-3pt]{$n$};
	\draw[very thick,->] (.5,.5) to (.5,1) node[above=-3pt]{$n$};
	\draw[very thick,directed=.65] (0,.25) to node[above=-2pt,xshift=-1pt]{$k$} (.5,.5);
	\draw[very thick,directed=.65] (0,-.25) to (0,.25);
	\draw[very thick,directed=.65] (.5,-.5) to [out=30,in=330] (.5,.5);
	\draw[very thick,directed=.65] (.5,-.5) to node[below=-1pt,xshift=-1pt]{$k$} (0,-.25);
	\draw[very thick,directed=.65] (.5,-1) node[below=-2pt]{$n$} to (.5,-.5);
	\draw[very thick,directed=.5] (-.25,-1)node[below=-2pt]{$n$} to [out=90,in=210] (0,-.25);
	\end{tikzpicture}
=
\begin{tikzpicture}[anchorbase,tinynodes,xscale=-1,scale=.625]
	\draw[very thick,->] (0,.25) to [out=150,in=270] (-.25,1) node[above=-3pt]{$n$};
	\draw[very thick,->] (.5,.5) to (.5,1) node[above=-3pt]{$n$};
	\draw[very thick,directed=.65] (0,.25) to node[above=-2pt,xshift=1pt]{$k$} (.5,.5);
	\draw[very thick,directed=.65] (0,-.25) to (0,.25);
	\draw[very thick,directed=.65] (.5,-.5) to [out=30,in=330] (.5,.5);
	\draw[very thick,directed=.65] (.5,-.5) to node[below=-1pt,xshift=1pt]{$k$} (0,-.25);
	\draw[very thick,directed=.65] (.5,-1) node[below=-2pt]{$n$} to (.5,-.5);
	\draw[very thick,directed=.5] (-.25,-1)node[below=-2pt]{$n$} to [out=90,in=210] (0,-.25);
	\end{tikzpicture}
\end{equation}
span the endomorphism space $\End_{\Web(\sln[2n])}(n \otimes n)$. 
Moreover, this unital algebra is generated by $\chi^{(1)}$, 
with the remaining elements $\chi^{(k)}$ determined recursively by the relation
\begin{equation}\label{eq:FE}
\chi^{(k)} \chi^{(1)} = [k][k+1] \chi^{(k)} + [k+1]^2 \chi^{(k+1)} \, .
\end{equation}
This relation continues to hold for all $k \geq 1$ if we impose the relation that 
$\chi^{(k)} = 0$ for $k >n$, which gives the algebra isomorphism
\begin{equation}\label{eq:FEalg}
\End_{\Web(\sln[2n])}(n \otimes n) \cong 
\C(q)[\chi^{(1)}] {\big /} \chi^{(n+1)} = 
\C(q)[\chi^{(1)}] {\Big /} {\textstyle \prod_{k=0}^n \big(\chi^{(1)} - [k][k+1]\big)} \, .
\end{equation}
There are $1$-morphisms $\FEk$ in $\mathbf{B}_2^n$ (for each $k \ge 1$) 
which categorify $\chi^{(k)}$ via \eqref{grothofbnm},
and equation \eqref{eq:FE} admits a categorification by an isomorphism
\begin{equation}\label{eq:FEcat}
\FEk \FE^{(1)} \cong [k][k+1] \FEk \oplus [k+1]^2 \FE^{(k+1)}.
\end{equation}
More precisely, there are two $1$-morphisms $\FE^{(k)}$ and $\FEalt^{(k)}$ 
which categorify the webs on the left- and right-hand sides of \eqref{eq:FEwebs} respectively, 
and $\FE^{(k)} \cong \FEalt^{(k)}$.

In \cite{BER} we discovered an involution $\tau$ on $\mathbf{B}_2^n$ 
that swaps $\FEk$ with $\FEalt^{(k)}$, hence fixing the $1$-morphisms $\FEk$ up to isomorphism. 
In \cite[Proposition 8.6]{BER}, we show that each of the 
$1$-morphisms $\FEk$ may be equipped with an equivariant structure, 
i.e.~a coherent choice of isomorphism $\FEk \cong \tau(\FEk)=\FEalt^{(k)}$. 
By slight abuse of notation, we use the same symbol $\FEk$ 
to denote the corresponding object in the equivariant category $(\mathbf{B}^n_2)^\tau$. 
In fact, $\FEk$ admits two equivariant structures which differ by a sign, 
and we will denote the object given by the other equivariant structure by $-\FEk$. 
The decomposition \eqref{eq:FEcat} then lifts to the equivariant category to produce a decomposition
\begin{equation}\label{eq:introbXkcat}
\FEk \FE 
	\cong (-1)^{k} ``[k][k+1]"\FEk \oplus (-1)^k``[k+1]^2"\FE^{(k+1)} \, .
\end{equation}
Briefly put, the signs in this decomposition are computed as the trace of $\tau$ 
acting on certain graded multiplicity spaces in $\mathbf{B}_2^n$.
Decompositions such as \eqref{eq:introbXkcat} are the origins of the devil's arithmetic.

The folding paradigm suggests that an appropriate decategorification of $(\mathbf{B}^n_2)^\tau$ 
should be related to type $B$ representation theory.
Specifically, we consider the \emph{equivariant Grothendieck group} 
$K_0^\tau((\mathbf{B}^n_2)^\tau)$, where one imposes the relation $[-\FEk] = -[\FEk]$. We propose the following.
\begin{conj}[{c.f.~\cite[Conjecture 1.10]{BER}}]
	\label{conj:SLH}
If $(\mathbf{B}^n_m)^\tau$ denotes the equivariant category 
associated with the involution $\tau$ acting on 
$\mathbf{B}^n_m := \End_{2n\mathbf{Foam_+}}(n^{\otimes m})$,
then
\[
\C(q) \otimes_{\Z[q^{\pm}]} K_0^{\tau}\big((\mathbf{B}^n_m)^\tau \big)
	\cong \End_{U_q(\son)}(S^{\otimes m}) \, .
	\]
\end{conj}

We establish the $m=2$ case of Conjecture \ref{conj:SLH} in \cite[Theorem 1.9]{BER}. 
In \cite[Definition 4.24]{BER}, 
we identify an element $\mathsf{X} \in \End_{U_q(\son)}(S^{\otimes 2})$ 
that corresponds to the class of $\FE^{(1)}$ and use the equation 
\begin{equation}\label{eq:introbXk}
\mathsf{X}^{(k)} \mathsf{X}  
	= (-1)^{k} ``[k][k+1]"\mathsf{X}^{(k)} + (-1)^k``[k+1]^2"\mathsf{X}^{(k+1)}
\end{equation}
(i.e.~the equivariant decategorification of \eqref{eq:introbXkcat})
to define elements $\mathsf{X}^{(k)} \in \End_{U_q(\son)}(S^{\otimes 2})$. 
We then show in \cite[Proposition 4.25 and Remark 4.26]{BER} that
\begin{equation}\label{eq:devilsalg}
\End_{U_q(\son)}(S \otimes S) \cong 
\C(q)[\mathsf{X}] {\big /} \mathsf{X}^{(n+1)} = 
\C(q)[\mathsf{X}] {\Big /} 
	{\textstyle \prod_{k=0}^n \big(\mathsf{X}^{(1)} - (-1)^k ``[k][k+1]" \big)} \, , 
\end{equation}
which is the type $B$ analogue of \eqref{eq:FEalg}.

Further considerations in equivariant categorification 
(which did not appear\footnote{But should appear in a planned extension of that work to 
non-minuscule fundamentals.} explicitly in \cite{BER})
suggest that $\mathsf{X} \in \End_{U_q(\son)}(S^{\otimes 2})$ corresponds to the web
\begin{equation}\label{eq:introX}
\xx :=
\begin{tikzpicture}[scale=.325, rotate=90, tinynodes, anchorbase]
	\draw[very thick, gray] (-1,.5) to (0,1);
	\draw[very thick, gray] (1,.5) to (0,1);
	\draw[very thick, gray] (0,2.5) to (-1,3);
	\draw[very thick, gray] (0,2.5) to (1,3);
	\draw[very thick] (0,1) to node[below,yshift=2pt]{$1$} (0,2.5);
\end{tikzpicture}
-\frac{1}{[2]}
\begin{tikzpicture}[scale=.35, tinynodes, anchorbase]
	\draw[very thick,gray] (1,-1) to (1,1);
	\draw[very thick,gray] (0,-1) to (0,1);
\end{tikzpicture} \in \End_{\Web(\son)}(\gS \otimes \gS) \, .
\end{equation}
(This is equivalent to our definition of $\mathsf{X}$ in \cite{BER}.)
One can then check that the coefficients in \eqref{eq:digonS} and \eqref{eq:spinH=I} 
are designed so that $\prod_{k=0}^n \big(\xx - (-1)^k ``[k][k+1]" \big) = 0$,
as is required by \eqref{eq:devilsalg}.

In summary, one finds candidate relations in $\End_{\Web(\son)}(\gS \otimes \gS)$ 
by considering the categorification of web relations in $\End_{\Web(\sln[2n])}(n \otimes n)$, 
passing to the equivariant category, and decategorifying. 
We propose that the following general framework governs webs in non-simply laced type. 
For the duration, $\Web(\mathfrak{g})$ denotes a $\C(q)$-linear monoidal category 
presented via generators-and-relations that is equivalent to $\Rep(U_q(\mathfrak{g}))$, 
as in Theorem \ref{thm:main}.

\begin{conj}\label{conj:metaweb}
Let $\mathfrak{g}$ be a simply laced simple Lie algebra.
If $\tau$ is a diagram automorphism of $\mathfrak{g}$, 
let $\mathfrak{g}_{\tau}$ denote the non-simply laced Lie algebra that is
Langlands dual\footnote{The root system for $\mathfrak{g}_{\tau}$ 
is the folding of the root system of $\mathfrak{g}$ \cite{Stembridge}, 
and is the dual of the root system of $\mathfrak{g}^{\tau}$; see \cite{MOfolding}.} 
to the fixed point subalgebra $\mathfrak{g}^{\tau}$.
(Thus, if $\mathfrak{g} = \sln[2n], \son[2n], \mathfrak{e}_6,$ or $\son[8]$, 
and $\tau$ is order $2, 2, 2$, and $3$, respectively,
then $\mathfrak{g}_{\tau}= \son, \spn[2n-2], \mathfrak{f}_4, \mathfrak{g}_2$.)
\begin{itemize}
\item There is a graded, linear, monoidal bicategory $\Foam(\mathfrak{g})$ such that 
\[
\C(q) \otimes_{\Z[q^{\pm}]} K_0(\Foam(\mathfrak{g}))
	\cong \Web(\mathfrak{g}) \cong \FRep(U_q(\mathfrak{g})) \, .
	\]
\item There is a full, monoidal, sub-bicategory $\mathbf{B}(\mathfrak{g}) \subset \Foam(\mathfrak{g})$ 
that admits an automorphism $\tau$ (of the same order as the corresponding diagram automorphism) 
such that
\[
\C(q) \otimes_{\Z[q^{\pm}]} K_0^{\tau}(\mathbf{B}(\mathfrak{g})^\tau)
	\cong \Web(\mathfrak{g}_{\tau}) \cong \FRep(U_q(\mathfrak{g}_{\tau})) \, .
	\]
\end{itemize}
\end{conj}

Taken together, 
the present work and \cite{BER} provide a wealth of evidence 
in support of Conjecture \ref{conj:metaweb}.
Support outside types A/B is given by the following.

\begin{example}
Let $\Lambda_q^2\C^8$ be the quantum analogue of the $\son[8]$ representation $\Lambda^2 \C^8$. 
This is the fundamental representation associated to the central Dynkin node in type $D_4$, 
and has quantum dimension
\begin{equation}\label{eq:sp8circle}
[11]+2[7]+[3] =
q^{10}+q^{8}+3 q^{6}+3q^{4}+4q^{2}+4+4q^{-2}+3q^{-4}+3q^{-6}+q^{-8}+q^{-10} \, .
\end{equation}
This is the value of a $\Lambda_q^2\C^8$-colored circle in $\Web(\son[8])$. 
(This latter category has not yet appeared in the literature.)
In $\Foam(\son[8])$ (also to-be-defined), this circle relation will be categorified 
by a direct sum decomposition involving a multiplicity space having \eqref{eq:sp8circle} 
as its graded dimension.

Such a vector space admits a grading-preserving action of $\Z/3$ whose 
trace equals 
\[
[11] - [7] + [3] = q^{10} + q^{8} + q^{2} + 1 + q^{-2} + q^{-8} + q^{-10} \, .
\]
This recovers Kuperberg's circle value for the $7$-dimensional fundamental $\mathfrak{g}_2$ representation
\cite[p.~14]{Kup}. (Note: Kuperberg's $q$ is our $q^2$.)
	\end{example}
	
\subsection{Web bases}
	\label{ss:webbasesintro}

For rank 1 $\mathfrak{g} (\cong \sln[2] \cong \son[3])$, 
$\Web(\mathfrak{g})$ is equivalent to the Temperley--Lieb category $\TL$,
which is pivotally generated by a single self-dual object $\gS$. 
In this case, there are no trivalent vertices in the presentation, 
and $\Hom_{\TL}(\gS^{\otimes m} , \gS^{\otimes n})$ is spanned by planar $(m,n)$-tangles 
modulo isotopy and the circle relation 
$\begin{tikzpicture}[scale =.5, anchorbase]
	\draw[very thick, gray] (0,0) circle (.5);
\end{tikzpicture} = - [2]$.
Since $\TL$ is pivotal, each morphism space $\Hom_{\TL}(\gS^{\otimes m} , \gS^{\otimes n})$ 
can be identified with the span of planar tangles in the closed disk $\mathbb{D}$ 
with $m+n$ boundary points on $\partial \mathbb{D}$ 
(by choosing an appropriate basepoint disjoint from the boundary of the tangles).
It is straightforward to show that this space has a basis consisting of planar matchings of 
$m+n$ points, i.e.~the planar tangles in $\mathbb{D}$ having no closed components.

Extending this, for rank 2 $\mathfrak{g} (= \sln[3], \son[5] \cong \spn[4], \text{ or } \mathfrak{g}_2)$, 
Kuperberg's web categories $\Web(\mathfrak{g})$ similarly admit bases characterized by certain 
topological/combinatorial properties. 
Again, all morphism spaces can be identified with the span of certain graphs in $\mathbb{D}$ 
with appropriate boundary, and the basis elements are the so-called \emph{non-elliptic webs}, 
which are characterized by the types of allowed faces 
(in particular, this again implies that no closed components are permitted).

For various applications/considerations, 
it is desirable to have higher rank versions of Kuperberg's non-elliptic web bases. 
Given that these bases are characterized by topological/combinatorial properties of the constituent graphs, 
these bases are \emph{rotation-invariant}, 
i.e.~ (appropriate) rotations of the disk preserve the non-elliptic basis.
Recent work of Gaetz--Pechenik--Pfannerer--Striker--Swanson \cite[Theorem A]{MR5008156}
constructs such a basis for $\mathfrak{g}=\mathfrak{sl}_4$, 
but, in general, it is an open problem to find rank $3$ (and higher) non-elliptic, rotation-invariant web bases
(although some initial steps for type $C$ appear in \cite{BERT}).

Taking $n=2$ in Definition \ref{def:webs}, our presentation for $\Web(\son[5])$ 
coincides with Kuperberg's $B_2$ webs. 
For example, in this case, relation \ref{eq:spinH=I} becomes
\[
\begin{tikzpicture}[scale=.35, rotate=90, tinynodes, anchorbase]
	\draw[very thick, gray] (-1,.5) node[below,yshift=2pt,xshift=2pt]{$\gS$} to (0,1);
	\draw[very thick, gray] (1,.5) node[above,yshift=-4pt,xshift=2pt]{$\gS$} to (0,1);
	\draw[very thick, gray] (0,2.5) to (-1,3) node[below,yshift=2pt,xshift=-2pt]{$\gS$};
	\draw[very thick, gray] (0,2.5) to (1,3) node[above,yshift=-4pt,xshift=-2pt]{$\gS$};
	\draw[very thick] (0,1) to node[below,yshift=2pt]{$1$} (0,2.5);
\end{tikzpicture}
=\frac{1}{[2]}
\begin{tikzpicture}[scale=.5, tinynodes, anchorbase]
	\draw[very thick,gray] (1,-1) to (1,1);
	\draw[very thick,gray] (0,-1) to (0,1);
\end{tikzpicture} \
+ 
\begin{tikzpicture}[scale=.3, smallnodes, anchorbase]
	\draw[very thick,gray] (-1,0) to (0,1);
	\draw[very thick,gray] (1,0) to (0,1);
	\draw[very thick,gray] (0,2.5) to (-1,3.5);
	\draw[very thick,gray] (0,2.5) to (1,3.5);
	\draw[very thick] (0,1) to node[right=-2pt]{$1$} (0,2.5);
\end{tikzpicture}
- \frac{1}{[2]}
\begin{tikzpicture}[scale=.5, tinynodes, anchorbase,gray]
	\draw[very thick] (0,0) to [out=90,in=180] (.5,.625) 
		to [out=0,in=90] (1,0);
	\draw[very thick] (0,2) to [out=270,in=180] (.5,1.375)
		to [out=0,in=270] (1,2);
\end{tikzpicture}
\]
which is the last relation in \cite[equation (3)]{Kup}, 
after rescaling Kuperberg's generating trivalent vertices 
(his equal $[2]^{\frac{1}{2}}$ times ours).
This relation inspired Kuperberg to introduce the quadrivalent vertex 
\[
\begin{tikzpicture}[scale=.3, tinynodes, anchorbase]
	\draw[very thick,gray] (-1,-1) to (1,1);
	\draw[very thick, gray] (-1,1) to (1,-1);
\end{tikzpicture}
:=
\begin{tikzpicture}[scale=.4, rotate=90, tinynodes, anchorbase]
	\draw[very thick, gray] (-1,.5) to (0,1);
	\draw[very thick, gray] (1,.5) to (0,1);
	\draw[very thick, gray] (0,2.5) to (-1,3);
	\draw[very thick, gray] (0,2.5) to (1,3);
	\draw[very thick] (0,1) to node[below,yshift=2pt]{$1$} (0,2.5);
\end{tikzpicture}
-\frac{1}{[2]}
\begin{tikzpicture}[scale=.5, tinynodes, anchorbase]
	\draw[very thick,gray] (1,-1) to (1,1);
	\draw[very thick,gray] (0,-1) to (0,1);
\end{tikzpicture}
\stackrel{(n=2)}{=}
\begin{tikzpicture}[scale=.3, smallnodes, anchorbase]
	\draw[very thick,gray] (-1,0) to (0,1);
	\draw[very thick,gray] (1,0) to (0,1);
	\draw[very thick,gray] (0,2.5) to (-1,3.5);
	\draw[very thick,gray] (0,2.5) to (1,3.5);
	\draw[very thick] (0,1) to node[right=-2pt]{$1$} (0,2.5);
\end{tikzpicture}
- \frac{1}{[2]}
\begin{tikzpicture}[scale=.5, tinynodes, anchorbase,gray]
	\draw[very thick] (0,0) to [out=90,in=180] (.5,.625) 
		to [out=0,in=90] (1,0);
	\draw[very thick] (0,2) to [out=270,in=180] (.5,1.375)
		to [out=0,in=270] (1,2);
\end{tikzpicture}
\]
which is the rotation-invariant building block for Kuperberg's $B_2$ non-elliptic web basis. 
Note that Kuperberg's vertex is exactly (the $n=2$ case of)
our morphism \eqref{eq:introX}.

For $n > 2$, the morphism $\xx$ from \eqref{eq:introX} is no longer fixed by rotation.
However, one can mimic \eqref{eq:introbXk} by defining
\[
\xx^{(0)} :=\begin{tikzpicture}[scale=.35, tinynodes,anchorbase]
	\draw[very thick,gray] (1,-1) to (1,1);
	\draw[very thick,gray] (0,-1) to (0,1);
\end{tikzpicture} \, , \quad
\xx^{(k+1)} := \dfrac{(-1)^k}{``[k+1]^2"}\left(\xx^{(k)} \xx - (-1)^{k} ``[k][k+1]"\xx^{(k)}\right)
\, \text{for } k \geq 0 \, .
\]
It follows from \eqref{eq:spinH=I} that 
$\xx^{(1)} = \xx =  
\sum_{i=1}^{n}(-1)^{\binom{i}{2}}\frac{``[i][i{+}1]"}{``{\textstyle {2i \brack i}}"}
\begin{tikzpicture}[scale=.175, tinynodes, anchorbase]
	\draw[very thick,gray] (-1,0) to (0,1);
	\draw[very thick,gray] (1,0) to (0,1);
	\draw[very thick,gray] (0,2.5) to (-1,3.5);
	\draw[very thick,gray] (0,2.5) to (1,3.5);
	\draw[very thick] (0,1) to node[right=-2pt]{$n{-}k$} (0,2.5);
\end{tikzpicture}$. 
More generally, 
it follows from \cite[Proposition 4.25]{BER} that 
\[\xx^{(k)} =  
\sum_{i=k}^{n} \frac{(-1)^{\binom{i-k+1}{2}}}{``{\textstyle {2i \brack i}}"} 
	\prod_{t=1}^{k}\frac{``[i+1-t][i+t]"}{``[t]^2"}
\begin{tikzpicture}[scale=.175, tinynodes, anchorbase]
	\draw[very thick,gray] (-1,0) to (0,1);
	\draw[very thick,gray] (1,0) to (0,1);
	\draw[very thick,gray] (0,2.5) to (-1,3.5);
	\draw[very thick,gray] (0,2.5) to (1,3.5);
	\draw[very thick] (0,1) to node[right=-2pt]{$n{-}i$} (0,2.5);
\end{tikzpicture} \, .
\]
A direct computation (also appearing in \cite[Proposition 4.25]{BER})
shows that the (seemingly complicated) coefficient on 
$\begin{tikzpicture}[scale=.175, tinynodes, anchorbase]
	\draw[very thick,gray] (-1,0) to (0,1);
	\draw[very thick,gray] (1,0) to (0,1);
	\draw[very thick,gray] (0,2.5) to (-1,3.5);
	\draw[very thick,gray] (0,2.5) to (1,3.5);
	\draw[very thick] (0,1) to node[right=-2pt]{$n{-}k$} (0,2.5);
\end{tikzpicture}$ in $\xx^{(k)}$ is in fact equal to $1$. 
This gives that $\xx^{(n)} = \begin{tikzpicture}[scale=.275, tinynodes, anchorbase,gray]
	\draw[very thick] (0,0) to [out=90,in=180] (.5,.625) 
		to [out=0,in=90] (1,0);
	\draw[very thick] (0,2) to [out=270,in=180] (.5,1.375)
		to [out=0,in=270] (1,2);
\end{tikzpicture}$ and a little more work shows that 
$\xx^{(n-1)} = \begin{tikzpicture}[scale=.175, tinynodes, anchorbase]
	\draw[very thick,gray] (-1,0) to (0,1);
	\draw[very thick,gray] (1,0) to (0,1);
	\draw[very thick,gray] (0,2.5) to (-1,3.5);
	\draw[very thick,gray] (0,2.5) to (1,3.5);
	\draw[very thick] (0,1) to node[right=-2pt]{$1$} (0,2.5);
\end{tikzpicture} 
- \frac{1}{[2]}
\begin{tikzpicture}[scale=.275, tinynodes, anchorbase,gray]
	\draw[very thick] (0,0) to [out=90,in=180] (.5,.625) 
		to [out=0,in=90] (1,0);
	\draw[very thick] (0,2) to [out=270,in=180] (.5,1.375)
		to [out=0,in=270] (1,2);
\end{tikzpicture}$. 
Noting that these are the $90^{\circ}$ rotations of $\xx^{(0)}$ and $\xx^{(1)}$ 
motivates the following (and proves it for $n \leq 3$).

\begin{conj}[{c.f.~\cite[Conjecture 4.44]{BER}}]\label{conj:rotationwebs}
In $\Web(\son)$, 
the $90^{\circ}$ rotation of $\xx^{(k)}$ is $\xx^{(n-k)}$. 
\end{conj}

Given this, we anticipate that the ``webs'' $\xx^{(k)}$ will be the building blocks 
of rotation-invariant non-elliptic basis for morphism spaces between arbitrary tensor powers 
of $\gS$. 
More generally, this should yield non-elliptic bases in all morphism spaces in $\Web(\son)$, 
since each of the remaining fundamental representations is a summand of the tensor square of 
the spin representation.

\begin{rem}
Conjecture \ref{conj:rotationwebs} is also predicted from the equivariant categorification 
perspective of \S \ref{ss:categorification}. 
Indeed, recall that the elements $\xx^{(k)}$ are categorified by 
equipping the $1$-morphisms $\mathbf{X}^{(k)}$ with equivariant structures, 
and that the latter (without these structures) categorify the webs $\chi^{(k)}$ from \eqref{eq:FEwebs}.
Rotating the web $\chi^{(k)}$ by $90^{\circ}$ yields the web
\begin{equation}\label{eq:chi90}
\begin{tikzpicture}[anchorbase,tinynodes,scale=.625,rotate=90]
	\draw[very thick,->] (0,.25) to [out=150,in=0] (-.75,.5) node[below=-2pt]{$n$};
	\draw[very thick,->] (.5,.5) to [out=90,in=180] (1,.75) node[above=-2pt]{$n$};
	\draw[very thick,directed=.65] (0,.25) to node[left=-2pt]{$k$} (.5,.5);
	\draw[very thick,directed=.65] (0,-.25) to (0,.25);
	\draw[very thick,directed=.65] (.5,-.5) to [out=30,in=330] (.5,.5);
	\draw[very thick,directed=.65] (.5,-.5) to node[right=-2pt,yshift=-1pt]{$k$} (0,-.25);
	\draw[very thick,directed=.45] (1,-.75) node[above=-2pt]{$n$} to [out=180,in=270] (.5,-.5);
	\draw[very thick,directed=.5] (-.75,-.5)node[below=-2pt]{$n$} to [out=0,in=210] (0,-.25);
	\end{tikzpicture} 
\in \Hom_{\Web(\sln[2n])}(n^\vee \otimes n, n \otimes n^\vee)\, .
\end{equation}
Using the isomorphisms 
$\Lambda^a(\C^{2n})^\vee \cong \Lambda^{2n-a}(\C^{2n})$ in $\Rep(U_q(\sln[2n]))$, 
one can show that \eqref{eq:chi90} corresponds to the web
\[
\begin{tikzpicture}[anchorbase,tinynodes,scale=.625,xscale=-1.5]
	\draw[very thick,->] (0,.25) to [out=150,in=270] (-.25,1) node[above=-3pt]{$n$};
	\draw[very thick,->] (.5,.5) to (.5,1) node[above=-3pt]{$n$};
	\draw[very thick,directed=.65] (0,.25) to node[above=-2pt,xshift=2pt]{$n{-}k$} (.5,.5);
	\draw[very thick,directed=.65] (0,-.25) to (0,.25);
	\draw[very thick,directed=.65] (.5,-.5) to [out=30,in=330] (.5,.5);
	\draw[very thick,directed=.65] (.5,-.5) to node[below=-1pt,xshift=2pt]{$n{-}k$} (0,-.25);
	\draw[very thick,directed=.65] (.5,-1) node[below=-2pt]{$n$} to (.5,-.5);
	\draw[very thick,directed=.5] (-.25,-1)node[below=-2pt]{$n$} to [out=90,in=210] (0,-.25);
	\end{tikzpicture}
	= \chi^{(n-k)} \in \Web(\sln[2n]) \, .
	\]
	\end{rem}

\begin{rem}\label{rem:topsummand}
Although every irreducible representation in $\Rep(U_q(\son))$ is a direct summand 
of some tensor power of the spin representation, 
one benefit of using all fundamental representations in our presentation for $\Web(\son)$ is that 
every irreducible representation is a multiplicity one summand 
of a particular tensor product of fundamental representations.  
This structure leads to the existence of object-adapted cellular bases, 
i.e.~light ladder bases, for morphism spaces; 
see \cite{EliasLL} and \cite{Bodish} 
for the construction of such bases using type $A$ webs and type $B_2$ webs, respectively.

While the rotation-invariant bases discussed thus far are tailored for applications 
in combinatorics and skein theory, 
light ladder bases are designed to have representation theoretic applications,
e.g.~modular representation theory \cite{EliasLL, Bodish}, 
and clasp formulae \cite{EliasLL, MR4413286, BodishWu2026}. 
\end{rem}

\subsection{Dualities}

A key technical tool for studying webs in type $A$ is quantum skew Howe duality,  
as pioneered by Cautis--Kamnitzer--Morrison (CKM) \cite{CKM}. 
They consider commuting actions of $U_q(\sln[N])$ and $U_q(\glm)$ on 
$\Lambda_q^*(\C^N\otimes \C^m)$, which yield a surjective homomorphism
\begin{equation}\label{CKMmap}
\dot{U}_q^{\le N}(\glm) \longrightarrow 
\bigoplus_{\sum k_i = \sum \ell_i} 
\Hom_{U_q(\sln[N])} \big(\Lambda_q^{k_1}\C^N \otimes \cdots \otimes \Lambda_q^{k_m}\C^N
	, \Lambda_q^{\ell_1} \C^N \otimes \cdots \otimes \Lambda_q^{\ell_m}\C^N \big) \, .
\end{equation}
Here, $\dot{U}_q^{\le N}(\glm)$ is the quotient of Lusztig's $\dot{U}_q(\glm)$, 
an algebra with orthogonal idempotents $1_{\wt}$ for $\wt\in \mathbb{Z}^m$, 
by the ideal generated by all $1_{\wt}$ which act by zero on $\Lambda_q^*(\C^N\otimes \C^m)$. 
Using known results on the structure theory and representation theory of $\dot{U}_q(\glm)$, 
CKM deduce that the map \ref{CKMmap} is injective and thus is an isomorphism.

The map \eqref{CKMmap} sends Chevalley generators in $U_q(\glm)$ to $U_q(\sln[N])$-module morphisms 
described diagrammatically by so-called ``ladder webs.''
As an example, taking $m=2$ and $N=2n$, the elements 
$F^{(k)}E^{(k)}1_{(n,n)} = E^{(k)}F^{(k)}1_{(n,n)}$ are mapped to the webs
$\chi^{(k)}$ from \eqref{eq:FEwebs}.
The defining relations of $\Web(\sln[N])$ are then obtained by 
writing the relations of $\dot{U}_q^{\le N}(\glm)$ in diagrammatic form.

This incarnation of skew Howe duality was categorified in \cite{QR1} to 
give the bicategories $N\mathbf{Foam_+}$ appearing in \S \ref{ss:categorification}.
Recall that it is the $\Z/2$-action on the sub-bicategories 
$\mathbf{B}^n_m \subset 2n\mathbf{Foam_+}$ appearing in Conjecture \ref{conj:SLH}
that gives rise to the relations in $\Web(\son)$.
It is therefore unsurprising that aspects of our proof of Theorem \ref{thm:main} 
share features with the CKM approach, as we now describe.

Playing the role of the right-hand side of \eqref{CKMmap} is the algebra
$\End_{U_q(\son)}(S^{\otimes m})$, which appears in Conjecture \ref{conj:SLH}.
This algebra was previously studied in \cite{Wenzl-Spin}, 
wherein Wenzl establishes a duality between $U_q(\son)$ 
and Gavrilik--Klimyk's \cite{GK-Usom} nonstandard quantum group $\iqUsom$.
The latter is a $\C(q)$-algebra generated by elements $b_1, \dots, b_{m-1}$
(and plays a role in the present story analogously to that of $U_q(\glm)$ in CKM). 
Wenzl defines endomorphisms of $S^{\otimes m}$ giving the action of each $b_i$, 
then proves that this gives a surjective homomorphism
\begin{equation}\label{Wenzlmap} 
\iqUsom^{\le n} \longrightarrow \End_{U_q(\son)}(S^{\otimes m}) \, ,
\end{equation}
where $\iqUsom^{\le n}$ is the quotient of $\iqUsom$ by the ideal generated 
by the minimal polynomials of the $b_i$'s action on $S^{\otimes m}$. 
Most likely, it is a folk conjecture among experts in this area that \eqref{Wenzlmap} is an isomorphism.
However, although the representation theory of $\iqUsom$ is well-studied \cite{GK-Usom, WenzlUprime}, 
a proof that \eqref{Wenzlmap}  is injective has not previously appeared in the literature.

In \cite{BER}, we studied Wenzl's homomorphism diagrammatically,
and observed that \eqref{Wenzlmap} acts on the generators of $\iqUsom$ as 
$b_i\mapsto \varphi\Big( \,
\begin{tikzpicture}[scale=.3, tinynodes, anchorbase]
	\draw[thick, black] (1,0) to node[below=-1pt]{$1$} (2,0);
	\draw[very thick,gray] (0,-1) to (0,1);
	\node at (.5,0) {$\mysdots$};
	\draw[very thick,gray] (1,-1) to (1,1);
	\draw[very thick,gray] (2,-1) to (2,1);
	\node at (2.5,0) {$\mysdots$};
	\draw[very thick,gray] (3,-1) to (3,1);	
\end{tikzpicture} \,
\Big) \in \End_{U_q(\son)}(S^{\otimes m})$; 
here $\varphi$ is the functor $\Web(\son)\rightarrow \FRep(U_q(\son))$ from Theorem \ref{thm:main}. 
Writing the defining ``$\iota$Serre relation'' of $\iqUsom$ 
(see \eqref{eqn:iSerre-alg-version} below) in diagrammatic form yields 
the following:
\begin{equation}\label{eq:iSW}
\begin{tikzpicture}[scale=.4, anchorbase]
	\draw[very thick] (0,1.5) to (1,1.5);
	\draw[very thick] (0,.75) to (1,.75);
	\draw[very thick] (1,2.25) to (2,2.25);
	\draw[very thick, gray] (0,0) to (0,3);
	\draw[very thick, gray] (1,0) to (1,3);
	\draw[very thick, gray] (2,0) to (2,3);
\end{tikzpicture}
+[2]_{q^2}
\begin{tikzpicture}[scale=.4, anchorbase]
	\draw[very thick, gray] (0,0) to (0,3);
	\draw[very thick, gray] (1,0) to (1,3);
	\draw[very thick, gray] (2,0) to (2,3);
	\draw[very thick] (0,.75) to (1,.75);
	\draw[very thick] (1,1.5) to (2,1.5);
	\draw[very thick] (0,2.25) to (1,2.25);
\end{tikzpicture}
+ 
\begin{tikzpicture}[scale=.4, anchorbase]
	\draw[very thick] (0,1.5) to (1,1.5);
	\draw[very thick] (0,2.25) to (1,2.25);
	\draw[very thick] (1,.75) to (2,.75);
	\draw[very thick, gray] (0,0) to (0,3);
	\draw[very thick, gray] (1,0) to (1,3);
	\draw[very thick, gray] (2,0) to (2,3);
\end{tikzpicture}
= 
\begin{tikzpicture}[scale=.4, anchorbase]
	\draw[very thick, gray] (0,0) to (0,3);
	\draw[very thick, gray] (1,0) to (1,3);
	\draw[very thick, gray] (2,0) to (2,3);
	\draw[very thick] (1,1.5) to (2,1.5);
\end{tikzpicture} \, .
\end{equation}
In Proposition \ref{prop:iQtoLad}, 
we prove this relation is a consequence of the defining relations of $\Web(\son)$.

The diagrammatic perspective also sheds light on the algebra $\iqUsom^{\le n}$.
Lifting \eqref{eq:introbXk} and \eqref{eq:introX} to $\iqUsom$, 
we define the \emph{nonclassical $\iota$divided power elements}
\begin{equation}\label{eq:ddp}
\xxb_i:=b_i-\frac{1}{[2]} \, , \quad
\xxb_i^{(k+1)}=\frac{(-1)^k}{``[k+1]^2"}(\xxb_i^{(k)}\xxb_i- (-1)^k``[k][k+1]"\xxb_i^{(k)}) \, .
\end{equation}
Setting $(``[k]^2")! := (``[k]^2")(``[k-1]^2")\cdots(``[2]^2")(``[1]^2")$,
this gives the formulae
\[
\xxb_i^{(k)}:= (-1)^{\binom{k}{2}}\dfrac{(\xxb_i+(-1)^k``[k-1][k]")\cdot (\xxb_i+(-1)^{k-1}``[k-2][k-1]")\cdots (\xxb_i+(-1)``[0][1]")}{(``[k]^2")!} \, .
\]
We show that, up to scalar, the minimal polynomial of 
$b_i$ in $\End_{U_q(\son)}(S^{\otimes m})$ is equal to $\xxb_i^{(n+1)}$.
Using this, and the braid group symmetries discussed in the following section, 
we prove that $\iqUsom^{\le n}$ is finite-dimensional in Proposition \ref{prop:finite-dim}. 
After establishing this, results from the literature on the representation theory of $\iqUsom$ 
imply that \eqref{Wenzlmap} is injective. We hence establish the aforementioned folk conjecture.
\begin{thm}[{Corollary \ref{cor:inj}}]
Wenzl's map $\iqUsom^{\le n} \longrightarrow \End_{U_q(\son)}(S^{\otimes m})$ from \eqref{Wenzlmap}
is an isomorphism of unital associative algebras. \qed
	\end{thm}
This is the representation-theoretic tool used to show 
that our functor in Theorem \ref{thm:main} is fully faithful.

\subsection{Relative braid group symmetries on nonclassical representations}

For each $\mathfrak{g}$, the monoidal category $\Rep(U_q(\mathfrak{g}))$ is braided, 
via the universal $R$-matrix. 
In type $A$, CKM show that the braiding isomorphism in 
$\Hom_{U_q(\sln[N])}(\Lambda_q^k\C^N \otimes \Lambda_q^{\ell}\C^N
	, \Lambda_q^{\ell}\C^N \otimes \Lambda_q^k \C^N)$
is given as the image under \eqref{CKMmap} of the quantum Weyl group generator 
in $\dot{U}_q(\gln[2])$ \cite[\S 6]{CKM}. 
Inspired by this result, and the folding paradigm from \S \ref{ss:categorification}, 
we showed in \cite[Theorem 1.26]{BER} that the elements 
\begin{equation}\label{eq:iQW}
\iQW_i:= \sum_{k \ge 0} q^{-k}\xxb_i^{(k)} \in \iqUsom^{\le n} 
\end{equation}
are sent by \eqref{Wenzlmap} to a scalar multiple of the braiding isomorphism 
acting on the $i$-th and $(i+1)$-st tensor factors of $S^{\otimes m}$.
Having now established that \eqref{Wenzlmap} is an isomorphism, 
this shows that the elements $\iQW_i$ satisfy the braid relations in the algebra $\iqUsom^{\le n}$ itself.
This has implications for $\iota$quantum group representation theory, 
as we now describe.

Gavrilik--Klimyk \cite{GK-Usom} originally defined the
nonstandard quantization $\iqUsom$ to study $q$-analogues of Gelfand-Zetlin bases 
of representations of $\mathfrak{so}_m$. 
However, Noumi \cite{Noumi} later discovered that $\iqUsom$ is isomorphic to the subalgebra of $U_{-q^2}(\slm)$ 
generated by the elements $f_i+ (-q^2)^{-1}e_ik_i^{-1}$, for $i=1, \dots, m-1$. 
Thus $\iqUsom$ is an $\iota$quantum group, a coideal subalgebra $\iqUsom \subset U_{-q^2}(\slm)$
which gives a quantum symmetric pair.

To provide context for the relative braid group symmetries \eqref{eq:iQW}, 
we consider the symmetric pair $\mathrm{SO}_m\subset \mathrm{SL}_m$. 
The Weyl group $W= N_{\mathrm{SL}_m}(T)/T$ has simple reflection generators $s_1, \ldots , s_{m-1}$ 
which lift to elements $\tilde{s_i}\in N_{\mathrm{SL}_m}(T)$, e.g.~$\tilde{s_1}=
\begin{pmatrix}
0 & -1 \\
1 & 0 \\
\end{pmatrix}\in N_{\mathrm{SL}_2}(T)$. 
Let $\tilde{W}$ denote the subgroup of $\mathrm{SL}_m$ 
generated by the elements $\tilde{s}_1, \dots, \tilde{s}_{m-1}$,
and let $\Br_m$ denote the $m$-strand braid group, 
which has generators $\beta_1, \dots, \beta_{m-1}$. 
The assignment $\beta_i\mapsto \tilde{s}_i$ gives a group homomorphism $\Br_m\rightarrow \tilde{W}$, 
thus there is an action of $\Br_m$ on any $\mathrm{SL}_m$ representation.
The $q$-analogue of this braid group action is realized by Lusztig's quantum Weyl group operators, 
which are defined via a $q$-analogue of the triple exponent formula $\tilde{s}_i=\exp(f_i)\exp(-e_i)\exp(f_i)$.

Since $\tilde{W}\subset \mathrm{SO}_m:=\{g\in \mathrm{SL}_m \mid gg^t=1\}$, 
there is also an action of $\Br_m$ on any $\mathrm{SO}_m$ representation.
The Lie algebra $\som$ is generated by the elements $f_i-e_i$, $i=1, \dots, m-1$, 
and using that $\tilde{s_i}=\exp(\frac{\pi}{2}(f_i-e_i))$ 
gives explicit formulae for the action of $\tilde{s_i}$ on any $\mathrm{SO}_m$ representation 
in terms of the (classical) $\iota$divided powers \cite[Definition 3.8]{MR4983792}. 
The $q$-analogues of these formulae, which are given in \cite[Section 3.1]{wang2026relativebraidgroupsymmetries}, 
describe relative braid group symmetries on the integrable $U_q^{\iota}(\som)$-modules.

These formulae are the key ingredient in the solution to \cite[Problem 1.1]{wang2026relativebraidgroupsymmetries}, 
which asks for an explicit description of the action of the relative braid group on integrable modules, satisfying some desired properties. 
Before the discovery of these explicit formulae, an implicit formula for the action was established in \cite{WangZhangintrinsic}, 
by combining Lusztig's braid group operators for $U_{q}(\slm)$-modules with the quasi $K$-matrix for $U_q^{\iota}(\som)\subset U_{q}(\slm)$. 
In the explicit approach, 
it is essential that for any integrable $U_q^{\iota}(\som)$-module, 
the spectrum of the action of the generators $b_i$ consists only of quantum integers, 
i.e.~the classical $\iota$divided powers are eventually zero.
In the implicit approach, 
it is essential that any finite-dimensional irreducible integrable $U_q^{\iota}(\som)$-module 
is isomorphic to a summand of a restriction of a type I $U_{q}(\slm)$-module. 

The representations of $\iqUsom$ which are summands of $S^{\otimes m}$ 
are (type I) \emph{nonclassical}\footnote{While the notion of nonclassical representations is in the literature, 
the description of some of them as being ``type I'' is not in the literature.}, see Definition \ref{def:nonclassical-typeI}.
In particular, they are \emph{not} integrable and do not appear in restrictions of representations of $U_{-q^2}(\slm)$.
As such, there did not exist explicit formulae for an action of the $m$-strand braid group $\Br_m$ on the 
nonclassical representations of $\iqUsom$; 
nor were there implicit formulae for an action in terms of Lusztig's symmetries and the quasi $K$-matrix 
(although implicit formulae can be deduced from Wenzl's results). 

As a consequence of our isomorphism $\iqUsom^{\le n}\rightarrow \End_{U_q(\son)}(S^{\otimes m})$,
established in Corollary \ref{cor:inj}, 
along with the relation between $\iQW_i$ and the braiding of the $U_{q}(\son)$ spin representations
from \cite[Theorem 1.26]{BER}, we prove the following theorem. 

\begin{thm}\label{thm:mainthm2}
Let $M$ be a finite-dimensional (type I) nonclassical representation of $\iqUsom$. 
For $m \in M$, the formulae $\beta_i(m):=\sum_{k \ge 0} q^{-k}\xxb^{(k)}_i\cdot m$
determine an action of $\Br_m$ on $M$.
\end{thm}

\begin{rem}
We prove the injectivity of $\iqUsom^{\le n}\rightarrow \End_{U_q(\son)}(S^{\otimes m})$ 
by first deducing that $\iqUsom^{\le n}$ is finite-dimensional 
in Proposition \ref{prop:finite-dim}.
This uses Proposition \ref{prop:longb-braid-conjugates}, which relates
the elements $\iQW_i$ to the PBW basis of $\iqUsom$.
Hence, the elements \eqref{eq:iQW} are crucial to our study.
\end{rem}

We view Theorem \ref{thm:mainthm2} as evidence for nonclassical generalizations 
of \cite[Problem 1.1]{wang2026relativebraidgroupsymmetries}, 
which we state as follows.

\begin{problem}
For non-integrable representations of split $\iota$quantum groups, find: 
quasi $K$-matrices, modified $\iota$quantum groups, integral $\mathbb{Z}[q^{\pm 1}]$-forms, 
and both implicit (using quasi $K$-matrix and Lusztig symmetries) and explicit (using rank $1$ formulae) 
descriptions of relative braid group symmetries that preserve the integral $\mathbb{Z}[q^{\pm 1}]$-forms.
\end{problem}

Wang's ICM address \cite{WangICM} presents a program to generalize the theory of quantum groups $U_q:=U_q(\mathfrak{g})$ 
to $\iota$quantum groups $U_q^{\iota}:=U_q^{\iota}(\mathfrak{k}\subset \mathfrak{g})$. 
The centerpiece of this program is Bao--Wang's $\iota$canonical basis \cite{BW-icanonical} 
for the $U_q^{\iota}$-modules that are restrictions of based $U_q$-modules (and for a related modified form of $U_q^{\iota}$).
The $\iota$canonical basis is defined using an $\iota$bar involution, 
which is constructed by combining the usual bar involution on the based $U_q$-module 
with the quasi $K$-matrix $\mathfrak{X}$ 
(denoted $\Upsilon$ and called the ``intertwiner'' in \cite{MR3864017, BW-icanonical}), 
which lies in a completion of $U_q(\mathfrak{g})$. 
The $\iota$canonical basis elements for the modified form of $U_q^{\iota}\subset U_q(\mathfrak{sl}_2)$ are the $\iota$divided powers \cite{MR3783013}. 

Bao--Wang's $\iota$canonical basis is adapted to \emph{classical} representations\footnote{More precisely, 
integrable representations \cite[Section 2.5]{wang2026relativebraidgroupsymmetries}, 
which are classical but do not exhaust all classical representations.} 
of $U_q^\iota$, 
but we believe that Wang's program extends to other representations, including the nonclassical representations appearing above.
Let $\mathfrak{g}$ be a simply laced Lie algebra. 
Consider the associated Drinfeld--Jimbo quantum group $U_{-q^2}$ 
and let $U_{-q^2}^{\iota}$ be the split $\iota$quantum group with generators $b_i = f_i+(-q^2)^{-1}e_ik_i^{-1}$. 
A (type I) nonclassical representation of $U_{-q^2}^{\iota}$ is a representation in which the $b_i$ are diagonalizable 
with eigenvalues of the form $\{(-1)^k\frac{[2k+1]}{[2]}\}_{k\in \mathbb{Z}_{\ge 0}}$.
Equivalently, a (type I) nonclassical representation is one in which the $\xxb_i^{(k)}$,
which again are defined via \eqref{eq:ddp}, act by zero for $k\gg0$. 
As a consequence, the formula in \eqref{eq:iQW} gives rise to well-defined operators $\iQW_i$ 
on these representations.

There is a one-dimensional (type I) nonclassical representation $\nctriv=\C(q) \{1_{\nctriv}\}$, 
where the generators of $U_{-q^2}^{\iota}$ act as $b_i\cdot 1_{\nctriv} = \frac{1}{[2]}1_{\nctriv}$. 
Since the coproduct of $U_{-q^2}$ is such that $\Delta(b_i) = b_i\otimes 1 + k_i^{-1}\otimes b_i$, 
if $V$ is a representation of $U_{-q^2}$, 
then $V\otimes \nctriv$ can be treated as a representation of $U_{-q^2}^{\iota}$.

\begin{prop}
Let $V$ be a finite-dimensional (type I) $U_{-q^2}$-module. 
The $U_{-q^2}^{\iota}$-module $V\otimes \nctriv$ is (type I) nonclassical.
As operators on $V\otimes \nctriv$, 
the nonclassical $\iota$quantum Weyl group generators $\Delta(\iQW_i)$ are invertible 
and satisfy the braid relations of $\Br_{\mathfrak{g}}$.
\end{prop}

\begin{proof}[Proof (sketch)]
Note that $V$ restricts to a completely reducible representation of any parabolic subalgebra of $U_{-q^2}$. 
This reduces the first claim to $\mathfrak{g}=\mathfrak{sl}_2$, 
where we leave it as an exercise to check that the spectrum of $b$ 
acting on $V\otimes \mathcal{N}$ is contained in $\{(-1)^k\frac{[2k+1]}{[2]}\}_{k\in \mathbb{Z}_{\ge 0}}$. 
It also reduces the second claim to a nontrivial rank $2$ calculation which, 
since $\mathfrak{g}$ is simply laced, follows from Theorem \ref{thm:mainthm2}.
\end{proof}

The quasi $K$-matrix $\mathfrak{X}$ gives rise to a $2$-tensor quasi $K$-matrix $\Theta^{\iota}$ 
which lies in a completion of $U_{-q^2}\otimes U_{-q^2}^{\iota}$ \cite{BalaKolb, kolb2026shortstarproductsquantum} and therefore acts on $V\otimes \nctriv$.

\begin{conj}
Let $(V, \mathbb{B}_V)$ be a based $U_{-q^2}$-module \cite[Chapter 27]{Lus4}, 
e.g.~$V$ is a finite-dimensional irreducible (type I) representation of $U_{-q^2}$ 
and $\mathbb{B}_V$ is Lusztig's canonical basis of $V$.
\begin{itemize}
\item As operators on $V\otimes \nctriv$, we have 
\[
\Theta^{\iota} = (T_{w_0}^{-1}\otimes \iQW_{w_0}^{-1})\circ \Delta(\iQW_{w_0}), 
\]
where $T_{w_0}$ and $\iQW_{w_0}$ are the longest elements of the quantum Weyl group for $U_{-q^2}$ 
and the $\iota$quantum Weyl group of $U_{-q^2}^{\iota}$ respectively. 
(Note that $\iQW_{w_0}\cdot 1_{\nctriv} = 1_{\nctriv}$.)
\item The $2$-tensor quasi $K$-matrix $\Theta^{\iota}$ 
gives rise to an $\iota$bar involution on $V\otimes \nctriv$. 
\item There is a unique $\iota$canonical basis $\mathbb{B}_{V\otimes \nctriv}^{\iota}$ of $V\otimes \nctriv$ 
which is $\iota$bar invariant and unitriangular with respect to $\mathbb{B}_V\otimes 1_{\nctriv}$. 
\item The $\iota$canonical basis $\mathbb{B}_{V\otimes \nctriv}^{\iota}$ respects the decomposition of $V\otimes \nctriv$ into isotypic components.
\item There is a nonclassical modified form of $U_{-q^2}^{\iota}$ 
which admits a nonclassical $\iota$canonical basis compatible with $\mathbb{B}_{V\otimes \nctriv}$ for each $V\otimes \nctriv$. 
\item In rank $1$, the nonclassical $\iota$canonical basis consists of the nonclassical $\iota$divided powers $\xxb^{(k)}$.
\end{itemize}
\end{conj}

\subsection{Further relations}
	\label{ss:furtherrelationsandapp}

While studying the spin-colored Reshetikhin--Turaev link polynomial in \cite{BER}, 
we found various formulae satisfied by the nonclassical $\iota$divided powers \eqref{eq:ddp}. 
In particular, our \cite[Conjecture 4.40]{BER} generalizes the aforementioned 
$\iota$Serre relation of $\iqUsom$,
and can be restated in terms of $\Web(\son)$ as follows. 
\begin{conj}\label{conj:divpowerrelns}
For $1\le a,b,c\le n$, there is a relation of the form
\[
\xx_i^{(a)}\xx_{i\pm 1}^{(b)}\xx_i^{(c)} = \xi\cdot \xx_{i\pm 1}^{(a')}\xx_i^{(b')}\xx_{i\pm 1}^{(c')} 
	+ \mathrm{LOT}_{a,b.c}
\]
where $\xi\in \C(q)$, $1\le a',b',c'\le n$, 
and $\mathrm{LOT}_{a,b,c}$ is a linear combination of terms of the form 
$\xx_i^{(k)}\xx_{i\pm 1}^{(\ell)}$ and $\xx_{i \pm 1}^{(\ell)}\xx_i^{(k)}$.
\end{conj}

We establish this conjecture for $n=1,2,3$ in \cite[Proposition 4.41]{BER}, 
and the relations appear to be independent of $n$.

\begin{example}
We have
\begin{equation}\label{eq:iserre-divpower-webs}
\xx_i\xx_{i\pm 1}\xx_i = \xx_i^{(2)}\xx_{i\pm 1} + \xx_{i\pm 1}\xx_i^{(2)} + [2]\xx_i^{(2)} + \xx_i
\end{equation}
and
\begin{equation}
\xx_i\xx_{i\pm 1}^{(2)}\xx_i = \xx_{i\pm 1}\xx_i^{(2)}\xx_{i\pm 1} \, .
\end{equation}
The simplest way to derive these relation is to prove that they in fact already hold in $\iqUsom$. 
From this point of view, 
\eqref{eq:iserre-divpower-webs} is exactly the defining $\iota$Serre relation of $\iqUsom$, 
but rewritten in terms of the nonclassical $\iota$divided powers. 
Our conjecture may then be a nonclassical generalization of \cite{MR4227167}.
\end{example}

\begin{rem}
As outlined in \cite[Theorem 4.45]{BER},
Conjecture \ref{conj:divpowerrelns} is the crucial step in a characterization of 
the spin-colored Reshetikhin--Turaev polynomial in terms of traces on $\iqUsom^{\le n}$. 
This is analogous to the results in \cite[Sections 5 and 6]{Jones2} that characterize the HOMFLYPT 
and $\sln[N]$ link polynomials in terms of the Jones--Ocneanu trace 
on the type $A$ Hecke algebras.
In type $B$, this characterization is desirable, as it is a proposed step in showing that 
the spin link homology constructed in \cite{BER}
indeed categorifies spin-colored $\son$ Reshetikhin--Turaev link polynomials. 
So far, that result is only established for $n\le 3$. 
\end{rem}

A different, but related, direction is the consideration of integral forms 
for the category $\Web(\son)$.
In \S \ref{ss:morerels}, we establish additional useful relations
that follow from those in Definition \ref{def:webs}
when working over $\C(q)$.
However, we suspect that some of these relations should be 
adjoined to the presentation \eqref{eq:webrel}
when working with the ``correct'' integral form of $\Web(\son)$.
By the latter, we mean a category that describes tilting modules for 
quantum $\son$, after all specializations. 
In this paper, we have focused on providing an efficient presentation over $\C(q)$, 
rather than on solving the (important) follow-up problem of finding this integral form.

Nevertheless, let $\mathcal{A}:=\Z[q^{\pm 1},``{2n \brack n}"^{-1}] 
= \mathbb{Z}[q^{\pm 1}, [2]^{-1}, [2]_{q^3}^{-1}, \dots, [2]_{q^{2n-1}}^{-1}]$,
and note that this ring contains all of the denominators 
appearing in our defining relations \eqref{eq:webrel} for $\Web(\son)$. 
Working over $\mathcal{A}$, 
we expect that it is still possible to adapt arguments from this paper to prove there is a functor 
$\Web_{\mathcal{A}}(\son)\rightarrow \FRep_{\mathcal{A}}(\son)$. 
Upon establishing the existence of this functor, 
we have the following result, 
which we learned from arguments in \cite[Section 3.6.2]{JMWParityTilting}. 

\begin{lem}
Assume the existence of the functor $\Web_{\mathcal{A}}(\son)\rightarrow \FRep_{\mathcal{A}}(\son)$.
After specialization to a field, via $\mathcal{A}\rightarrow \mathbb{F}$, 
the Karoubi envelope of $\FRep_{\mathbb{F}}(\son)$ is equivalent to the category of tilting modules.
\end{lem}

\begin{proof}[Proof (sketch)]
Since $S=V_{\varpi_n}$ is minuscule, it is tilting (it is both a Weyl and a dual Weyl module), 
and since a tensor product of tilting modules is tilting, so is $S^{\otimes 2}$. 
By \eqref{eq:digonS}, and invertibility of $``{2(n-i) \brack n-i}"$ in $\mathbb{F}$, 
the fundamental representations $V_{\varpi_i}$, 
$i=1, \dots, n-1$, are irreducible direct summands of $S^{\otimes 2}$. 
Direct summands of tilting modules are tilting, 
so $\FRep_{\mathbb{F}}(\son)$ is contained in the category of tilting modules, thus so is its Karoubi envelope. 
By highest weight theory, this Karoubi envelope contains 
an indecomposable tilting module for each dominant weight, 
and thus it contains the entire category of tilting modules. \end{proof}

At the present moment, however,
we do not know if the generating morphisms for $\Web(\son)$ 
still generate all morphisms in $\FRep_{\mathcal{A}}(\son)$.
Further, we do not know if the relations defining $\Web(\son)$ 
describe all relations between the corresponding morphisms in $\FRep_{\mathcal{A}}(\son)$. 

\begin{rem}\label{rem:integralform}
To aid future research, we draw attention to instances when 
we divide by scalars other than units in $\mathcal{A}$. 
In \eqref{dividingbyk}, we divide by $``[k]^2"$ to derive the equation \eqref{eq:zerodigonB}.
Thus, it is possible that \eqref{eq:zerodigonB} is not implied by our current relations 
for specializations where  $``[k]^2"=0$. 
If so, equation \eqref{eq:zerodigonB} should be added to the integral version of our presentation. 
Another example is in \eqref{eq:flowvertex} below, 
where we define new trivalent (``flow'') vertices via formulae that involve dividing by $``[k]^2"$. 
Thus, when working integrally, the flow vertices should be generators in their own right, 
and we expect general flow vertex versions of all the defining relations in \eqref{eq:webrel} 
in this presentation.
\end{rem}

\begin{rem}
More generally,
tilting modules are defined for any specialization $\mathbb{Z}[q^{\pm 1}]\rightarrow \mathbb{F}$. 
However, we have observed that new web generators-and-relations 
are required for the ``correct" $\mathbb{Z}[q^{\pm 1}]$ version of $\Web(\son)$.
This presentation will most likely include nonclassical $\iota$divided power generators, 
and the relations from Conjecture \ref{conj:divpowerrelns}.
\end{rem}

\subsection{Related work}

Various categories related to $\Web(\son)$ have previously appeared in the literature. 
We now comment on some of these categories.

\subsubsection{Ranks $1$ and $2$}

As already noted in \S \ref{ss:webbasesintro}, 
our Definition \ref{def:webs} recovers existing descriptions of $\FRep(U_q(\son))$ when $n=1,2$.
To briefly recall, 
when $n=1$ it gives a pivotal category with one generating objects $\gS$, 
no generating morphisms (besides identity and cap/cup morphisms), 
and one relation \eqref{eq:circleS}. 
Theorem \ref{thm:main} thus states the equivalence between 
$\FRep(U_q(\son[3]))$ and the Temperley--Lieb category $\mathbf{TL}$ (at circle value $-[2]$).

Similarly, when $n=2$, Definition \ref{def:webs} recovers Kuperberg's $\son[5]$ web category. 
For this, one should use the substitution that Kuperberg's generating trivalent vertex is 
$[2]^{\frac{1}{2}}$ ours, and his $q$ is equal to our $q^2$. 
After doing so, his relations exactly agree with ours 
(note: there are no nontrivial ``all black'' relations when $n=2$), 
except that he includes a relation for the $1$-labeled circle. 
See Proposition \ref{prop:easyrel} which derives this relation from the others.
Thus, in this case, Theorem \ref{thm:main} recovers Kuperberg's (type $B_2$) 
equivalence from \cite{Kup}.

Hence, in both of these cases, Theorem \ref{thm:main} is simply a restatement of known results.
Given this, to help streamline aspects of our proof, 
we will at times assume that $n \geq 3$.
Saliently, for these values of $n$, 
we have access to both $1$- and $2$-labeled web edges in $\Web(\son)$. 
Our ``ladderization" result (Theorem \ref{thm:ladder}) used in our
proof of fully faithfulness utilizes the braiding morphism 
$c_{1,1} \colon 1 \otimes 1 \to 1 \otimes 1$.
Hence, we must argue differently at this step when $n=1$, 
since we don't have access to $1$-labeled edges in this case. 
Further, $c_{1,1}$ has a simpler description when written in terms of $2$-labeled edges, 
so certain assertions concerning this braiding morphism 
are justified by different computations when $n=2$.
We will comment more about the workarounds, 
which would allow our proof to carry over to the $n=1,2$ cases, 
at the relevant parts of the paper.
See Remarks \ref{rem:n=2} and \ref{rem:n=1}.

\subsubsection{Westbury's $\son[7]$-webs}

In \cite{WesSpin}, Westbury studies the full subcategory 
$\Rep_{S,1} \subset \Rep(U_q(\son[7]))$ 
tensor-generated by the quantum spin and vector representations. 
He defines a category of webs $\Web_{S,1}$ which 
provides a complete generators-and-relations presentation of $\Rep_{S,1}$, 
i.e.~he proves the analogue of our Theorem \ref{thm:main} in this setting.
Since $\Rep_{S,1}$ is a proper subcategory of $\FRep(U_q(\son[7]))$, 
Westbury's presentation has fewer generating objects and morphisms; 
hence, it is necessarily different from ours.
Indeed, restricting to these generators necessitates relations 
in $\Web_{S,1}$ involving webs having
faces with more than three edges.
Nevertheless, up to a rescaling 
(since Westbury's trivalent vertex is equal to $[2]^{\frac{1}{2}}$ times ours),
one can derive all relations in \cite[Fig.~1--4]{WesSpin} from ours
(after fixing the apparent typo in \cite[Fig.~4]{WesSpin} so that it matches 
the corresponding equation on p.~222 of loc.~cit.).

\subsubsection{Bodish--Wu's quantum orthogonal webs}

In classical work on invariant theory in type $B$ and $D$ \cite{Braueroriginal, LZbrauer}, 
the first object to study is the orthogonal group $\mathrm{O}_m$, 
which is a $\mathbb{Z}/2$ extension of $\mathrm{SO}_m$. 
One reason for this is that the representation theory of $\mathrm{O}_m$ 
is more uniform than $\mathrm{SO}_m$, 
e.g.~the exterior powers $\Lambda^k(\mathbb{C}^m)$ are mutually 
non-isomorphic irreducible $\mathrm{O}_m$-representations, for $k=0, \dots, m$. 
However, the main reason is that there are no spin representations 
for the groups $\mathrm{SO}_m$ or $\mathrm{O}_m$. 

The algebra $U_q(\mathfrak{o}_m)$, 
defined as a $\mathbb{Z}/2$ extension of $U_q(\mathfrak{so}_m)$, 
is an associative algebra such that (type I) modules over it, 
with integer weights (i.e.~ignoring spin representations), 
give a $q$-analogue of $\Rep(\mathrm{O}_m)$. 
In \cite{BodWu}, Bodish--Wu define $\Web(\mathrm{O}_m)$ 
and prove that it gives a complete generators-and-relations presentation 
of the quantum analogue of the full subcategory of representations tensor-generated 
by the exterior powers, 
denoted $\FRep(U_q(\mathfrak{o}_m))\subset \Rep(U_q(\mathfrak{o}_m))$. 
The relations in \cite[Definition 1.4]{BodWu} are the same (up to signs accounting for different conventions)
as those relations in our definition of $\Web(\son)$ which do not involve gray strands. 
For more discussion see the beginning of the proof of Theorem \ref{thm:functor}.

\subsubsection{McNamara--Savage's quantum spin Brauer category}

In type $A$, another approach to the study of quantum group representation theory 
is via the \emph{HOMFLY--PT (oriented) skein category} $\mathbf{OS}$, 
wherein morphisms are given by oriented tangles, modulo relations;
see e.g.~\cite{BrunHOMFLY} and references within.
This category encodes morphisms between tensor powers of the vector 
representation $V$ of $U_q(\gln)$ and its dual\footnote{It is more natural to consider quantum $\gln$ in this setting, 
to avoid subtleties involving the isomorphism of $U_q(\sln)$-modules between the determinant and trivial modules.}, 
as a ribbon category. 
Although $\mathbf{OS}$ may also be presented by generators and relations, 
its generators are the braiding morphisms (as well as cups and caps), 
in contrast to the trivalent vertices that generate web categories.

However, note that the category $\mathbf{OS}$ does not describe $\Rep(U_q(\gln))$ for any fixed value of $n$; 
rather, it can be viewed as the $n \to \infty$ limit of these categories. 
To pass from $\mathbf{OS}$ to a particular $\Rep(U_q(\gln))$ involves two steps: 
a specialization and a quotient, both depending on $n$.
To elaborate, for each $n$ there is a category $\mathbf{OS}(n)$ obtained by specializing the 
parameters in the definition of $\mathbf{OS}$, 
and there is a full functor $\mathbf{OS}(n) \to \Rep(U_q(\gln))$.
We emphasize that this latter functor is \emph{not faithful}, 
so a further quotient, described explicitly in \cite[Theorem 1.3]{BrunHOMFLY}, 
must be imposed on $\mathbf{OS}(n)$ before it ``describes'' $\Rep(U_q(\gln))$. 

The \emph{quantum Brauer category} is a type $BCD$ analogue of $\mathbf{OS}$, 
wherein morphisms are described by unoriented tangles,
accounting for the fact that the vector representation in type $BCD$ is self-dual. 
Recently, McNamara--Savage \cite{MS-qspin} introduced and studied a type $B/D$ 
enlargement $\mathbf{QSB}$ of the quantum Brauer category that also incorporates the spin representation, 
and which they call the \emph{quantum spin Brauer category}. 
Paralleling the type $A$ case described above, 
McNamara--Savage show that there is a full functor
\begin{equation}\label{eq:MSfunctor}
\mathbf{QSB}(2n+1) \twoheadrightarrow \FRep(U_q(\son)) \, ,
\end{equation}
where $\mathbf{QSB}(2n+1)$ denotes an appropriate specialization of the parameters of $\mathbf{QSB}$.

This functor is not faithful, and, unlike in type $A$, 
there is no explicit description of the kernel of this functor in the literature (this kernel depends on $n$).
As a consequence of Theorem \ref{thm:main}, there is a full functor 
\[
\mathbf{QSB}(2n+1) \twoheadrightarrow \Web(\son), 
\]
which is essentially surjective after additive Karoubi completion. 
This result could also be established directly from the respective presentations of these categories, 
although to match conventions one must replace $q$ with $q^2$ 
in the choice of parameters from \cite[Section 6]{MS-qspin}.
The main feature of our result, besides incorporating the remaining fundamental representations, 
is the \emph{faithfulness} established in Theorem \ref{thm:main}: 
we have all the relations needed to exactly describe $\Rep(U_q(\son))$.

\begin{ack}
Thanks to 
William Ballinger, 
Pavel Etingof, 
Artem Kalmykov,
Greg Kuperberg,
Peter McNamara, 
Ivan Motorin, 
Alistair Savage, 
David Speyer, 
Dani Tubbenhauer,
Weiqiang Wang, 
Hans Wenzl, 
and Haihan Wu 
for helpful discussion and correspondence.
The results of this work were presented at the conference 
\emph{Webs in Algebra, Geometry, Topology and Combinatorics}
held at ICERM in December 2025, 
and we thank the organizers for the opportunity to preview our results.
\end{ack}

\begin{fund}
E.B.~was partially supported by the NSF MSPRF-2202897 
and is grateful for the support and hospitality of the Sydney Mathematical Research Institute (SMRI). 
B.E.~was partially supported by NSF grant DMS-2201387, 
and appreciates the support given to his research group by DMS-2039316.
D.E.V.R.~was partially supported by NSF CAREER grant DMS-2144463.
\end{fund}

%
\section{Preliminaries}
%

\subsection{The type $B$ quantum group}

The Lie algebra $\son$ corresponds to the type $B_n$ root system $\Phi_{B_n}$.
Therein, the simple roots are the vectors
\[
\{\ee_i = \epsilon_i- \epsilon_{i+1} \}_{i=1}^{n-1} \cup \{\ee_n = \epsilon_n\} 
\]
living in $\R\{\epsilon_i\}_{i=1}^n \cong \R^n$. 
This vector space is endowed with the inner product $(\epsilon_i , \epsilon_j) := 2\delta_{ij}$, a multiple
of the standard inner product on $\R^n$.
The corresponding (type $B_n$) Cartan matrix is
\[
a_{ij} := 2 \frac{(\ee_i ,\ee_j)}{(\ee_i , \ee_i)}
=
\begin{pmatrix}
2 & -1 & 0 & \cdots & 0 \\
-1 & 2 & \ddots & \ddots & \vdots \\
0 & \ddots & \ddots & -1 & 0 \\
\vdots & \ddots & -1 & 2 & -1 \\
0 & \cdots & 0 & -2 & 2
\end{pmatrix}.
\]

Let $q$ be an indeterminate.
Given $i \in \{1,\ldots,n\}$ and $m \in \N$, set 
\[
q_i := q^{\frac{1}{2}(\ee_i , \ee_i)} = 
\begin{cases}
q^2 & i = 1 \ldots, n-1 \\
q & i = n
\end{cases} \, , \quad
[m]_{q_i} = \frac{q_i^m- q_i^{-m}}{q_i- q_i^{-1}}
	\]
and
\[
[n]_{q_i}! = [n]_{q_i}[n-1]_{q_i}\cdots[2]_{q_i}[1]_{q_i} \, ,  \quad 
{n \brack k}_{q_i}= \frac{[n]_{q_i}!}{[n-k]_{q_i}![k]_{q_i}!} \, .
\]
We will typically write $[m]$ in place of $[m]_q$.
Note also that $[m]_{q^2} = \tfrac{[2m]}{[2]}$.

\begin{defn}
	\label{def:quantumgroup}
Let $U_q(\son)$ be the unital $\C(q)$-algebra 
generated by elements 
$e_i, f_i, k_i^{\pm 1}$ for $i \in \{1,\ldots,n\}$, subject to the following relations:
\begin{gather*}
k_i k_i^{-1} = 1= k_i^{-1} k_i, \quad k_i k_j = k_j k_i, 
\quad  k_i e_j= q^{(\alpha_i, \alpha_j)} e_j k_i, \quad  k_i f_j = q^{-(\alpha_i, \alpha_j)} f_j k_i \, , \\
e_i f_j - f_je_i = \delta_{ij}\dfrac{k_i- k_i^{-1}}{q_i - q_i^{-1}} \, , \\
\sum_{s= 0}^{1- a_{ij}} (-1)^s{1- a_{ij}\brack s}_{q_i} e_i^{1- a_{ij}- s} e_j e_i^s = 0, 
\quad \sum_{s= 0}^{1- a_{ij}} (-1)^s{1- a_{ij}\brack s}_{q_i} f_i^{1- a_{ij}- s} f_j f_i^s = 0 \, .
	\end{gather*}
\end{defn}

This algebra is a Hopf algebra, and, 
following the conventions in \cite{CP,ST},
the structure maps are given on generators as follows:
\begin{itemize}
\item $\Delta(e_i) = e_i \otimes k_i + 1\otimes e_i$, 
	$\Delta(f_i) = f_i\otimes 1 + k_i^{-1}\otimes f_i$, 
	and $\Delta(k_i^{\pm}) = k_i ^{\pm}\otimes k_i^{\pm}$.
\item $\mathbf{S}(e_i) = -e_i k_i^{-1}$, 
	$\mathbf{S}(f_i) = - k_i f_i$, 
	and $\mathbf{S}(k_i^{\pm})= k_i^{\mp}$.
\item $\epsilon(e_i)= 0$, 
	$\epsilon(f_i) = 0$, 
	and $\epsilon(k_i) = 1$.
\end{itemize}

Let $\Rep(U_q(\son))$ denote the category of finite-dimensional, type I representations of $U_q(\son)$ over $\C(q)$
and let $\Rep(\son)$ denote the category of finite-dimensional representations of the Lie algebra $\son$ over $\C$. 
These categories are rigid monoidal, so in particular are closed under tensor product and 
monoidal duals, and contain a trivial representation. 
Further, both categories are braided (with the braiding on $\Rep(\son)$ additionally being symmetric) and pivotal.

\begin{conv}\label{conv:pivotal}
Throughout, we work with the pivotal structure on $\Rep(U_q(\son))$ 
associated to the ribbon element $\nu = X^{-2}$ from \cite{ST}, 
wherein all self-dual irreducible representations have Frobenius--Schur indicator equal to $+1$.
\end{conv}

The categories $\Rep(U_q(\son))$ and $\Rep(\son)$ are also 
\emph{split semisimple} over $\C(q)$ and $\C$ (respectively), 
meaning that they are semisimple and endomorphism algebras of simple objects $V$ 
are equal to $\C(q) \cdot \id_V$ and $\C \cdot \id_V$ (respectively).
In both cases, the simple objects are classified (up to isomorphism)
by the set $X_+(\son)$ of dominant integral weights, 
which are the weights $\lambda \in \R\{\epsilon_i\}_{i=1}^n$ 
such that $2\frac{(\alpha_i,\lambda)}{(\alpha_i, \alpha_i)}\in \N$ for $i=1, \dots, n$.
In coordinates,
\[
X_+(\son) = \{ (a_1,\ldots,a_n) \mid 
	a_i - a_{i+1} \in \N \text{ and } a_n \in \tfrac{1}{2} \N \}.
	\]
In particular, either all indices $a_i$ are integers, or all indices $a_i$ are half-integers (see Convention \ref{convention:halfinteger}).
For $\lambda \in X_+(\son)$ we write $L_{\lambda}$ for the corresponding irreducible representation of $\son$ 
and $V_{\lambda}$ for the irreducible representation of $U_q(\son)$. 
The representations $V_{\lambda}$ and $L_{\lambda}$ have the same formal characters; 
i.e.~if $V[\mu]$ denotes the $\mu$ weight space of $V$, then $\dim(V_{\lambda}[\mu]) = \dim(L_{\lambda}[\mu])$.

\begin{conv} \label{convention:halfinteger} 
In this paper, a \emph{half-integer} is an element of $\mathbb{Q}$ of the form $\frac{2j+1}{2}$ for some integer $j$. 
Integers are not half-integers, 
and $\tfrac{1}{2} \Z$ is the disjoint union of the integers and the half-integers. \end{conv}

For the Lie algebra $\son$ and $k \in \{1,\ldots,n-1\}$, 
the fundamental weights are 
\[
\varpi_k := \sum_{i=1}^k \epsilon_i = (\underbrace{1,\ldots,1}_{k},0,\ldots,0)
	\]
and the $k$-th fundamental representation $L_k :=L_{\varpi_k}$ 
is isomorphic to the exterior product $\Lambda^k \C^{2n+1}$. 
The corresponding fundamental $U_q(\son)$-representations $V_k := V_{\varpi_k}$ 
are quantized analogues of these exterior products \cite{BerZwick, BodWu}. 
We also write $L_0$ and $V_0$ to denote the trivial representation, 
isomorphic to $\Lambda^0 \C^{2n+1}$ and its quantized analogue, respectively.

The remaining fundamental representation is the (quantum) spin representation,
which in both classical and quantum cases we denote by $S$ (by slight abuse of notation).
Its highest weight is the final fundamental weight
\[
\varpi_{n} := \frac{1}{2} \sum_{i=1}^n \epsilon_i = (\tfrac{1}{2},\ldots,\tfrac{1}{2}) \, .
\]

\begin{rem}
Let $\Rep_{\text{even}}(U_q(\son))$ denote the full subcategory of $\Rep(U_q(\son))$ 
consisting of representations whose weights have integer coefficients.
For example, $V_k\in \Rep_{\text{even}}(U_q(\son))$ for $0 \le k \le n-1$, 
$S \otimes S \in \Rep_{\text{even}}(U_q(\son))$, 
and $S \notin \Rep_{\text{even}}(U_q(\son))$.
Every irreducible representation in $\Rep_{\text{even}}(U_q(\son))$ is a summand 
of a tensor power of $V_1$.
In the classical setting, the representations in $\Rep_{\text{even}}(\son)$ 
are those that integrate to the Lie group $\mathrm{SO}_{2n+1}$, 
while one must pass to the double cover 
$\mathrm{Spin}_{2n+1}$ to integrate the remaining representations.
\end{rem}

The categories $\Rep(U_q(\son))$ and $\Rep(\son)$ share the same rules for decomposing tensor products. 
Additionally, if a tensor product $L \otimes L'$ of two finite-dimensional irreducible representations 
contains the trivial representation, then $L^\vee \cong L'$, 
where $L^\vee$ denotes the dual representation of $L$.
Thanks to monoidal self-duality of fundamental representations in $\Rep(\son)$, 
we deduce that all of the fundamental representations in $\Rep(U_q(\son))$ are monoidal self-dual.

Next, we record some pertinent decompositions.

\begin{prop}
	\label{prop:decomp}
The following decompositions hold in $\Rep(U_q(\son))$:
\begin{equation} \label{eq:V1Vk}
V_1 \otimes V_k \cong
\begin{cases}
V_{\varpi_1+\varpi_k} \oplus V_{k+1} \oplus V_{k-1} & 1 \leq k \leq n-2 \\
V_{\varpi_1+\varpi_{n-1}} \oplus V_{2\varpi_n} \oplus V_{n-2} & k = n-1 \, ,
	\end{cases}
\end{equation}
\begin{equation} \label{eq:V1V2n}
V_1 \otimes V_{2\varpi_n} \cong
V_{\varpi_1+2\varpi_n} \oplus V_{2\varpi_n} \oplus V_{n-1} \, ,
	\end{equation}
\begin{equation} \label{eq:SSdecomp}
S \otimes S \cong V_{2\varpi_n} \oplus \bigoplus_{i=0}^{n-1} V_i \, ,
	\end{equation}
and
\begin{equation} \label{eq:1Sdecomp} 
V_1 \otimes S \cong V_{\varpi_1 + \varpi_n} \oplus S \, . 
	\end{equation}
	\end{prop}	
\begin{proof}
This is classical; 
see e.g.~\cite{YDM} for \eqref{eq:V1Vk} and \eqref{eq:V1V2n}
and \cite[Exercise 19.16]{FultonHarris} for \eqref{eq:SSdecomp}.
For \eqref{eq:1Sdecomp}, 
observe that we must have the highest weight summand $V_{\varpi_1 + \varpi_n}$ 
with multiplicity one and all other summands having strictly lower highest weight in the dominance order.
Further, all weights appearing in $V_1 \otimes S$ have half-integer weights, 
so the only possible remaining summands are isomorphic to $S$.
To see that there is only one such summand, we use monoidal self-duality of the spin representation to compute
\[
\dim \big( \Hom_{U_q(\son)}(V_1 \otimes S, S) \big) = \dim \big( \Hom_{U_q(\son)}(V_1, S \otimes S) \big)
\stackrel{{\eqref{eq:SSdecomp}}}{=} 1\, . \qedhere
	\]
\end{proof}

\begin{rem} \label{rem:morelabels}
If we write $V_n := V_{2 \varpi_n}$ and $V_{2n+1-k} := V_k$ for $0 \le k \le n$, 
and write $\mu_k$ to denote the highest weight of
$V_k$ for $1 \le k \le 2n$, 
then 
\[ 
V_1 \ot V_k \cong V_{\varpi_1+\mu_k} \oplus V_{k+1} \oplus V_{k-1} 
\] 
for all $ 1 \le k \le 2n$. 
In this paper the index $k$ in
$V_k$ is constrained so that $1 \le k \le n-1$, 
but we expect many diagrammatic formulae to continue to hold for all $1 \le k \le 2n$ when extended in this fashion. 
Trivalent vertices involving $k$ will not represent the same morphism as those with $2n+1-k$, 
but some scaling factors need to be inserted, c.f. Remark \ref{rem:2nplus1minusk}.

In this extended graphical calculus, 
one should interpret a strand of thickness $n$ as representing the object in the Karoubi envelope of $\Web(\son)$ 
given by the idempotent
\[
\begin{tikzpicture}[scale=.5, tinynodes, anchorbase]
	\draw[very thick,gray] (1,-1) to (1,1);
	\draw[very thick,gray] (0,-1) to (0,1);
\end{tikzpicture} \
- \sum_{k=1}^{n} \frac{(-1)^{\binom{k+1}{2}}}{``{\textstyle {2k \brack k}}"}
\begin{tikzpicture}[scale=.275, smallnodes, anchorbase]
	\draw[very thick,gray] (-1,0) to (0,1);
	\draw[very thick,gray] (1,0) to (0,1);
	\draw[very thick,gray] (0,2.5) to (-1,3.5);
	\draw[very thick,gray] (0,2.5) to (1,3.5);
	\draw[very thick] (0,1) to node[right=-2pt]{$n{-}k$} (0,2.5);
\end{tikzpicture}
\]
in $\End(\gS\otimes \gS)$. The trivalent vertex $n \to \gS \otimes \gS$
is the inclusion map, and the trivalent vertex $\gS \otimes \gS \to n$ is the projection map, which is compatible with
\eqref{eq:digonS} for $k=n$. The equality \eqref{eq:blackspinH=I} for $k=n-1$, 
when correctly interpreted, can be viewed
as a definition of the trivalent vertex $(n-1) \otimes 1 \to n$.

We do not pursue this extended calculus further in this paper.
\end{rem}

\subsection{Additional relations}\label{ss:morerels}

In this section, we record some additional web relations 
which follow as direct consequences of our defining relations \eqref{eq:webrel}.
While the enterprising reader can view their establishment as an exercise 
(of varying difficulty, depending on the relation), 
we provide some detail, especially for those relations 
that play a role in the proof of Theorem \ref{thm:main}.

Henceforth we let 
\[
\dd_k := ``{\textstyle {2k \brack k}}" = \prod_{i=1}^{k} [2]_{2i-1}
\]
denote the devil's central binomial coefficient
(this was our notation from \cite{BER}).
For example, $(-1)^{n+1 \choose 2}\dd_n$ is the value of the $\gS$-colored circle by \eqref{eq:circleS}. 
Note that \eqref{eq:digonS} contains $\dd_{n-k}$, and \eqref{eq:spinH=I} contains $\dd_k^{-1}$.
We also remind the reader that $``[k]^2"=[k]_{q^{2}}$.

\begin{prop}\label{prop:easyrel}
The following relations hold in $\Web(\son)$:
\newline\begin{minipage}{.5\textwidth} 
\begin{equation}
	\label{eq:otherdigonS}
\begin{tikzpicture}[scale=.175,tinynodes, anchorbase]
	\draw [very thick, gray] (0,.75) to (0,2.5) ;
	\draw [very thick] (0,-2.75) to [out=30,in=330] node[right,xshift=-2pt]{$1$} (0,.75);
	\draw [very thick, gray] (0,-2.75) to [out=150,in=210]  (0,.75);
	\draw [very thick, gray] (0,-4.5) to (0,-2.75);
\end{tikzpicture}
= 
(-1)^n\frac{[2n+1]}{[2]} \
\begin{tikzpicture}[scale=.175, tinynodes, anchorbase]
	\draw [very thick, gray] (0,-4.5) to (0,2.5);
\end{tikzpicture}
		\end{equation}
\end{minipage}
\begin{minipage}{.5\textwidth} 
\begin{equation}
	\label{eq:zerodigonB}
\begin{tikzpicture}[scale=.175,tinynodes, anchorbase]
	\draw [very thick] (0,.75) to (0,2.5) node[above,yshift=-3pt]{$k{+}2$};
	\draw [very thick] (0,-2.75) to [out=30,in=330] node[right,xshift=-2pt]{$k{+}1$} (0,.75);
	\draw [very thick] (0,-2.75) to [out=150,in=210] node[left,xshift=2pt]{$1$} (0,.75);
	\draw [very thick] (0,-4.5) node[below,yshift=2pt]{$k$} to (0,-2.75);
\end{tikzpicture}
= 0 \qquad \qquad
	\end{equation}
	\end{minipage}
\begin{minipage}{.4\textwidth} 
\begin{gather}
	\label{eq:circle1}
\begin{tikzpicture}[scale =.625, anchorbase]
	\draw[very thick] (0,0) node[left=7pt]{\scs$1$} circle (.5);
\end{tikzpicture}
= [2]_{2n-1}\frac{[2n+1]}{[2]} \qquad \qquad
\end{gather}
	\end{minipage}
\begin{minipage}{.6\textwidth} 
\begin{equation}
	\label{eq:otherdigonB}
\begin{tikzpicture}[scale=.175,tinynodes, anchorbase]
	\draw [very thick] (0,.75) to (0,2.5) node[above,yshift=-3pt]{$k$};
	\draw [very thick] (0,-2.75) to [out=30,in=330] node[right,xshift=-2pt]{$1$} (0,.75);
	\draw [very thick] (0,-2.75) to [out=150,in=210] node[left,xshift=2pt]{$k{+}1$} (0,.75);
	\draw [very thick] (0,-4.5) node[below,yshift=2pt]{$k$} to (0,-2.75);
\end{tikzpicture}
= (-1)^{k} 
``[2n-k+1]^2"
\dfrac{[2]_{2n-2k-1}}{[2]_{2n-2k+1}}
\begin{tikzpicture}[scale=.175, tinynodes, anchorbase]
	\draw [very thick] (0,-4.5) node[below,yshift=2pt]{$k$} to (0,2.5) node[above,yshift=-3pt]{$k$};
\end{tikzpicture}
\end{equation}
\end{minipage}
\begin{minipage}{.4\textwidth} 
\begin{equation}\label{eq:ggb-triangle+}
\begin{tikzpicture}[scale=.25,tinynodes,anchorbase]
	\draw[very thick, gray] (-1,0) to (1,0);
	\draw[very thick] (-1,0) to node[above=-2pt, left=-1pt]{$1$} (0,1.732);
	\draw[very thick] (1,0) to node[above=-2pt, right=-1pt]{$k$} (0,1.732);
	\draw[very thick] (0,1.732) to (0,3.232) node[above=-3pt]{$k{+}1$};
	\draw[very thick, gray] (-2.3,-.75) to (-1,0);
	\draw[very thick, gray] (2.3,-.75) to (1,0);
\end{tikzpicture}
= (-1)^k 
``[k+1]^2"
\begin{tikzpicture}[scale =.5, smallnodes,anchorbase]
	\draw[very thick, gray] (0,0) to [out=90,in=210] (.5,.75);
	\draw[very thick, gray] (1,0) to [out=90,in=330] (.5,.75);
	\draw[very thick] (.5,.75) to (.5,1.5) node[above=-3pt]{$k{+}1$};
\end{tikzpicture}
	\end{equation}
	\end{minipage}
\begin{minipage}{.6\textwidth} 
\begin{equation}\label{eq:ggb-triangle-}
\begin{tikzpicture}[scale=.25,tinynodes,anchorbase]
	\draw[very thick, gray] (-1,0) to (1,0);
	\draw[very thick] (-1,0) to node[above=-2pt, left=-1pt]{$1$} (0,1.732);
	\draw[very thick] (1,0) to node[above=-2pt, right=-1pt]{$k$} (0,1.732);
	\draw[very thick] (0,1.732) to (0,3.232) node[above=-3pt]{$k{-}1$};
	\draw[very thick, gray] (-2.3,-.75) to (-1,0);
	\draw[very thick, gray] (2.3,-.75) to (1,0);
\end{tikzpicture}
= (-1)^n 
``[2n-k+2]^2"
\frac{1}{[2]_{2n-2k+3}} 
\begin{tikzpicture}[scale =.5, smallnodes,anchorbase]
	\draw[very thick, gray] (0,0) to [out=90,in=210] (.5,.75);
	\draw[very thick, gray] (1,0) to [out=90,in=330] (.5,.75);
	\draw[very thick] (.5,.75) to (.5,1.5) node[above=-3pt]{$k{-}1$};
\end{tikzpicture} \, .
	\end{equation}
	\end{minipage}
\end{prop}
\begin{proof}
To prove \eqref{eq:otherdigonS} we precompose \eqref{eq:spinH=I} with an $\gS$-labeled cup. 
Applying \eqref{eq:lolliS} will kill every term in the sum except for $k=n$. 
The $k=n$ term yields an $\gS$-labeled circle, which cancels the denominator $\dd_n$ using \eqref{eq:circleS}. 
Thus it remains to show that 
\begin{equation} (-1)^n \frac{[2n+1]}{[2]} = \frac{1}{[2]} + (-1)^{\binom{n}{2}}(-1)^{\binom{n+1}{2}} ``[n][n+1]", \end{equation}
which we leave to the reader.
(Hint: multiply both sides by $[2]$.)

We next prove \eqref{eq:zerodigonB} by induction on $k$. We first establish the $k=0$ case, 
which is the relation
\begin{gather} \label{eq:lolli1}
\begin{tikzpicture}[scale=.175,tinynodes,anchorbase,rotate=180]
	\draw [very thick] (0,-2.75) to [out=30,in=0] (0,.75) node[below=-1pt]{$1$};
	\draw [very thick] (0,-2.75) to [out=150,in=180] (0,.75);
	\draw [very thick] (0,-4.5) node[above=-2pt]{$2$} to (0,-2.75);
\end{tikzpicture}
= 0 \, .
\end{gather}
For this, we compute 
\begin{align*}
\begin{tikzpicture}[scale=.175,tinynodes,anchorbase,rotate=180]
	\draw [very thick] (0,-2.75) to [out=30,in=0] (0,.75) node[below=-1pt]{$1$};
	\draw [very thick] (0,-2.75) to [out=150,in=180] (0,.75);
	\draw [very thick] (0,-4.5) node[above=-2pt]{$2$} to (0,-2.75);
\end{tikzpicture}
\stackrel{\eqref{eq:digonS}}{=}
(-1)^{\binom{n-1}{2}} \dd_{n-2}^{-1}
\begin{tikzpicture}[scale=.175,tinynodes,anchorbase,rotate=180]
	\draw [very thick] (0,-2.75) to [out=30,in=0] (0,0) node[below=-1pt]{$1$};
	\draw [very thick] (0,-2.75) to [out=150,in=180] (0,0);
	\draw [very thick] (0,-4) to node[right=-1pt]{$2$} (0,-2.75);
	\draw[very thick,gray] (0,-6) to [out=30,in=330] (0,-4);
	\draw[very thick,gray] (0,-6) to [out=150,in=210] (0,-4);
	\draw [very thick] (0,-7) node[above=-2pt]{$2$} to (0,-6);
\end{tikzpicture}
\stackrel{\eqref{eq:blackspinH=I}}{=}
(-1)^{\binom{n-1}{2}} \dd_{n-2}^{-1}
\left(
\begin{tikzpicture}[scale=.175,tinynodes,anchorbase,rotate=180]
	\draw[very thick,gray] (0,-6) to [out=30,in=0] (1.5,-2.5);
	\draw[very thick,gray] (0,-6) to [out=150,in=180] (-1.5,-2.5);
	\draw [very thick] (0,-7) node[above=-2pt]{$2$} to (0,-6);
	\draw[very thick] (1.5,-2.5) to [out=120,in=60] node[below=-2pt]{$1$} (-1.5,-2.5);
	\draw[very thick,gray] (1.5,-2.5) to [out=240,in=300] (-1.5,-2.5);
\end{tikzpicture}
- (-1)^n \frac{1}{[2]_{2n-1}}
\begin{tikzpicture}[scale=.175,tinynodes,anchorbase,rotate=180]
	\draw[very thick,gray] (0,-6) to [out=30,in=0] (0,-4);
	\draw[very thick,gray] (0,-6) to [out=150,in=180] (0,-4);
	\draw [very thick] (0,-7) node[above=-2pt]{$2$} to (0,-6);
	\draw[very thick] (0,-2) node[left=2pt]{\scs$1$} circle (.75);
\end{tikzpicture}
\right)
\stackrel{\substack{\eqref{eq:otherdigonS} \\ \eqref{eq:lolliS}}}{=} 0 \, .
	\end{align*}
The inductive step for \eqref{eq:zerodigonB} follows from the computation
\begin{equation} \label{dividingbyk}
(-1)^{k-1}
``[k]^2"
\begin{tikzpicture}[scale=.2, xscale=-1,tinynodes, anchorbase]
	\draw [very thick] (-2,-4) node[below=-2pt]{$k$} to [out=90,in=210] (0,.75);
	\draw[very thick] (0,.75) to [out=330,in=90] (1,-1) to [out=270,in=180] (2,-2) 
		node[below=-2pt]{$1$} to [out=0,in=270] (3,-1) to [out=90,in=330] (1,2.5);
	\draw [very thick] (0,.75) to [out=90,in=210] node[right=-1pt]{$k{+}1$} (1,2.5);
	\draw [very thick] (1,2.5) to (1,4.25) node[above=-2pt]{$k{+}2$};
\end{tikzpicture}
\stackrel{\eqref{eq:digon1}}{=}
\begin{tikzpicture}[scale=.2, xscale=-1,tinynodes, anchorbase]
	\draw [very thick] (-2,-4) node[below=-2pt]{$k$} to (-2,-3) node[left=-1pt]{$k{-}1$}
	to [out=30,in=330] (-2,-1.5) to [out=90,in=210] node[right]{$k$} (0,.75);
	\draw[very thick] (-2,-3) node[right=-1pt]{$1$} to [out=150,in=210] (-2,-1.5);
	\draw[very thick] (0,.75) to [out=330,in=90] (1,-1) to [out=270,in=180] (2,-2) 
		node[below=-2pt]{$1$} to [out=0,in=270] (3,-1) to [out=90,in=330] (1,2.5);
	\draw [very thick] (0,.75) to [out=90,in=210] node[right=-1pt]{$k{+}1$} (1,2.5);
	\draw [very thick] (1,2.5) to (1,4.25) node[above=-2pt]{$k{+}2$};
\end{tikzpicture}
\stackrel{\eqref{eq:assoc}}{=}
\begin{tikzpicture}[scale=.25,tinynodes,anchorbase,xscale=-1]
	\draw[very thick] (-2,-2) node[below=-2pt]{$k$} to [out=90,in=210] (-1,0);
	\draw[very thick] (-1,0) to node[below=-1pt]{$k{-}1$} (1,0);
	\draw[very thick] (-1,0) to node[pos=.7,right=-1pt]{$1$} (0,1.732);
	\draw[very thick] (1,0) to node[pos=.7,left=-1pt]{$k$} (0,1.732);
	\draw[very thick] (0,1.732) to [out=90,in=210] node[right=-1pt]{$k{+}1$} (1.25,3.5);
	\draw[very thick] (1,0) to [out=330,in=225] (3,0) node[below]{$1$} 
		to [out=45,in=330] (1.25,3.5);
	\draw[very thick] (1.25,3.5) to (1.25,5) node[above=-2pt]{$k{+}2$};
\end{tikzpicture}
\stackrel{\eqref{eq:assoc}}{=}
\begin{tikzpicture}[scale=.25,tinynodes,anchorbase,xscale=-1]
	\draw[very thick] (-2,-2) node[below=-2pt]{$m$} to [out=90,in=210] (-1,0);
	\draw[very thick] (1,.5) to [out=330,in=225] (3,0) node[below]{$1$} 
		to [out=45,in=330] (2.5,2);
	\draw[very thick] (-1,0) to node[below=-1pt,xshift=-1pt]{$k{-}1$} (1,.5);
	\draw[very thick] (2.5,2) to [out=90,in=330] node[above=-2pt,xshift=-5pt]{$k{+}1$} (1.25,3.5);
	\draw[very thick] (1,.5) to [out=90,in=210] node[above=-2pt]{$k$} (2.5,2);
	\draw[very thick] (-1,0) to [out=90,in=210] node[right=-1pt]{$1$} (1.25,3.5);
	\draw[very thick] (1.25,3.5) to (1.25,5) node[above=-2pt]{$k{+}2$};
\end{tikzpicture}
\stackrel{\eqref{eq:zerodigonB}}{=}
0.
\end{equation}
We divide by $(-1)^{k-1}``[k]^2"$ 
to draw our conclusion.

Next, we will show \eqref{eq:circle1} and \eqref{eq:otherdigonB};
although the former is simply the $k=0$ case of the latter, 
we establish it first since it is used for the general case.
We thus compute
\begin{align*}
\begin{tikzpicture}[scale =.625, anchorbase]
	\draw[very thick] (0,0) node[left=7pt]{\scs$1$} circle (.5);
\end{tikzpicture}
&\stackrel{\eqref{eq:digonS}}{=}
(-1)^{n \choose 2} \dd_{n-1}^{-1} \;
\begin{tikzpicture}[scale=.5,tinynodes, anchorbase]
	\draw[very thick, gray] (0,-.5) to [out=150,in=210] (0,.5);
	\draw[very thick, gray] (0,-.5) to [out=30,in=330] (0,.5);
	\draw[very thick] (0,.5) to [out=90,in=180] (.5,1) to [out=0,in=90] (1,0) node[right=-2pt]{$1$}
		to [out=270,in=0] (.5,-1) to [out=180,in=270] (0,-.5);
\end{tikzpicture}
=
(-1)^{n \choose 2} \dd_{n-1}^{-1} \;
\begin{tikzpicture}[scale=.5,tinynodes, anchorbase]
	\draw[very thick] (0,-.5) to [out=150,in=210] node[left=-2pt]{$1$} (0,.5);
	\draw[very thick, gray] (0,-.5) to [out=30,in=330] (0,.5);
	\draw[very thick,gray] (0,.5) to [out=90,in=180] (.5,1) to [out=0,in=90] (1,0) 
		to [out=270,in=0] (.5,-1) to [out=180,in=270] (0,-.5);
\end{tikzpicture} \\
&\stackrel{\eqref{eq:otherdigonS}}{=}
(-1)^{{n \choose 2}+n}
\frac{[2n+1]}{[2]}
\dd_{n-1}^{-1} \;
\begin{tikzpicture}[scale =.625, anchorbase]
	\draw[very thick, gray] (0,0) circle (.5);
\end{tikzpicture}
\stackrel{\eqref{eq:circleS}}{=}
[2]_{2n-1}\frac{[2n+1]}{[2]} \, .
\end{align*}
In the last step, we also
used that $\dd_n \dd_{n-1}^{-1} = [2]_{2n-1}$.

Placing a $1$-labeled cap on \eqref{eq:blackH=I} and turning the result by 90 degrees, we have
\begin{align*}
\begin{tikzpicture}[scale=.175,tinynodes, anchorbase]
	\draw [very thick] (0,.75) to (0,2.5) node[above,yshift=-3pt]{$k$};
	\draw [very thick] (0,-2.75) to [out=30,in=330] node[right,xshift=-2pt]{$1$} (0,.75);
	\draw [very thick] (0,-2.75) to [out=150,in=210] node[left,xshift=2pt]{$k{+}1$} (0,.75);
	\draw [very thick] (0,-4.5) node[below,yshift=2pt]{$k$} to (0,-2.75);
\end{tikzpicture}
&\stackrel{\eqref{eq:blackH=I}}{=}
\begin{tikzpicture}[scale=.175,tinynodes, anchorbase]
	\draw [very thick] (0,.75) to (0,2.5) node[above,yshift=-3pt]{$k$};
	\draw[very thick] (0,.75) to (1.5,-1);
	\draw[very thick] (0,-2.75) to (1.5,-1);
	\draw[very thick] (1.5,-1) to node[above=-2pt]{$2$} (3,-1);
	\draw[very thick] (3,-1) to [out=60,in=90] (4.5,-1) to [out=270,in=300] (3,-1);
	\draw [very thick] (0,-2.75) to [out=150,in=210] node[left,xshift=2pt]{$k{-}1$} (0,.75);
	\draw [very thick] (0,-4.5) node[below,yshift=2pt]{$k$} to (0,-2.75);
\end{tikzpicture}
+\tfrac{[2]_{2n{-}2k{-}1}}{[2]_{2n{-}2k{+}1}}
\begin{tikzpicture}[scale=.175,tinynodes, anchorbase]
	\draw [very thick] (0,.75) to (0,2.5) node[above,yshift=-3pt]{$k$};
	\draw [very thick] (0,-2.75) to [out=30,in=330] node[right,xshift=-2pt]{$1$} (0,.75);
	\draw [very thick] (0,-2.75) to [out=150,in=210] node[left,xshift=2pt]{$k{-}1$} (0,.75);
	\draw [very thick] (0,-4.5) node[below,yshift=2pt]{$k$} to (0,-2.75);
\end{tikzpicture}
+ (-1)^k
\tfrac{[2]_{2n{-}2k{-}1}}{[2]_{2n{-}1}}
\begin{tikzpicture}[scale=.175,tinynodes, anchorbase]
	\draw [very thick] (0,-4.5) node[below,yshift=2pt]{$k$} to (0,2.5) node[above,yshift=-3pt]{$k$};
	\draw[very thick] (1.5,-1) node[right=2pt]{\scs$1$} circle (.75);
\end{tikzpicture} \\
&\stackrel{\eqref{eq:lolli1},\eqref{eq:digon1},\eqref{eq:circle1}}{=}
\left(
(-1)^{k-1}
\tfrac{``[k]^2"[[2]_{2n{-}2k{-}1}}{[2]_{2n{-}2k{+}1}} 
+ (-1)^k
\tfrac{[2n+1][2]_{2n{-}2k{-}1}}{[2]}
	\right)
\begin{tikzpicture}[scale=.175,tinynodes, anchorbase]
	\draw [very thick] (0,-4.5) node[below,yshift=2pt]{$k$} to (0,2.5) node[above,yshift=-3pt]{$k$};
\end{tikzpicture} \;.
	\end{align*}
From here, \eqref{eq:otherdigonB} is an unilluminating exercise in quantum arithmetic.

Relations \eqref{eq:ggb-triangle+} and \eqref{eq:ggb-triangle-} 
then follow by postcomposing \eqref{eq:blackspinH=I} with an appropriate trivalent vertex to match the left-hand side of the equations. 
The right-hand side will have two terms, one of which is zero by \eqref{eq:zerodigonB}, 
and the other can be resolved using \eqref{eq:digon1} or \eqref{eq:otherdigonB}; no quantum arithmetic required.
	\end{proof}

\begin{rem} \label{rem:2nplus1minusk} 
Following the idea of Remark \ref{rem:morelabels}, 
one might naively try to resolve the left-hand side of \eqref{eq:otherdigonB} by replacing $k$ with $2n+1-k$, 
replacing $k+1$ with $2n-k$, and using \eqref{eq:digon1}. 
This would produce the scalar $(-1)^{2n-k} ``[2n+1-k]^2"$, 
which disagrees with the scalar from \eqref{eq:otherdigonB} by a factor of $\frac{[2]_{2n-2k-1}}{[2]_{2n-2k+1}}$. 
This rescaling factor (invisible at $q=1$) also appears in \eqref{eq:blackH=I}, 
and the individual factors $[2]_{2n-2k\pm 1}$ appear frequently in relations that involve strands labeled both $k+1$ and $k-1$. \end{rem}

In order to establish further relations, 
we find it convenient to introduce additional trivalent vertices, 
called \emph{flow vertices},
for which we need to divide by quantum numbers of the form $``[k]^2"$. 
For $k,\ell$ satisfying $k+\ell \leq n-1$
we let
\begin{equation}\label{eq:flowvertex}
\begin{tikzpicture}[scale =.5, smallnodes,anchorbase]
	\draw[very thick] (0,0) node[below=-1pt]{$k$} to [out=90,in=210] (.5,.75);
	\draw[very thick] (1,0) node[below=-1pt]{$\ell$} to [out=90,in=330] (.5,.75);
	\draw[very thick] (.5,.75) to (.5,1.5) node[above=-2pt]{$k{+}\ell$};
\end{tikzpicture}
:=
(-1)^{k-1}
\frac{1}{``[k]^2"}
\begin{tikzpicture}[scale=.3,smallnodes,anchorbase]
	\draw[very thick] (-1,0) to node[below=-1pt]{$k{-}1$} (1,0);
	\draw[very thick] (-1,0) to node[left,xshift=2pt]{$1$} (0,1.732);
	\draw[very thick] (1,0) to node[right,xshift=-2pt]{$k{+}\ell{-}1$} (0,1.732);
	\draw[very thick] (0,1.732) to (0,3.232) node[above=-2pt]{$k{+}\ell$};
	\draw[very thick] (-2.3,-.75) node[below=-2pt,xshift=-2pt]{$k$} to (-1,0);
	\draw[very thick] (2.3,-.75) node[below=-2pt,xshift=2pt]{$\ell$} to (1,0);
\end{tikzpicture}
=
(-1)^{\ell-1}
\frac{1}{``[\ell]^2"}
\begin{tikzpicture}[scale=.3,smallnodes,anchorbase]
	\draw[very thick] (-1,0) to node[below=-1pt]{$\ell{-}1$} (1,0);
	\draw[very thick] (-1,0) to node[left,xshift=2pt]{$k{+}\ell{-}1$} (0,1.732);
	\draw[very thick] (1,0) to node[right,xshift=-2pt]{$1$} (0,1.732);
	\draw[very thick] (0,1.732) to (0,3.232) node[above=-2pt]{$k{+}\ell$};
	\draw[very thick] (-2.3,-.75) node[below=-2pt,xshift=-2pt]{$k$} to (-1,0);
	\draw[very thick] (2.3,-.75) node[below=-2pt,xshift=2pt]{$\ell$} to (1,0);
\end{tikzpicture} \, .
\end{equation}
The equality in \eqref{eq:flowvertex} is not obvious, and is part of the content of the following proposition.

\begin{prop}
The flow vertices in \eqref{eq:flowvertex} are well-defined and 
satisfy the relations \\
\begin{minipage}{.5\textwidth} 
\begin{equation}
	\label{eq:flowassoc}
	\begin{tikzpicture}[scale=.2, xscale=-1,tinynodes, anchorbase]
		\draw [very thick] (-1,-1) node[below=-2pt]{$k$} to [out=90,in=210] (0,.75);
		\draw [very thick] (1,-1) node[below=-2pt]{$\ell$} to [out=90,in=330] (0,.75);
		\draw [very thick] (3,-1) node[below=-2pt]{$m$} to [out=90,in=330] (1,2.5);
		\draw [very thick] (0,.75) to [out=90,in=210] (1,2.5);
		\draw [very thick] (1,2.5) to (1,4.25) node[above=-2pt]{$k{+}\ell{+}m$};
	\end{tikzpicture}
	=
	\begin{tikzpicture}[scale=.2, tinynodes, anchorbase]
		\draw [very thick] (-1,-1) node[below=-2pt]{$k$} to [out=90,in=210] (0,.75);
		\draw [very thick] (1,-1) node[below=-2pt]{$\ell$} to [out=90,in=330] (0,.75);
		\draw [very thick] (3,-1) node[below=-2pt]{$m$} to [out=90,in=330] (1,2.5);
		\draw [very thick] (0,.75) to [out=90,in=210] (1,2.5);
		\draw [very thick] (1,2.5) to (1,4.25) node[above=-2pt]{$k{+}\ell{+}m$};
	\end{tikzpicture}
	\end{equation}
\end{minipage}
\begin{minipage}{.5\textwidth}
\begin{equation}
	\label{eq:flowdigon}
	\begin{tikzpicture}[scale=.175,tinynodes, anchorbase]
		\draw [very thick] (0,.75) to (0,2.5) node[above,yshift=-3pt]{$k{+}\ell$};
		\draw [very thick] (0,-2.75) to [out=30,in=330] node[right,xshift=-2pt]{$\ell$} (0,.75);
		\draw [very thick] (0,-2.75) to [out=150,in=210] node[left,xshift=2pt]{$k$} (0,.75);
		\draw [very thick] (0,-4.5) node[below,yshift=2pt]{$k{+}\ell$} to (0,-2.75);
	\end{tikzpicture}
	= (-1)^{k \ell} {k+\ell \brack k}_{q^2}
	\begin{tikzpicture}[scale=.175,tinynodes, anchorbase]
		\draw [very thick] (0,-4.5) node[below,yshift=2pt]{$k{+}\ell$} to (0,2.5);
	\end{tikzpicture}
	\end{equation}
\end{minipage}
for $k+\ell+m \leq n-1$ or $k+\ell \leq n-1$ respectively.
	\end{prop}
	
\begin{proof}
An analogous result was proven for type $C$ in \cite[Lemma 3.4]{BERT}. 
The proof can be copied almost verbatim, adding signs, and replacing quantum numbers in $q$ with quantum numbers in $q^2$. 
The difference originates in the digon relation \eqref{eq:digon1}, which involves $``[k]^2"=[k]_{q^{2}}$ 
(c.f.~\cite[Equation (1.2c)]{BERT}).
This leads to the present formulae \eqref{eq:flowvertex} and \eqref{eq:flowdigon}
(c.f.~\cite[Equations (3.3) and (3.4)]{BERT}).
\end{proof}

In the remainder of this section we prove formulae which do not involve flow vertices, 
though the proofs are simplified by the use of flow vertices. 
In the spirit of Remark \ref{rem:integralform}, we point out that
the proofs thus implicitly (or explicitly) divide by the values $``[k]^2"$. 
This is the only purpose for which we use flow vertices in this paper, 
though they are intrinsically important.
	
\begin{lem}
	\label{lem:hardtriangle}
For $1 \leq k \leq n-1$,
\begin{equation}\label{eq:blacktriangle}
 \, .
		\end{equation}
\end{lem}
\begin{proof}
We assume that $n \geq 3$ (this allows the use of $2$-labeled edges), 
since otherwise $k=1$ and the result holds tautologically.
Let $\tau_k$ denote the scalar on the right-hand side of \eqref{eq:blacktriangle}.
We proceed by induction on $k$, with the base $k=1$ case holding trivially.
We then compute
\begingroup
\allowdisplaybreaks
\begin{align*}
%
 \, .
	\end{align*}
\endgroup
The result follows since 
$\tau_k 
= \Big( \tfrac{[4n-4k+2][2n-2k+3]}{[2n-2k+1][4n-4k+6]} -\tfrac{[4]}{[2]} \Big) \tau_{k-1}$, 
which can be confirmed via direct computation.
	\end{proof}

\begin{prop}\label{prop:hardrel}
The following relations hold in $\Web(\son)$:
\begin{equation}
	\label{eq:refblackspinH=I}
%
 \, .
\end{equation}
	\end{prop}
\begin{proof}
For \eqref{eq:refblackspinH=I}, we proceed inductively, 
with the base $k=1$ case provided by \eqref{eq:blackspinH=I}.
We then compute
\begin{align*}
%
 \, .
	\end{align*}
Having established \eqref{eq:refblackspinH=I}, 
we deduce \eqref{eq:otherggb-triangle} by composing with appropriate 
trivalent vertices and applying \eqref{eq:digon1}, \eqref{eq:zerodigonB}, 
and \eqref{eq:otherdigonB}.

Next, we consider \eqref{eq:zerodigonS}. 
Without loss of generality, it suffices to establish this in the case that $\ell < k$.
When $\ell = 0$, this is simply \eqref{eq:lolliS}, 
and when $k=n-1$,
this is given by
\[
%
\stackrel{\eqref{eq:n-1triangle}}{=} 0 \, .
	\]
It remains to establish the relation for all $0<\ell<k<n-1$.
This follows by induction, since the computation
\[
%
\right)
	\]
shows that \eqref{eq:zerodigonS} holds for the pair $(k,\ell)$ if it holds 
for both $(k+1,\ell-1)$ and $(k-1,\ell-1)$.
Having proved \eqref{eq:zerodigonS}, 
equation \eqref{eq:blackrungtriangle} follows by composing \eqref{eq:spinH=I} 
with appropriate trivalent vertices and equation \eqref{eq:refn-1triangle} 
follows from \eqref{eq:refblackspinH=I}.

Finally, the proof of \eqref{eq:altblackH=I} employs \eqref{eq:blacktriangle} via 
\eqref{eq:blackH=I}, and is similar to the proof of \eqref{eq:refblackspinH=I}. 
Since this relation plays no role in the remainder of the paper, 
we leave the details to the reader.
	\end{proof}

%
\section{Outline of the proof} \label{S:outline}
%

We now give a condensed and informal overview of the proof of Theorem \ref{thm:main}. 

\subsection{Step 1: Produce the functor}
First, we show the existence of a pivotal, 
$\C(q)$-linear, essentially surjective functor $\varphi \colon \Web(\son) \to \FRep(U_q(\son))$.
To begin, we note that the images of our generating trivalent morphisms and of the cap/cup morphisms 
(implicit in the pivotal structure) lie in $1$-dimensional morphism spaces in $\FRep(U_q(\son))$. 
Hence, the choice of where $\varphi$ sends these morphisms \eqref{eq:WebGen} is unique up to scalar. 

It remains to show that, for an appropriate choice of these scalars, 
our relations in \eqref{eq:webrel} are satisfied.
Here, most of the heavy lifting has been done in our prior work \cite{BER} and the work of 
the first named author and Wu \cite{BodWu}. 
Indeed, in \cite[Section 4]{BER} we exhibit explicit images of ``gray'' cap, cup, and ``gray-gray-black'' trivalent morphisms 
so that 
\eqref{eq:digonS}, 
\eqref{eq:lolliS}, 
and \eqref{eq:spinH=I} are satisfied.
Meanwhile, in \cite{BodWu}, images of ``all black'' cap, cup, and trivalent morphisms are determined 
for which, after a slight adjustment, relations 
\eqref{eq:digon1}, 
\eqref{eq:assoc}, 
and \eqref{eq:blackH=I} hold.
It thus remains to show that these choices can be made compatibly, 
and to verify that 
\eqref{eq:n-1triangle} and \eqref{eq:blackspinH=I} hold 
in the image. 
This is done in \S \ref{SS:functor} and Appendix \ref{Appendix:comparing-cups-and-caps}.

Essential surjectivity is a statement about objects and is evident from the definition of $\varphi$. 

\subsection{Step 2: Analyze the braiding}\label{ss:O2}

The $R$-matrix for $U_q(\son)$ gives rise to isomorphisms $R_{V,V'} \colon V \ot V' \to V' \ot V$ in $\Rep(U_q(\son))$ 
which equip the category $\Rep(U_q(\son))$ with the structure of a braided monoidal category. 
As a consequence, the morphisms $R_{V,V'}$ satisfy a number of desirable properties, 
for example invertibility (Reidemeister II), the braid relation (Reidemeister III), and, generalizing the latter,
naturality (e.g.~the ``fork-sliding'' relation; see \cite[Relation 11]{RT1}).

Inspired by prior work of the authors (and Wu) \cite{BERT,BER,BodWu}, 
it is easy to see how some of these braiding morphisms $R_{V,V'}$ 
lift to linear combinations of webs in $\Web(\son)$. 
Specifically, we provide the linear combinations of webs $c_{1,1}$, $c_{1,\gS}$, $c_{\gS,1}$, and $c_{\gS,\gS}$, 
lifting the braiding morphisms $R_{V_1,V_1}$, $R_{V_1,S}$, $R_{S, V_1}$, and $R_{S,S}$ respectively.
We make explicit use of the first three of these morphisms in the proof of fully faithfulness of our functor, 
and it is hence necessary to show that these maps satisfy the same desirable properties in 
$\Web(\son)$ that $R_{V,V'}$ enjoy in $\Rep(U_q(\son))$. 
Once established, we may use ``topological'' arguments in our simplification of diagrams.

While the morphism $c_{\gS,\gS}$ does not play an explicit role in our proof, 
it informs some of our algebraic arguments given in \S \ref{ss:FF}. 
For this reason, we also discuss this morphism and present a trick to 
quickly derive some of its desirable properties.

\subsection{Step 3: Reduce to a full subcategory and ladderize}\label{ss:ladder}

Having produced an essentially surjective functor $\varphi \colon \Web(\son) \to \FRep(U_q(\son))$, 
it remains to show that $\varphi$ is fully faithful. 
Our first step is to use monoidal duality (i.e.~cup and cap morphisms, see \eqref{eqn:monoidalduality} below) 
and an elementary result
(stated as Lemma \ref{lem:reduction} below and proved in \cite[Lemma 5.5]{BERT})
to reduce to proving that,
for each $m \geq 1$, the map
\begin{equation}
	\label{eq:spinalgiso?}
\End_{\Web(\son)}(\gS^{\otimes m}) \xrightarrow{\varphi}  \End_{U_q(\son)}(S^{\otimes m})
\end{equation}
is an algebra isomorphism.

To help with this latter task, 
we consider the \emph{ladder subalgebra} $\Lad_m \subset \End_{\Web(\son)}(\gS^{\otimes m})$, 
defined as the $\C(q)$-subalgebra generated by webs called \emph{rungs}:
\[
\rung_i := 
\begin{tikzpicture}[scale=.5, tinynodes, anchorbase]
	\draw[thick, black] (1,0) to node[below=-1pt]{\scs$1$} (2,0);
	\draw[very thick,gray] (0,-1) to (0,1);
	\node at (.5,0) {$\mydots$};
	\draw[very thick,gray] (1,-1) to (1,1);
	\draw[very thick,gray] (2,-1) to (2,1);
	\node at (2.5,0) {$\mydots$};
	\draw[very thick,gray] (3,-1) to (3,1);	
\end{tikzpicture} \; .
\]
Here the $1$-labeled ``rung'' $r_i$ lies between the $i$-th and $(i+1)$-st $\gS$-labeled ``uprights.''
We then employ a diagrammatic/topological argument 
using properties of the morphisms $c_{1,1}$, $c_{1, \gS}$, and $c_{\gS, 1}$ established in Step 2
to show that the inclusion 
$\Lad_m \hookrightarrow \End_{\Web(\son)}(\gS^{\otimes m})$ is in fact an equality. 
This ``ladderization'' result is Theorem \ref{thm:ladder} and Corollary \ref{cor:ladder0} below.

\subsection{Step 4: Use duality with the $\iota$quantum group $\iqUsom$
	to show fully faithfulness}

Let $\iqUsom$ be the non-standard quantization of $\som$, 
defined in work of Gavrilik--Klimyk \cite{GK-Usom}.
Equivalently, this is the $\iota$\emph{quantum group} associated with 
the split symmetric pair $(\slm, \som)$. 

In \cite[Theorem 5.2]{Wenzl-Spin}, 
Wenzl constructs a surjective $\C(q)$-algebra homomorphism
\begin{equation}\label{E:wenzls-map}
\iqUsom \twoheadrightarrow \End_{U_q(\son)}(S^{\otimes m}) \, .
\end{equation}
Further, he describes the $\iqUsom$ representation $S^{\otimes m}$ as a direct sum of finite-dimensional irreducible $\iqUsom$ representations, 
such that the generators $b_1, \dots, b_{m-1}$ of $\iqUsom$ have a specific spectrum \cite[Theorem 5.3(a)]{Wenzl-Spin}.

In Proposition \ref{prop:iQtoLad}, we show that there is a surjective $\C(q)$-algebra morphism
\begin{equation}
	\label{eq:iqtoLad}
\psi \colon \iqUsom \twoheadrightarrow \Lad_m \, ,
\end{equation}
and our results from \cite{BER} imply that the composition
\begin{equation}
	\label{eq:iqtoSpin}
\iqUsom \xrightarrow{\psi} \Lad_m \xrightarrow{\varphi} \End_{U_q(\son)}(S^{\otimes m})
\end{equation}
agrees with Wenzl's homomorphism. 
Since \eqref{E:wenzls-map} is surjective, 
this implies that \eqref{eq:spinalgiso?} is indeed surjective.

It remains to show that the map $\varphi$ from \eqref{eq:spinalgiso?} is injective. 
For this, we first observe that \eqref{eq:iqtoLad} 
factors through a quotient $\iqUsom^{\leq n}$ of $\iqUsom$
defined in our previous work \cite{BER}.
As predicted in \cite[Remark 1.21]{BER}, 
we are now able to show that this quotient is finite-dimensional and semisimple.
This allows us to establish injectivity via the strategy outlined in loc.~cit., 
which we now recall.

Let $\psi^{\le n}$ be the induced map $\iqUsom^{\leq n} \twoheadrightarrow \Lad_m$. 
The composition $\varphi \circ \psi^{\le n}$ is a surjection 
\begin{equation}
	\label{eq:iqtoSpin2}
\iqUsom^{\leq n} \twoheadrightarrow \End_{U_q(\son)}(S^{\otimes m})
\end{equation}
between finite-dimensional semisimple algebras.  
In \S \ref{s:reptheory}, we use results from \cite{MR2143754} to show that
all finite-dimensional irreducible representations of $\iqUsom^{\leq n}$ 
appear as summands of $S^{\otimes m}$, so these representations factor through the map $\varphi \circ \psi^{\le n}$. 
This shows that $\varphi \circ \psi^{\le n}$ is injective, and thus so is $\psi^{\le n}$. 
As both $\varphi \circ \psi^{\le n}$ and $\psi^{\le n}$ are isomorphisms, so too is $\varphi$.

\medskip

In the remaining sections, we carry out the technical details of the proof just sketched.

\section{The essentially surjective functor}\label{SS:functor}

\begin{thm}\label{thm:functor}
There is a pivotal, essentially surjective functor $\varphi \colon \Web(\son) \to \FRep(U_q(\son))$.
\end{thm}

\begin{proof}
Given a subset $\mathsf{W}$ of the generating morphisms in \eqref{eq:WebGen}, 
let $\Free(\mathsf{W})$ denote the free $\C(q)$-linear pivotal category generated by this subset. 
(Implicitly, the objects in $\Free(\mathsf{W})$ are monoidally generated by the tensor factors of the source and target of morphism in $\mathsf{W}$, 
and these generators are assumed to be self-dual.)
Let
\[
\mathsf{W}_{\bullet} := \Big\{
\begin{tikzpicture}[scale=.375, tinynodes,anchorbase]
	\draw[very thick] (0,0) node[below]{$1$} to [out=90,in=210] (.5,.75);
	\draw[very thick] (1,0) node[below]{$k$} to [out=90,in=330] (.5,.75);
	\draw[very thick] (.5,.75) to (.5,1.5) node[above=-2pt]{$k{+}1$};
\end{tikzpicture}
\, , \, 
\begin{tikzpicture}[scale =.375, tinynodes,anchorbase]
	\draw[very thick] (0,0) node[below]{$k$} to [out=90,in=210] (.5,.75);
	\draw[very thick] (1,0) node[below]{$1$} to [out=90,in=330] (.5,.75);
	\draw[very thick] (.5,.75) to (.5,1.5) node[above=-2pt]{$k{+}1$};
\end{tikzpicture}
\Big\}_{k=1}^{n-2}
\, , \quad
\mathsf{W}_{\gray{\bullet}} := \Big\{
\begin{tikzpicture}[scale =.375, smallnodes,anchorbase]
	\draw[very thick, gray] (0,0) node[below]{$\gS$} to [out=90,in=210] (.5,.75);
	\draw[very thick, gray] (1,0) node[below]{$\gS$} to [out=90,in=330] (.5,.75);
	\draw[very thick] (.5,.75) to (.5,1.5) node[above=-2pt]{$k$};
\end{tikzpicture}
\Big\}_{k=1}^{n-1}
\, ,
\]
and denote $\FreeWeb(\son) := \Free(\mathsf{W}_{\bullet} \cup \mathsf{W}_{\gray{\bullet}})$.
The latter is simply the free pivotal category on all the generators in \eqref{eq:WebGen}.

The arguments\footnote{To quickly summarize that argument: 
since $\Rep(U_q(\son))$ is pivotal, we know that it affords some graphical calculus.
In our chosen pivotal structure from Convention \ref{conv:pivotal}, 
all objects of $\Rep(U_q(\son))$
are naturally self-dual (in the sense of \cite{Sel2}), 
so, in particular, one obtains an unoriented graphical calculus for $\FRep(U_q(\son))$.
We then obtain our functor by sending generators and cap/cup morphisms 
in $\FreeWeb(\son)$
to the unique (up to scalar multiple) morphisms in the corresponding 
$\Hom$-spaces in $\FRep(U_q(\son))$.} 
in the proof of \cite[Theorem 5.1]{BERT} apply 
mutatis mutandis to show that there is a pivotal functor 
\begin{equation}
	\label{eq:Free}
\hat{\varphi} \colon \FreeWeb(\son) \to \FRep(U_q(\son))
	\end{equation}
that sends the generating morphisms in \eqref{eq:WebGen} 
to morphisms in the $\Hom$-spaces
\[
\Hom_{U_q(\son)}(V_1 \otimes V_k,V_{k+1}) \, , \quad 
\Hom_{U_q(\son)}(V_k \otimes V_1,V_{k+1}) \, , \quad 
\Hom_{U_q(\son)}(S \otimes S, V_k)
\]
respectively. 
By Proposition \ref{prop:decomp}, 
each of these spaces is $1$-dimensional,
so such a functor amounts to choosing an appropriate multiple of any 
fixed non-zero element in these $\Hom$-spaces.
Similarly, the cup and cap morphisms
\footnote{Existence of these morphisms is implicit in the word ``pivotal,'' which in particular requires that all objects have duals. 
We emphasize, however, that a particular choice of cap/cup morphisms is \emph{not} part of the pivotal structure.}
in $\FreeWeb(\son)$ are sent to maps in the $1$-dimensional spaces
\[
\Hom_{U_q(\son)}(\C(q), V_{\varpi_k} \otimes V_{\varpi_k}) 
\text{ and }
\Hom_{U_q(\son)}(V_{\varpi_k} \otimes V_{\varpi_k}, \C(q)) \, . 
\]
While these latter morphisms cannot be chosen arbitrarily (the isotopy relations in the domain must be satisfied), 
they can be rescaled as long as the scalar multiple of the cup is equal 
to the inverse of the scalar multiple of the cap.
In all, it thus suffices to show that we may choose appropriate scalar multiples for 
the images of $\mathsf{W}_{\bullet}$, $\mathsf{W}_{\gray{\bullet}}$, and the cap/cup morphisms, so that 
the relations in \eqref{eq:webrel} are satisfied.

Our results in \cite[\S 4.2]{BER} imply that there is a pivotal functor
\begin{equation}
	\label{eq:FreeSpin}
\hat{\varphi}_{BER} \colon \Free(\mathsf{W}_{\gray{\bullet}}) \to \FRep(U_q(\son))
	\end{equation}
that factors through the quotient by 
\eqref{eq:digonS}, 
\eqref{eq:lolliS},
and \eqref{eq:spinH=I}.
Meanwhile, 
the results in \cite{BodWu} imply that there is a pivotal functor 
$\hat{\varphi}_{O} \colon \Free(\mathsf{W}_{\bullet}) \to \FRep(U_q(\son))$.
Therein the authors show that one may choose the images of the 
all black trivalent and cap/cup morphisms so that 
\eqref{eq:assoc} 
holds on the nose, 
and so that the relations \eqref{eq:digon1} and \eqref{eq:blackH=I} hold,
but with different signs. 

In Appendix \ref{Appendix:comparing-cups-and-caps}, we show that 
\[
\hat{\varphi}_{O}\bigg(
\begin{tikzpicture}[scale=.4, anchorbase,tinynodes]
	\draw[very thick] (0,1.25) to [out=90,in=270] (0,1.5) to [out=90,in=180] (.5,2.25) node[above=-2pt]{$k$}
		to [out=0,in=90] (1,1.5) to (1,1.25);
\end{tikzpicture}  
\bigg)
= (-1)^{\binom{k}{2}}\hat{\varphi}_{BER}\bigg(
\begin{tikzpicture}[scale=.4, anchorbase,tinynodes]
	\draw[very thick] (0,1.25) to [out=90,in=270] (0,1.5) to [out=90,in=180] (.5,2.25) node[above=-2pt]{$k$}
		to [out=0,in=90] (1,1.5) to (1,1.25);
\end{tikzpicture}  
\bigg) \, .
\]
Since both functors are pivotal, 
the black cups are also related by the scalar $(-1)^{\binom{k}{2}}$. 
Clearly, the functors in \eqref{eq:FreeBlack} and \eqref{eq:FreeSpin} 
do not agree on the intersection of their domains: 
i.e.~they do not take the same values on black caps and cups.
To remedy this, define a new pivotal functor 
\begin{equation}
	\label{eq:FreeBlack}
\hat{\varphi}_{BW} \colon \Free(\mathsf{W}_{\bullet}) \to \FRep(U_q(\son))
	\end{equation}
such that
\begin{gather*}
\hat{\varphi}_{BW}\bigg(
\begin{tikzpicture}[scale=.4, anchorbase,tinynodes]
	\draw[very thick] (0,1.25) to [out=90,in=270] (0,1.5) to [out=90,in=180] (.5,2.25) node[above=-2pt]{$k$}
		to [out=0,in=90] (1,1.5) to (1,1.25);
\end{tikzpicture}  
\bigg) 
:=
\hat{\varphi}_{BER}\bigg(
\begin{tikzpicture}[scale=.4, anchorbase,tinynodes]
	\draw[very thick] (0,1.25) to [out=90,in=270] (0,1.5) to [out=90,in=180] (.5,2.25) node[above=-2pt]{$k$}
		to [out=0,in=90] (1,1.5) to (1,1.25);
\end{tikzpicture}  
\bigg) 
= 
(-1)^{\binom{k}{2}}\hat{\varphi}_{O}\bigg(
\begin{tikzpicture}[scale=.4, anchorbase,tinynodes]
	\draw[very thick] (0,1.25) to [out=90,in=270] (0,1.5) to [out=90,in=180] (.5,2.25) node[above=-2pt]{$k$}
		to [out=0,in=90] (1,1.5) to (1,1.25);
\end{tikzpicture}  
\bigg) \, , \\
\hat{\varphi}_{BW}\Bigg(
\begin{tikzpicture}[scale=.375, tinynodes,anchorbase]
	\draw[very thick] (0,0) node[below=-2pt]{$1$} to [out=90,in=210] (.5,.75);
	\draw[very thick] (1,0) node[below=-2pt]{$k$} to [out=90,in=330] (.5,.75);
	\draw[very thick] (.5,.75) to (.5,1.5) node[above=-2pt]{$k{+}1$};
\end{tikzpicture}
\Bigg) 
:=
\hat{\varphi}_{O}\Bigg(
\begin{tikzpicture}[scale=.375, tinynodes,anchorbase]
	\draw[very thick] (0,0) node[below=-2pt]{$1$} to [out=90,in=210] (.5,.75);
	\draw[very thick] (1,0) node[below=-2pt]{$k$} to [out=90,in=330] (.5,.75);
	\draw[very thick] (.5,.75) to (.5,1.5) node[above=-2pt]{$k{+}1$};
\end{tikzpicture}
\Bigg) \, ,
\quad \text{and} \quad 
\hat{\varphi}_{BW}\Bigg(
\begin{tikzpicture}[scale =.375, tinynodes,anchorbase]
	\draw[very thick] (0,0) node[below=-2pt]{$k$} to [out=90,in=210] (.5,.75);
	\draw[very thick] (1,0) node[below=-2pt]{$1$} to [out=90,in=330] (.5,.75);
	\draw[very thick] (.5,.75) to (.5,1.5) node[above=-2pt]{$k{+}1$};
\end{tikzpicture}
\Bigg) 
:=
\hat{\varphi}_{O}\Bigg(
\begin{tikzpicture}[scale =.375, tinynodes,anchorbase]
	\draw[very thick] (0,0) node[below=-2pt]{$k$} to [out=90,in=210] (.5,.75);
	\draw[very thick] (1,0) node[below=-2pt]{$1$} to [out=90,in=330] (.5,.75);
	\draw[very thick] (.5,.75) to (.5,1.5) node[above=-2pt]{$k{+}1$};
\end{tikzpicture}
\Bigg) \, .
\end{gather*}
It is an easy exercise to check that the image of the 
caps, cups, and trivalent vertices under $\hat{\varphi}_{BW}$ 
satisfy relations 
\eqref{eq:digon1}, 
\eqref{eq:assoc},
and \eqref{eq:blackH=I} on the nose. 
Since the functors $\hat{\varphi}_{BER}$ and $\hat{\varphi}_{BW}$ agree on the intersection of their domains, 
we obtain a pivotal functor as in \eqref{eq:Free} that factors through the 
quotient by all of the relations in \eqref{eq:webrel}, 
save for \eqref{eq:blackspinH=I} and \eqref{eq:n-1triangle}.

It thus remains to show that the images under $\hat{\varphi}$ 
of these two relations hold in $\FRep(U_q(\son))$.
For \eqref{eq:n-1triangle}, this is straightforward: 
the image of the left-hand side lies in 
$\Hom_{U_q(\son)}(V_k \otimes V_1,V_{n-1})$, 
which for $k \neq n-2$ is zero by \eqref{eq:V1Vk}.
It thus suffices to establish the image of \eqref{eq:blackspinH=I}. 
This proof will last until the end of this chapter, with several intervening lemmata.

Fix $1 \leq k \leq n-2$. 
Using Proposition \ref{prop:decomp} and Schur's lemma, 
one deduces that $\Hom(V_k \otimes V_1, S \otimes S)$ is two-dimensional. 
Precisely, one element in a basis will be the projection map $V_k \otimes V_1 \to V_{k+1}$ 
composed with the inclusion map $V_{k+1} \to S \otimes S$ 
and the other will factor through $V_{k-1}$ instead. 
Up to non-zero scalar, these projection and inclusion maps agree with the images of our generators and their rotations. 
Hence, there exists some relation of the form
\begin{equation}
	\label{eq:lastrel}
\hat{\varphi} \Bigg(
\begin{tikzpicture}[scale=.25, rotate=90, tinynodes, anchorbase]
	\draw[very thick] (-1,0) node[below,yshift=2pt,xshift=2pt]{$1$} to (0,1);
	\draw[very thick, gray] (1,0) node[above,yshift=-4pt,xshift=2pt]{$\gS$} to (0,1);
	\draw[very thick] (0,2.5) to (-1,3.5) node[below,yshift=2pt,xshift=-2pt]{$k$};
	\draw[very thick, gray] (0,2.5) to (1,3.5) node[above,yshift=-4pt,xshift=-2pt]{$\gS$};
	\draw[very thick, gray] (0,1) to node[below,yshift=1pt]{$\gS$} (0,2.5);
	\end{tikzpicture}
		\Bigg)
= 
a_k \
\hat{\varphi} \Bigg(
\begin{tikzpicture}[scale=.2, tinynodes, anchorbase]
	\draw[very thick] (-1,0) node[below=-2pt]{$k$} to (0,1);
	\draw[very thick] (1,0) node[below=-2pt]{$1$} to (0,1);
	\draw[very thick,gray] (0,2.5) to (-1,3.5);
	\draw[very thick,gray] (0,2.5) to (1,3.5);
	\draw[very thick] (0,1) to node[right=-2pt]{$k{+}1$} (0,2.5);
	\end{tikzpicture}
		\Bigg)
+ b_k \
\hat{\varphi} \Bigg(
\begin{tikzpicture}[scale=.2, tinynodes, anchorbase]
	\draw[very thick] (-1,0) node[below=-2pt]{$k$} to (0,1);
	\draw[very thick] (1,0) node[below=-2pt]{$1$} to (0,1);
	\draw[very thick,gray] (0,2.5) to (-1,3.5);
	\draw[very thick,gray] (0,2.5) to (1,3.5);
	\draw[very thick] (0,1) to node[right=-2pt]{$k{-}1$} (0,2.5);
	\end{tikzpicture}
		\Bigg)
	\end{equation}
for scalars $a_k, b_k \in \C(q)$, 
which we now compute inductively.
While doing so, we use the relations in \eqref{eq:webrel}
which we have already confirmed hold in the codomain of the functor $\hat{\varphi}$. 
We may also use all the relations from Proposition \ref{prop:easyrel} and relation \eqref{eq:zerodigonS}. 
The proof of Proposition \ref{prop:easyrel} does use \eqref{eq:blackspinH=I} exactly once, 
in the proof of \eqref{eq:zerodigonB}, 
but \eqref{eq:zerodigonB} holds obviously after applying $\hat{\varphi}$ since $\Hom_{U_q(\son)}(V_k,V_\ell) = 0$ when $k \neq \ell$.
Similarly, \eqref{eq:zerodigonS} holds after applying $\hat{\varphi}$.
Thus \eqref{eq:blackspinH=I} is not required to confirm that the relations 
in Proposition \ref{prop:easyrel} and relation \eqref{eq:zerodigonS} hold after applying $\hat{\varphi}$.

By convention, we say that \eqref{eq:lastrel} holds for $k=0$, where $a_0=1$ and $b_0=0$.
This is our base case.
Next, we compose \eqref{eq:lastrel} 
on the top with
(the image under $\hat{\varphi}$ of) 
trivalent vertices in $\mathsf{W}_{\gray{\bullet}}$, yielding maps to $V_{k+1}$ or $V_{k-1}$, namely
\begin{equation}
	\label{eq:spintriangle1}
\hat{\varphi}\Bigg(
\begin{tikzpicture}[scale=.2,tinynodes,anchorbase]
	\draw[very thick,gray] (-1,0) to (1,0);
	\draw[very thick,gray] (-1,0) to (0,1.732);
	\draw[very thick,gray] (1,0) to (0,1.732);
	\draw[very thick] (0,1.732) to (0,3.232) node[above=-2pt]{$k{+}1$};
	\draw[very thick] (-2.3,-.75) node[below=-1pt]{$k$} to (-1,0);
	\draw[very thick] (2.3,-.75) node[below=-1pt]{$1$} to (1,0);
\end{tikzpicture}
	\Bigg)
=
a_k \
\hat{\varphi} \Bigg(
\begin{tikzpicture}[scale=.2, tinynodes, anchorbase]
	\draw[very thick] (-1,0) node[below=-2pt]{$k$} to (0,1);
	\draw[very thick] (1,0) node[below=-2pt]{$1$} to (0,1);
	\draw[very thick] (0,1) to node[right=-2pt]{$k{+}1$} (0,2.5);
	\draw[very thick,gray] (0,2.5) to [out=150,in=210] (0,4.5);
	\draw[very thick,gray] (0,2.5) to [out=30,in=330] (0,4.5);
	\draw[very thick] (0,4.5) to node[above]{$k{+}1$} (0,5.5);
	\end{tikzpicture}
		\Bigg)
+ b_k \
\hat{\varphi} \Bigg(
\begin{tikzpicture}[scale=.2, tinynodes, anchorbase]
	\draw[very thick] (-1,0) node[below=-2pt]{$k$} to (0,1);
	\draw[very thick] (1,0) node[below=-2pt]{$1$} to (0,1);
	\draw[very thick] (0,1) to node[right=-2pt]{$k{-}1$} (0,2.5);
	\draw[very thick,gray] (0,2.5) to [out=150,in=210] (0,4.5);
	\draw[very thick,gray] (0,2.5) to [out=30,in=330] (0,4.5);
	\draw[very thick] (0,4.5) to node[above]{$k{+}1$} (0,5.5);
	\end{tikzpicture}
		\Bigg)
\stackrel{\substack{\eqref{eq:digonS} \\ \eqref{eq:zerodigonS}}}{=}
(-1)^{\binom{n-k}{2}}
\dd_{n-k-1}
a_k
\hat{\varphi} \Bigg(
\begin{tikzpicture}[scale =.375, tinynodes,anchorbase]
	\draw[very thick] (0,0) node[below=-1pt]{$k$} to [out=90,in=210] (.5,.75);
	\draw[very thick] (1,0) node[below=-1pt]{$1$} to [out=90,in=330] (.5,.75);
	\draw[very thick] (.5,.75) to (.5,1.5) node[above=-2pt]{$k{+}1$};
\end{tikzpicture}
		\Bigg)
			\end{equation}
and
\begin{equation}
	\label{eq:spintriangle2}
\hat{\varphi}\Bigg(
\begin{tikzpicture}[scale=.2,tinynodes,anchorbase]
	\draw[very thick,gray] (-1,0) to (1,0);
	\draw[very thick,gray] (-1,0) to (0,1.732);
	\draw[very thick,gray] (1,0) to (0,1.732);
	\draw[very thick] (0,1.732) to (0,3.232) node[above=-2pt]{$k{-}1$};
	\draw[very thick] (-2.3,-.75) node[below=-1pt]{$k$} to (-1,0);
	\draw[very thick] (2.3,-.75) node[below=-1pt]{$1$} to (1,0);
\end{tikzpicture}
	\Bigg)
=
a_k \
\hat{\varphi} \Bigg(
\begin{tikzpicture}[scale=.2, tinynodes, anchorbase]
	\draw[very thick] (-1,0) node[below=-2pt]{$k$} to (0,1);
	\draw[very thick] (1,0) node[below=-2pt]{$1$} to (0,1);
	\draw[very thick] (0,1) to node[right=-2pt]{$k{+}1$} (0,2.5);
	\draw[very thick,gray] (0,2.5) to [out=150,in=210] (0,4.5);
	\draw[very thick,gray] (0,2.5) to [out=30,in=330] (0,4.5);
	\draw[very thick] (0,4.5) to node[above]{$k{-}1$} (0,5.5);
	\end{tikzpicture}
		\Bigg)
+ b_k \
\hat{\varphi} \Bigg(
\begin{tikzpicture}[scale=.2, tinynodes, anchorbase]
	\draw[very thick] (-1,0) node[below=-2pt]{$k$} to (0,1);
	\draw[very thick] (1,0) node[below=-2pt]{$1$} to (0,1);
	\draw[very thick] (0,1) to node[right=-2pt]{$k{-}1$} (0,2.5);
	\draw[very thick,gray] (0,2.5) to [out=150,in=210] (0,4.5);
	\draw[very thick,gray] (0,2.5) to [out=30,in=330] (0,4.5);
	\draw[very thick] (0,4.5) to node[above]{$k{-}1$} (0,5.5);
	\end{tikzpicture}
		\Bigg)
\stackrel{\substack{\eqref{eq:digonS} \\ \eqref{eq:zerodigonS}}}{=}
(-1)^{\binom{n-k+2}{2}}
\dd_{n-k+1}
b_k
\hat{\varphi} \Bigg(
\begin{tikzpicture}[scale =.375, tinynodes,anchorbase]
	\draw[very thick] (0,0) node[below=-1pt]{$k$} to [out=90,in=210] (.5,.75);
	\draw[very thick] (1,0) node[below=-1pt]{$1$} to [out=90,in=330] (.5,.75);
	\draw[very thick] (.5,.75) to (.5,1.5) node[above=-2pt]{$k{-}1$};
\end{tikzpicture}
		\Bigg) \, .
	\end{equation}
Using pivotality (i.e.~rotating), 
\eqref{eq:spintriangle2} yields
\begin{equation}
	\label{eq:spintriangle3}
\hat{\varphi}\Bigg(
\begin{tikzpicture}[scale=.2,tinynodes,anchorbase]
	\draw[very thick,gray] (-1,0) to (1,0);
	\draw[very thick,gray] (-1,0) to (0,1.732);
	\draw[very thick,gray] (1,0) to (0,1.732);
	\draw[very thick] (0,1.732) to (0,3.232) node[above=-2pt]{$k$};
	\draw[very thick] (-2.3,-.75) node[below=-1pt]{$1$} to (-1,0);
	\draw[very thick] (2.3,-.75) node[below=-1pt]{$k{-}1$} to (1,0);
\end{tikzpicture}
	\Bigg)
=
(-1)^{\binom{n-k+2}{2}}
\dd_{n-k+1}
b_k
\hat{\varphi} \Bigg(
\begin{tikzpicture}[scale =.375, tinynodes,anchorbase]
	\draw[very thick] (0,0) node[below=-1pt,xshift=-1pt]{$1$} to [out=90,in=210] (.5,.75);
	\draw[very thick] (1,0) node[below=-1pt,xshift=2pt]{$k{-}1$} to [out=90,in=330] (.5,.75);
	\draw[very thick] (.5,.75) to (.5,1.5) node[above=-2pt]{$k$};
\end{tikzpicture}
		\Bigg) \, .
			\end{equation}
			
We now argue that \eqref{eq:spintriangle3} also holds with the diagrams reflected horizontally. 
Having done so, we can compare the reflected \eqref{eq:spintriangle3} (for $k$) with \eqref{eq:spintriangle1} (for $k-1$) 
to obtain a relationship between $b_k$ and $a_{k-1}$. 
To compute the relationship between reflected diagrams, we instead ``flip'' the diagrams over
using the braiding. 
Flipping and reflecting are not the same operation\footnote{Generally, one expects flipping and reflecting to be related by ribbon twists.} 
but in this case they agree up to scalar, and we compute those scalars.
			
We precompose \eqref{eq:spintriangle3} with the braiding isomorphism 
$R_{k-1,1} \colon V_{k-1} \otimes V_1 \xrightarrow{\cong} V_1 \otimes V_{k-1}$ in $\Rep(U_q(\son))$. 
To resolve the right-hand side, we use the following lemma, a generalization of \cite[equation (2.7)]{BodWu}.

\begin{lem} For either $\ell = 1$ or $m = 1$, we have
\begin{equation}\label{generalblackuntwist}
\hat{\varphi} \Bigg(
\begin{tikzpicture}[scale =.375, tinynodes,anchorbase]
	\draw[very thick] (0,0) node[below=-2pt]{$m$} to [out=90,in=210] (.5,.75);
	\draw[very thick] (1,0) node[below=-2pt]{$\ell$} to [out=90,in=330] (.5,.75);
	\draw[very thick] (.5,.75) to (.5,1.5) node[above=-2pt]{$m{+}\ell$};
\end{tikzpicture}
		\Bigg)
\circ R_{\ell,m}
= (-q^{-2})^{m\ell}
\hat{\varphi} \Bigg(
\begin{tikzpicture}[scale =.375, tinynodes,anchorbase]
	\draw[very thick] (0,0) node[below=-2pt]{$\ell$} to [out=90,in=210] (.5,.75);
	\draw[very thick] (1,0) node[below=-2pt]{$m$} to [out=90,in=330] (.5,.75);
	\draw[very thick] (.5,.75) to (.5,1.5) node[above=-2pt]{$m{+}\ell$};
\end{tikzpicture}
		\Bigg) \, .
\end{equation}
\end{lem}

\begin{proof} The base case $m = \ell = 1$ follows from \cite[equation (2.7) and Theorem 4.10]{BodWu}. 
We now establish the case of $\ell = 1$ and general $m$ via induction on $m$. 
(The case of general $\ell$ and $m=1$ is analogous, so we omit the argument.)
In the following computation, we will notationally depict $R_{1,m}$ as the image under $\hat{\varphi}$ 
of an appropriately labeled positive crossing. 
We emphasize that we have not yet defined\footnote{We eventually do so in Remark \ref{rem:braiding}, 
but one could also establish an explicit formulae akin to \cite[Theorem 2.5]{HigginsWu} by following the proof of that result.}  
this $(1,m)$-labeled crossing in $\Web(\son)$; all calculations take place in $\FRep(U_q(\son))$.

Using this, and naturality of the braiding morphisms in $\FRep(U_q(\son))$ (abbreviated by nat.~below), 
we inductively compute
\begin{equation}\label{eq:blackforktwistproof}
\begin{aligned}
\hat{\varphi} \Bigg(
\begin{tikzpicture}[scale =.375, tinynodes,anchorbase]
	\draw[very thick] (0,0) node[below=-2pt]{$m$} to [out=90,in=210] (.5,.75);
	\draw[very thick] (1,0) node[below=-2pt]{$1$} to [out=90,in=330] (.5,.75);
	\draw[very thick] (.5,.75) to (.5,1.5) node[above=-2pt]{$m{+}1$};
\end{tikzpicture}
		\Bigg)
\circ R_{1,m}
&=:
\hat{\varphi} \Bigg(
\begin{tikzpicture}[scale =.375, tinynodes,anchorbase]
	\draw[very thick] (1,-1.5) node[below=-2pt]{$m$} to [out=90,in=270] (0,0) to [out=90,in=210] (.5,.75);
	\draw[overcross] (0,-1.5) to [out=90,in=270] (1,0);
	\draw[very thick] (0,-1.5) node[below=-2pt]{$1$} to [out=90,in=270] (1,0) to [out=90,in=330] (.5,.75);
	\draw[very thick] (.5,.75) to (.5,1.5) node[above=-2pt]{$m{+}1$};
\end{tikzpicture}
		\Bigg)
\stackrel{\eqref{eq:digon1}}{=} 
\frac{(-1)^{m-1}}{``[m]^2"}
\hat{\varphi} \Bigg(
\begin{tikzpicture}[scale =.375, tinynodes,anchorbase]
	\draw[very thick] (1,-2.5) node[below=-2pt]{$m$} to (1,-2.25);
	\draw[very thick] (1,-2.25) to [out=150,in=210] node[left=-3pt]{$1$} (1,-1.5);
	\draw[very thick] (1,-2.25) to [out=30,in=330] node[right=-3pt]{$m{-}1$} (1,-1.5);
	\draw[very thick] (1,-1.5) to [out=90,in=270] (0,0) to [out=90,in=210] (.5,.75);
	\draw[overcross] (0,-1.5) to [out=90,in=270] (1,0);
	\draw[very thick] (0,-2.5) node[below=-2pt]{$1$} to (0,-1.5) to [out=90,in=270] (1,0) to [out=90,in=330] (.5,.75);
	\draw[very thick] (.5,.75) to (.5,1.5) node[above=-2pt]{$m{+}1$};
\end{tikzpicture}
		\Bigg)
\stackrel{\text{nat.}}{=} 
\frac{(-1)^{m-1}}{``[m]^2"}
\hat{\varphi} \Bigg(
\begin{tikzpicture}[scale =.375, tinynodes,anchorbase]
	\draw[very thick] (1,-2.5) node[below=-2pt]{$m$} to (1,-2.25);
	\draw[very thick] (1,-2.25) to [out=150,in=210] node[pos=.75, left=-3pt]{$1$} (0,0);
	\draw[very thick] (1,-2.25) to [out=30,in=330] node[right=-3pt]{$m{-}1$} (0,0);
	\draw[very thick] (0,0) to [out=90,in=210] (.5,.75);
	\draw[overcross] (0,-2.25) to (0,-1.5) to [out=90,in=270] (1,0);
	\draw[very thick] (0,-2.5) node[below=-2pt]{$1$} to (0,-1.5) to [out=90,in=270] (1,0) to [out=90,in=330] (.5,.75);
	\draw[very thick] (.5,.75) to (.5,1.5) node[above=-2pt]{$m{+}1$};
\end{tikzpicture}
		\Bigg) \\
\stackrel{\eqref{eq:assoc}}{=}&
\frac{(-1)^{m-1}}{``[m]^2"}
\hat{\varphi} \Bigg(
\begin{tikzpicture}[scale =.375, tinynodes,anchorbase]
	\draw[very thick] (1,-2.5) node[below=-2pt]{$m$} to (1,-2.25);
	\draw[very thick] (1,-2.25) to [out=150,in=210] node[pos=.75, left=-3pt]{$1$} (.5,.75);
	\draw[very thick] (1,-2.25) to [out=30,in=210] node[pos=.25, right=-3pt]{$m{-}1$} (1,0);
	\draw[overcross] (0,-2.25) to (0,-1.5) to [out=90,in=270] (1.25,-.5);
	\draw[very thick] (0,-2.5) node[below=-2pt]{$1$} to (0,-1.5) to [out=90,in=270] (1.25,-.5) 
		to [out=90,in=330] (1,0) to [out=90,in=330] (.5,.75);
	\draw[very thick] (.5,.75) to (.5,1.5) node[above=-2pt]{$m{+}1$};
\end{tikzpicture}
		\Bigg)
\stackrel{\eqref{generalblackuntwist}}{=}
\frac{(q^{-2})^{m-1}}{``[m]^2"}
\hat{\varphi} \Bigg(
\begin{tikzpicture}[scale =.375, tinynodes,anchorbase]
	\draw[very thick] (1,-2.5) node[below=-2pt]{$m$} to (1,-2.25);
	\draw[very thick] (1,-2.25) to [out=150,in=210] node[pos=.75, left=-3pt]{$1$} (.5,.75);
	\draw[very thick] (1,-2.25) to [out=30,in=330] node[pos=.25, right=-3pt]{$m{-}1$} (1,0);
	\draw[overcross] (0,-2.25) to (0,-1.5) to [out=90,in=270] (.75,-.5);
	\draw[very thick] (0,-2.5) node[below=-2pt]{$1$} to (0,-1.5) to [out=90,in=270] (.75,-.5) 
		to [out=90,in=210] (1,0) to [out=90,in=330] (.5,.75);
	\draw[very thick] (.5,.75) to (.5,1.5) node[above=-2pt]{$m{+}1$};
\end{tikzpicture}
		\Bigg) \\
\stackrel{\eqref{eq:flowassoc}}{=}&
\frac{(q^{-2})^{m-1}}{``[m]^2"}
\hat{\varphi} \Bigg(
\begin{tikzpicture}[scale =.375, tinynodes,anchorbase]
	\draw[very thick] (1,-2.5) node[below=-2pt]{$m$} to (1,-2.25);
	\draw[very thick] (1,-2.25) to [out=150,in=210] node[pos=.75, left=-3pt]{$1$} (0,0) to [out=90,in=210] (.5,.75);
	\draw[very thick] (1,-2.25) to [out=30,in=330] node[right=-3pt]{$m{-}1$} (.5,.75);
	\draw[overcross] (0,-2.25) to (0,-1.5) to [out=90,in=270] (.25,-.5);
	\draw[very thick] (0,-2.5) node[below=-2pt]{$1$} to (0,-1.5) to [out=90,in=270] (.25,-.5) 
		to [out=90,in=330] (0,0);
	\draw[very thick] (.5,.75) to (.5,1.5) node[above=-2pt]{$m{+}1$};
\end{tikzpicture}
		\Bigg)
\stackrel{\eqref{generalblackuntwist}}{=}
\frac{-(q^{-2})^{m}}{``[m]^2"}
\hat{\varphi} \Bigg(
\begin{tikzpicture}[scale =.375, tinynodes,anchorbase]
	\draw[very thick] (1,-2.5) node[below=-2pt]{$m$} to (1,-2.25);
	\draw[very thick] (1,-2.25) to [out=150,in=330] node[left=-3pt]{$1$} (0,0) to [out=90,in=210] (.5,.75);
	\draw[very thick] (1,-2.25) to [out=30,in=330] node[right=-3pt]{$m{-}1$} (.5,.75);
	\draw[very thick] (0,-2.5) node[below=-2pt]{$1$} to [out=90,in=210] (0,0);
	\draw[very thick] (.5,.75) to (.5,1.5) node[above=-2pt]{$m{+}1$};
\end{tikzpicture}
		\Bigg)
\stackrel{\substack{\eqref{eq:flowassoc} \\ \eqref{eq:digon1}}}{=}
(-q^{-2})^{m}
\hat{\varphi} \Bigg(
\begin{tikzpicture}[scale =.375, tinynodes,anchorbase]
	\draw[very thick] (0,0) node[below=-2pt]{$1$} to [out=90,in=210] (.5,.75);
	\draw[very thick] (1,0) node[below=-2pt]{$m$} to [out=90,in=330] (.5,.75);
	\draw[very thick] (.5,.75) to (.5,1.5) node[above=-2pt]{$m{+}1$};
\end{tikzpicture}
		\Bigg) \, .
\end{aligned}
\end{equation}
(Here, our uses of \eqref{generalblackuntwist} are applications of the inductive hypothesis.)
\end{proof}

We pause to note the following, which we do not use.

\begin{lem} Equation \eqref{generalblackuntwist} holds for all $m, \ell \ge 0$, 
when interpreted using the flow vertices of \eqref{eq:flowvertex}. \end{lem}

\begin{proof}[Proof (sketch)]
The base case where $m=1$ (and $\ell$ is arbitrary) is the previous lemma. 
Now one inducts on $m$, following the exact procedure of the previous proof.
\end{proof}

Now we treat the left-hand side of the flipped-over \eqref{eq:spintriangle3}. 
Similar to the above, 
we will depict $R_{S,S}$ as the image under $\hat{\varphi}$ of an $\gS$-labeled positive crossing, 
and we again emphasize that this is merely notation: all computations take place in $\Rep(U_q(\son))$.
(But see Proposition \ref{prop:SpinBraiding} where we give a formula for this crossing in $\Web(\son)$ 
which is backwards-compatible with our usage here.)

Using the naturality of the braiding, we again compute
\begin{equation} \label{twistytriangle}
\hat{\varphi}\Bigg(
\begin{tikzpicture}[scale=.2,tinynodes,anchorbase]
	\draw[very thick,gray] (-1,0) to (1,0);
	\draw[very thick,gray] (-1,0) to (0,1.732);
	\draw[very thick,gray] (1,0) to (0,1.732);
	\draw[very thick] (0,1.732) to (0,3.232) node[above=-2pt]{$k$};
	\draw[very thick] (-2.3,-.75) node[below=-2pt]{$1$} to (-1,0);
	\draw[very thick] (2.3,-.75) node[below=-2pt]{$k{-}1$} to (1,0);
\end{tikzpicture}
	\Bigg)
\circ
R_{k-1,1}
=:
\hat{\varphi}\Bigg(
\begin{tikzpicture}[scale=.2,tinynodes,anchorbase]
	\draw[very thick] (1.5,-3) node[below=-2pt]{$1$} to [out=90,in=210] (-1,0);
	\draw[overcross] (-1.5,-3) to [out=90,in=330] (1,0);
	\draw[very thick] (-1.5,-3) node[below=-2pt]{$k{-}1$} to [out=90,in=330] (1,0);
	\draw[very thick,gray] (-1,0) to (1,0);
	\draw[very thick,gray] (-1,0) to (0,1.732);
	\draw[very thick,gray] (1,0) to (0,1.732);
	\draw[very thick] (0,1.732) to (0,3.232) node[above=-2pt]{$k$};
\end{tikzpicture}
	\Bigg)
\stackrel{\text{nat.}}{=}
\hat{\varphi}\Bigg(
\begin{tikzpicture}[scale=.2,tinynodes,anchorbase]
	\draw[very thick,gray] (1.5,-2.25) to [out=150,in=270] (-2,.5) to [out=90,in=210] (0,2.5);
	\draw[very thick,gray] (1.5,-2.25) to [out=30,in=270] (-.75,.5) to [out=90,in=180] (0,1.25);
	\draw[overcros] (-1.5,-2.25) to [out=30,in=270] (2,.5);
	\draw[overcros] (-1.5,-2.25) to [out=150,in=270] (.75,.5);
	\draw[very thick,gray] (-1.5,-2.25) to [out=30,in=270] (2,.5) to [out=90,in=330] (0,2.5);
	\draw[very thick,gray] (-1.5,-2.25) to [out=150,in=270] (.75,.5) to [out=90,in=0] (0,1.25) to [out=180,in=90] (-.75,.5);
	\draw[very thick] (0,2.5) to (0,3.25) node[above=-2pt]{$k$};
	\draw[very thick] (1.5,-3) node[below=-2pt]{$1$} to (1.5,-2.25);
	\draw[very thick] (-1.5,-3) node[below=-2pt]{$k{-}1$} to (-1.5,-2.25);
\end{tikzpicture}
	\Bigg)
\end{equation}
We now simplify this latter morphism using \cite[Lemma 4.11]{BER}, 
which states that
\begin{equation}\label{eq:untwistovercrossing}
\hat{\varphi} \Bigg(
\begin{tikzpicture}[scale=.375, tinynodes,anchorbase,rotate=180]
	\draw[very thick,gray] (1,-1.5) node[above=-2pt]{$\gS$} to [out=90,in=270] (0,0) to [out=90,in=210] (.5,.75);
	\draw[overcross] (0,-1.5) to [out=90,in=270] (1,0);
	\draw[very thick,gray] (0,-1.5) node[above=-2pt]{$\gS$} to [out=90,in=270] (1,0) to [out=90,in=330] (.5,.75);
	\draw[very thick] (.5,.75) to (.5,1.5) node[below=-2pt]{$\ell$};
\end{tikzpicture}
		\Bigg)
:=
R_{S,S} \circ 
\hat{\varphi} \Bigg(
\begin{tikzpicture}[scale =.375, tinynodes,anchorbase,yscale=-1]
	\draw[very thick,gray] (0,0) node[above=-3pt]{$\gS$} to [out=90,in=210] (.5,.75);
	\draw[very thick,gray] (1,0) node[above=-3pt]{$\gS$} to [out=90,in=330] (.5,.75);
	\draw[very thick] (.5,.75) to (.5,1.5) node[below=-2pt]{$\ell$};
\end{tikzpicture}
		\Bigg)
= q^{\frac{2\ell-n}{2}} (-1)^{\binom{n-\ell+1}{2}} q^{-(n-\ell)^2}
\hat{\varphi} \Bigg(
\begin{tikzpicture}[scale =.375, tinynodes,anchorbase,yscale=-1]
	\draw[very thick,gray] (0,0) node[above=-3pt]{$\gS$} to [out=90,in=210] (.5,.75);
	\draw[very thick,gray] (1,0) node[above=-3pt]{$\gS$} to [out=90,in=330] (.5,.75);
	\draw[very thick] (.5,.75) to (.5,1.5) node[below=-2pt]{$\ell$};
\end{tikzpicture}
		\Bigg)
\end{equation}
for $0 \leq \ell \leq n-1$.
The $\ell=0$ case of \eqref{eq:untwistovercrossing} allows for an application of a $\gS$-colored Reidemeister I move in the right-hand side of \eqref{twistytriangle}. 
After then applying a $\gS$-colored Reidemeister III move, it becomes possible to apply \eqref{eq:untwistovercrossing} (with $\ell=k$)
and apply \eqref{eq:untwistovercrossing} for a negative crossing twice (with $\ell=1$ and $k-1$).
In the end, this gives
\begin{equation}
	\hat{\varphi}\Bigg(
	\begin{tikzpicture}[scale=.2,tinynodes,anchorbase]
		\draw[very thick,gray] (-1,0) to (1,0);
		\draw[very thick,gray] (-1,0) to (0,1.732);
		\draw[very thick,gray] (1,0) to (0,1.732);
		\draw[very thick] (0,1.732) to (0,3.232) node[above=-2pt]{$k$};
		\draw[very thick] (-2.3,-.75) node[below=-1pt]{$1$} to (-1,0);
		\draw[very thick] (2.3,-.75) node[below=-1pt]{$k{-}1$} to (1,0);
	\end{tikzpicture}
		\Bigg) \circ R_{k-1,1}
	=
	(-1)^{k-1} q^{2-2k}
\hat{\varphi}\Bigg(
\begin{tikzpicture}[scale=.2,tinynodes,anchorbase]
	\draw[very thick,gray] (-1,0) to (1,0);
	\draw[very thick,gray] (-1,0) to (0,1.732);
	\draw[very thick,gray] (1,0) to (0,1.732);
	\draw[very thick] (0,1.732) to (0,3.232) node[above=-2pt]{$k$};
	\draw[very thick] (-2.3,-.75) node[below=-1pt]{$k{-}1$} to (-1,0);
	\draw[very thick] (2.3,-.75) node[below=-1pt]{$1$} to (1,0);
\end{tikzpicture}
	\Bigg) \, .
\end{equation}
Observe that the scalar here agrees with the scaling factor from \eqref{generalblackuntwist} for $\ell=k-1$ and $m=1$. 
Dividing by this common factor, we obtain the horizontal flip of \eqref{eq:spintriangle3}, namely
\begin{equation}
	\label{eq:spintriangle4}
\hat{\varphi}\Bigg(
\begin{tikzpicture}[scale=.2,tinynodes,anchorbase]
	\draw[very thick,gray] (-1,0) to (1,0);
	\draw[very thick,gray] (-1,0) to (0,1.732);
	\draw[very thick,gray] (1,0) to (0,1.732);
	\draw[very thick] (0,1.732) to (0,3.232) node[above=-2pt]{$k$};
	\draw[very thick] (-2.3,-.75) node[below=-1pt]{$k{-}1$} to (-1,0);
	\draw[very thick] (2.3,-.75) node[below=-1pt]{$1$} to (1,0);
\end{tikzpicture}
	\Bigg)
=
(-1)^{\binom{n-k+2}{2}}
\dd_{n-k+1}
b_k
\hat{\varphi} \Bigg(
\begin{tikzpicture}[scale =.375, tinynodes,anchorbase]
	\draw[very thick] (0,0) node[below=-1pt,xshift=-1pt]{$k{-}1$} to [out=90,in=210] (.5,.75);
	\draw[very thick] (1,0) node[below=-1pt,xshift=2pt]{$1$} to [out=90,in=330] (.5,.75);
	\draw[very thick] (.5,.75) to (.5,1.5) node[above=-2pt]{$k$};
\end{tikzpicture}
		\Bigg) \, .
			\end{equation}
Together, \eqref{eq:spintriangle1}, \eqref{eq:spintriangle4},
and the equality $d_{n-k+1}/d_{n-k} = [2]_{2n-2k+1}$ imply that
\begin{equation}
	\label{eq:b=a}
(-1)^{n-k+1} [2]_{2n-2k+1} b_k =  a_{k-1}
\end{equation}
for $1 \leq k \leq n-2$. 
This also holds for $k=0$, provided we declare that $a_{-1}=0$.

Next, composing \eqref{eq:lastrel} on the bottom with 
(the image under $\hat{\varphi}$ of) 
trivalent vertices in $\mathsf{W}_{\bullet}$ yields the equations:
\begin{equation}
	\label{eq:othertriangles}
\begin{aligned}
\hat{\varphi}\Bigg(
\begin{tikzpicture}[scale=.2,tinynodes,anchorbase, rotate=180]
	\draw[very thick,gray] (-1,0) to (1,0);
	\draw[very thick] (-1,0) to node[right=-1pt,yshift=-1pt]{$1$} (0,1.732);
	\draw[very thick] (1,0) to node[left=-1pt,yshift=-1pt]{$k$} (0,1.732);
	\draw[very thick] (0,1.732) to (0,3.232) node[below=-2pt]{$k{+}1$};
	\draw[very thick,gray] (-2.3,-.75) node[above=-3pt]{$\gS$} to (-1,0);
	\draw[very thick,gray] (2.3,-.75) node[above=-3pt]{$\gS$} to (1,0);
\end{tikzpicture}
	\Bigg)
&=
a_k \
\hat{\varphi} \Bigg(
\begin{tikzpicture}[scale=.2, tinynodes, anchorbase]
	\draw[very thick] (0,-2) node[below=-2pt]{$k{+}1$} to (0,-1);
	\draw[very thick] (0,-1) to [out=150,in=210] node[left=-2pt]{$k$} (0,1);
	\draw[very thick] (0,-1) to [out=30,in=330] node[right=-2pt]{$1$} (0,1);
	\draw[very thick,gray] (0,2.5) to (-1,3.5);
	\draw[very thick,gray] (0,2.5) to (1,3.5);
	\draw[very thick] (0,1) to node[right=-2pt]{$k{+}1$} (0,2.5);
	\end{tikzpicture}
		\Bigg)
+ b_k \
\hat{\varphi} \Bigg(
\begin{tikzpicture}[scale=.2, tinynodes, anchorbase]
	\draw[very thick] (0,-2) node[below=-2pt]{$k{+}1$} to (0,-1);
	\draw[very thick] (0,-1) to [out=150,in=210] node[left=-2pt]{$k$} (0,1);
	\draw[very thick] (0,-1) to [out=30,in=330] node[right=-2pt]{$1$} (0,1);
	\draw[very thick,gray] (0,2.5) to (-1,3.5);
	\draw[very thick,gray] (0,2.5) to (1,3.5);
	\draw[very thick] (0,1) to node[right=-2pt]{$k{-}1$} (0,2.5);
	\end{tikzpicture}
		\Bigg)
\stackrel{\substack{\eqref{eq:digon1} \\ \eqref{eq:zerodigonB}}}{=}
(-1)^k
``[k+1]^2" 
a_k
\hat{\varphi} \Bigg(
\begin{tikzpicture}[scale =.375, tinynodes,anchorbase,yscale=-1]
	\draw[very thick,gray] (0,0) to [out=90,in=210] (.5,.75);
	\draw[very thick,gray] (1,0) to [out=90,in=330] (.5,.75);
	\draw[very thick] (.5,.75) to (.5,1.5) node[below=-2pt]{$k{+}1$};
\end{tikzpicture}
		\Bigg) \\
\hat{\varphi}\Bigg(
\begin{tikzpicture}[scale=.2,tinynodes,anchorbase, rotate=180]
	\draw[very thick,gray] (-1,0) to (1,0);
	\draw[very thick] (-1,0) to node[right=-1pt,yshift=-1pt]{$1$} (0,1.732);
	\draw[very thick] (1,0) to node[left=-1pt,yshift=-1pt]{$k$} (0,1.732);
	\draw[very thick] (0,1.732) to (0,3.232) node[below=-2pt]{$k{-}1$};
	\draw[very thick,gray] (-2.3,-.75) node[above=-3pt]{$\gS$} to (-1,0);
	\draw[very thick,gray] (2.3,-.75) node[above=-3pt]{$\gS$} to (1,0);
\end{tikzpicture}
	\Bigg)
&=
a_k \
\hat{\varphi} \Bigg(
\begin{tikzpicture}[scale=.2, tinynodes, anchorbase]
	\draw[very thick] (0,-2) node[below=-2pt]{$k{-}1$} to (0,-1);
	\draw[very thick] (0,-1) to [out=150,in=210] node[left=-2pt]{$k$} (0,1);
	\draw[very thick] (0,-1) to [out=30,in=330] node[right=-2pt]{$1$} (0,1);
	\draw[very thick,gray] (0,2.5) to (-1,3.5);
	\draw[very thick,gray] (0,2.5) to (1,3.5);
	\draw[very thick] (0,1) to node[right=-2pt]{$k{+}1$} (0,2.5);
	\end{tikzpicture}
		\Bigg)
+ b_k \
\hat{\varphi} \Bigg(
\begin{tikzpicture}[scale=.2, tinynodes, anchorbase]
	\draw[very thick] (0,-2) node[below=-2pt]{$k{-}1$} to (0,-1);
	\draw[very thick] (0,-1) to [out=150,in=210] node[left=-2pt]{$k$} (0,1);
	\draw[very thick] (0,-1) to [out=30,in=330] node[right=-2pt]{$1$} (0,1);
	\draw[very thick,gray] (0,2.5) to (-1,3.5);
	\draw[very thick,gray] (0,2.5) to (1,3.5);
	\draw[very thick] (0,1) to node[right=-2pt]{$k{-}1$} (0,2.5);
	\end{tikzpicture}
		\Bigg)
\stackrel{\substack{\eqref{eq:zerodigonB} \\ \eqref{eq:otherdigonB}}}{=}
(-1)^{k-1}
``[2n-k+2]^2" 
\frac{[2]_{2n-2k+1}}{[2]_{2n-2k+3}}
b_k
\hat{\varphi} \Bigg(
\begin{tikzpicture}[scale =.375, tinynodes,anchorbase,yscale=-1]
	\draw[very thick,gray] (0,0) to [out=90,in=210] (.5,.75);
	\draw[very thick,gray] (1,0) to [out=90,in=330] (.5,.75);
	\draw[very thick] (.5,.75) to (.5,1.5) node[below=-2pt]{$k{-}1$};
\end{tikzpicture}
	\Bigg) \, .
	\end{aligned}
		\end{equation}
If we instead compose \eqref{eq:lastrel} on the right side 
with the first generator in $\mathsf{W}_{\gray{\bullet}}$, 
this gives
\begin{equation}
	\label{eq:prequadratic}
(-1)^n \frac{[2n+1]}{[2]} 
\hat{\varphi} \Bigg(
\begin{tikzpicture}[scale =.375, tinynodes,anchorbase,yscale=-1]
	\draw[very thick, gray] (0,0) to [out=90,in=210] (.5,.75);
	\draw[very thick, gray] (1,0) to [out=90,in=330] (.5,.75);
	\draw[very thick] (.5,.75) to (.5,1.5) node[below=-2pt]{$k$};
	\end{tikzpicture}
		\Bigg)
\stackrel{\eqref{eq:otherdigonS}}{=}
\hat{\varphi}\Bigg(
\begin{tikzpicture}[scale=.2, tinynodes, anchorbase,rotate=90]
	\draw[very thick,gray] (1,-2) to [out=180,in=270] (0,-1);
	\draw[very thick] (0,-1) to [out=150,in=210] node[below=-2pt]{$1$} (0,1);
	\draw[very thick,gray] (0,-1) to [out=30,in=330] (0,1);
	\draw[very thick] (0,2.5) to (-1,3.5) node[below=-2pt]{$k$};
	\draw[very thick,gray] (0,2.5) to (1,3.5);
	\draw[very thick, gray] (0,1) to (0,2.5);
	\end{tikzpicture}
		\Bigg)
= a_k \hat{\varphi}\Bigg(
\begin{tikzpicture}[scale=.2,tinynodes,anchorbase, rotate=180]
	\draw[very thick,gray] (-1,0) to (1,0);
	\draw[very thick] (-1,0) to node[right=-1pt,yshift=-1pt]{$1$} (0,1.732);
	\draw[very thick] (1,0) to node[left=-1pt,yshift=-1pt]{$k{+}1$} (0,1.732);
	\draw[very thick] (0,1.732) to (0,3.232) node[below=-2pt]{$k$};
	\draw[very thick,gray] (-2.3,-.75) node[above=-3pt]{$\gS$} to (-1,0);
	\draw[very thick,gray] (2.3,-.75) node[above=-3pt]{$\gS$} to (1,0);
\end{tikzpicture}
	\Bigg)
+ b_k \hat{\varphi}\Bigg(
\begin{tikzpicture}[scale=.2,tinynodes,anchorbase, rotate=180]
	\draw[very thick,gray] (-1,0) to (1,0);
	\draw[very thick] (-1,0) to node[right=-1pt,yshift=-1pt]{$1$} (0,1.732);
	\draw[very thick] (1,0) to node[left=-1pt,yshift=-1pt]{$k{-}1$} (0,1.732);
	\draw[very thick] (0,1.732) to (0,3.232) node[below=-2pt]{$k$};
	\draw[very thick,gray] (-2.3,-.75) node[above=-3pt]{$\gS$} to (-1,0);
	\draw[very thick,gray] (2.3,-.75) node[above=-3pt]{$\gS$} to (1,0);
\end{tikzpicture}
	\Bigg) \, .
\end{equation}
Now, recall that $\Hom_{U_q(\son)}(V_k,S \otimes S)$ is one-dimensional, 
with basis vector
$\hat{\varphi} \Big(
\begin{tikzpicture}[scale =.25, tinynodes,anchorbase,yscale=-1]
	\draw[very thick, gray] (0,0) to [out=90,in=210] (.5,.75);
	\draw[very thick, gray] (1,0) to [out=90,in=330] (.5,.75);
	\draw[very thick] (.5,.75) to (.5,1.5) node[below=-2pt]{$k$};
	\end{tikzpicture}
		\Big)$. 
Expanding the right-hand side of \eqref{eq:prequadratic} using \eqref{eq:othertriangles}, 
and computing the coefficient of that basis vector, we obtain
\[ 
(-1)^n \frac{[2n+1]}{[2]} 
= a_k (-1)^k ``[2n-k+1]^2" \frac{[2]_{2n-2k-1}}{[2]_{2n-2k+1}} b_{k+1} + b_k (-1)^{k-1} ``[k]^2" a_{k-1} \, . \]
Applying \eqref{eq:b=a} and simplifying, we then get
\[
[2]_{2n-2k+1} \frac{[2n+1]}{[2]} = ``[2n-k+1]^2" \cdot a_k^2 + ``[k]^2" \cdot a_{k-1}^2
\]
for $0 \leq k \leq n-2$.
If $a_{k-1}^2 = 1$, then quantum arithmetic implies $a_k^2 = 1$ (an exercise for the reader). 
Since $a_0 = 1$, this implies that $a_k^2 =1$ for all $1 \leq k \leq n-2$.

We have thus shown that the equation
\[
\hat{\varphi} \Bigg(
\begin{tikzpicture}[scale=.25, rotate=90, tinynodes, anchorbase]
	\draw[very thick] (-1,0) node[below,yshift=2pt,xshift=2pt]{$1$} to (0,1);
	\draw[very thick, gray] (1,0) node[above,yshift=-4pt,xshift=2pt]{$\gS$} to (0,1);
	\draw[very thick] (0,2.5) to (-1,3.5) node[below,yshift=2pt,xshift=-2pt]{$k$};
	\draw[very thick, gray] (0,2.5) to (1,3.5) node[above,yshift=-4pt,xshift=-2pt]{$\gS$};
	\draw[very thick, gray] (0,1) to node[below,yshift=1pt]{$\gS$} (0,2.5);
	\end{tikzpicture}
		\Bigg)
= 
a_k \
\hat{\varphi} \Bigg(
\begin{tikzpicture}[scale=.2, tinynodes, anchorbase]
	\draw[very thick] (-1,0) node[below=-2pt]{$k$} to (0,1);
	\draw[very thick] (1,0) node[below=-2pt]{$1$} to (0,1);
	\draw[very thick,gray] (0,2.5) to (-1,3.5);
	\draw[very thick,gray] (0,2.5) to (1,3.5);
	\draw[very thick] (0,1) to node[right=-2pt]{$k{+}1$} (0,2.5);
	\end{tikzpicture}
		\Bigg)
+ (-1)^{n-k+1}
\frac{1}{[2]_{2n-2k+1}}
a_{k-1} \
\hat{\varphi} \Bigg(
\begin{tikzpicture}[scale=.2, tinynodes, anchorbase]
	\draw[very thick] (-1,0) node[below=-2pt]{$k$} to (0,1);
	\draw[very thick] (1,0) node[below=-2pt]{$1$} to (0,1);
	\draw[very thick,gray] (0,2.5) to (-1,3.5);
	\draw[very thick,gray] (0,2.5) to (1,3.5);
	\draw[very thick] (0,1) to node[right=-2pt]{$k{-}1$} (0,2.5);
	\end{tikzpicture}
		\Bigg)
			\]
holds in $\Rep(U_q(\son))$, where $a_0=1$ and $a_k = \pm1$ for $1 \leq k \leq n-2$.
We finally claim that, possibly after adjusting the functor $\hat{\varphi}$, 
we can assume that $a_k=1$ for all $0 \leq k \leq n-2$. 
Indeed, were this not the case for a particular value of $k$, 
we could redefine our functor to rescale the image of both generators
\[
\begin{tikzpicture}[scale=.375, tinynodes,anchorbase]
	\draw[very thick] (0,0) node[below=-1pt]{$1$} to [out=90,in=210] (.5,.75);
	\draw[very thick] (1,0) node[below=-1pt]{$k$} to [out=90,in=330] (.5,.75);
	\draw[very thick] (.5,.75) to (.5,1.5) node[above=-2pt]{$k{+}1$};
\end{tikzpicture}
\, , \, 
\begin{tikzpicture}[scale =.375, tinynodes,anchorbase]
	\draw[very thick] (0,0) node[below=-1pt]{$k$} to [out=90,in=210] (.5,.75);
	\draw[very thick] (1,0) node[below=-1pt]{$1$} to [out=90,in=330] (.5,.75);
	\draw[very thick] (.5,.75) to (.5,1.5) node[above=-2pt]{$k{+}1$};
\end{tikzpicture}
\]
by $-1$. 
This leaves all other relations invariant, 
as they all involve an even number of the rescaled trivalent vertices.
Hence, we obtain
a functor, denoted $\varphi$, that factors through the quotient by all relations in \eqref{eq:webrel}.
\end{proof}

\section{Some results concerning the braiding}

In this section, we consider morphisms $c_{i,j} \in \Web(\son)$ which are sent
to braiding morphisms $R_{i,j} \in \Rep(U_q(\son))$ under 
the functor $\varphi$ from Theorem \ref{thm:functor}. 
Using this, we show that these morphisms satisfy various 
Reidemeister-like relations for webs built from these morphisms. 
These relations are crucial in our subsequent arguments.

\subsection{Braidings involving $1$-labeled edges}\label{ss:1braid}

To begin, we define the following morphisms:
\begin{subequations}\label{eq:otherbraidings}
\begin{equation}\label{eq:blackbraiding}
\begin{aligned}
c_{1,1}:=
\begin{tikzpicture}[scale=.4, anchorbase,tinynodes]
	\draw[very thick] (1,0) to [out=90,in=270] (0,1.5);
	\draw[overcross] (0,0) to [out=90,in=270] (1,1.5);
	\draw[very thick] (0,0) to [out=90,in=270] (1,1.5);
\end{tikzpicture}
& := q^2 \
\begin{tikzpicture}[scale=.2, tinynodes, anchorbase]
	\draw[very thick] (-.75,0) to (-.75,3.5);
	\draw[very thick] (.75,0) to (.75,3.5);
	\end{tikzpicture}
+ 
\frac{(-1)^{n-1 \choose 2}}{``{\textstyle {2(n-2) \brack n-2}}"}
\begin{tikzpicture}[scale=.4, anchorbase,tinynodes]
	\draw[very thick, gray] (-.25,-.25) to (.25,-.25) to (.25,.25) to (-.25,.25) to (-.25,-.25);
	\draw[very thick] (-.5,-.75) to [out=90,in=225] (-.25,-.25);
	\draw[very thick] (.5,-.75) to [out=90,in=315] (.25,-.25);
	\draw[very thick] (-.5,.75) to [out=270,in=135] (-.25,.25);
	\draw[very thick] (.5,.75) to [out=270,in=45] (.25,.25);
\end{tikzpicture} 
+q^{-2}
\begin{tikzpicture}[scale=.4, tinynodes, anchorbase]
	\draw[very thick] (0,0) to [out=90,in=180] (.5,.625) 
		to [out=0,in=90] (1,0);
	\draw[very thick] (0,2) to [out=270,in=180] (.5,1.375)
		to [out=0,in=270] (1,2);
\end{tikzpicture} \\
&=
q^2 \
\begin{tikzpicture}[scale=.2, tinynodes, anchorbase]
	\draw[very thick] (-.75,0) to (-.75,3.5);
	\draw[very thick] (.75,0) to (.75,3.5);
	\end{tikzpicture}
+ 
\begin{tikzpicture}[scale=.2, tinynodes, anchorbase]
	\draw[very thick] (-1,0) to (0,1);
	\draw[very thick] (1,0) to (0,1);
	\draw[very thick] (0,2.5) to (-1,3.5);
	\draw[very thick] (0,2.5) to (1,3.5);
	\draw[very thick] (0,1) to node[right=-2pt]{$2$} (0,2.5);
\end{tikzpicture}
- \frac{q^{1-2n} (q^2-q^{-2})}{[2]_{2n-1}}
\begin{tikzpicture}[scale=.4, tinynodes, anchorbase]
	\draw[very thick] (0,0) to [out=90,in=180] (.5,.625) 
		to [out=0,in=90] (1,0);
	\draw[very thick] (0,2) to [out=270,in=180] (.5,1.375)
		to [out=0,in=270] (1,2);
\end{tikzpicture} \\
c_{1,1}' :=
\begin{tikzpicture}[scale=.4, anchorbase,xscale=-1]
	\draw[very thick] (1,0) to [out=90,in=270] (0,1.5);
	\draw[overcross] (0,0) to [out=90,in=270] (1,1.5);
	\draw[very thick] (0,0) to [out=90,in=270] (1,1.5);
\end{tikzpicture}
& := q^{-2} \
\begin{tikzpicture}[scale=.2, tinynodes, anchorbase]
	\draw[very thick] (-.75,0) to (-.75,3.5);
	\draw[very thick] (.75,0) to (.75,3.5);
	\end{tikzpicture}
+ 
\frac{(-1)^{n-1 \choose 2}}{``{\textstyle {2(n-2) \brack n-2}}"}
\begin{tikzpicture}[scale=.4, anchorbase,tinynodes]
	\draw[very thick, gray] (-.25,-.25) to (.25,-.25) to (.25,.25) to (-.25,.25) to (-.25,-.25);
	\draw[very thick] (-.5,-.75) to [out=90,in=225] (-.25,-.25);
	\draw[very thick] (.5,-.75) to [out=90,in=315] (.25,-.25);
	\draw[very thick] (-.5,.75) to [out=270,in=135] (-.25,.25);
	\draw[very thick] (.5,.75) to [out=270,in=45] (.25,.25);
\end{tikzpicture} 
+q^{2}
\begin{tikzpicture}[scale=.4, tinynodes, anchorbase]
	\draw[very thick] (0,0) to [out=90,in=180] (.5,.625) 
		to [out=0,in=90] (1,0);
	\draw[very thick] (0,2) to [out=270,in=180] (.5,1.375)
		to [out=0,in=270] (1,2);
\end{tikzpicture} \\
&= 
q^{-2} \
\begin{tikzpicture}[scale=.2, tinynodes, anchorbase]
	\draw[very thick] (-.75,0) to (-.75,3.5);
	\draw[very thick] (.75,0) to (.75,3.5);
	\end{tikzpicture}
+ 
\begin{tikzpicture}[scale=.2, tinynodes, anchorbase]
	\draw[very thick] (-1,0) to (0,1);
	\draw[very thick] (1,0) to (0,1);
	\draw[very thick] (0,2.5) to (-1,3.5);
	\draw[very thick] (0,2.5) to (1,3.5);
	\draw[very thick] (0,1) to node[right=-2pt]{$2$} (0,2.5);
\end{tikzpicture}
+ \frac{q^{2n-1} (q^2-q^{-2})}{[2]_{2n-1}}
\begin{tikzpicture}[scale=.4, tinynodes, anchorbase]
	\draw[very thick] (0,0) to [out=90,in=180] (.5,.625) 
		to [out=0,in=90] (1,0);
	\draw[very thick] (0,2) to [out=270,in=180] (.5,1.375)
		to [out=0,in=270] (1,2);
\end{tikzpicture}\\
\end{aligned}
\end{equation}
\begin{equation}\label{eq:blackgraybraiding}
\begin{gathered}
c_{\gS,1} :=
\begin{tikzpicture}[scale=.4, anchorbase]
	\draw[very thick] (1,0) to [out=90,in=270] (0,1.5);
	\draw[overcross] (0,0) to [out=90,in=270] (1,1.5);
	\draw[very thick, gray] (0,0) to [out=90,in=270] (1,1.5);
\end{tikzpicture}
:= 
q
\begin{tikzpicture}[scale=.2, rotate=90, tinynodes, anchorbase]
	\draw[very thick] (-1,0) to (0,1);
	\draw[very thick, gray] (1,0) to (0,1);
	\draw[very thick, gray] (0,2.5) to (-1,3.5);
	\draw[very thick] (0,2.5) to (1,3.5);
	\draw[very thick, gray] (0,1) to (0,2.5);
\end{tikzpicture}
+q^{-1}
\begin{tikzpicture}[scale=.2, tinynodes, anchorbase]
	\draw[very thick, gray] (-1,0) to (0,1);
	\draw[very thick] (1,0) to (0,1);
	\draw[very thick] (0,2.5) to (-1,3.5);
	\draw[very thick, gray] (0,2.5) to (1,3.5);
	\draw[very thick, gray] (0,1) to (0,2.5);
\end{tikzpicture}
\quad , \quad
c_{\gS,1}' :=
\begin{tikzpicture}[scale=.4, anchorbase,xscale=-1]
	\draw[very thick] (1,0) to [out=90,in=270] (0,1.5);
	\draw[overcross] (0,0) to [out=90,in=270] (1,1.5);
	\draw[very thick, gray] (0,0) to [out=90,in=270] (1,1.5);
\end{tikzpicture}
:= 
q^{-1}
\begin{tikzpicture}[scale=.2, rotate=90, tinynodes, anchorbase,xscale=-1]
	\draw[very thick] (-1,0) to (0,1);
	\draw[very thick, gray] (1,0) to (0,1);
	\draw[very thick, gray] (0,2.5) to (-1,3.5);
	\draw[very thick] (0,2.5) to (1,3.5);
	\draw[very thick, gray] (0,1) to (0,2.5);
\end{tikzpicture}
+q
\begin{tikzpicture}[scale=.2, tinynodes, anchorbase,xscale=-1]
	\draw[very thick, gray] (-1,0) to (0,1);
	\draw[very thick] (1,0) to (0,1);
	\draw[very thick] (0,2.5) to (-1,3.5);
	\draw[very thick, gray] (0,2.5) to (1,3.5);
	\draw[very thick, gray] (0,1) to (0,2.5);
\end{tikzpicture} \\
c_{1,\gS}:=
\begin{tikzpicture}[scale=.4, anchorbase]
	\draw[very thick,gray] (1,0) to [out=90,in=270] (0,1.5);
	\draw[overcross] (0,0) to [out=90,in=270] (1,1.5);
	\draw[very thick] (0,0) to [out=90,in=270] (1,1.5);
\end{tikzpicture}
:=
q
\begin{tikzpicture}[scale=.2, rotate=90, tinynodes, anchorbase,xscale=-1]
	\draw[very thick] (-1,0) to (0,1);
	\draw[very thick, gray] (1,0) to (0,1);
	\draw[very thick, gray] (0,2.5) to (-1,3.5);
	\draw[very thick] (0,2.5) to (1,3.5);
	\draw[very thick, gray] (0,1) to (0,2.5);
\end{tikzpicture}
+q^{-1}
\begin{tikzpicture}[scale=.2, tinynodes, anchorbase,xscale=-1]
	\draw[very thick, gray] (-1,0) to (0,1);
	\draw[very thick] (1,0) to (0,1);
	\draw[very thick] (0,2.5) to (-1,3.5);
	\draw[very thick, gray] (0,2.5) to (1,3.5);
	\draw[very thick, gray] (0,1) to (0,2.5);
\end{tikzpicture}
\quad , \quad
c_{1,\gS}' :=
\begin{tikzpicture}[scale=.4, anchorbase,xscale=-1]
	\draw[very thick,gray] (1,0) to [out=90,in=270] (0,1.5);
	\draw[overcross] (0,0) to [out=90,in=270] (1,1.5);
	\draw[very thick] (0,0) to [out=90,in=270] (1,1.5);
\end{tikzpicture}
:=
q^{-1}
\begin{tikzpicture}[scale=.2, rotate=90, tinynodes, anchorbase]
	\draw[very thick] (-1,0) to (0,1);
	\draw[very thick, gray] (1,0) to (0,1);
	\draw[very thick, gray] (0,2.5) to (-1,3.5);
	\draw[very thick] (0,2.5) to (1,3.5);
	\draw[very thick, gray] (0,1) to (0,2.5);
\end{tikzpicture}
+q
\begin{tikzpicture}[scale=.2, tinynodes, anchorbase]
	\draw[very thick, gray] (-1,0) to (0,1);
	\draw[very thick] (1,0) to (0,1);
	\draw[very thick] (0,2.5) to (-1,3.5);
	\draw[very thick, gray] (0,2.5) to (1,3.5);
	\draw[very thick, gray] (0,1) to (0,2.5);
\end{tikzpicture}
	\end{gathered}
	\end{equation}
	\end{subequations}
Here, and in the following, all unlabeled black strands carry the label $1$.

\begin{rem}\label{rem:n=2}
The second formulation of $c_{1,1}$ and $c_{1,1}'$, which uses
\eqref{eq:blackspinH=I}, \eqref{eq:digonS}, and \eqref{eq:lolliS} to write
\[
\begin{tikzpicture}[scale=.4, anchorbase,tinynodes]
	\draw[very thick, gray] (-.25,-.25) to (.25,-.25) to (.25,.25) to (-.25,.25) to (-.25,-.25);
	\draw[very thick] (-.5,-.75) to [out=90,in=225] (-.25,-.25);
	\draw[very thick] (.5,-.75) to [out=90,in=315] (.25,-.25);
	\draw[very thick] (-.5,.75) to [out=270,in=135] (-.25,.25);
	\draw[very thick] (.5,.75) to [out=270,in=45] (.25,.25);
\end{tikzpicture} 
= (-1)^{n-1 \choose 2}
``{\textstyle {2(n-2) \brack n-2}}"
\begin{tikzpicture}[scale=.2, tinynodes, anchorbase]
	\draw[very thick] (-1,0) to (0,1);
	\draw[very thick] (1,0) to (0,1);
	\draw[very thick] (0,2.5) to (-1,3.5);
	\draw[very thick] (0,2.5) to (1,3.5);
	\draw[very thick] (0,1) to node[right=-2pt]{$2$} (0,2.5);
\end{tikzpicture}
+ \frac{(-1)^{n+1 \choose 2} ``{\textstyle {2(n-1) \brack n-1}}"}{[2]_{2n-1}}
\begin{tikzpicture}[scale=.4, tinynodes, anchorbase]
	\draw[very thick] (0,0) to [out=90,in=180] (.5,.625) 
		to [out=0,in=90] (1,0);
	\draw[very thick] (0,2) to [out=270,in=180] (.5,1.375)
		to [out=0,in=270] (1,2);
\end{tikzpicture} \, ,  
\]
is only valid when $n>2$, 
since it involves $2$-labeled strands. 
Nevertheless, it is this second formulation that we will use in the proofs of the 
main results of this section (Lemmata \ref{lem:otherR2} and \ref{lem:1R13} and Proposition \ref{prop:S1}).
Hence, this is a step in the proof of Theorem \ref{thm:main} where we assume that $n>2$.

However, each of these results is straightforward to verify when $n=2$.
Having done so, our proof works uniformly for all $n>1$.
We comment on the stand-in for these results in the $n=1$ case below in Remark \ref{rem:n=1}.
	\end{rem}

\begin{lem}\label{lem:otherR2}
Each of the morphisms $c_{i,j}$ in \eqref{eq:otherbraidings} are invertible, 
with inverses $c_{i,j}'$, and $\varphi(c_{i,j}) = R_{i,j}$.
\end{lem}

\begin{proof}
The assertions that $\varphi(c_{i,j}) = R_{i,j}$ follow from 
\cite[Theorem 4.8]{BodWu} when $i=1=j$, 
\cite[Proposition 4.35]{BER} when $i=S,j=1$, 
and an argument analogous to the latter when $i=1,j=S$.
That $c_{1,1} c_{1,1}' = \id_{1 \ot 1} = c_{1,1}' c_{1,1}$ 
is a fairly straightforward computation using \eqref{eq:webrel}; 
however, it can also be proved using following trick that we dub the 
\emph{linear independence trick}.

This trick proceeds as follows:
First, it is an immediate consequence of \eqref{eq:digon1} and \eqref{eq:lolli1} 
that the compositions $c_{1,1} c_{1,1}'$ and $c_{1,1}' c_{1,1}$
may be written as linear combinations of the webs 
$\begin{tikzpicture}[scale=.15, tinynodes, anchorbase]
	\draw[very thick] (-.75,0) to (-.75,3.5);
	\draw[very thick] (.75,0) to (.75,3.5);
	\end{tikzpicture} \ $,
$\begin{tikzpicture}[scale=.15, tinynodes, anchorbase]
	\draw[very thick] (-1,0) to (0,1);
	\draw[very thick] (1,0) to (0,1);
	\draw[very thick] (0,2.5) to (-1,3.5);
	\draw[very thick] (0,2.5) to (1,3.5);
	\draw[very thick] (0,1) to node[right=-2pt]{$2$} (0,2.5);
\end{tikzpicture}$, and 
$\begin{tikzpicture}[scale=.25, tinynodes, anchorbase]
	\draw[very thick] (0,0) to [out=90,in=180] (.5,.625) 
		to [out=0,in=90] (1,0);
	\draw[very thick] (0,2) to [out=270,in=180] (.5,1.375)
		to [out=0,in=270] (1,2);
\end{tikzpicture}$. 
Using the $k=1$ case of the decomposition \eqref{eq:V1Vk}, 
it is easy to see that these three webs are sent by $\varphi$ 
to a basis of $\End_{U_q(\son)}(V_1 \ot V_1)$, thus are linearly independent in $\Web(\son)$.
After applying $\varphi$, we obtain $R_{1,1} R_{1,1}^{-1} = \id_{V_1 \ot V_1} = R_{1,1}^{-1} R_{1,1}$, 
thus the compositions $c_{1,1} c_{1,1}'$ and $c_{1,1}' c_{1,1}$ 
must have been equal to $\id_{1 \ot 1}$ before applying $\varphi$.
(Note: the efficiency in this argument is that one need not keep track of the coefficients 
in the relevant compositions.)

Similar arguments establish the equations 
$c_{1,\gS} c_{1,\gS}' = \id_{1 \otimes \gS} = c_{1,\gS}' c_{1,\gS}$, 
and
$c_{\gS,1} c_{\gS,1}' = \id_{\gS \otimes 1} = c_{\gS,1}' c_{\gS,1}$.
One can either perform a direct computation,
or one can 
argue that the compositions 
are in the spans of the diagrams
$\left\{ 
 
\right\}$
and then use the linear independence trick.
	\end{proof}
	
Lemma \ref{lem:otherR2} asserts that the $1$-labeled Reidemeister II move
\[
%
\]
holds in $\Web(\son)$, 
as do the variants where one of these strands is $\gS$-colored.
We next consider additional moves corresponding 
to isotopies of all-black colored tangles and tangled webs.
The following are established in \cite{BodWu} 
via direct computation and/or arguments analogous to those used in \cite{BERT}.

\begin{lem}\label{lem:1R13}
We have
\begin{equation}\label{eq:11crossingrotate}
%
 \, .
	\end{equation}
	\end{lem}
\begin{proof}
Equation \eqref{eq:11crossingrotate} is 
immediate from \eqref{eq:blackbraiding} 
(or see \cite[Proposition 2.9]{BodWu}) and
the relations \eqref{eq:1R1} and \eqref{eq:111R3} are given in \cite[Proposition 2.11]{BodWu}.
	\end{proof}
	
We next consider tangle and tangled web relations involving both $1$- and $\gS$-labeled strands.
	
\begin{prop}\label{prop:S1}
We have
\begin{equation}\label{eq:S1crossingrotate}
%
 \, .
	\end{equation}
	\end{prop}
\begin{proof}
The relations \eqref{eq:S1crossingrotate} are immediate from the definitions of the 
crossings in \eqref{eq:blackgraybraiding}.
Similarly, \eqref{eq:forktwist} follows from \eqref{eq:blackgraybraiding} using a 
direct computation employing \eqref{eq:otherdigonS} and \eqref{eq:blackrungtriangle}.

For \eqref{eq:1forkslide}, 
a direct computation via repeated use of 
\eqref{eq:blackspinH=I}, 
\eqref{eq:refblackspinH=I},
and \eqref{eq:blackrungtriangle}
gives that
\[
%
 \, .
\]
The relation then follows from the equality
\begin{equation}\label{eq:nicec11}
%
	\end{equation}
which can be confirmed by plugging the $k=1$ case of \eqref{eq:blackH=I} in 
for the second web appearing in \eqref{eq:nicec11} 
and comparing to \eqref{eq:blackbraiding}.

The derivation of \eqref{eq:S11R3} is more involved.
First, by expanding the black-black crossings using \eqref{eq:blackbraiding} 
and using the black-gray Reidemeister II invariance from Lemma \ref{lem:otherR2}, 
we see that it suffices to show
\begin{equation}\label{eq:S11MSslide}
%
 \, .
	\end{equation}
Observe that relation \eqref{eq:S11MSslide} will follow by writing the right-hand side
as a linear combination of webs that are invariant under a $180^{\circ}$ rotation.
For this, we must first establish some additional relations. 

To start, we compute
\[
q^{-2}
%
	\]
where the second equality uses both Lemma \ref{lem:otherR2} and \eqref{eq:1forkslide}.
From this, we conclude that
\begin{equation}\label{eq:preiQ}
%
 \, .
	\end{equation}
We now compute
\begin{align*}
%
 \, .
	\end{align*}
Since the first two webs here are invariant under $180^{\circ}$ rotation, 
it remains to show the same for the third.
This follows by computing
\[
%
	\]
which completes the proof.
	\end{proof}

\subsection{An aside on the spin-spin braiding}

In our previous work \cite{BER}, we gave simple and explicit formulae for the braiding morphisms
\[
R_{S,S} \colon S \otimes S \xrightarrow{\cong} S \otimes S
\, , \quad
R_{S,S}^{-1} \colon S \otimes S \xrightarrow{\cong} S \otimes S
\]
in $\Rep(U_q(\son))$. 
These formulae do not play a technical role in our arguments, 
but they do inform some of the results in \S \ref{ss:FF} below. 
Given this, we now recall the formulae and lift them to $\Web(\son)$.

To begin, the decomposition \eqref{eq:SSdecomp} implies that the 
functor $\varphi \colon \Web(\son) \to \FRep(U_q(\son))$
sends the webs:
\begin{equation} \label{SSbasis} \left\{ 
\begin{tikzpicture}[scale=.2, tinynodes, anchorbase]
	\draw[very thick,gray] (-.75,0) to (-.75,3.5);
	\draw[very thick,gray] (.75,0) to (.75,3.5);
	\end{tikzpicture} \right\} \cup \left\{
\begin{tikzpicture}[scale=.2, tinynodes, anchorbase]
	\draw[very thick,gray] (-1,0) to (0,1);
	\draw[very thick,gray] (1,0) to (0,1);
	\draw[very thick,gray] (0,2.5) to (-1,3.5);
	\draw[very thick,gray] (0,2.5) to (1,3.5);
	\draw[very thick] (0,1) to node[right=-2pt]{$k$} (0,2.5);
	\end{tikzpicture} \right\}_{k=0}^{n-1} 
\end{equation}
to a basis for $\End_{\son}(S \otimes S)$.
In \cite{BER}, we described the braiding morphisms $R_{S,S}$ and $R_{S,S}^{-1}$ in terms of this basis. 
Unfortunately, these coefficients are rather involved; 
see \cite[Proposition 4.21]{BER}.
However, motivated by our categorification results ibid., 
we found another basis for $\End_{\son}(S \otimes S)$ with respect to which 
the braiding admits a much simpler formula.
The following is a lift of that basis to $\Web(\son)$.

\begin{defn}\label{def:grayX}
Set $\xx^{(0)}:=\id_{\gS \otimes \gS}$, 
\begin{equation}
	\label{eq:x1defn}
\xx := \xx^{(1)} := 
\begin{tikzpicture}[scale=.5, tinynodes, anchorbase]
	\draw[very thick] (0,0) to node[below=-2pt]{$1$} (1,0);
	\draw[very thick,gray] (1,-1) to (1,1);
	\draw[very thick,gray] (0,-1) to (0,1);
\end{tikzpicture}
- \frac{1}{[2]} \,
\begin{tikzpicture}[scale=.5, tinynodes, anchorbase]
	\draw[very thick,gray] (1,-1) to (1,1);
	\draw[very thick,gray] (0,-1) to (0,1);
\end{tikzpicture}
\in \End_{\Web(\son)}(\gS \otimes \gS) \, ,
	\end{equation}
and, for $0 \leq i \leq n-1$, let
\begin{equation}
	\label{eq:xidefn}
\xx^{(i+1)} := \dfrac{(-1)^i}{``[i+1]^2"}\left(\xx^{(i)}\xx - (-1)^{i} ``[i][i+1]"\xx^{(i)}\right) \, .
\end{equation}
	\end{defn}

We again make use of the devil's arithmetic \eqref{eq:devils}. 
Unpacking the recursive formula of \eqref{eq:xidefn}, 
we see that
\begin{equation}
	\label{E:xxk-expanded}
\xx^{(i)} = (-1)^{\binom{i}{2}}\dfrac{(\xx+(-1)^i``[i-1][i]")\cdot (\xx+(-1)^{i-1}``[i-2][i-1]")\cdots (\xx+(-1)``[0][1]")}{``[i]^2"\cdot ``[i-1]^2"\cdots ``[1]^2"}
\end{equation}
for all $1 \leq i \leq n$.

It is clear from \eqref{eq:spinH=I} that $\xx^{(1)}$ lies in the span of \eqref{SSbasis}. 
By \eqref{eq:blackrungtriangle} and induction, 
we deduce that $\xx^{(i)}$ lies in this span for all $0 \le i \le n$, namely
\begin{equation}
	\label{eq:ItoX}
\xx^{(i)} = 
\prescript{n-i}{}\lambda_{n} \
\begin{tikzpicture}[scale=.2, tinynodes, anchorbase]
	\draw[very thick,gray] (-.75,0) to (-.75,3.5);
	\draw[very thick,gray] (.75,0) to (.75,3.5);
	\end{tikzpicture}
+ \sum_{\ell=1}^n \prescript{n-i}{}\lambda_{n-\ell}
\begin{tikzpicture}[scale=.2, tinynodes, anchorbase]
	\draw[very thick,gray] (-1,0) to (0,1);
	\draw[very thick,gray] (1,0) to (0,1);
	\draw[very thick,gray] (0,2.5) to (-1,3.5);
	\draw[very thick,gray] (0,2.5) to (1,3.5);
	\draw[very thick] (0,1) to node[right=-2pt]{$n{-}\ell$} (0,2.5);
	\end{tikzpicture}
	\end{equation}
for some scalars $\prescript{n-i}{}\lambda_{n-\ell} \in \C(q)$.
The proof of \cite[Proposition 4.25]{BER}
(which is established there in the context of $\Rep(U_q(\son))$) 
applies here mutatis mutandis to show that
\begin{equation} \label{eq:lambdas}
\prescript{n-i}{}\lambda_{n-\ell} 
= \begin{cases}
\dfrac{(-1)^{\binom{\ell-i+1}{2}}}{\prod_{j=1}^\ell (q^{2j-1}+ q^{-2j+ 1})} 
	\cdot \prod_{t=1}^i\dfrac{``[\ell+1- t][\ell+t]"}{``[t]^2"} & i \leq \ell \leq n \\ \\
0 & \text{else} \, .
	\end{cases}
	\end{equation}
In particular, $\prescript{n-i}{}\lambda_{n-i}=1$, so
\begin{equation}\label{eq:Xspan}
\operatorname{span} \big\{ \xx^{(i)} \big\}_{i=0}^{n}
=
\operatorname{span}
\left\{ 
\begin{tikzpicture}[scale=.2, tinynodes, anchorbase]
	\draw[very thick,gray] (-.75,0) to (-.75,3.5);
	\draw[very thick,gray] (.75,0) to (.75,3.5);
	\end{tikzpicture} \, , \,
\begin{tikzpicture}[scale=.2, tinynodes, anchorbase]
	\draw[very thick,gray] (-1,0) to (0,1);
	\draw[very thick,gray] (1,0) to (0,1);
	\draw[very thick,gray] (0,2.5) to (-1,3.5);
	\draw[very thick,gray] (0,2.5) to (1,3.5);
	\draw[very thick] (0,1) to node[right=-2pt]{$n{-}1$} (0,2.5);
	\end{tikzpicture} \, , \,
\ldots
 \, , \,
\begin{tikzpicture}[scale=.2, tinynodes, anchorbase]
	\draw[very thick,gray] (-1,0) to (0,1);
	\draw[very thick,gray] (1,0) to (0,1);
	\draw[very thick,gray] (0,2.5) to (-1,3.5);
	\draw[very thick,gray] (0,2.5) to (1,3.5);
	\draw[very thick] (0,1) to node[right=-2pt]{$1$} (0,2.5);
	\end{tikzpicture} \, , \,
\begin{tikzpicture}[scale=.35, tinynodes, anchorbase]
	\draw[very thick,gray] (0,0) to [out=90,in=180] (.5,.625) 
		to [out=0,in=90] (1,0);
	\draw[very thick,gray] (0,2) to [out=270,in=180] (.5,1.375)
		to [out=0,in=270] (1,2);
\end{tikzpicture} \right\} \, .
\end{equation}

Having introduced the relevant morphisms, 
we now lift our description of $R_{S,S}$ in \cite[Proposition 4.28]{BER}
to the setting of $\Web(\son)$.

\begin{prop}\label{prop:SpinBraiding}
Extend scalars in $\Web(\son)$ to $\C(q^{1/2})$ and set 
\begin{equation}
	\label{eq:SpinBraiding}
c_{\gS, \gS} = q^{\frac{n}{2}} \sum_{i=0}^n q^{-i}\xx^{(i)}
\, , \quad
c_{\gS, \gS}' = q^{-\frac{n}{2}} \sum_{i=0}^n q^{i}\xx^{(i)} \, .
	\end{equation}
Then, $c_{\gS, \gS}$ is invertible, 
with inverse $c_{\gS,\gS}'$, 
and $\varphi(c_{\gS,\gS}) = R_{S,S}$.
	\end{prop}
\begin{proof}
It follows from \cite[Proposition 4.28 and Remark 4.29]{BER} that
\[
\varphi(c_{\gS,\gS}) = R_{S,S} \quad \text{and} \quad \varphi(c_{\gS,\gS}') = R_{S,S}^{-1} \, .
	\]
Invertibility of $c_{\gS,\gS}$ and $c_{\gS,\gS}'$ then follows from \eqref{eq:Xspan}
using the linear independence trick employed in the proof of Lemma \ref{lem:otherR2}.
	\end{proof}
	
Henceforth, denote diagrammatically as follows
\[
c_{\gS,\gS} =:
\begin{tikzpicture}[scale=.4, anchorbase]
	\draw[very thick,gray] (1,0) to [out=90,in=270] (0,1.5);
	\draw[overcross] (0,0) to [out=90,in=270] (1,1.5);
	\draw[very thick, gray] (0,0) to [out=90,in=270] (1,1.5);
\end{tikzpicture}
\quad , \qquad
c_{\gS,\gS}'= c_{\gS,\gS}^{-1} =:
\begin{tikzpicture}[scale=.4, anchorbase]
	\draw[very thick,gray] (0,0) to [out=90,in=270] (1,1.5);
	\draw[overcross] (1,0) to [out=90,in=270] (0,1.5);
	\draw[very thick, gray] (1,0) to [out=90,in=270] (0,1.5);
\end{tikzpicture} \, .
	\]
Proposition \ref{prop:SpinBraiding} implies that these morphisms 
satisfy the Reidemeister II relation.

\begin{rem}\label{rem:couldproveSSSRIII}
It is possible to establish various additional tangle and tangled web 
relations involving $c_{\gS,\gS}^{\pm 1}$,
either by (sometimes difficult) direct computation
or by using the linear independence trick from the proof of Lemma \ref{lem:otherR2}.
\end{rem}

\begin{example}
For instance, it is possible to show
\begin{equation}\label{eq:Sforkslide}
\begin{tikzpicture}[scale=.375, tinynodes,anchorbase]
	\draw[very thick,gray] (0,0) to [out=90,in=210] (.5,.75);
	\draw[very thick,gray] (1,0) to [out=90,in=330] (.5,.75);
	\draw[very thick] (.5,.75) to [out=90,in=270] (-.5,2.5);
	\draw[overcross] (-1,0) to [out=90,in=270] (.5,2.5);
	\draw[very thick,gray] (-1,0) to [out=90,in=270] (.5,2.5);
\end{tikzpicture}
=
\begin{tikzpicture}[scale=.375, tinynodes,anchorbase]
	\draw[very thick,gray] (0,0) to [out=90,in=210] (-.5,1.75);
	\draw[very thick,gray] (1,0) to [out=90,in=330] (-.5,1.75);
	\draw[very thick] (-.5,1.75) to [out=90,in=270] (-.5,2.5);
	\draw[overcross] (-1,0) to [out=90,in=270] (.5,2.5);
	\draw[very thick,gray] (-1,0) to [out=90,in=270] (.5,2.5);
\end{tikzpicture}
	\end{equation}
via a tedious direct computation involving complicated coefficients. 
Alternatively, it is possible to show that both sides of this equation 
are in the span of the webs
\[
\Bigg\{ \!\!\!
\begin{tikzpicture}[scale=.375, tinynodes,anchorbase]
	\draw[very thick,gray] (-.5,0) to [out=90,in=210] (0,.75);
	\draw[very thick,gray] (.5,0) to [out=90,in=330] (0,.75);
	\draw[very thick] (0,.75) to [out=90,in=270] node[left=-2pt]{$\ell{-}1$} (0,1.5);
	\draw[very thick] (-1,3) to [out=270,in=150] (0,1.5);
	\draw[very thick,gray] (0,3) to [out=270,in=150] (.5,2.25);
	\draw[very thick,gray] (1,3) to [out=270,in=30] (.5,2.25);
	\draw[very thick] (.5,2.25) to [out=270,in=30] node[right=-2pt,yshift=-1pt]{$\ell$} (0,1.5);
\end{tikzpicture}
\Bigg\}_{\ell=1}^{n-1}
\cup
\Bigg\{ \!\!\!
\begin{tikzpicture}[scale=.375, tinynodes,anchorbase]
	\draw[very thick,gray] (-.5,0) to [out=90,in=210] (0,.75);
	\draw[very thick,gray] (.5,0) to [out=90,in=330] (0,.75);
	\draw[very thick] (0,.75) to [out=90,in=270] node[left=-2pt]{$\ell{+}1$} (0,1.5);
	\draw[very thick] (-1,3) to [out=270,in=150] (0,1.5);
	\draw[very thick,gray] (0,3) to [out=270,in=150] (.5,2.25);
	\draw[very thick,gray] (1,3) to [out=270,in=30] (.5,2.25);
	\draw[very thick] (.5,2.25) to [out=270,in=30] node[right=-2pt,yshift=-1pt]{$\ell$} (0,1.5);
\end{tikzpicture}
\Bigg\}_{\ell=0}^{n-2}
\cup
\Bigg\{
\begin{tikzpicture}[scale=.375, tinynodes,anchorbase]
	\draw[very thick,gray] (-.5,0) to [out=90,in=210] (0,.75);
	\draw[very thick,gray] (.5,0) to [out=90,in=330] (0,.75);
	\draw[very thick] (0,.75) to [out=90,in=270] node[left=-2pt]{$n{-}1$} (0,1.5);
	\draw[very thick] (-1,3) to [out=270,in=150] (-.5,2.25);
	\draw[very thick,gray] (0,3) to [out=270,in=30] (-.5,2.25);
	\draw[very thick,gray] (1,3) to [out=270,in=30] (0,1.5);
	\draw[very thick,gray] (-.5,2.25) to [out=270,in=150] (0,1.5);
\end{tikzpicture}
\, , \,
\begin{tikzpicture}[scale=.375, tinynodes,anchorbase]
	\draw[very thick,gray] (-.5,0) to [out=90,in=240] (-.375,1.25);
	\draw[very thick,gray] (.5,0) to [out=90,in=300] (.375,1.25);
	\draw[very thick,gray] (-.375,1.25) to (.375,1.25);
	\draw[very thick] (-1,3) to [out=270,in=120] (-.375,1.25);
	\draw[very thick,gray] (0,3) to [out=270,in=150] (.5,2.25);
	\draw[very thick,gray] (1,3) to [out=270,in=30] (.5,2.25);
	\draw[very thick] (.5,2.25) to [out=270,in=60] node[right=-2pt,yshift=-1pt]{$n{-}1$} (.375,1.25);
\end{tikzpicture}
\, , \,
\begin{tikzpicture}[scale=.375, tinynodes,anchorbase]
	\draw[very thick] (-1,3) to [out=270,in=150] (-.5,2.25);
	\draw[very thick,gray] (0,3) to [out=270,in=30] (-.5,2.25);
	\draw[very thick,gray] (1,3) to [out=270,in=90] (.5,0);
	\draw[very thick,gray] (-.5,2.25) to [out=270,in=90] (-.5,0);
\end{tikzpicture}
	\Bigg\} \, .
	\]
Then, one can show (via a somewhat non-trivial argument) 
that these webs are sent by $\varphi$ to a basis of 
$\Hom_{U_q(\son)}(S \otimes S , V_1 \otimes S \otimes S)$.
Naturality of the braiding guarantees that the image of \eqref{eq:Sforkslide} 
under $\varphi$ holds in $\Rep(U_q(\son))$, 
so \eqref{eq:Sforkslide} must hold in $\Web(\son)$ via the linear independence trick.
\end{example}

\begin{rem}\label{rem:n=1}
When $n=1$, there are no $1$-labeled edges. 
Hence, the results from \S\ref{ss:1braid} are vacuous. 
Nevertheless, we use these results in the following section 
in the ``ladderization'' step in our proof of Theorem \ref{thm:main}
(recall \S \ref{ss:ladder}).
We now discuss how to extend these results to the $n=1$ case.

The reader will observe in \S \ref{s:ladder} that, if we attempt to run the ladderization argument when $n=1$, 
it will only require $1$-labeled edges that terminate at trivalent vertices containing $\gS$-labeled edges.
In this case, we can set
\[
\begin{tikzpicture}[scale=.2, tinynodes, anchorbase]
	\draw[very thick,gray] (-1,0) to (0,1);
	\draw[very thick,gray] (1,0) to (0,1);
	\draw[very thick,gray] (0,2.5) to (-1,3.5);
	\draw[very thick,gray] (0,2.5) to (1,3.5);
	\draw[very thick] (0,1) to node[right=-2pt]{$1$} (0,2.5);
	\end{tikzpicture}
=
\begin{tikzpicture}[scale=.2, tinynodes, anchorbase]
	\draw[very thick,gray] (-.75,0) to (-.75,3.5);
	\draw[very thick,gray] (.75,0) to (.75,3.5);
	\end{tikzpicture}
+
\frac{1}{[2]}
\begin{tikzpicture}[scale=.4, tinynodes, anchorbase]
	\draw[very thick,gray] (0,0) to [out=90,in=180] (.5,.625) 
		to [out=0,in=90] (1,0);
	\draw[very thick,gray] (0,2) to [out=270,in=180] (.5,1.375)
		to [out=0,in=270] (1,2);
\end{tikzpicture}
\]
which is the $2$-strand \emph{Jones--Wenzl projector}.
Having done so, the stand-ins for the results in \S \ref{ss:1braid} are results involving 
the morphisms \eqref{eq:SpinBraiding} that are easy to check.
(They amount to standard facts concerning the Kauffman bracket \cite{Kauff} approach to the Jones polynomial.)
In this $n=1$ case, Theorem \ref{thm:ladder} gives a set of ``ladder webs'' 
that is a unitriangular change of basis from the usual Temperley--Lieb basis.
\end{rem}

\begin{rem}\label{rem:promisedStanglerlns}
Save for the $n=1$ case just discussed,
\textbf{we do not use tangled web relations involving $c_{\gS,\gS}^{\pm 1}$ to prove Theorem \ref{thm:main}}.
Hence, we can avoid both direct computation and the linear independence trick in establishing such relations. 
Indeed, we will more economically prove that tangled web relations involving $c_{\gS,\gS}^{\pm 1}$ hold in $\Web(\son)$
as a consequence of the equivalence $\Web(\son)\cong \FRep(U_q(\son))$ provided by Theorem \ref{thm:main}. 
See Corollary \ref{cor:Stangledweb}.
\end{rem}

\section{Ladderization}\label{s:ladder}

We now begin our main task of showing that the functor 
from Theorem \ref{thm:functor} is fully faithful. 
We first need the following easy result.

\begin{lem}[{\cite[Lemma 5.5]{BERT}}]
	\label{lem:reduction}
Let $\mathcal{D}$ and $\mathcal{R}$ be $\mathbb{C}(q)$-linear monoidal categories. 
Suppose that $\mathcal{D}'$ is a full monoidal subcategory of $\mathcal{D}$ such that
every object in $\mathcal{D}$ is a direct summand
of an iterated tensor product of objects in $\mathcal{D}'$, 
after embedding both in $\Kar(\mathcal{D})$.
Let $\Psi \colon \mathcal{D}\to \mathcal{R}$ be a $\mathbb{C}(q)$-linear (monoidal) functor,
then $\Psi$ is full if and only if $\Psi|_{\mathcal{D}'}$ is full 
and $\Psi$ is faithful if and only if $\Psi|_{\mathcal{D}'}$ is faithful. \qed
	\end{lem}
	
Using this, we now reduce to investigating endomorphism spaces 
of tensor powers of the spin representation.

\begin{prop}\label{prop:FF}
The functor in Theorem \ref{thm:functor} is faithful (respectively, full) if and only if 
all of the corresponding algebra homomorphisms
\[
\varphi \colon \End_{\Web(\son)}(\gS^{\otimes m}) \to \End_{U_q(\son)}(S^{\otimes m})
	\]
are injective (respectively, surjective).
	\end{prop}
\begin{proof}
We apply Lemma \ref{lem:reduction} in the case that $\mathcal{D} = \Web(\son)$, 
$\mathcal{D}'$ is the full subcategory tensor-generated by the object $\gS$, 
$\mathcal{R} = \Rep(U_q(\son))$, and $\Psi$ is the functor $\varphi$ from Theorem \ref{thm:functor}.
The hypotheses are satisfied since equations \eqref{eq:digonS} and \eqref{eq:zerodigonS} 
guarantee that each of the objects $\{1,\ldots,n-1\} \in \Web(\son)$ are summands of $\gS \otimes \gS$ 
in $\Kar(\Web(\son))$.

We thus see that $\varphi$ is faithful (respectively, full) if and only if the maps
\[
\varphi \colon \Hom_{\Web(\son)}(\gS^{\otimes m_1},\gS^{\otimes m_2}) 
	\to \Hom_{U_q(\son)}(S^{\otimes m_1},S^{\otimes m_2})
	\]
are injective (respectively, surjective). If $m_1+m_2$ is odd, then the domain is zero 
(the gray $1$-manifold with boundary obtained by erasing all black edges 
must have an even number of boundary points) 
and the codomain is zero 
(the trivial representation and an odd tensor power of $S$ have different central character). 
Since the domain and codomain are both zero, injectivity and surjectivity is immediate if $m_1+m_2$ is odd. 

Otherwise $m_1+m_2=2m$, and we have the commutative diagram
\begin{equation}\label{eqn:monoidalduality}
\begin{tikzcd}
\Hom_{\Web(\son)}(\gS^{\otimes m_1},\gS^{\otimes m_2}) \ar[rr,"\varphi"] \ar[d,"\cong"]
& & \Hom_{U_q(\son)}(S^{\otimes m_1},S^{\otimes m_2}) \ar[d,"\cong"] \\
\End_{\Web(\son)}(\gS^{\otimes m}) \ar[rr,"\varphi"]
& & \End_{U_q(\son)}(S^{\otimes m})
	\end{tikzcd}
\end{equation}
where the vertical isomorphisms are given by cap/cup morphisms in $\Web(\son)$ 
and their images under $\varphi$. 
The map in the top row is thus injective (respectively, surjective)
if and only if the map in the bottom row is injective (respectively, surjective).
	\end{proof}

Consequently, we turn our attention to the algebras $\End_{\Web(\son)}(\gS^{\otimes m})$. 
We first aim to show that this algebra is equal to the following, 
which a priori is simply a subalgebra.

\begin{defn}\label{def:ladder}
For $1 \leq i \leq m-1$, set
\[
\rung_i := \rung_{i,i+1} :=
\begin{tikzpicture}[scale=.5, tinynodes, anchorbase]
	\draw[thick, black] (1,0) to node[below=-1pt]{$1$} (2,0);
	\draw[very thick,gray] (0,-1) to (0,1);
	\node at (.5,.5) {$\mydots$};
	\draw[very thick,gray] (1,-1) to (1,1);
	\draw[very thick,gray] (2,-1) to (2,1);
	\node at (2.5,.5) {$\mydots$};
	\draw[very thick,gray] (3,-1) to (3,1);	
\end{tikzpicture}
\in \End_{\Web(\son)}(\gS^{\otimes m})
\]
where the $1$-labeled ``rung''-edge lies between 
the $i$-th and $(i+1)$-st $\gS$-labeled ``uprights.''
Let $\Lad_m \subseteq \End_{\Web(\son)}(\gS^{\otimes m})$ 
be the unital subalgebra generated by the elements $\{r_i\}_{i=1}^{m-1}$.
\end{defn}

\begin{defn}
For $1 \leq i < j \leq m$, set
\begin{equation}\label{eq:grung}
\rung_{i,j} :=
\begin{tikzpicture}[scale=.5, tinynodes, anchorbase]
	\draw[thick, black] (1,0) to node[below=-1pt]{$1$} (4,0);
	\draw[very thick,gray] (0,-1) to (0,1);
	\node at (.5,.5) {$\mydots$};
	\draw[very thick,gray] (1,-1) to (1,1);
	\draw[overcross] (2,-1) to (2,1);
	\draw[very thick,gray] (2,-1) to (2,1);
	\node at (2.5,.5) {$\mydots$};
 	\draw[overcross] (3,-1) to (3,1);
	\draw[very thick,gray] (3,-1) to (3,1);
	\draw[very thick,gray] (4,-1) to (4,1);
	\node at (4.5,.5) {$\mydots$};
	\draw[very thick,gray] (5,-1) to (5,1);
\end{tikzpicture}
\in \End_{\Web(\son)}(\gS^{\otimes m})
	\end{equation}
where here the rung connects the $i$-th and $j$-th uprights.
\end{defn}
	
\begin{lem}\label{lem:genrung} 
The subalgebra $\Lad_m$ is equal to the subalgebra generated by $\{r_{i,j}\}_{1 \leq i < j \leq m}$. 
	\end{lem}
	
\begin{proof} 
One need only prove that $r_{i,j} \in \Lad_m$ for all $1 \leq i < j \leq m$.
This is an immediate consequence of \eqref{eq:blackgraybraiding}. \end{proof}

\begin{example}
A single application of \eqref{eq:blackgraybraiding} results in the equation
\begin{equation}\label{eq:webrung13}
\begin{tikzpicture}[scale=.5, tinynodes, anchorbase]
	\draw[thick, black] (1,0) to (3,0);
	\draw[very thick,gray] (0,-1) to (0,1);
	\node at (.5,.5) {$\mydots$};
	\draw[very thick,gray] (1,-1) to (1,1);
	\draw[overcross] (2,-1) to (2,1);
	\draw[very thick,gray] (2,-1) to (2,1);
	\draw[very thick,gray] (3,-1) to (3,1);
	\node at (3.5,.5) {$\mydots$};
	\draw[very thick,gray] (4,-1) to (4,1);
\end{tikzpicture} \ 
= q^{-1} \ 
\begin{tikzpicture}[scale=.5, tinynodes, anchorbase]
	\draw[thick, black] (2,-.5) to (3,-.5);
	\draw[thick, black] (1,.5) to (2,.5);
	\draw[very thick,gray] (0,-1) to (0,1);
	\node at (.5,.5) {$\mydots$};
	\draw[very thick,gray] (1,-1) to (1,1);
	\draw[very thick,gray] (2,-1) to (2,1);
	\draw[very thick,gray] (3,-1) to (3,1);
	\node at (3.5,.5) {$\mydots$};
	\draw[very thick,gray] (4,-1) to (4,1);
\end{tikzpicture}
+ q \ 
\begin{tikzpicture}[scale=.5, tinynodes, anchorbase]
	\draw[thick, black] (1,-.5) to (2,-.5);
	\draw[thick, black] (2,.5) to (3,.5);
	\draw[very thick,gray] (0,-1) to (0,1);
	\node at (.5,.5) {$\mydots$};
	\draw[very thick,gray] (1,-1) to (1,1);
	\draw[very thick,gray] (2,-1) to (2,1);
	\draw[very thick,gray] (3,-1) to (3,1);
	\node at (3.5,.5) {$\mydots$};
	\draw[very thick,gray] (4,-1) to (4,1);
\end{tikzpicture}
\end{equation}
that is, $\rung_{i,i+2}= q^{-1}\rung_{i,i+1}\rung_{i+1, i+2} + q\rung_{i+1,i+2}r_{i,i+1}\in \Lad_m$.
\end{example}
	
The main result of this section is the following.

\begin{thm}\label{thm:ladder}
For $m \geq 1$, 
$\Lad_m = \End_{\Web(\son)}(\gS^{\otimes m})$.
	\end{thm}

In fact, the $m=0$ case holds as well; see Corollary \ref{cor:ladder0}. 
Before proving Theorem \ref{thm:ladder}, 
we establish some simplifying results concerning 
``all-black'' webs: those which have no edge labeled $\gS$.

\begin{defn} For $k \ge 2$, let the full merge and full split webs be given recursively by
\[
\begin{tikzpicture}[scale=.375, anchorbase,yscale=-1]
	\draw[very thick] (0,-2) node[above=-2pt]{\scs$k$} to (0,-1);
	\draw[very thick] (0,-1) to (1,1) node[below=-2pt]{\scs$1$};
	\draw[very thick] (0,-1) to (-1,1) node[below=-2pt]{\scs$1$};
	\node at (0,.5) {$\mydots$};
\end{tikzpicture}
:=
\begin{tikzpicture}[scale=.375, anchorbase,yscale=-1]
	\draw[very thick] (0,-2) node[above=-2pt]{\scs$k$} to (0,-1);
	\draw[very thick] (0,-1) to (1,1) node[below=-2pt]{\scs$1$};
	\draw[very thick] (-.5,0) to (0,1) node[below=-2pt]{\scs$1$};
	\draw[very thick] (0,-1) to node[pos=.25,left=-2pt]{\scs$k{-}1$} (-1,1) node[below=-2pt]{\scs$1$};
	\node at (-.5,.75) {\scs$\mysdots$};
\end{tikzpicture}
\quad , \quad
\begin{tikzpicture}[scale=.375, anchorbase]
	\draw[very thick] (0,-2) node[below=-2pt]{\scs$k$} to (0,-1);
	\draw[very thick] (0,-1) to (1,1) node[above=-2pt]{\scs$1$};
	\draw[very thick] (0,-1) to (-1,1) node[above=-2pt]{\scs$1$};
	\node at (0,.5) {$\mydots$};
\end{tikzpicture}
:=
\begin{tikzpicture}[scale=.375, anchorbase]
	\draw[very thick] (0,-2) node[below=-2pt]{\scs$k$} to (0,-1);
	\draw[very thick] (0,-1) to (1,1) node[above=-2pt]{\scs$1$};
	\draw[very thick] (-.5,0) to (0,1) node[above=-2pt]{\scs$1$};
	\draw[very thick] (0,-1) to node[pos=.25,left=-2pt]{\scs$k{-}1$} (-1,1) node[above=-2pt]{\scs$1$};
	\node at (-.5,.75) {\scs$\mysdots$};
\end{tikzpicture} \, .
	\]
\end{defn}	

\begin{lem} For $k \geq 2$, we have
\begin{equation}
	\label{eq:fulldigon}
\begin{tikzpicture}[scale=.175,tinynodes, anchorbase]
	\draw [very thick] (0,.75) to (0,2.5) node[above,yshift=-3pt]{$k$};
	\draw [very thick] (0,-2.75) to [out=30,in=330] node[right,xshift=-2pt]{$1$} (0,.75);
	\node at (0,-1) {$\mydots$};
	\draw [very thick] (0,-2.75) to [out=150,in=210] node[left,xshift=2pt]{$1$} (0,.75);
	\draw [very thick] (0,-4.5) node[below,yshift=2pt]{$k$} to (0,-2.75);
\end{tikzpicture}
= (-1)^{\binom{k}{2}} 
(``[k]^2")!
\begin{tikzpicture}[scale=.175,tinynodes, anchorbase]
	\draw [very thick] (0,-4.5) node[below,yshift=2pt]{$k$} to (0,2.5);
\end{tikzpicture} \, .
	\end{equation}
\end{lem}
\begin{proof} 
This follows by induction using \eqref{eq:digon1}. 
	\end{proof}

\begin{lem}
	\label{lem:fullMS}
For all $k \geq 2$, the morphism
\begin{equation}\label{eq:fullMS}
\begin{tikzpicture}[scale=.375, anchorbase]
	\draw[very thick] (-1,-1.75) node[below=-2pt]{\scs$1$} to (0,-.5);
	\node at (0,-1.375) {$\mydots$};
	\draw[very thick] (1,-1.75) node[below=-2pt]{\scs$1$} to (0,-.5);
	\draw[very thick] (0,-.5) to node[right=-2pt]{\scs$k$} (0,.5);
	\draw[very thick] (0,.5) to (1,1.75) node[above=-2pt]{\scs$1$};
	\node at (0,1.375) {$\mydots$};
	\draw[very thick] (0,.5) to (-1,1.75) node[above=-2pt]{\scs$1$};
\end{tikzpicture} \in \End_{\Web(\son)}(1^{\otimes k})
	\end{equation}
is in the span of $1$-labeled $(k,k)$-tangles.
	\end{lem}
\begin{proof}
We argue via induction on $k\geq 2$.
For the base case, \eqref{eq:blackbraiding} gives
\begin{equation}\label{eq:1labelbase}
\begin{tikzpicture}[scale=.2, tinynodes, anchorbase]
	\draw[very thick] (-1,0) to (0,1);
	\draw[very thick] (1,0) to (0,1);
	\draw[very thick] (0,2.5) to (-1,3.5);
	\draw[very thick] (0,2.5) to (1,3.5);
	\draw[very thick] (0,1) to node[right=-2pt]{$2$} (0,2.5);
\end{tikzpicture}
=
\begin{tikzpicture}[scale=.4, anchorbase,tinynodes]
	\draw[very thick] (1,0) to [out=90,in=270] (0,1.5);
	\draw[overcross] (0,0) to [out=90,in=270] (1,1.5);
	\draw[very thick] (0,0) to [out=90,in=270] (1,1.5);
\end{tikzpicture}
-
q^2 \
\begin{tikzpicture}[scale=.2, tinynodes, anchorbase]
	\draw[very thick] (-.75,0) to (-.75,3.5);
	\draw[very thick] (.75,0) to (.75,3.5);
	\end{tikzpicture}
+ \frac{q^{1-2n} (q^2-q^{-2})}{[2]_{2n-1}}
\begin{tikzpicture}[scale=.4, tinynodes, anchorbase]
	\draw[very thick] (0,0) to [out=90,in=180] (.5,.625) 
		to [out=0,in=90] (1,0);
	\draw[very thick] (0,2) to [out=270,in=180] (.5,1.375)
		to [out=0,in=270] (1,2);
\end{tikzpicture} \, .
	\end{equation}
	Let $k> 2$. It follows from \eqref{eq:blackH=I} that the diagram in \eqref{eq:fullMS} is equal to
\begin{equation}\label{eq:1labelinduction}
\begin{tikzpicture}[scale=.375, anchorbase,yscale=1]
	\draw[very thick] (1.5,0) to (1.5,3) node[above=2pt]{\scs$1$};
	\draw[very thick] (1.5,0) to (1.5,-2) node[below=2pt]{\scs$1$};
	\draw[very thick] (-.5,2) node[left=-2pt]{\scs${k{-}1}$} to (-.5,.5);
	\draw[very thick] (-.5,.5) to (-.5,-1) node[left=-2pt]{\scs${k{-}1}$};
	\draw[very thick] (-.5,2) to (0,3) node[above=2pt]{\scs$1$};
	\draw[very thick] (-.5,2) to (-1,3) node[above=2pt]{\scs$1$};
	\node at (-.5,2.75) {\scs$\mysdots$};
	\draw[very thick] (-.5,-1) to (0,-2) node[below=-2pt]{\scs$1$};
	\draw[very thick] (-.5,-1) to (-1,-2) node[below=-2pt]{\scs$1$};
	\draw[very thick] (-.5,1) to (.5,.5);
	\draw[very thick] (-.5,0) to (.5,.5);
	\draw[very thick] (.5,.5) to (1,.5);
	\draw[very thick] (1,.5) node[above=-2pt]{\scs$2$} to (1.5,.5);
	\node at (-.5,-1.75) {\scs$\mysdots$};
	\node at (-1.5,.5) {\scs$k{-}2$};
\end{tikzpicture}
+\frac{[2]_{2n{-}2k{+}1}}{[2]_{2n{-}2k{+}3}}
\begin{tikzpicture}[scale=.375, anchorbase,yscale=1]
	\draw[very thick] (0,1) to (1,3) node[above=2pt]{\scs$1$};
	\draw[very thick] (-.5,2) to (0,3) node[above=2pt]{\scs$1$};
	\draw[very thick] (0,1) to node[pos=.25,left=-2pt]{\scs$k{-}1$} (-1,3) node[above=2pt]{\scs$1$};
	\node at (-.5,2.75) {\scs$\mysdots$};
	\draw[very thick] (0,1) to (0,0.5) node[right=-2pt]{\scs${k{-}2}$};
	\draw[very thick] (0,0) to (0,0.5);
	\draw[very thick] (0,0) to (1,-2) node[below=-2pt]{\scs$1$};
	\draw[very thick] (-.5,-1) to (0,-2) node[below=-2pt]{\scs$1$};
	\draw[very thick] (0,0) to node[pos=.25,left=-2pt]{\scs$k{-}1$} (-1,-2) node[below=-2pt]{\scs$1$};
	\node at (-.5,-1.75) {\scs$\mysdots$};
\end{tikzpicture}
+ (-1)^k
\frac{[2]_{2n{-}2k{+}1}}{[2]_{2n{-}1}}
\begin{tikzpicture}[scale=.375, anchorbase,yscale=1]
	\draw[very thick] (1,0) to (1,3) node[above=2pt]{\scs$1$};
	\draw[very thick] (1,0) to (1,-2) node[below=-2pt]{\scs$1$};
	\draw[very thick] (-.5,2) to (-.5,.5) node[left=-2pt]{\scs${k{-}1}$};
	\draw[very thick] (-.5,.5) to (-.5,-1);
	\draw[very thick] (-.5,2) to (0,3) node[above=2pt]{\scs$1$};
	\draw[very thick] (-.5,2) to (-1,3) node[above=2pt]{\scs$1$};
	\node at (-.5,2.75) {\scs$\mysdots$};
	\draw[very thick] (-.5,-1) to (0,-2) node[below=-2pt]{\scs$1$};
	\draw[very thick] (-.5,-1) to (-1,-2) node[below=-2pt]{\scs$1$};
	\node at (-.5,-1.75) {\scs$\mysdots$};
\end{tikzpicture}
\ .
\end{equation}
So it suffices to show that each diagram in \eqref{eq:1labelinduction} is in the span of $1$-labelled tangles. 

It is an immediate consequence of induction that the third diagram in \eqref{eq:1labelinduction} 
is in the span of $1$-labelled tangles. 
Next, applying \eqref{eq:fulldigon} to the $k-2$ labelled edge in the second diagram in \eqref{eq:1labelinduction} 
introduces two subdiagrams of the form \eqref{eq:fullMS} (with $k$ replaced by $k-1$). 
It follows by induction that the second diagram in \eqref{eq:1labelinduction} is in the span of $1$-labelled tangles. 
Lastly, applying \eqref{eq:fulldigon} to the $k-2$ labelled edge in the first diagram in \eqref{eq:1labelinduction} 
introduces two subdiagrams of the form \eqref{eq:fullMS} (with $k$ replaced by $k-1$), 
so the first diagram is in the span of $1$-labelled tangles by induction and \eqref{eq:1labelbase}. 
This result is 
\[
\begin{tikzpicture}[scale=.625, anchorbase]
	\draw[very thick] (-1,-1.75) node[below=-2pt]{\scs$1$} to (0,-.5);
	\node at (0,-1.375) {$\mydots$};
	\draw[very thick] (1,-1.75) node[below=-2pt]{\scs$1$} to (0,-.5);
	\draw[very thick] (0,-.5) to node[right=-2pt]{\scs$k$} (0,.5);
	\draw[very thick] (0,.5) to (1,1.75) node[above=-2pt]{\scs$1$};
	\node at (0,1.375) {$\mydots$};
	\draw[very thick] (0,.5) to (-1,1.75) node[above=-2pt]{\scs$1$};
\end{tikzpicture} \ 
= \ 
\begin{tikzpicture}[scale=.375, anchorbase,yscale=1]
	\draw[very thick] (1.5,1.5) to (1.5,3.5) node[above=-2pt]{\scs$1$};
	\draw[very thick] (0,2.5) to (0,3.5) node[above=-2pt]{\scs$1$};
	\draw[very thick] (-1,2.5) to (-1,3.5) node[above=-2pt]{\scs$1$};
		\draw[very thick] (.2,1.5) rectangle (-1.2,2.5);
	\node at (-.5,3.25) {\scs$\mysdots$};
	\draw[very thick] (1.5,-0.5) to (1.5,-2.5) node[below=-2pt]{\scs$1$};
	\draw[very thick] (0,-1.5) to (0,-2.5) node[below=-2pt]{\scs$1$};
	\draw[very thick] (-1,-1.5) to (-1,-2.5) node[below=-2pt]{\scs$1$};
		\draw[very thick] (-1.2,0-.5) rectangle (0.2,-1.5);
	\node at (-.5,-2.25) {\scs$\mysdots$};
		\draw[very thick] (0,-.5) to [out=90,in=180] (1,.125) 
		to [out=0,in=90] (1.5,-.5);
	\draw[very thick] (0,1.5) to [out=270,in=180] (1,.875)
		to [out=0,in=270] (1.5,1.5);
	\draw[very thick] (-1,-0.5) to (-1,1.5);
	\node at (-.5,.5) {\scs$\mysdots$};
	\draw[very thick, fill=white] (1.1,1.5) rectangle (0.6,-.5);
\end{tikzpicture}
+ \ 
\begin{tikzpicture}[scale=.375, anchorbase,yscale=1]
	\draw[very thick] (1.5,1.5) to (1.5,3.5) node[above=-2pt]{\scs$1$};
	\draw[very thick] (0,2.5) to (0,3.5) node[above=-2pt]{\scs$1$};
	\draw[very thick] (-1,2.5) to (-1,3.5) node[above=-2pt]{\scs$1$};
		\draw[very thick] (.2,1.5) rectangle (-1.2,2.5);
	\node at (-.5,3.25) {\scs$\mysdots$};
	\draw[very thick] (1.5,-0.5) to (1.5,-2.5) node[below=-2pt]{\scs$1$};
	\draw[very thick] (0,-1.5) to (0,-2.5) node[below=-2pt]{\scs$1$};
	\draw[very thick] (-1,-1.5) to (-1,-2.5) node[below=-2pt]{\scs$1$};
		\draw[very thick] (-1.2,0-.5) rectangle (0.2,-1.5);
	\node at (-.5,-2.25) {\scs$\mysdots$};
		\draw[very thick] (0,-.5) to [out=90,in=180] (.75,.125) 
		to [out=0,in=90] (1.5,-.5);
	\draw[very thick] (0,1.5) to [out=270,in=180] (.75,.875)
		to [out=0,in=270] (1.5,1.5);
	\draw[very thick] (-1,-0.5) to (-1,1.5);
	\node at (-.5,.5) {\scs$\mysdots$};
\end{tikzpicture}
+ \ 
\begin{tikzpicture}[scale=.375, anchorbase,yscale=1]
	\draw[very thick] (1.5,0) to (1.5,3.5) node[above=-2pt]{\scs$1$};
	\draw[very thick] (0,1) to (0,3.5) node[above=-2pt]{\scs$1$};
	\draw[very thick] (-1,1) to (-1,3.5) node[above=-2pt]{\scs$1$};
	\node at (-.5,3.25) {\scs$\mysdots$};
	\draw[very thick] (1.5,0) to (1.5,-2.5) node[below=-2pt]{\scs$1$};
	\draw[very thick] (0,0) to (0,-2.5) node[below=-2pt]{\scs$1$};
	\draw[very thick] (-1,0) to (-1,-2.5) node[below=-2pt]{\scs$1$};
		\draw[very thick] (-1.2,1) rectangle (0.2,0);
	\node at (-.5,-2.25) {\scs$\mysdots$};
\end{tikzpicture}
\ ,
\]
where each rectangle represents some linear combination of $1$-labelled tangles.
	\end{proof}
	
\begin{lem} \label{lem:rewriteblack} 
Any all-black web whose boundary strands are all $1$-labeled is in the span of $1$-labeled tangles. 
	\end{lem}

\begin{proof} 
Apply \eqref{eq:fulldigon} in the middle of every edge labeled $\geq 2$. 
After doing so, each $k$-labeled edge either abuts two full merge/split diagrams or one full merge/split diagram and one trivalent vertex. 
In the first case, the edge takes the form \eqref{eq:fullMS}, so we can apply Lemma \ref{lem:fullMS} to replace it with a 
linear combination of $1$-labeled tangles.

In the second case, we necessarily have three edges meeting at a trivalent vertex, 
with each edge having a full merge/split at their other end. 
In this case, we have
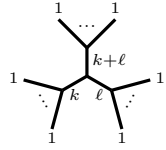
\[
\begin{tikzpicture}[scale=.375, anchorbase,tinynodes]
	\draw[very thick] (0,0) to node[right=-2pt]{$k{+}\ell$} (0,1);
	\draw[very thick] (0,1) to (1,2) node[above=-2pt]{$1$};
	\draw[very thick] (0,1) to (-1,2) node[above=-2pt]{$1$};
	\node at (0,1.75) {$\mydots$};
\begin{scope}[rotate=120]
	\draw[very thick] (0,0) to node[below=-1pt]{$k$} (0,1);
	\draw[very thick] (0,1) to (1,2) node[left=-2pt]{$1$};
	\draw[very thick] (0,1) to (-1,2) node[below=-2pt]{$1$};
	\node[rotate=120] at (0,1.75) {$\mydots$};
	\end{scope}
\begin{scope}[rotate=240]
	\draw[very thick] (0,0) to node[below=-1pt]{$\ell$} (0,1);
	\draw[very thick] (0,1) to (1,2) node[below=-2pt]{$1$};
	\draw[very thick] (0,1) to (-1,2) node[right=-2pt]{$1$};
	\node[rotate=240] at (0,1.75) {$\mydots$};
	\end{scope}
\end{tikzpicture}
	\]
where either $k=1$ or $\ell=1$.
Repeat application of \eqref{eq:assoc}  shows that this is equal to a $(k+\ell)$-labeled edge having a full merge/split at each end. 
We can then again appeal to Lemma \ref{lem:fullMS} to replace this with a linear combination of $1$-labeled tangles.
\end{proof}

We thus turn our attention to the span of $1$-labeled tangles.
Observe that \eqref{eq:blackbraiding} implies that the BMW skein relation
\begin{equation}
	\label{eq:BMWskein}
\begin{tikzpicture}[scale=.4, anchorbase,tinynodes]
	\draw[very thick] (1,0) to [out=90,in=270] (0,1.5);
	\draw[overcross] (0,0) to [out=90,in=270] (1,1.5);
	\draw[very thick] (0,0) to [out=90,in=270] (1,1.5);
\end{tikzpicture}
-
\begin{tikzpicture}[scale=.4, anchorbase,xscale=-1]
	\draw[very thick] (1,0) to [out=90,in=270] (0,1.5);
	\draw[overcross] (0,0) to [out=90,in=270] (1,1.5);
	\draw[very thick] (0,0) to [out=90,in=270] (1,1.5);
\end{tikzpicture}
=
(q^2 - q^{-2})
\left(
\begin{tikzpicture}[scale=.2, tinynodes, anchorbase]
	\draw[very thick] (-.75,0) to (-.75,3.5);
	\draw[very thick] (.75,0) to (.75,3.5);
	\end{tikzpicture}
-
\begin{tikzpicture}[scale=.4, tinynodes, anchorbase]
	\draw[very thick] (0,0) to [out=90,in=180] (.5,.625) 
		to [out=0,in=90] (1,0);
	\draw[very thick] (0,2) to [out=270,in=180] (.5,1.375)
		to [out=0,in=270] (1,2);
\end{tikzpicture}
\right)
	\end{equation}
is satisfied by all-black $1$-labeled crossings in $\Web(\son)$. 
This allows us to replace over-crossings with under-crossings (and vice-versa) 
up to lower terms, i.e.~terms with fewer crossings. 
This suggests that we focus on which boundary points are connected by strands in a tangle
and ignore the topology of its embedding.

\begin{defn}
A \emph{matching} of $2v$ points in the boundary of a disk 
is given by $v$ immersed curves in the disk intersecting generically 
and meeting the $2v$ boundary points transversely, 
e.g.~
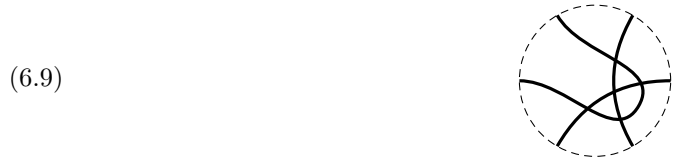
\begin{equation}\label{eq:non-red-match}
\begin{tikzpicture}[scale=1, smallnodes, anchorbase]
	\draw[densely dashed] (0,0) circle (1);
	\draw[very thick] (1,0) to [out=180,in=60] (-.5,-.866) ;
	\draw[very thick] (.5,.866) to [out=240,in=120] (.5,-.866) ;
	\draw[very thick] (-.5,.866) to [out=300,in=60] (.577,-.333) to [out=240,in=0] (-1,0) ;
\end{tikzpicture} \, .
\end{equation}
Such a matching is \emph{reduced} if no curve intersects itself, 
and if any given pair of curves intersects at most once. 
The matching in \eqref{eq:non-red-match} is not reduced, 
but the matching
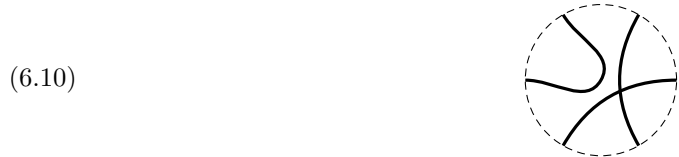
\begin{equation}\label{eq:red-match}
\begin{tikzpicture}[scale=1, smallnodes, anchorbase]
	\draw[densely dashed] (0,0) circle (1);
	\draw[very thick] (1,0) to [out=180,in=60] (-.5,-.866) ;
	\draw[very thick] (.5,.866) to [out=240,in=120] (.5,-.866) ;
	\draw[very thick] (-.5,.866) to [out=300,in=60] (0,0) to [out=240,in=0] (-1,0);
\end{tikzpicture}
\end{equation}
is reduced.
A \emph{lift} of a given matching
replaces each intersection point with either an over- or an under-crossing.
\end{defn}

\begin{lem}\label{lem:BMW-span}
The space spanned by tangles in the disk with $2v$ boundary points
modulo the BMW skein relation \eqref{eq:BMWskein}, 
the circle relation \eqref{eq:circle1}, 
and the 1-labeled Reidemeister moves from Lemmata \ref{lem:otherR2} and \ref{lem:1R13}
has a basis consisting of any chosen lifts of the set of reduced matchings of $2v$ points.
\end{lem}

\begin{proof}
This is a standard folklore fact (it can be inferred from e.g.~\cite{BW}).
For details proofs, see e.g.~\cite{WilBMW,MortonBMW}.
\end{proof}

\begin{example}
The following give such a basis when $v=2$ 
(here, we replace the disc with a half-plane):
\begin{equation}\label{eq:matchingex}
\begin{tikzpicture}[scale=.5, tinynodes, anchorbase]
	\draw[densely dashed] (0,3.5) to (0,-0.5);
	\draw[very thick] (0,3) to [out=0,in=90] (.75,2.5) to [out=270,in=0] (0,2);
	\draw[very thick] (0,1) to [out=0,in=90] (.75,.5) to [out=270,in=0] (0,0);
\end{tikzpicture}
\quad , \quad
\begin{tikzpicture}[scale=.5, tinynodes, anchorbase]
	\draw[densely dashed] (0,3.5) to (0,-0.5);
	\draw[very thick] (0,3) to [out=0,in=90] (1.5,1.5) to [out=270,in=0] (0,0);
	\draw[very thick] (0,2) to [out=0,in=90] (.75,1.5) to [out=270,in=0] (0,1);
\end{tikzpicture}
\quad , \quad
\begin{tikzpicture}[scale=.5, tinynodes, anchorbase]
	\draw[densely dashed] (0,3.5) to (0,-0.5);
	\draw[very thick] (0,2) to [out=0,in=90] (1.125,1) to [out=270,in=0] (0,0);
	\draw[overcross] (0,3) to [out=0,in=90] (1.125,2) to [out=270,in=0] (0,1);
	\draw[very thick] (0,3) to [out=0,in=90] (1.125,2) to [out=270,in=0] (0,1);
\end{tikzpicture}
	\end{equation}
\end{example}

As a final preparatory result, 
we record the interaction between $\gS$-labeled strands and $1$-labeled tangles.

\begin{lem}\label{lem:tangleeating}
\begin{equation}\label{eq:Reshetikhin}
\begin{tikzpicture}[scale=.4, anchorbase]
	\draw[very thick] (1,0) to [out=90,in=270] (0,1.5);
	\draw[overcross] (0,0) to [out=90,in=270] (1,1.5);
	\draw[very thick] (0,0) to [out=90,in=270] (1,1.5);
	\draw[very thick, gray] (1,0) to (0,0);
	\draw[very thick, gray] (1,0) to (1,-1);
	\draw[very thick, gray] (0,0) to (0,-1);
\end{tikzpicture}
= -q^{-2} \;
\begin{tikzpicture}[scale=.4, anchorbase]
	\draw[very thick] (1,0) to (1,1);
	\draw[very thick] (0,0) to (0,1);
	\draw[very thick, gray] (1,0) to (0,0);
	\draw[very thick, gray] (1,0) to (1,-1);
	\draw[very thick, gray] (0,0) to (0,-1);
\end{tikzpicture}
+ (-1)^nq^{-2n-1}
\begin{tikzpicture}[scale=.4, tinynodes, anchorbase]
	\draw[very thick, gray] (0,0) to [out=90,in=180] (.5,.625) 
		to [out=0,in=90] (1,0);
	\draw[very thick] (0,2) to [out=270,in=180] (.5,1.375)
		to [out=0,in=270] (1,2);
\end{tikzpicture}
\quad , \qquad 
\begin{tikzpicture}[scale=.4, anchorbase]
	\draw[very thick] (0,0) to [out=90,in=270] (1,1.5);
	\draw[overcross] (1,0) to [out=90,in=270] (0,1.5);
	\draw[very thick] (1,0) to [out=90,in=270] (0,1.5);
	\draw[very thick, gray] (1,0) to (0,0);
	\draw[very thick, gray] (1,0) to (1,-1);
	\draw[very thick, gray] (0,0) to (0,-1);
\end{tikzpicture}
= -q^{2} \;
\begin{tikzpicture}[scale=.4, anchorbase]
	\draw[very thick] (1,0) to (1,1);
	\draw[very thick] (0,0) to (0,1);
	\draw[very thick, gray] (1,0) to (0,0);
	\draw[very thick, gray] (1,0) to (1,-1);
	\draw[very thick, gray] (0,0) to (0,-1);
\end{tikzpicture}
+ (-1)^nq^{2n+1}
\begin{tikzpicture}[scale=.4, tinynodes, anchorbase]
	\draw[very thick, gray] (0,0) to [out=90,in=180] (.5,.625) 
		to [out=0,in=90] (1,0);
	\draw[very thick] (0,2) to [out=270,in=180] (.5,1.375)
		to [out=0,in=270] (1,2);
\end{tikzpicture}
\end{equation}
	\end{lem}
\begin{proof}
Direct computation, using \eqref{eq:blackbraiding}, \eqref{eq:ggb-triangle+}, 
and \eqref{eq:blackspinH=I}.
\end{proof}

Consequently, $\gS$-labeled strands can ``absorb'' $1$-labeled crossings up to scalar, 
modulo lower terms. 
By \eqref{eq:otherdigonS}, $\gS$-labeled strands can also absorb $1$-labeled cap/cups.

\begin{proof}[Proof of Theorem \ref{thm:ladder}]
Let $W$ be a web in $\End_{\Web(\son)}(\gS^{\otimes m})$.
We will show that $W$ is a linear combination of products of the elements 
$\{r_{i,j}\}_{1 \leq i < j \leq m}$.
Observe that, if we ignore all black edges in $W$, 
we obtain a $1$-manifold $\gM_W$ embedded in $\R \times [0,1]$ with 
$m$ boundary points at both its top and bottom, e.g.~
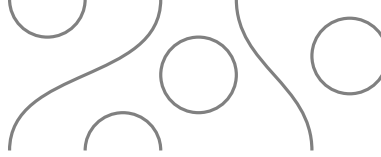
\[
\begin{tikzpicture}
	\draw[very thick, gray] (0,0) to [out=90,in=270] (2,2);
	\draw[very thick, gray] (4,0) to [out=90,in=270] (3,2);
	\draw[very thick, gray] (0,2) to [out=270,in=180] (.5,1.5) to [out=0,in=270] (1,2);
	\draw[very thick, gray] (1,0) to [out=90,in=180] (1.5,.5) to [out=0,in=90] (2,0);
	\draw[very thick, gray] (2.5,1) circle (.5);
	\draw[very thick, gray] (4.5,1.25) circle (.5);
	\end{tikzpicture} \, .
\]
The web $W$ in its entirety consists of $\gM_W$ as well as its ``all-black'' components, 
which live in the complement of $\gM_W$ and are attached to $\gM_W$ via strands 
carrying the labels $\{1,\ldots,n-1\}$.
We now undertake a step-by-step process which shows that 
$W$ lies in the span of successively simpler webs.
\newline

\noindent \textbf{Step 1: reduce black components to $1$-labeled tangles.}
Our first reduction is to rewrite $W$ as a non-zero multiple of a web 
where the interface between $\gM_W$ and the black portion of the web only carries the label $1$. 
Using \eqref{eq:ggb-triangle+}, 
we can rewrite any $(\gS, \gS, k+1)$ trivalent vertex as a web which only has $(\gS, \gS, 1)$ 
and $(\gS, \gS, k)$ trivalent vertices (up to scalar). 
Doing this repeatedly achieves our goal. 
Now each ``all-black'' component can be viewed as a web whose boundary 
consists only of $1$-labeled strands. 
We use Lemma \ref{lem:rewriteblack} 
to write all of the black portions of the web in the span of $1$-labeled tangles.

Consequently, we can assume $W$ takes this form; 
note that it is no longer a web, but rather a ``tangled web'', e.g.~
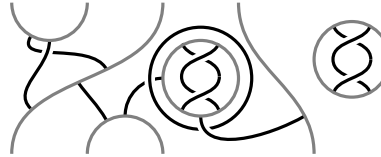
\[
\begin{tikzpicture}
		\draw[very thick] (1.25,.45) to [out=120,in=305] (.875,.95);
		\draw[very thick] (2,1) to [out=180,in=90] (1.5,.5);
		\draw[overcross] (2.5,1) circle (.675);
		\draw[very thick] (2.5,1) circle (.675);
		\draw[overcross] (2.5,.5) to [out=270,in=210] (3.875,.5);
		\draw[very thick] (2.5,.5) to [out=270,in=210] (3.875,.5);
		\begin{scope}[xshift=-57pt,yshift=-6.5pt]
		\clip (4.5,1.25) circle (.5);
		\draw[very thick] (4.75,1.25) to [out=90,in=270] (4.25,1.75);
		\draw[overcross] (4.75,.75) to [out=90,in=270] (4.25,1.25) to [out=90,in=270] (4.75,1.75);
		\draw[very thick] (4.75,.75) to [out=90,in=270] (4.25,1.25) to [out=90,in=270] (4.75,1.75);
		\draw[overcross] (4.25,.75) to [out=90,in=270] (4.75,1.25);
		\draw[very thick] (4.25,.75) to [out=90,in=270] (4.75,1.25);
		\end{scope}
	\draw[very thick, gray] (2.5,1) circle (.5);
		\begin{scope}
		\clip (4.5,1.25) circle (.5);
		\draw[very thick] (4.75,1.25) to [out=90,in=270] (4.25,1.75);
		\draw[overcross] (4.75,.75) to [out=90,in=270] (4.25,1.25) to [out=90,in=270] (4.75,1.75);
		\draw[very thick] (4.75,.75) to [out=90,in=270] (4.25,1.25) to [out=90,in=270] (4.75,1.75);
		\draw[overcross] (4.25,.75) to [out=90,in=270] (4.75,1.25);
		\draw[very thick] (4.25,.75) to [out=90,in=270] (4.75,1.25);
		\end{scope}
	\draw[very thick, gray] (4.5,1.25) circle (.5);
		\draw[very thick] (.25,1.55) to [out=240,in=120] (1.25,1.125);
		\draw[overcross] (.5,1.5) to [out=270,in=135] (.25,.6);
		\draw[very thick] (.5,1.5) to [out=270,in=135] (.25,.6);
	\draw[very thick, gray] (0,0) to [out=90,in=270] (2,2);
	\draw[very thick, gray] (4,0) to [out=90,in=270] (3,2);
	\draw[very thick, gray] (0,2) to [out=270,in=180] (.5,1.5) to [out=0,in=270] (1,2);
	\draw[very thick, gray] (1,0) to [out=90,in=180] (1.5,.5) to [out=0,in=90] (2,0);
	\end{tikzpicture} \, .
	\]
All black strands now have the label $1$, and, by construction, 
the tangled black portions of the web lie entirely in the complement 
of the planar gray $1$-manifold $\gM_W$. 
In particular, at this point there are no black-gray (or gray-gray) crossings.
\newline

\noindent \textbf{Step 2: Create uprights.}
We now show that the tangled webs $W$ produced in Step 1
are in the span of tangled webs wherein the planar gray $1$-manifold $\gM_W$ 
agrees with the $1$-manifold $\gray{I}_W$ consisting of $m$ ``vertical'' segments, 
i.e.~$\gM_W$ agrees with the gray $1$-manifold underlying any diagram in $\Lad_m$. 
This comes at a (temporary) cost: the introduction of black-gray crossings.

Observe that there exists a sequence of saddle moves
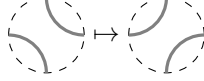
\[
\begin{tikzpicture}[scale=.5, tinynodes, anchorbase]
	\draw[dashed] (0,0) circle (1);
	\draw[very thick,gray] (-1,0) to [out=0,in=90] (0,-1);
	\draw[very thick,gray] (1,0) to [out=180,in=270] (0,1);
	\end{tikzpicture}
\mapsto
\begin{tikzpicture}[scale=.5, tinynodes, anchorbase]
	\draw[dashed] (0,0) circle (1);
	\draw[very thick,gray] (-1,0) to [out=0,in=270] (0,1);
	\draw[very thick,gray] (1,0) to [out=180,in=90] (0,-1);
	\end{tikzpicture}
\]
taking $\gM_W$ to $\gray{I}_W$: 
for example, we can first use saddle moves to iteratively remove 
any closed components in $\gM_W$ by joining them to components meeting the boundary,
and then use saddle moves to pass from the resulting crossingless matching to $\gray{I}_W$. 
Note that we use the $m\ge 1$ hypothesis of the theorem to ensure there is some gray component meeting the boundary. 
(If $m=0$, so there is no gray component meeting the boundary, 
then we certainly cannot saddle to join a given closed gray component to the boundary! 
For the $m=0$ case, see instead Corollary \ref{cor:ladder0}.)

In order to implement such saddle moves in $\Web(\son)$, 
we use \eqref{eq:spinH=I}. 
This relation implies that there exist scalars $A_{i,j} \in \C(q)$ for $0 \leq i,j \leq n$ so that 
\[
(\rung_1)^i :=
i \Big\{
\begin{tikzpicture}[scale=.375, tinynodes, anchorbase]
	\draw[very thick, gray] (0,-1.25) to (0,1.25);
	\draw[very thick, gray] (1,-1.25) to (1,1.25);
	\draw[very thick] (0,-.5) to node[below=-2pt]{$1$} (1,-.5);
	\node at (.5,-.25) {$\cdot$}; \node at (.5,0) {$\cdot$}; \node at (.5,.25) {$\cdot$};
	\draw[very thick] (0,.5) to node[above=-2pt]{$1$} (1,.5);
	\end{tikzpicture}
= A_{i,0} \begin{tikzpicture}[scale=.2, tinynodes, anchorbase]
	\draw[very thick,gray] (-.75,0) to (-.75,3.5);
	\draw[very thick,gray] (.75,0) to (.75,3.5);
	\end{tikzpicture}
+ \sum _{j=1}^{n}
A_{i,j}
\begin{tikzpicture}[scale=.2, tinynodes, anchorbase]
	\draw[very thick,gray] (-1,0) to (0,1);
	\draw[very thick,gray] (1,0) to (0,1);
	\draw[very thick,gray] (0,2.5) to (-1,3.5);
	\draw[very thick,gray] (0,2.5) to (1,3.5);
	\draw[very thick] (0,1) to node[right=-2pt]{$n{-}j$} (0,2.5);
	\end{tikzpicture} \, .
	\]
It follows from Definition \ref{eq:x1defn} and \cite[Proposition 4.25]{BER} that
$\{ \varphi\left((\rung_1)^i\right) \}_{i=0}^n \in \End_{U_q(\son)}(S \otimes S)$ is a basis. 
Since the webs \eqref{SSbasis} are also sent by $\varphi$ to a basis, 
this implies that the matrix $(A_{i,j})_{0 \leq i,j \leq n}$ is invertible.
Consequently, we find that 
\[
\begin{tikzpicture}[scale=.4, tinynodes, anchorbase]
	\draw[very thick,gray] (0,0) to [out=90,in=180] (.5,.625) 
		to [out=0,in=90] (1,0);
	\draw[very thick,gray] (0,2) to [out=270,in=180] (.5,1.375)
		to [out=0,in=270] (1,2);
\end{tikzpicture}
=
\sum_{i=0}^n A^{-1}_{n,i} \
\begin{tikzpicture}[scale=.375, tinynodes, anchorbase]
	\draw[very thick, gray] (0,-1.25) to (0,1.25);
	\draw[very thick, gray] (1,-1.25) to (1,1.25);
	\draw[very thick] (0,-.5) to node[below=-2pt]{$1$} (1,-.5);
	\node at (.5,-.25) {$\cdot$}; \node at (.5,0) {$\cdot$}; \node at (.5,.25) {$\cdot$};
	\draw[very thick] (0,.5) to node[above=-2pt]{$1$} (1,.5);
	\end{tikzpicture} \Big\} i \, .
\]
Using this, 
together with the gray-black Reidemeister II move established in Lemma \ref{lem:otherR2}, 
we can perform saddle moves on tangled webs, e.g.~
\[
\begin{tikzpicture}[scale=.5, tinynodes, anchorbase]
	\draw[very thick] (-.707,.707) to (.707,-.707);
	\draw[very thick,gray] (-1,0) to [out=0,in=90] (0,-1);
	\draw[very thick,gray] (1,0) to [out=180,in=270] (0,1);
	\draw[dashed] (0,0) circle (1);
	\end{tikzpicture}
=
\begin{tikzpicture}[scale=.5, tinynodes, anchorbase]
	\draw[very thick] (-.707,.707) to [out=315,in=135] (-.5,-.5) to [out=315,in=135] (.707,-.707);
	\draw[overcross] (-1,0) to [out=0, in=135] (0,0) to [out=315,in=90] (0,-1);
	\draw[very thick,gray] (-1,0) to [out=0, in=135] (0,0) to [out=315,in=90] (0,-1);
	\draw[very thick,gray] (1,0) to [out=180,in=270] (0,1);
	\draw[dashed] (0,0) circle (1);
	\end{tikzpicture}
=
\sum_{i=0}^n A^{-1}_{n,i} \
\begin{tikzpicture}[scale=.5, tinynodes, anchorbase]
	\draw[very thick] (-.707,.707) to [out=315,in=135] (-.5,-.5) to [out=315,in=135] (.707,-.707);
	\draw[overcross] (-1,0) to [out=0,in=270] (0,1);
	\draw[overcross] (1,0) to [out=180,in=90] (0,-1);	
	\draw[very thick] (-.5,.125) to (.125,-.5);
	\draw[very thick] (.5,-.125) to (-.125,.5);
	\draw[very thick,gray] (-1,0) to [out=0,in=270] (0,1);
	\draw[very thick,gray] (1,0) to [out=180,in=90] (0,-1);
	\draw[dashed] (0,0) circle (1);
	\node at (-.125,-.125) {$\cdot$}; \node at (0,0) {$\cdot$}; \node at (.125,.125) {$\cdot$};
	\end{tikzpicture} \, .
\]
As in this example, if we use gray-black Reidemeister II moves during this step, 
we will always use the version of this move that keeps the gray strands on top of any black strands.
\newline

\noindent \textbf{Step 3: Swing the black strands outside.} 
We consider a tangled web $W$ arising from Step 2, 
which necessarily satisfies $\gM_W = \gray{I}_W$ and 
may contain both black-black crossings and gray-black crossings with 
gray strands always passing over black strands.

Since the gray strands all lie above the black strands, 
when viewing this tangled web 3-dimensionally, it e.g.~takes the form
\begin{equation}\label{eq:web3d}
\begin{tikzpicture}[scale=.5, anchorbase]
	\draw[very thick] (.5,1) to (1.5,-.5);
	\draw[very thick] (.75,1.5) to (1.75,-.5);
	\draw[very thick,gray] (0,0) to (1,2);
	\draw[overcros] (1,0) to (2,2);
	\draw[very thick] (1.5,1) to (2,-.5);
	\draw[very thick] (1.75,1.5) to (2.25,-.5);	
	\draw[very thick,gray] (1,0) to (2,2);
	\node at (2.5,1) {\scs$\cdots$};
	\draw[very thick] (4.5,1) to (3.5,-.5);
	\draw[very thick] (4.25,1.5) to (3.25,-.5);
	\draw[very thick,gray] (5,0) to (4,2);
	\draw[overcros] (4,0) to (3,2);
	\draw[very thick] (3.5,1) to (3,-.5);
	\draw[very thick] (3.25,1.5) to (2.75,-.5);
	\draw[very thick,gray] (4,0) to (3,2);
	\draw[very thick] (1,-.5) rectangle (4,-1.5);
	\node at (2.5,-1) {$T'$};
	\end{tikzpicture} \, .
	\end{equation}
where $T'$ denotes some all-black $1$-labeled tangle.
Via isotopy, pulling all interesting parts of the tangle $T'$ to the right
results in a tangled web e.g.~of the form
\begin{equation}\label{eq:twebright}
\begin{tikzpicture}[scale=.5, anchorbase]
	\draw[very thick] (5,3.5) rectangle (6,.5);
	\node at (5.5,2) {$T$};
	\draw[very thick] (0,3.25) to (5,3.25);
	\draw[very thick] (0,3) to (5,3);
	\draw[overcross] (1,0) to (1,4);
	\draw[very thick] (1,2.625) to (5,2.625);
	\draw[very thick] (1,2.375) to (5,2.375);
	\draw[overcross] (3,0) to (3,4);	
	\draw[very thick] (3,1.625) to (5,1.625);
	\draw[very thick] (3,1.375) to (5,1.375);
	\draw[overcross] (4,0) to (4,4);		
	\draw[very thick] (4,1) to (5,1);
	\draw[very thick] (4,.75) to (5,.75);	
	\draw[very thick,gray] (0,0) to (0,4);
	\draw[very thick,gray] (1,0) to (1,4);
	\node at (2,2) {$\cdots$};
	\draw[very thick,gray] (3,0) to (3,4);
	\draw[very thick,gray] (4,0) to (4,4);
	\end{tikzpicture} \, .
	\end{equation}
Here, again, $T$ denotes some all-black $1$-labeled tangle.

We now claim that the original tangled web $W$ is equal in $\Web(\son)$ 
to the isotopic tangled web in \eqref{eq:twebright}, 
up to an invertible scalar.
In terms of the original tangled web $W$, 
we arrive at one of the form \eqref{eq:twebright} via a sequence of 
all-black Reidemeister moves (Lemmata \ref{lem:otherR2} and \ref{lem:1R13}), 
gray-black Reidemeister II moves (Lemma \ref{lem:otherR2}), 
gray-black-black Reidemeister III moves (Proposition \ref{prop:S1}), 
and fork twist moves \eqref{eq:forktwist}. 
As given in those results, these moves all hold up to factors of the form $\pm q^d$.

The isotopy relating $W$ and \eqref{eq:twebright} 
does not change the number of gray-gray-black trivalent vertices. 
The number of such vertices is necessarily even, as it equals the 
the number of black strands entering the tangle $T$.
(In both \eqref{eq:web3d} and \eqref{eq:twebright} we have depicted 
exactly two black strands attaching to each gray upright, 
but in general there may be more or fewer.)
\newline

\noindent \textbf{Step 4: Create (generalized) rungs.}
We now assume that $W$ takes the form \eqref{eq:twebright} and argue,
via induction on the number $2v$ of gray-gray-black trivalent vertices, 
that $W$ lies in the subalgebra of $\End_{\Web(\son)}(\gS^{\otimes m})$ generated by the $r_{i,j}$.

Given \eqref{eq:BMWskein}, 
we can use Lemma \ref{lem:BMW-span} to write $T$ in the span of a chosen set of 
lifts of reduced matchings
(here, our ``disk'' is a rectangle with all of the boundary points along its left side).
We choose our lifts so that curves with lower bottom boundary are behind 
those with higher bottom boundary (as in \eqref{eq:matchingex}). 
We thus can assume that $T$ takes this form, which, for the duration, 
we refer to simply as a lifted matching.
In particular, this establishes the base $v=0$ case of our induction, 
since the tangle $T$ is reduced to a multiple of the empty diagram
(the only reduced matching of $0$ points is the empty matching).

Now, consider $W$ as in \eqref{eq:twebright}, 
such that $T$ is a lifted matching having $2v$ boundary points with $v \geq 1$. 
We consider the strand (highlighted in \bl{blue}) 
corresponding to the lowest boundary point of $T$, 
and now consider two cases. 
In the following computations, we use the symbol $\approx$ to mean 
``equal up to non-zero scalar.''

In the first case, this blue strand connects two distinct uprights, and we compute e.g.~
\begin{equation}\label{eq:createrung}
W =
\begin{tikzpicture}[scale=.5, anchorbase]
	\draw[very thick] (0,3.25) to (5,3.25);
	\draw[very thick] (0,3) to (5,3);
	\draw[overcross] (1,0) to (1,4);
	\draw[very thick,blue] (1,2.625) to (5,2.625);
	\draw[very thick] (1,2.375) to (5,2.375);
	\draw[overcross] (3,0) to (3,4);	
	\draw[very thick] (3,1.625) to (5,1.625);
	\draw[very thick] (3,1.375) to (5,1.375);
	\draw[overcross] (4,0) to (4,4);		
	\draw[very thick] (4,1) to (5,1);
	\draw[very thick,blue] (4,.75) to (5,.75);	
	\draw[very thick,gray] (0,0) to (0,4);
	\draw[very thick,gray] (1,0) to (1,4);
	\node at (2,2) {$\cdots$};
	\draw[very thick,gray] (3,0) to (3,4);
	\draw[very thick,gray] (4,0) to (4,4);
	\draw[very thick] (5,3.5) rectangle (6,.5);
	\node at (5.5,2) {$T$};
	\end{tikzpicture}
=
\begin{tikzpicture}[scale=.5, anchorbase]
	\draw[very thick] (0,3.25) to (5,3.25);
	\draw[very thick] (0,3) to (5,3);
	\draw[overcross] (1,0) to (1,4);
	\draw[very thick,blue] (1,2.625) to (4.25,2.625);
	\draw[very thick, blue] (4.2,2.625) to (4.25,2.625) to [out=0,in=0]
		(4.25,.75) to (4.2,.75);
	\draw[overcros] (1,2.375) to (5,2.375);	
	\draw[very thick] (1,2.375) to (5,2.375);
	\draw[overcross] (3,0) to (3,4);
	\draw[overcros] (3,1.625) to (5,1.625);	
	\draw[very thick] (3,1.625) to (5,1.625);
	\draw[overcros] (3,1.375) to (5,1.375);	
	\draw[very thick] (3,1.375) to (5,1.375);
	\draw[overcross] (4,0) to (4,4);
	\draw[overcros] (4,1) to (5,1);
	\draw[very thick] (4,1) to (5,1);
	\draw[very thick,blue] (4,.75) to (4.25,.75);	
	\draw[very thick,gray] (0,0) to (0,4);
	\draw[very thick,gray] (1,0) to (1,4);
	\node at (2,2) {$\cdots$};
	\draw[very thick,gray] (3,0) to (3,4);
	\draw[very thick,gray] (4,0) to (4,4);
	\draw[very thick] (5,3.5) rectangle (6,.5);
	\node at (5.5,2) {$\tilde{T}$};
	\end{tikzpicture}
\approx
\begin{tikzpicture}[scale=.5, anchorbase]
	\draw[very thick] (0,3.25) to (5,3.25);
	\draw[very thick] (0,3) to (5,3);
	\draw[overcross] (1,0) to (1,4);
	\draw[very thick,blue] (1,2.625) to [out=0,in=180] (3,.75) to (4,.75);
	\draw[overcros] (1,2.375) to (5,2.375);	
	\draw[very thick] (1,2.375) to (5,2.375);
	\draw[overcross] (3,0) to (3,4);
	\draw[overcros] (3,1.625) to (5,1.625);	
	\draw[very thick] (3,1.625) to (5,1.625);
	\draw[overcros] (3,1.375) to (5,1.375);	
	\draw[very thick] (3,1.375) to (5,1.375);
	\draw[overcross] (4,0) to (4,4);
	\draw[very thick,blue] (3.5,.75) to (4,.75);	
	\draw[overcros] (4,1) to (5,1);
	\draw[very thick] (4,1) to (5,1);
	\draw[very thick,gray] (0,0) to (0,4);
	\draw[very thick,gray] (1,0) to (1,4);
	\node at (2,2.75) {$\cdots$};
	\draw[very thick,gray] (3,0) to (3,4);
	\draw[very thick,gray] (4,0) to (4,4);
	\draw[very thick] (5,3.5) rectangle (6,.5);
	\node at (5.5,2) {$\tilde{T}$};
	\end{tikzpicture} \, .
\end{equation}
The last step here uses \eqref{eq:S11R3} and \eqref{eq:forktwist}.
Continuing the computation in \eqref{eq:createrung}, 
we have
\[
\stackrel{{\eqref{eq:Reshetikhin}}}{\approx}
\begin{tikzpicture}[scale=.5, anchorbase]
	\draw[very thick] (0,3.25) to (5,3.25);
	\draw[very thick] (0,3) to (5,3);
	\draw[overcross] (1,0) to (1,4);
	\draw[very thick,blue] (1,.5) to [out=0,in=180] (3,.5) to (4,.5);
	\draw[overcros] (1,2.375) to (5,2.375);	
	\draw[very thick] (1,2.375) to (5,2.375);
	\draw[overcross] (3,0) to (3,4);
	\draw[overcros] (3,1.625) to (5,1.625);	
	\draw[very thick] (3,1.625) to (5,1.625);
	\draw[overcros] (3,1.375) to (5,1.375);	
	\draw[very thick] (3,1.375) to (5,1.375);
	\draw[overcross] (4,0) to (4,4);
	\draw[very thick,blue] (3.5,.5) to (4,.5);	
	\draw[overcros] (4,1) to (5,1);
	\draw[very thick] (4,1) to (5,1);
	\draw[very thick,gray] (0,0) to (0,4);
	\draw[very thick,gray] (1,0) to (1,4);
	\node at (2,2) {$\cdots$};
	\draw[very thick,gray] (3,0) to (3,4);
	\draw[very thick,gray] (4,0) to (4,4);
	\draw[very thick] (5,3.5) rectangle (6,.5);
	\node at (5.5,2) {$\tilde{T}$};
	\end{tikzpicture}
	+ LOT_{< 2v} \, .
	\]
Here, 
$LOT_{< 2v}$ denotes a linear combination of terms of the form \eqref{eq:twebright} 
with strictly fewer than $2v$ trivalent vertices.
The result now follows by induction, since $\tilde{T}$ (also) has strictly fewer than $2v$ boundary points.

In the second case, the blue strand connects an upright to itself.
We do a similar computation, but at the last step instead arrive at
\[
\stackrel{{\eqref{eq:Reshetikhin}}}{\approx}
\begin{tikzpicture}[scale=.5, anchorbase]
	\draw[very thick] (0,3.25) to (5,3.25);
	\draw[very thick] (0,3) to (5,3);
	\draw[overcross] (1,0) to (1,4);
	\draw[overcros] (1,2.375) to (5,2.375);	
	\draw[very thick] (1,2.375) to (5,2.375);
	\draw[overcross] (3,0) to (3,4);
	\draw[overcros] (3,1.625) to (5,1.625);	
	\draw[very thick] (3,1.625) to (5,1.625);
	\draw[overcros] (3,1.375) to (5,1.375);	
	\draw[very thick] (3,1.375) to (5,1.375);
	\draw[overcross] (4,0) to (4,4);
	\draw[very thick,blue] (4,.25) to [out=0,in=270] (4.25,.5) to [out=90,in=0] (4,.75);	
	\draw[overcros] (4,1) to (5,1);
	\draw[very thick] (4,1) to (5,1);
	\draw[very thick,gray] (0,0) to (0,4);
	\draw[very thick,gray] (1,0) to (1,4);
	\node at (2,2) {$\cdots$};
	\draw[very thick,gray] (3,0) to (3,4);
	\draw[very thick,gray] (4,0) to (4,4);
	\draw[very thick] (5,3.5) rectangle (6,.5);
	\node at (5.5,2) {$\tilde{T}$};
	\end{tikzpicture}
	+ LOT_{< 2v}
\stackrel{{\eqref{eq:otherdigonS}}}{=} LOT_{< 2v} \, . \qedhere
	\]
\end{proof}

\begin{cor}\label{cor:ladder0}
The space $\End_{\Web(\son)}(\varnothing)$ of endomorphisms of the 
empty sequence $\varnothing$ (the monoidal unit in $\Web(\son)$) 
is spanned by the empty web. 
In other words, Theorem \ref{thm:ladder} also holds when $m=0$.
Consequently, $\End_{\Web(\son)}(\varnothing) \cong \C(q)$.
\end{cor}

\begin{proof}
Consider the linear map
\[
\iota_{\gS} \colon \End_{\Web(\son)}(\varnothing) \to \End_{\Web(\son)}(\gS)
\]
which adds a gray strand to the left of a closed web, 
and the linear map
\[
\pi_{\gS} \colon \End_{\Web(\son)}(\gS) \to \End_{\Web(\son)}(\varnothing)
\]
which takes the closure (to the left) of the web, 
connecting the gray boundary points. 
The composition $\pi_{\gS} \circ \iota_{\gS}$ adds a gray circle to the left of the closed web, 
which by \eqref{eq:circleS} is just multiplication by a unit. 
Thus $\pi_{\gS}$ is surjective. 
By Theorem \ref{thm:ladder}, 
$\End_{\Web(\son)}(\gS) = \Lad_1$ is spanned by a single gray upright. 
Applying $\pi_{\gS}$, 
$\End_{\Web(\son)}(\varnothing)$ is spanned by the gray circle, 
or, equivalently, by the empty diagram.
\end{proof}

\section{Fully faithfullness via duality with $\iota$quantum groups}\label{ss:FF}

To finish our proof of Theorem \ref{thm:main}, 
we study the following algebra, 
which is in duality with $U_q(\son)$ on the module $S^{\otimes m}$.

\begin{defn}\label{def:iQ}
The algebra
$\iqUsom$ is the $\C(q)$-algebra with generators $b_1, \dots, b_{m-1}$ and relations 
\begin{equation}\label{eqn:farcommute}
b_ib_j = b_jb_i \quad |i-j|>1
\end{equation}
\begin{equation}\label{eqn:iSerre-alg-version}
b_i^2b_{i\pm 1} + b_{i\pm 1}b_i^2 = -[2]_{q^2}b_ib_{i\pm 1}b_i + b_{i\pm 1} \, .
\end{equation}
\end{defn}

As outlined in \S\ref{S:outline}, 
we will use the structure theory and representation theory of this algebra to prove fully faithfulness. 
Some prefatory comments on its representation theory are helpful to set the context.

The algebra $\iqUsom$ appears in two distinct guises. 
As noted previously, 
$\iqUsom$ is the $\iota$\emph{quantum group} associated with the split symmetric pair $(\slm, \som)$, 
meaning that it appears as a (coideal) subalgebra of the quantum group $U_q(\slm)$. 
Finite-dimensional (type I) representations of the latter
can thus be restricted to produce finite-dimensional representations of $\iqUsom$.
These restrictions are called \emph{classical} representations. 
By \cite{MR2143754}, the operators $b_i$ act diagonalizably on these classical representations
with eigenvalues $[k]_{-q^2}$ for $k\in \mathbb{Z}$. 
There are also additional representations which are called classical; 
see Remark \ref{rem:moreclassical}.

Meanwhile, Wenzl \cite{Wenzl-Spin} 
proved that $\iqUsom$ also acts by $U_q(\son)$-intertwiners on the representation $S^{\otimes m}$. 
On these representations, 
the operators $b_i$ again act diagonalizably, 
but now the eigenvalues are the signed quantum numbers $(-1)^{j-1} \frac{[2j-1]}{[2]}$ for $j \ge 1$. 
The same is true for direct summands of $S^{\otimes m}$, 
and irreducible representations of this form are called \emph{type I nonclassical} representations. 
Thus, one can distinguish between classical and nonclassical representations using the spectrum of $b_i$. 
In this paper, we focus on type I nonclassical representations of $\iqUsom$.

\begin{rem} 
Many aspects of $\iota$quantum groups 
(e.g.~canonical bases, $K$-matrices, etcetera) 
have recently been developed in the literature \cite{WangICM}, 
but this literature focuses on classical representations restricted from $U_q(\slm)$. 
These developments are not known to have analogues for nonclassical representations. 
\end{rem}

\subsection{Relating $\iota$quantum groups to ladders}

To begin, we record the following consequence of \eqref{eq:preiQ}.

\begin{prop}\label{prop:iQtoLad}
The assignments 
\[
b_i\mapsto
\rung_{i,i+1} = 
\begin{tikzpicture}[scale=.5, tinynodes, anchorbase]
	\draw[thick, black] (1,0) to node[below=-1pt]{\scs$1$} (2,0);
	\draw[very thick,gray] (0,-1) to (0,1);
	\node at (.5,0) {$\mydots$};
	\draw[very thick,gray] (1,-1) to (1,1);
	\draw[very thick,gray] (2,-1) to (2,1);
	\node at (2.5,0) {$\mydots$};
	\draw[very thick,gray] (3,-1) to (3,1);	
\end{tikzpicture}
\]
define a surjective $\C(q)$-algebra homomorphism 
$\psi \colon \iqUsom \to
\End_{\Web(\son)}(\gS^{\otimes m})$.
\end{prop}
\begin{proof}
By Definition \ref{def:ladder} and Theorem \ref{thm:ladder}, 
it suffices to show that these assignments give a well-defined algebra homomorphism 
$\iqUsom \to \Lad_m$.
Thus, we check that the relations in Definition \ref{def:iQ} are satisfied in $\Lad_m$.
Equation \eqref{eqn:farcommute} holds via isotopy, 
so it suffices to check \eqref{eqn:iSerre-alg-version} which
is the web relation
\begin{equation}\label{eqn:iSerre}
\begin{tikzpicture}[scale=.4, anchorbase]
	\draw[very thick] (0,1.5) to (1,1.5);
	\draw[very thick] (0,.75) to (1,.75);
	\draw[very thick] (1,2.25) to (2,2.25);
	\draw[very thick, gray] (0,0) to (0,3);
	\draw[very thick, gray] (1,0) to (1,3);
	\draw[very thick, gray] (2,0) to (2,3);
\end{tikzpicture}
+ 
\begin{tikzpicture}[scale=.4, anchorbase]
	\draw[very thick] (0,1.5) to (1,1.5);
	\draw[very thick] (0,2.25) to (1,2.25);
	\draw[very thick] (1,.75) to (2,.75);
	\draw[very thick, gray] (0,0) to (0,3);
	\draw[very thick, gray] (1,0) to (1,3);
	\draw[very thick, gray] (2,0) to (2,3);
\end{tikzpicture}
= 
-(q^2 + q^{-2})
\begin{tikzpicture}[scale=.4, anchorbase]
	\draw[very thick, gray] (0,0) to (0,3);
	\draw[very thick, gray] (1,0) to (1,3);
	\draw[very thick, gray] (2,0) to (2,3);
	\draw[very thick] (0,.75) to (1,.75);
	\draw[very thick] (1,1.5) to (2,1.5);
	\draw[very thick] (0,2.25) to (1,2.25);
\end{tikzpicture}
+
\begin{tikzpicture}[scale=.4, anchorbase]
	\draw[very thick, gray] (0,0) to (0,3);
	\draw[very thick, gray] (1,0) to (1,3);
	\draw[very thick, gray] (2,0) to (2,3);
	\draw[very thick] (1,1.5) to (2,1.5);
\end{tikzpicture} \, .
\end{equation}
This is a straightforward computation using 
\eqref{eq:blackspinH=I}, \eqref{eq:preiQ}, and \eqref{eq:blackrungtriangle} to simplify both sides.
	\end{proof}

The functor $\varphi$ from Theorem \ref{thm:functor} 
gives algebra homomorphisms 
\begin{equation}\label{eq:phi}
\varphi \colon
\End_{\Web(\son)}(\gS^{\otimes m}) \rightarrow \End_{U_q(\son)}(S^{\otimes m}) \, ,
	\end{equation}
so pre-composing with $\psi$ yields an algebra homomorphism
\begin{equation}\label{eq:phipsi}
\varphi \circ \psi \colon \iqUsom \to \End_{U_q(\son)}(S^{\otimes m}) \, .
	\end{equation}
Our previous work gives the following.

\begin{prop}[{\cite[Proposition B.10]{BER}}]\label{prop:phipsi-isWenzl}
The algebra homomorphism in \eqref{eq:phipsi}
agrees with the homomorphism constructed by Wenzl in \cite{Wenzl-Spin}. \qed
\end{prop}

\begin{prop}\label{prop:alg-hom-surj}
The algebra homomorphism in \eqref{eq:phipsi} is surjective.
\end{prop}

\begin{proof}
In \cite[Theorem 5.3]{Wenzl-Spin}, Wenzl shows that his homomorphism, 
and thus \eqref{eq:phipsi}, is surjective.
\end{proof}

It remains to show that \eqref{eq:phi} is injective, 
which in particular requires that its domain is finite-dimensional.
However, the algebra $\iqUsom$ is infinite-dimensional, 
thus the homomorphism $\psi$ from Proposition \ref{prop:iQtoLad} should have a kernel. 
We now begin to study this kernel.
The preimage of the elements in $\Web(\son)$ from Definition \ref{def:grayX} will be crucial to 
our analysis.

\begin{lem}\label{lem:xpsiX}
For $1 \leq i \leq m-1$, set
\begin{equation}\label{eq:blackx}
\xxb_i := \xxb_i^{(1)} := b_i - \frac{1}{[2]}
\quad \text{and} \quad
\xxb_i^{(k+1)} := \dfrac{(-1)^k}{``[k+1]^2"}\left(\xxb_i^{(k)} \xxb_i - (-1)^{k} ``[k][k+1]"\xxb_i^{(k)}\right)
\, \text{for } k \geq 1 \, .
	\end{equation}
Then, 
\begin{equation}\label{eq:xpsiX}
\psi(\xxb_i^{(k)}) 
= \id_{\gS^{\otimes i-1}} \otimes \xx^{(k)} \otimes \id_{\gS^{\otimes m-i-1}}
=: \xx_i^{(k)} \, .
\end{equation}
	\end{lem}
\begin{proof}
This is immediate from \eqref{eq:x1defn}, \eqref{eq:xidefn}, 
and Proposition \ref{prop:iQtoLad}.
\end{proof}

\begin{lem}\label{lem:minpoly-vs-dividedpower}
If $p_k(t):= \prod_{i=1}^k (t -(-1)^{i-1}\frac{[2i-1]}{[2]})$, then 
$(-1)^{\binom{k}{2}}\frac{p_k(b_i)}{(``[k]^2")!} = \xxb_i^{(k)}$.
\end{lem}
\begin{proof}
The recursion defining $\xxb_i^{(k)}$ can be rewritten as
\[
\xxb_i^{(k)} = (-1)^{k-1}\xxb_i^{(k-1)} \left( \frac{\xxb_i - (-1)^{k-1} ``[k-1][k]" }{``[k]^2"} \right) \, .
\]
Using $\xxb_i = b_i - \frac{1}{[2]}$, 
along with the identity 
$``[k-1][k]" = \frac{[2k-1] - (-1)^{k-1}[1]}{[2]}$, 
we find
\[
\xxb_i^{(k)} = (-1)^{\binom{k}{2}}\frac{(b_i - \frac{1}{[2]})(b_i + \frac{[3]}{[2]}) 
	\cdots (b_i - (-1)^{k-1}\frac{[2k-1]}{[2]})}{(``[k]^2")!} \, . \qedhere
\]
\end{proof}

\begin{prop}\label{prop:minpoly}
$\psi(p_{n+1}(b_i))=0$ for $i=1, \ldots, m-1$. 
\end{prop}
\begin{proof}
If we define $\xx_i^{(n+1)} := \psi(\xxb_i^{(n+1)})$, then \eqref{eq:blackx} implies that
\[
\xx_i^{(n+1)} = \dfrac{(-1)^n}{``[n+1]^2"}\left(\xx_i^{(n)} \xx_i - (-1)^{n} ``[n][n+1]"\xx_i^{(n)}\right)
\]
(i.e.~the extension of \eqref{eq:xidefn} to the case $i=n$).
Now, \eqref{eq:ItoX} and \eqref{eq:lambdas} imply that 
$\xx^{(n)} = 
\begin{tikzpicture}[scale=.3, tinynodes, anchorbase]
	\draw[very thick,gray] (0,0) to [out=90,in=180] (.5,.625) 
		to [out=0,in=90] (1,0);
	\draw[very thick,gray] (0,2) to [out=270,in=180] (.5,1.375)
		to [out=0,in=270] (1,2);
\end{tikzpicture}$, thus
\[
\xx^{(n)} \xx = 
\begin{tikzpicture}[scale=.3, tinynodes, anchorbase]
	\draw[very thick,gray] (0,-.5) to (0,0) to [out=90,in=180] (.5,.625) 
		to [out=0,in=90] (1,0) to (1,-.5);
	\draw[very thick] (0,0) to (1,0);
	\draw[very thick,gray] (0,2) to [out=270,in=180] (.5,1.375)
		to [out=0,in=270] (1,2);
\end{tikzpicture}
-
\frac{1}{[2]}
\begin{tikzpicture}[scale=.3, tinynodes, anchorbase]
	\draw[very thick,gray] (0,0) to [out=90,in=180] (.5,.625) 
		to [out=0,in=90] (1,0);
	\draw[very thick,gray] (0,2) to [out=270,in=180] (.5,1.375)
		to [out=0,in=270] (1,2);
\end{tikzpicture}
\stackrel{{\eqref{eq:otherdigonS}}}{=}
(-1)^n \frac{[2n+1]}{[2]}
\begin{tikzpicture}[scale=.3, tinynodes, anchorbase]
	\draw[very thick,gray] (0,0) to [out=90,in=180] (.5,.625) 
		to [out=0,in=90] (1,0);
	\draw[very thick,gray] (0,2) to [out=270,in=180] (.5,1.375)
		to [out=0,in=270] (1,2);
\end{tikzpicture}
-
\frac{1}{[2]}
\begin{tikzpicture}[scale=.3, tinynodes, anchorbase]
	\draw[very thick,gray] (0,0) to [out=90,in=180] (.5,.625) 
		to [out=0,in=90] (1,0);
	\draw[very thick,gray] (0,2) to [out=270,in=180] (.5,1.375)
		to [out=0,in=270] (1,2);
\end{tikzpicture}
= (-1)^n ``[n][n+1]" \xx^{(n)} \, ,
\]
so we conclude that $\xx_i^{(n+1)}= 0$.
The result then follows from Lemma \ref{lem:minpoly-vs-dividedpower}.
\end{proof}

\begin{defn}\label{def:iQn}
Let the algebra $\iqUsom^{\le n}$ be the quotient of $\iqUsom$ 
by the two-sided ideal with generators $p_{n+1}(b_i)$ for $i=1, \ldots, m-1$. 
\end{defn}

Note that $\psi$ factors through $\iqUsom^{\le n}$. 
We denote the induced homomorphism 
$\iqUsom^{\le n}\rightarrow \End_{\Web(\son)}(\gS^{\otimes m})$
by $\psi^{\le n}$.

\begin{rem}\label{rem:bi-spectrum} 
The quotient defining $\iqUsom^{\le n}$ specifies the spectrum of the elements $b_i$.
Specifically, the representations of $\iqUsom^{\le n}$ are precisely 
the representations of $\iqUsom$ for which the eigenvalues of each $b_i$ 
are contained in the set $\{(-1)^{j-1} \frac{[2j-1]}{[2]}\}_{1 \le j \le n}$. 
\end{rem}

\subsection{Structure theory}

In this section, we establish the finite-dimensionality of $\iqUsom^{\le n}$. 
Our main tool is the PBW basis for $\iqUsom$ given by Iorgov--Klimyk \cite{MR1903121},
which we argue descends to a finite spanning set after quotienting by the relations $p_{n+1}(b_i)=0$.
Our approach is novel, and involves interpreting the image of the PBW basis in $\iqUsom^{\le n}$ 
using an analogue of Lusztig's quantum Weyl group for $\iqUsom$ suggested by the 
formulae \eqref{eq:SpinBraiding}. 

\begin{defn}
For $i=1, \ldots, m-1$, let
$\iQW_i^{\pm 1}\in \iqUsom^{\le n}$ be given by the formulae
\begin{equation} \label{defofbraidingviaX}
\iQW_i = \sum_{k\ge 0}q^{-k}\xxb_i^{(k)} 
\qquad \text{and} \qquad 
\iQW_i^{-1} = \sum_{k\ge 0}q^k\xxb_i^{(k)}
\end{equation}
\end{defn}

By Lemma \ref{lem:minpoly-vs-dividedpower} and Definition \ref{def:iQn}, 
$\xxb_i^{(k)} = 0$ in $\iqUsom^{\le n}$ for $k \geq n+1$. Thus, 
\eqref{defofbraidingviaX} are indeed finite sums.

\begin{rem}\label{rem:TisR}
Denote $\iQW^{\pm} := \iQW_1^{\pm} \in \iqUsom[2]^{\le n}$.
By \eqref{eq:SpinBraiding}, we have that
\[
\psi^{\leq n}(\iQW) = q^{-n/2} c_{\gS,\gS} 
\quad \text{and} \quad
\psi^{\leq n}(\iQW^{-1}) = q^{n/2} c_{\gS,\gS}'
\]
and thus Proposition \ref{prop:SpinBraiding} gives
\[
\varphi \circ \psi^{\leq n}(\iQW^{\pm 1}) = q^{\mp n/2} R_{S,S}^{\pm} \in \End_{U_q(\son)}(S \otimes S) \, .
\]
\end{rem}

\begin{prop}\label{prop:iQW-invertible}
We have $\iQW_i \iQW_i^{-1} = 1 = \iQW_i^{-1} \iQW_i$ in $\iqUsom^{\le n}$. 
\end{prop}
\begin{proof} 
Since $\iqUsom[2] = \C(q)[b]$ is a polynomial ring in one variable, it follows that $\iqUsom[2]^{\le n}$ is $n+1$ dimensional 
(it is the quotient of $\C(q)[b]$ by the degree $n+1$ polynomial $p_{n+1}$).
By \eqref{eq:SSdecomp} and Schur's lemma, the $\C(q)$-algebra $\End_{U_q(\son)}(S \otimes S)$ is also $n+1$ dimensional.

By Proposition \ref{prop:alg-hom-surj} and the rank-nullity theorem, $\varphi\circ \psi^{\leq n}$ is an isomorphism. 
So it suffices to show that the images of $\iQW$ and $\iQW^{-1}$ are mutually inverse under 
$\varphi\circ \psi^{\le n}$. This follows from Remark \ref{rem:TisR}.
\end{proof}

\begin{rem}
We will show that $\iQW_i\iQW_{i\pm1}\iQW_i = \iQW_{i\pm 1}\iQW_i\iQW_{i\pm1}$ in Theorem \ref{thm:iquantumWeylR3}.
\end{rem}

Motivated by \eqref{eq:longb} and \eqref{eq:blackgraybraiding},
we now set
\begin{equation}\label{eq:longb}
b_{i,i+1}:=q^{-1}b_ib_{i+1}+ qb_{i+1}b_{i} \, .
\end{equation}
Note that \eqref{eq:webrung13} gives
\[
\psi(b_{i,i+1}) = 
\rung_{i,i+2} :=
\begin{tikzpicture}[scale=.5, tinynodes, anchorbase]
	\draw[thick, black] (1,0) to (3,0);
	\draw[very thick,gray] (0,-1) to (0,1);
	\node at (.5,.5) {$\mydots$};
	\draw[very thick,gray] (1,-1) to (1,1);
	\draw[overcross] (2,-1) to (2,1);
	\draw[very thick,gray] (2,-1) to (2,1);
	\draw[very thick,gray] (3,-1) to (3,1);
	\node at (3.5,.5) {$\mydots$};
	\draw[very thick,gray] (4,-1) to (4,1);
\end{tikzpicture} \, .
	\]
(Warning: the indexing in $\psi(b_{i,i+1}) = r_{i,i+2}$ does not directly align; 
we remedy this in \eqref{eq:longHdef} below.)

The following essential ingredient in our proof of finite-dimensionality 
implies that the elements in \eqref{eq:longb} also satisfy $p_{n+1}$ in $\iqUsom^{\le n}$.

\begin{prop}\label{prop-braid-conjugates}
For $1 \leq i \leq m-2$, we have 
\begin{equation}\label{eq:braidconj}
b_{i,i+1} = \iQW_{i} b_{i+1} \iQW_{i}^{-1}
\end{equation}
in $\iqUsom^{\le n}$.
\end{prop}

Thanks to Proposition \ref{prop:iQW-invertible}, 
the equality \eqref{eq:braidconj} claimed in Proposition \ref{prop-braid-conjugates} is equivalent to 
\[
b_{i,i+1}\iQW_i = \iQW_{i}b_{i+1}, 
\]
which by \eqref{defofbraidingviaX} is equivalent to
\begin{equation}\label{eqn:longBbraidconjugate-suffices}
\sum_{k\ge 0}q^{-k}b_{i,i+1}\xxb_{i}^{(k)} = \sum_{k \ge 0}q^{-k}\xxb_{i}^{(k)} b_{i+1} \, .
\end{equation}
To establish this equivalent formulation of Proposition \ref{prop-braid-conjugates}, 
we first state a technical lemma proven in Appendix \ref{Appendix:proof-of-commutator-lemma}.

\begin{lem}\label{lem:techcomm}
For $k\in \mathbb{Z}_{\ge 0}$, the following relation holds in $\iqUsom$:
\[
\pushQED{\qed} 
q^{-k}b_{i,i+1}\xxb_{i}^{(k)} = (-1)^kq^{k}\xxb_{i}^{(k)}b_{i,i+1} 
	+(-1)^kq^{k-1}\xxb_i^{(k-1)}b_{i,i+1} + q^{-k+1}\xxb_i^{(k-1)}b_{i+1} \, . 
\qedhere \popQED
	\]
\end{lem}

\begin{proof}[Proof of Proposition \ref{prop-braid-conjugates}]
It follows from Lemma \ref{lem:techcomm} that
\begin{equation}\label{eqn:longBbraidconjugate-sumtoshow}
\sum_{k\ge 0}q^{-k}b_{i,i+1}\xxb_{i}^{(k)} 
= \sum_{k\ge 0}\bigg((-1)^kq^{k}\xxb_{i}^{(k)}b_{i,i+1} +(-1)^kq^{k-1}\xxb_i^{(k-1)}b_{i,i+1} + q^{-k+1}\xxb_i^{(k-1)}b_{i+1}\bigg)
\end{equation}
The sum 
\[
\sum_{k\ge 0}\bigg((-1)^kq^{k}\xxb_{i}^{(k)}b_{i,i+1} +(-1)^kq^{k-1}\xxb_i^{(k-1)}b_{i,i+1}\bigg)
\]
telescopes and therefore is zero. 
This leaves $\sum_{k \ge 0}q^{-k+1}\xxb_i^{(k-1)}b_{i+1}$ 
on the right hand side of \eqref{eqn:longBbraidconjugate-sumtoshow}, 
which establishes \eqref{eqn:longBbraidconjugate-suffices}.
\end{proof}

\begin{cor}
For $1 \leq i \leq m-2$, we also have
\begin{equation}\label{eqn:braid-conj-both}
b_{i,i+1} = \iQW_{i+1}^{-1}b_{i}\iQW_{i+1} \, .
\end{equation}
\end{cor}

We content ourselves with a sketch, 
since this result does not play a technical role in our remaining arguments.

\begin{proof}[Proof (sketch)]
It suffices to establish the $m=3$, $i=1$ case. 
Observe that the assignments 
$b_1\mapsto b_2$, $b_2\mapsto b_{1,2}$ determine an algebra automorphism of $\iqUsom[3]$
(see Lemma \ref{lem:shiftauto} below, which records its inverse).
It is possible to check that this automorphism descends to the quotient $\iqUsom[3]^{\le n}$.
Further, it sends $b_{1,2} \mapsto b_1$ and $\iQW_{1} \mapsto \iQW_{2}$, 
and therefore acts on the $m=3$, $i=1$ case of \eqref{eq:braidconj} as follows:
\[
\Big(b_{1,2} = \iQW_{1} b_{2} \iQW_{1}^{-1} \Big) \mapsto
\Big(b_{1} = \iQW_{2} b_{1,2} \iQW_{2}^{-1} \Big) \, .
	\]
Conjugating the latter by $\iQW_{2}^{-1}$ gives
the desired relation.
	\end{proof}

\begin{rem}
Applying $\psi^{\le n}$ to the relations \eqref{eq:braidconj} and \eqref{eqn:braid-conj-both} in 
$\iqUsom^{\le n}$ results in the following web relations.
\begin{equation}\label{eq:TbTweb}
\begin{tikzpicture}[scale=.5, tinynodes, anchorbase, rotate=180]
	\draw[thick, black] (1,0) to (2,0);
	\draw[very thick,gray] (1,-1) to (1,1);
	\draw[very thick,gray] (3,-1) to [out=90,in=270] (2,0) to [out=90,in=270] (3,1);
	\draw[overcross] (2,-1) to [out=90,in=270] (3,0) to [out=90,in=270] (2,1);
	\draw[very thick,gray] (2,-1) to [out=90,in=270] (3,0) to [out=90,in=270] (2,1);
\end{tikzpicture} \
= \
\begin{tikzpicture}[scale=.5, anchorbase]
	\draw[thick, black] (1,0) to (3,0);
	\draw[very thick,gray] (1,-1) to (1,1);
	\draw[overcross] (2,-1) to (2,1);
	\draw[very thick,gray] (2,-1) to (2,1);
	\draw[very thick,gray] (3,-1) to (3,1);
\end{tikzpicture} \
= \
\begin{tikzpicture}[scale=.5, anchorbase]
	\draw[thick, black] (1,0) to (2,0);
	\draw[very thick,gray] (1,-1) to (1,1);
	\draw[very thick,gray] (3,-1) to [out=90,in=270] (2,0) to [out=90,in=270] (3,1);
	\draw[overcross] (2,-1) to [out=90,in=270] (3,0) to [out=90,in=270] (2,1);
	\draw[very thick,gray] (2,-1) to [out=90,in=270] (3,0) to [out=90,in=270] (2,1);
\end{tikzpicture} \, 
	\end{equation}
The web relations in \eqref{eq:TbTweb} also follow from the 
\emph{not yet established} web relation \eqref{eq:Sforkslide}.
However, even if we first established web relation \eqref{eq:Sforkslide}, 
then we still could not conclude \eqref{eq:braidconj} and \eqref{eqn:braid-conj-both} 
without Lemma \ref{lem:techcomm}, 
since we have  \emph{not yet shown} that $\psi^{\le n}$ is injective. 
Instead, we will use \eqref{eq:braidconj} and \eqref{eqn:braid-conj-both} 
in our proof that $\psi^{\le n}$ is injective 
(see Corollary \ref{cor:inj}). 
From this, we deduce Theorem \ref{thm:main} 
and finally conclude that the web relation \eqref{eq:Sforkslide} holds, 
in Corollary \ref{cor:Stangledweb}. 
\end{rem}

We now introduce the notation
$\hh_{i,i+1}:=b_i$ and $\hh_{i,i+2}:=b_{i,i+1}$, for $1\leq i < j \leq m$. 
More generally, for $1 \leq i < j \leq m$ we set
\begin{equation}\label{eq:longHdef}
\hh_{i,j}:= q^{-1}\hh_{i,j-1}\hh_{j-1,j}+ q\hh_{j-1,j}\hh_{i,j-1}
	\end{equation}
and observe that \eqref{eq:blackgraybraiding} implies that
\begin{equation}\label{eq:longHweb}
\psi(\hh_{i,j}) = 
\begin{tikzpicture}[scale=.5, tinynodes, anchorbase]
	\draw[thick, black] (1,0) to (4,0);
	\draw[very thick,gray] (0,-1) to (0,1);
	\node at (.5,.5) {$\mydots$};
	\draw[very thick,gray] (1,-1) to (1,1);
	\draw[overcross] (2,-1) to (2,1);
	\draw[very thick,gray] (2,-1) to (2,1);
	\node at (2.5,.5) {$\mydots$};
 	\draw[overcross] (3,-1) to (3,1);
	\draw[very thick,gray] (3,-1) to (3,1);
	\draw[very thick,gray] (4,-1) to (4,1);
	\node at (4.5,.5) {$\mydots$};
	\draw[very thick,gray] (5,-1) to (5,1);
\end{tikzpicture}
=
\rung_{i,j} \, .
	\end{equation}

In fact, the elements $\hh_{i,j}$ vastly predate 
our topological interpretation provided by \eqref{eq:longHweb}.
They are the building blocks for the PBW basis for $\iqUsom$ 
constructed in \cite[Theorem 2]{MR1903121}.
We now recall that result.

\begin{thm}[Iorgov--Klimyk]\label{thm:PBWbasis}
The set 
\[
\left\{\hh_{1,2}^{e_{1,2}}\hh_{1,3}^{e_{1,3}}\dots \hh_{1,m}^{e_{1,m}}\hh_{2,3}^{e_{2,3}}\hh_{2,4}^{e_{2,4}} 
	\dots \hh_{2,m}^{e_{2,m}}\dots \hh_{m-1,m}^{e_{m-1,m}}\right\}_{e_{i,j}\in \mathbb{Z}_{\ge 0}}
	\]
is a $\C(q)$-basis for $\iqUsom$. \qed
\end{thm}

We now show that this basis descends to a finite spanning 
set for the algebra $\iqUsom^{\leq n}$.
First, we need the
following generalization of \eqref{eq:braidconj} for the elements $\hh_{i,j}$.

\begin{prop}\label{prop:longb-braid-conjugates}
For $2 \leq i+1 < j \leq m$, 
we have $\hh_{i,j} = \iQW_{i}\hh_{i+1,j}\iQW_{i}^{-1}$
in $\iqUsom^{\leq n}$.
\end{prop}

\begin{proof}
By induction on $j-i\ge 2$. 
The base case $j=i+2$ is Proposition \ref{prop-braid-conjugates}. 
Suppose that $j-i > 2$, then we compute
\begin{align*}
\hh_{i,j} &\!\!\!\!\stackrel{\eqref{eq:longHdef}}{=} q^{-1}\hh_{i,j-1}\hh_{j-1,j}+ q\hh_{j-1,j}\hh_{i,j-1} \\
&=q^{-1}\iQW_{i}\hh_{i+1,j-1}\iQW_{i}^{-1}\hh_{j-1,j}+ q\hh_{j-1,j}\iQW_{i}\hh_{i+1,j-1}\iQW_{i}^{-1} \\
&=q^{-1}\iQW_{i}\hh_{i+1,j-1}\hh_{j-1,j}\iQW_{i}^{-1} + q\iQW_{i}\hh_{j-1,j}\hh_{i+1,j-1}\iQW_{i}^{-1}
\stackrel{\eqref{eq:longHdef}}{=} \iQW_{i}\hh_{i+1,j}\iQW_{i}^{-1} \, .
\end{align*}
Here, 
the second equality is by induction, 
the third equality uses that $(j-1) - i > 1$ 
(so $\iQW_{i}$ and and $\iQW_{i}^{-1}$ commute with $\hh_{j-1,j}$ by \eqref{eqn:farcommute}).
\end{proof}

Next, we prove that the elements $\hh_{i,j}$ 
have the same minimal polynomial as the $b_i$'s in $\iqUsom^{\le n}$. 

\begin{lem}\label{lem:higherlongbsarezero}
In the algebra $\iqUsom^{\le n}$, 
we have $p_{n+1}(\hh_{i,j})=0$ for all $1\le i< j\le m$. 
\end{lem}

\begin{proof}
Iterating Proposition \ref{prop:longb-braid-conjugates}, we find that
\[
\hh_{i,j} = \iQW_{i}\hh_{i+1,j}\iQW_{i}^{-1} = \cdots
= \iQW_{i}\iQW_{i+1} \cdots \iQW_{j-2} \hh_{j-1,j} \iQW_{j-2}^{-1} \cdots \iQW_{i+1}^{-1} \iQW_{i}^{-1}
	\]
and thus 
\begin{align*}
p_{n+1}(\hh_{i,j}) &= p_{n+1}\big((\iQW_{i} \iQW_{i+1} \cdots \iQW_{j-2}) \hh_{j-1,j} (\iQW_{i} \iQW_{i+1} \cdots \iQW_{j-2})^{-1}\big) \\
	&= (\iQW_{i} \iQW_{i+1} \cdots \iQW_{j-2}) p_{n+1}(\hh_{j-1,j}) (\iQW_{i} \iQW_{i+1} \cdots \iQW_{j-2})^{-1} \\
	&= (\iQW_{i} \iQW_{i+1} \cdots \iQW_{j-2}) p_{n+1}(b_{j-1}) (\iQW_{i} \iQW_{i+1} \cdots \iQW_{j-2})^{-1} = 0 \, . \qedhere
	\end{align*}
	\end{proof}

At last, we pair this result with Iorgov--Klimyk's basis to establish the main result of this section.

\begin{prop}\label{prop:finite-dim}
The $\C(q)$-algebra $\iqUsom^{\le n}$ is finite-dimensional. 
\end{prop}

\begin{proof}
The map $\iqUsom\rightarrow \iqUsom^{\le n}$ necessarily sends the basis from Theorem \ref{thm:PBWbasis} to a spanning set. 
It follows from Lemma \ref{lem:higherlongbsarezero} that, if $N> n$, 
then $\hh_{i,j}^N$ is in the span of $1, \hh_{i,j}, \hh_{i,j}^2, \dots, \hh_{i,j}^n$. 
Therefore, the finite set
\[
\left\{\hh_{1,2}^{e_{1,2}}\hh_{1,3}^{e_{1,3}}\dots \hh_{1,m}^{e_{1,m}}\hh_{2,3}^{e_{2,3}}\hh_{2,4}^{e_{2,4}}
	\dots \hh_{2,m}^{e_{2,m}}\dots \hh_{m-1,m}^{e_{m-1,m}}\right\}_{e_{i,j}\in [0, n]}
	\]
spans $\iqUsom^{\le n}$.
\end{proof}

\begin{rem}
Relations \eqref{eq:braidconj} and \eqref{eqn:braid-conj-both} imply that 
each $b_i$ is conjugate to $b_1$ in $\iqUsom^{\le n}$.
Given this, it is tempting to suspect that only $p_{n+1}(b_1)$ 
is required to generate the kernel of $\iqUsom \to \iqUsom^{\le n}$,
as happens for cyclotomic quotients of Hecke algebras.

However, it is easy to see that this cannot be the case. 
For any representation of $\iqUsom^{\le n}$, 
the spectrum of $b_i$ is a subset of $\{(-1)^k\tfrac{[2k+1]}{[2]}\}_{0 \leq k \leq n}$ 
and is the same for all $b_i$ (since they are conjugate).
But, in a general ``nonclassical'' representation of $\iqUsom$, 
the spectrum of the $b_i$ need only be a subset of $\{\pm \tfrac{[2k+1]}{[2]}\}_{k > 0}$ 
and need not agree for each $b_i$ (see the next section for the terminology and more details).

As a simple example, let $m=3$. 
By \eqref{eqn:farcommute} and \eqref{eqn:iSerre-alg-version}, 
the assignments $b_1 \mapsto b_1$, $b_2 \mapsto -b_2$ determine an algebra 
automorphism $\rho$ of $\iqUsom[3]$.
Now, given a representation $V$ of $\iqUsom[3]^{\le n}$, 
consider the representation $\tilde{V}$ of $\iqUsom[3]$ given by
\[
\iqUsom[3] \xrightarrow{\rho} \iqUsom[3] \to \iqUsom[3]^{\le n} \to \End_{\C(q)}(V) \, .
	\]
This gives a representation of $\iqUsom[3] / (p_{n+1}(b_1))$ 
which is not a representation of $\iqUsom[3]^{\le n}$.
	\end{rem}
	
For later use, we pause to record some results concerning automorphisms of $\iqUsom$ 
which act like conjugation by $\iQW_i^{\pm 1}$, 
but on $\iqUsom$ (i.e.~before taking the quotient to $\iqUsom^{\le n}$). 

\begin{defn}
Let $\overline{(-)} \colon \iqUsom\rightarrow \iqUsom$ be the $\C$-algebra automorphism that is defined by 
$\overline{q}=q^{-1}$ and $\overline{b_i} = b_i$.
\end{defn}

It is clear from the presentation in Definition \ref{def:iQ} that this is indeed an automorphism of $\iqUsom$.
Note that the standard generators $b_i= \hh_{i,i+1}=\overline{\hh_{i,i+1}}$ of $\iqUsom$ are fixed by $\overline{(-)}$, 
while the $\hh_{i,j}$ are not fixed when $j-i>1$, 
e.g.~$\overline{\hh_{i,i+2}} = \overline{q^{-1}b_ib_{i+1} + qb_{i+1}b_{i}}= qb_ib_{i+1} + q^{-1}b_{i+1}b_i\ne \hh_{i,i+2}$.

\begin{lem}\label{lem:internalbraid}
The assignments
\begin{align*}
\hh_{k,k+1}&\mapsto \hh_{k, k+1} \quad \text{if } k\ne i-1, i+1 \, , \\
\hh_{i-1, i}&\mapsto \overline{\hh_{i-1, i+1}} \, , \quad \text{and} \\
\hh_{i+1, i+2}&\mapsto \hh_{i, i+2}
\end{align*}
determine a $\C(q)$-algebra automorphism $\Ad_i \colon \iqUsom\rightarrow \iqUsom$. 
The inverse $\C(q)$-algebra automorphims $\Ad_i^{-1} \colon \iqUsom\rightarrow \iqUsom$
is given by
\begin{align*}
\hh_{k,k+1}&\mapsto \hh_{k, k+1} \quad \text{if } k\ne i-1, i+1 \, , \\
\hh_{i-1, i}&\mapsto \hh_{i-1, i+1} \, , \quad \text{and} \\
\hh_{i+1, i+2}&\mapsto \overline{\hh_{i, i+2}} \, .
\end{align*}
\end{lem}
\begin{proof}
See the proof of \cite[Theorem 2.1]{MolevRagoucyiBraid}.
\end{proof}

\begin{rem}
The automorphism $\Ad_i$ originates in \cite{MolevRagoucyiBraid} and \cite{ChekhoviBraid}.
It was later interpreted in \cite{KolbPellegriniiBraid} as the $\iota$quantum analogue 
of Lusztig's braid group action on Drinfeld--Jimbo quantum groups.
\end{rem}

\begin{lem}\label{lem:shiftauto}
The assignments $b_1 \mapsto \hh_{1,m}$ and 
$b_i \mapsto b_{i-1}$ for $2 \leq i \leq m-1$ determine an algebra automorphism of $\iqUsom$. 
Consequently, $\iqUsom$ is generated by $\hh_{1,m}, b_1,\ldots,b_{m-2}$.
	\end{lem}
	
\begin{proof}
One can inductively compute the action of $\Ad_i^{\pm 1}$ on the elements $\hh_{i, j}$ \cite[Corollary 2.2]{MolevRagoucyiBraid}, in particular
\begin{align*}
\Ad_{i}^{-1}(\hh_{k, i}) &= \hh_{k, i+1} \quad k< i, \\
\Ad_{i}^{-1}(\overline{\hh_{i+1, k}}) &= \overline{\hh_{i, k}} \quad i+1< k, \qquad \text{and}\\
\Ad_{i}^{-1}(\overline{\hh_{k,i+1}}) &= \overline{\hh_{k,i}} \quad k< i
\end{align*}
A consequence is that 
\begin{equation}\label{eq:rotateauto}
\Ad_{m-1}^{-1}\circ \dots \circ \Ad_2^{-1}\circ \Ad_1^{-1}
\end{equation}
acts on $b_1=\hh_{1,2}$ as
\[
\hh_{1,2}\xmapsto{\Ad_1^{-1}}\hh_{1,2}\xmapsto{\Ad_2^{-1}}\hh_{1,3}
	\xmapsto{\Ad_3^{-1}}\hh_{1,4}\xmapsto{\Ad_4^{-1}}\dots\xmapsto{\Ad_{m-1}^{-1}}\hh_{1,m}
\]
and on $b_i=\hh_{i,i+1}=\overline{\hh_{i,i+1}}$ as
\[
\overline{\hh_{i,i+1}}\xmapsto{\Ad_{i-2}^{-1}\circ\dots\circ \Ad_1^{-1}}\overline{\hh_{i,i+1}}
	\xmapsto{\Ad_{i-1}^{-1}}\overline{\hh_{i-1,i+1}}\xmapsto{\Ad_{i}^{-1}}\overline{\hh_{i-1,i}}=\hh_{i-1,i}
		\xmapsto{\Ad_{m-1}^{-1}\circ\dots\circ\Ad_{i+1}^{-1}}\hh_{i-1,i}
\]
Thus, \eqref{eq:rotateauto} is an automorphism of $\iqUsom$ sending $b_1\mapsto \hh_{1,m}$ and $b_i\mapsto b_{i-1}$ for $2\le i\le m-1$. 
	\end{proof}

\subsection{Representation theory}\label{s:reptheory}

In this section, we establish the final ingredient in our proof of Theorem \ref{thm:main}; 
namely, that
\begin{equation}\label{eq:phipsin}
\varphi \circ \psi^{\leq n} \colon \iqUsom^{\le n} \rightarrow \End_{U_q(\son)}(S^{\otimes m}) 
\end{equation}
is injective. 
Indeed, since $\psi^{\leq n}$ is surjective, 
this implies that 
\eqref{eq:phi}
is injective and then
Proposition \ref{prop:FF} shows that the functor $\varphi \colon \Web(\son) \to \Rep(U_q(\son))$ is faithful.

We establish this injectivity using the representation theory of $\iqUsom^{\le n}$.
We first argue that $\iqUsom^{\le n}$ is a finite-dimensional semisimple algebra, 
with simple modules corresponding to certain so-called ``nonclassical'' irreducible representations of $\iqUsom$. 
To show that $S^{\otimes m}$ is a faithful representation of $\iqUsom^{\le n}$, 
it then suffices to prove that each of these simple modules appears with non-zero multiplicity in $S^{\otimes m}$. 
We show this by matching the combinatorics controlling the irreducible representations of $\iqUsom^{\le n}$ 
with the irreducible summands of the $U_q(\son)$-representation $S^{\otimes m}$.

\begin{rem}\label{rem:Wenzldoesmore}
This matching appears in the work of Wenzl \cite[Theorem 5.3]{Wenzl-Spin}, 
but we re\"{e}stablish it here, in our conventions, for the sake of clarity. 
Wenzl actually proves more:
he relates the combinatorial matching to the correspondence induced by the double centralizer theorem for $\End_{\C(q)}(S^{\otimes m})$,
and also matches the branching graph for $\iqUsom^{\le n}$ with the tensor product graph for $S^{\otimes m}$, 
viewed as a representation of $U_q(\son)$. 
In Appendix \ref{Appendix:doublecentralizerbijection}, 
we re\"{e}stablish these refinements of the matching in our conventions, for the sake of completeness.  
\end{rem}

To begin, we recall some general facts\footnote{Most of the results we need are established 
by Iorgov--Klimyk in \cite{MR2143754};
when referencing that work, we will give
specific theorem/section/equation references to the arXiv version, as it is more readily available.} 
about finite-dimensional representations of $\iqUsom$.
The algebra $\iqUsom$ was originally studied from the following perspective: 
it is a $q$-analogue of the enveloping algebra $U(\som)$ which fits into a chain of algebras 
\begin{equation}\label{eqn:chain}
\cdots \subset  \iqUsom[m-1] \subset \iqUsom\subset \iqUsom[m+1] \subset \cdots \, .
\end{equation}
Here, $\iqUsom[m-1] \hookrightarrow \iqUsom$ is the standard inclusion 
of the subalgebra generated by $b_1,\ldots,b_{m-2}$.

This should be contrasted with the usual $q$-analogue of $\som$, 
namely, the Drinfeld--Jimbo quantum group $U_q(\mathfrak{so}_m)$,
which is defined via a $q$-analogue of the Serre presentation of $\mathfrak{so}_m$.
The Serre presentation is determined by the corresponding Dynkin diagram, 
which is type $B$ when $m$ is odd and type $D$ when $m$ is even.
Since neither of these diagrams naturally embed into one another, there is not 
an analogous chain of algebras for $U_q(\mathfrak{so}_m)$.

On the other hand, the Serre presentation is adapted to a triangular decomposition. 
Therefore, Drinfeld--Jimbo quantum groups admit triangular decompositions, 
which allows for their representation theory to be analyzed by highest weight theory. 
By contrast, there is no obvious triangular decomposition of $\iqUsom$, 
so the finite-dimensional representation theory thereof is studied 
by exploiting the chain of algebras \eqref{eqn:chain}.
The main tool is the following.

\begin{prop}[{\cite[Theorem 1]{MR2143754}}]\label{prop:mult-free}
The restriction of a finite-dimensional irreducible representation from $\iqUsom$ to $\iqUsom[m-1]$ 
is a direct sum of mutually non-isomorphic finite-dimensional irreducible representations. \qed
\end{prop}

This result is used in \cite{MR2143754} to classify the irreducible finite-dimensional $\iqUsom$-representations.
Since the restriction of finite-dimensional irreducible representations at each step of the chain \eqref{eqn:chain} is multiplicity free, 
these representations have Gelfand--Tsetlin bases, which we describe below in certain cases.
(See e.g.~\cite{Molev} for a survey on such bases in the setting of classical Lie algebras.) 
In these bases, the action of the generators $b_i$ is given by explicit formulae \cite[Section 3]{MR2143754}. 
The irreducible representations are then split into two classes, depending on the eigenvalues of the generator $b_1$.
We now discuss this dichotomy.

The formulae in \cite[equations (9), (12), (15), and (16)]{MR2143754}
for the action of $b_1$ on a finite-dimensional irreducible representation 
show that this element is diagonalizable, and they explicitly describe its spectrum.
In our conventions, their $q$ is replaced with $-q^2=(\mathsf{i}q)^2$, where $\mathsf{i}=\sqrt{-1}$, 
and their generators $I_{i+1,i}$ correspond to $\mathsf{i} b_i$.

\begin{defn}\label{defn:classical}
A finite-dimensional irreducible representation of $\iqUsom$ is called \emph{classical} 
if the spectrum of $b_1$ consists of quantum numbers of the form
$\pm [\ell]_{-q^2}$ for  $\ell \in \tfrac{1}{2}\mathbb{Z}$. 

Here, when $\ell = \frac{2k+1}{2} \in \tfrac{1}{2}\mathbb{Z} \smallsetminus \Z$, 
the \emph{quantum half-integers} (in the variable $-q^2$) are defined as
\begin{equation}\label{eq:qhi}
\left[\frac{2k+1}{2}\right]_{-q^2} := \frac{(-q^2)^{\frac{2k+1}{2}} - (-q^2)^{\frac{-2k-1}{2}}}{(-q^2)-(-q^2)^{-1}} 
	= \frac{[2k+1]_{\mathsf{i}q}}{[2]_{\mathsf{i}q}} \, .
\end{equation}
\end{defn}

These classical representations are thusly named because they deform the irreducible representations of $U(\som)$, 
the classical specialization, which is obtained from $\iqUsom$ by setting $q \mapsto \mathsf{i}$.
Observe that the quantum half-integers indeed admit such a specialization:
$\frac{[2k+1]_{\mathsf{i}q}}{[2]_{\mathsf{i}q}} \mapsto \frac{[2k+1]_{-1}}{[2]_{-1}}=  -\frac{2k+1}{2}$.

\begin{rem}\label{rem:moreclassical}
As noted earlier, one can obtain representations of $\iqUsom$ by restriction from $U_{-q^2}(\slm)$. 
Representations obtained in this way have $b_1$-eigenvalues of the form $\pm [\ell]_{-q^2}$ for $\ell \in \mathbb{Z}$ 
(thus are classical), and they deform representations of $U(\som)$ whose weights lie in the root lattice. 
This is the expected behavior: 
the representations of $U(\som)$ with weights in the root lattice correspond to representations of $\mathrm{SO}_m$, 
which is a subgroup of $\mathrm{SL}_m$. 
On the other hand, representations of $U(\som)$ with weights not in the root lattice only come from 
the double cover $\mathrm{Spin}_m$ of $\mathrm{SO}_m$, which is not a subgroup of $\mathrm{SL}_m$. 
Classical representations with $b_1$-eigenvalues of the form $\pm [\ell]_{-q^2}$ for $\ell \in \frac{1}{2} + \mathbb{Z}$ 
deform these representations of $U(\son)$ whose weights do not lie in the root lattice. 
\end{rem}

On the other hand, there exist representations of $\iqUsom$ which do not deform representations of $U(\som)$. 

\begin{defn}\label{defn:nonclassical}
A finite-dimensional irreducible is called \emph{nonclassical} if 
the spectrum of $b_1$ consists of ratios of quantum integers of the form
\begin{equation}\label{eqn:no-class-spec}
\pm \frac{[2k+1]_q}{[2]_q} \, , \quad k \in \mathbb{Z} \, .
\end{equation}
\end{defn}

These \emph{half quantum integers} are distinct from both the quantum integers $\pm [\ell]_{-q^2}, \ell \in \Z$
and the quantum half-integers \eqref{eq:qhi}. 
Since $[2]_{q=\mathsf{i}} = 0$, these $b_1$-eigenvalues cannot be specialized at $q= \mathsf{i}$,
whence the term nonclassical.

In order to relate to the quotient $\iqUsom^{\le n}$ from Definition \ref{def:iQn}, 
we require information about the remaining generators $\{b_i\}_{i=2}^{m-1}$ of $\iqUsom$. 
The following is obtained in work of Wenzl.

\begin{lem}[{\cite[Theorem 3.11, part (d)]{WenzlUprime}}]\label{lem:Wenzlso3}
In every finite-dimensional irreducible representation of $\iqUsom[3]$, 
$b_1$ is conjugate to either $b_2$ or $-b_2$. \qed
	\end{lem}
	
\begin{cor}\label{cor:biss}
Each generator $b_i$ is diagonalizable in every finite-dimensional irreducible representation of $\iqUsom$.
Further, the spectrum of $b_i$ is a subset of $\{\pm [\ell]_{-q^2} \mid \ell \in \tfrac{1}{2}\mathbb{Z} \}$ if the representation is classical
and is a subset of $\{\pm \frac{[2k+1]_q}{[2]_q} \mid k \in \mathbb{Z} \}$ if the representation is nonclassical.
	\end{cor}
\begin{proof}
That $b_1$ is diagonalizable in any finite-dimensional irreducible follows 
from repeated application of Proposition \ref{prop:mult-free}. 
The claim then follows from Lemma \ref{lem:Wenzlso3} by restricting, in turn, 
to the subalgebras of $\iqUsom$ generated by $b_i$ and $b_{i+1}$ for $i=1,\ldots,m-2$.
\end{proof}

We now look more closely at the signs that appear in the eigenvalues of the generators $b_i$ 
acting on a finite-dimensional irreducible representation, 
since they play a special role in our considerations.
Observe that, for each $i\in\{1, \dots, m-1\}$, there is an algebra automorphism of $\iqUsom$ which sends $b_i \mapsto -b_i$ 
and $b_j \mapsto b_j$ for $j\ne i$. 
More generally, for any sequence $\mathbf{\epsilon} = (\epsilon_1, \dots, \epsilon_{m-1})$ with $\epsilon_i \in \{\pm 1\}$, 
there is an algebra automorphism such that $b_i\mapsto \epsilon_ib_i$.
Thus, given a finite-dimensional irreducible representation of $\iqUsom$, 
one obtains another irreducible representation by precomposing the action map with such an automorphism.
We call this a \emph{sign twist} of the original representation.

\begin{rem}\label{rem:sgntwistclassical}
In the classical case, an irreducible representation will have an eigenvector for $b_i$ 
with eigenvalue $[\ell]_{-q^2}$ if and only if there is an eigenvector with eigenvalue $-[\ell]_{-q^2}$. 
This parallels how weights in $\som$-representations are evenly distributed about the origin 
(a consequence of invariance under the Weyl group). 
Correspondingly, any sign twist of a classical representation of $\iqUsom$ 
is isomorphic to the original representation.
\end{rem}

\begin{rem}\label{rem:moremoreclassical}
Continuing Remark \ref{rem:moreclassical}, 
specializing $q$ to $\mathsf{i}$ in the classical irreducible representations, we recover the irreducible representations of $\som$. 
The Lie subalgebra generated by $b_1, b_3, b_5, \ldots$ then becomes a Cartan subalgebra of $\som$, 
which acts with weights such that $b_i\mapsto \pm [\ell]_{-q^2= 1}= \pm \ell$ 
for some $\ell \in \frac{1}{2}\mathbb{Z}$. 
For example, consider $\som[3]\cong\mathfrak{sl}_2$. 
The generator $b_1$ spans a Cartan subalgebra which is conjugate (over $\C$) to the usual Cartan subalgebra. 
Specifically, $b_1\in \mathfrak{so}_3$ corresponds to $\frac{1}{2}h\in \mathfrak{sl}_2$.
\end{rem}

The behavior in Remark \ref{rem:sgntwistclassical} is directly contrasted in the nonclassical representations.

\begin{lem}\label{lem:sigmaplusorminus}
The spectrum of $b_1$ acting on a 
nonclassical irreducible representation is either a subset of 
\begin{equation}\label{eq:type+}
\sigma_{+} :=
\left\{\frac{[1]}{[2]}, -\frac{[3]}{[2]}, \frac{[5]}{[2]}, \dots, (-1)^{j}\frac{[2j+1]}{[2]}, \dots\right\}
\end{equation}
or a subset of 
\begin{equation}\label{eq:type-}
\sigma_{-} :=
\left\{-\frac{[1]}{[2]}, \frac{[3]}{[2]}, -\frac{[5]}{[2]}, \dots, (-1)^{j+1}\frac{[2j+1]}{[2]}, \dots\right\} \, .
\end{equation}
\end{lem}
\begin{proof}
By direct construction, this is shown in \cite{MR2143754}. 
\end{proof}

Implicit in \cite{MR2143754} is the following result; 
we provide full details since it is essential for our remaining arguments.

\begin{lem}\label{lem:typeIexist}
Let $M$ be a nonclassical irreducible representation of $\iqUsom$. 
There exists a sign twist of $M$ so that the spectrum of every generator $b_i$ is a subset of $\sigma_{+}$.
	\end{lem}
\begin{proof}
First, since $M$ is nonclassical and irreducible, 
Corollary \ref{cor:biss} says $b_i$ is diagonalizable on $M$ with spectrum in $\sigma_{+}\cup \sigma_{-}$. 
Moreover, Lemma \ref{lem:sigmaplusorminus} says that $b_1$ has spectrum entirely contained
in either $\sigma_{+}$ or $\sigma_{-}$.

We now proceed via induction on $m$.
The base case ($m=3$) follows from Lemma \ref{lem:Wenzlso3} since we may first sign twist 
to ensure that $b_1$ has spectrum contained in $\sigma_{+}$ and then,
if necessary, sign twist again to guarantee that $b_1$ is conjugate to $b_2$.

Now, suppose that $m \geq 4$ and that the result holds for $\iqUsom[m-1]$. 
Given a nonclassical irreducible representation $M$ of $\iqUsom$, 
decompose the restriction $M|_{\iqUsom[m-1]}$ into irreducible summands.
By the induction hypothesis, the spectrum of each $b_i$ on each of these summands is 
contained in either $\sigma_{+}$ or $\sigma_{-}$.
Denote the direct sum of all summands on which the spectrum of $b_i$ is contained in 
$\sigma_{\pm}$ by $M_{i}^{\pm}$. Note that each of $M_{i}^{\pm}$ is the direct sum 
of all $b_{i}$-eigenspaces for eigenvalues in $\sigma_{\pm}$.

Suppose that $1 \leq i \leq m-3$, so $b_i$ commutes with $b_{m-1}$. 
It follows that $b_{m-1}$ preserves each $b_i$-eigenspace. 
Since $\sigma_{+} \cap \sigma_{-} = 0$, this implies that $b_{m-1}(M_{i}^{+}) \subset M_{i}^{+}$ 
and $b_{m-1}(M_{i}^{-}) \subset M_{i}^{-}$. 
Since $\iqUsom$ is generated by $\iqUsom[m-1]$ and $b_{m-1}$, 
it follows that $M_{i}^{+}$ and $M_{i}^{-}$ are subrepresentations, hence one of them must be zero.
Consequently, we may sign twist $M$ so that the spectrum of each $b_i$ with $1 \leq i \leq m-3$ 
is a subset of $\sigma_{+}$. 

Next, consider $b_{m-2}$. 
Again, we may decompose $M = M_{m-2}^{+} \oplus M_{m-2}^{-}$.
By Lemma \ref{lem:shiftauto}, $\iqUsom$ is generated by $\iqUsom[m-1]$ and $\hh_{1,m}$. 
Moreover, a routine check shows that $\hh_{1,m}$ commutes with $b_{m-2}$.  
Thus, we may repeat the above argument to see that 
one of $M_{m-2}^{\pm}$ must be zero. Again, we may therefore sign twist $M$ so that 
the spectrum of $b_{m-2}$ is a subset of $\sigma_{+}$.

Finally, we consider $b_{m-1}$. 
Repeating the above arguments for the restriction of $M$ to the copy of $\iqUsom[m-1]$ 
generated by $b_2,\ldots,b_{m-1}$ and using the fact that $b_{m-1}$ commutes with $b_1$ (since $m\ge 4$), 
we again see that the spectrum of $b_{m-1}$ is contained in one of $\sigma_{\pm}$. 
If necessary, a final sign twist ensures it is contained in $\sigma_{+}$.
	\end{proof}

We now give a name to the nonclassical representations from Lemma \ref{lem:typeIexist}.

\begin{defn}\label{def:nonclassical-typeI}
A nonclassical representation of $\iqUsom$ is of \emph{type I} 
provided the eigenvalues of all the $b_i$ are contained in the set $\sigma_+$ 
from \eqref{eq:type+}.
\end{defn}

By Remark \ref{rem:bi-spectrum}, 
all irreducible representations of the algebras $\iqUsom^{\le n}$
are type I nonclassical representations of $\iqUsom$.
With this terminology in place, we now record the main result of Iorgov--Klimyk.

\begin{prop}[{\cite[Theorem 4 and its Corollary]{MR2143754}}]\label{prop-classification-thm}
Each finite-dimensional irreducible 
representation of $\iqUsom$ is either classical or nonclassical.
Moreover, 
\begin{itemize}
\item The isomorphism classes of classical (finite-dimensional) irreducible representations are in bijection with dominant integral weights for $\som$. 
\item The isomorphism classes of type I nonclassical (finite-dimensional) irreducible representations are in bijection 
with weakly decreasing $k$-tuples $\ab = (a_1 \ge a_2 \ge \cdots \ge a_k)$ of positive half-integers 
(recall Convention \ref{convention:halfinteger}). Here, $k$ is determined by $m = 2k$ or $m=2k+1$, depending on the parity of $m$.
\item Every nonclassical (finite-dimensional) irreducible representation is isomorphic to a sign twist of a type I nonclassical irreducible representation
(and no two distinct such sign twists are isomorphic).
\end{itemize}
Finally, all finite-dimensional representations of $\iqUsom$ are completely reducible.
\end{prop}

\begin{cor}\label{cor:complete-reducibility}
The algebra $\iqUsom^{\le n}$ is a finite-dimensional semisimple algebra.
\end{cor}
\begin{proof}
By Proposition \ref{prop:finite-dim}, $\iqUsom^{\le n}$ is finite-dimensional 
and then, by Proposition \ref{prop-classification-thm}, the left regular module is completely reducible.
	\end{proof}

Henceforth, we exclusively focus on the type I nonclassical 
(finite-dimensional) irreducible representations for the algebra $\iqUsom$. 
To begin, we establish notation and terminology for the data which parametrize these representations.

\begin{defn} 
Let $m \ge 0$. 
An \emph{$m$-spin-partition} is a weakly decreasing $k$-tuple $\ab = (a_1 \ge a_2 \ge \ldots \ge a_k)$ of positive half-integers, 
where $m = 2k$ or $m=2k+1$, depending on parity. 
Its \emph{leading term} is $\lead(\ab) = a_1$, which is defined so long as $m \ge 2$. 
We write $\SP_{m}$ for the set of all $m$-spin-partitions and $\SP_m^{\le n}$ for the set of all $m$-spin-partitions 
with $\lead(\ab)\le \frac{2n+1}{2}$.
\end{defn}

For example, $\ab = (\frac{7}{2}, \frac{5}{2}, \frac{5}{2}, \frac{1}{2})$ is both an $8$-spin-partition 
and a $9$-spin-partition, and $\lead(\ab) = 7/2$. 
The empty tuple is the unique $0$-spin-partition and the unique $1$-spin-partition.

\begin{rem} 
Although an $m$-spin-partition determines a dominant weight in type $B_k$ (where $m=2k$ or $2k+1$), 
we encourage the reader to not think about $m$-spin-partitions this way. 
Below, we will use a different relationship between $m$-spin-partitions 
and dominant weights of type $B_n$ for some $n$ unrelated to $m$.
\end{rem} 

In this language, Proposition \ref{prop-classification-thm} asserts that the 
type I nonclassical representations of $\iqUsom$ are indexed by $m$-spin-partitions. 
We now aim to describe bases of these representations. 
The following combinatorial data will help in indexing these bases.

\begin{defn}
Let $\ab = (a_1 \ge a_2 \ge \cdots \ge a_k)$ be a $(2k+1)$-spin-partition (for $2k+1 \ge 1$) 
and let $\ab' = (a_1' \ge a_2' \ge \cdots \ge a_k')$ be a $2k$-spin-partition (both are $k$-tuples). 
We say that $\ab$ and $\ab'$ \emph{interlace} if
\[ 
a_1 \ge a_1' \ge a_2 \ge a_2' \ge \cdots \ge a_k \ge a_k' \, .
\]
Meanwhile, let $\ab = (a_1 \ge a_2 \ge \cdots \ge a_k)$ be a $2k$-spin-partition (for $2k \ge 2$) 
and let $\ab' = (a_1' \ge a_2' \ge \cdots \ge a_{k-1}')$ be a $(2k-1)$-spin-partition 
(now, the former is a $k$-tuple and latter a $(k-1)$-tuple). 
We say that $\ab$ and $\ab'$ \emph{interlace} if
\[
a_1 \ge a_1' \ge a_2 \ge a_2' \ge \cdots \ge a_{k-1} \ge a_{k-1}' \ge a_k \, .
\]
The unique $1$-spin-partition and $0$-spin-partition interlace, 
and any $2$-spin-partition interlaces with the unique $1$-spin-partition.
\end{defn}

\begin{rem}\label{rem:interlacesaturated}
If $\abb\in \SP_{m-1}$ interlaces $\ab\in \SP_{m}^{\le n}$, then $\abb\in \SP_{m-1}^{\le n}$. 
\end{rem}

\begin{defn} 
Let $m \ge 0$. 
An \emph{$m$-spin-tableau} is an $m$-tuple 
$(\ab_m,\ab_{m-1}, \ldots, \ab_1, \ab_0)$
where each $\ab_i\in \SP_i$ and where $\ab_i$ and $\ab_{i-1}$ interlace for all $1 \le i \le m$.
\end{defn}

\begin{example}\label{ex:spintable}
The following describes a $9$-spin-tableau:
\[
\begin{tikzcd}[column sep = 0.1 em, row sep =0.1 em]
\ab_9: & 9/2 & & 5/2 & & 5/2 & & 1/2 &  \\
\ab_8: &  & 9/2 & & 5/2 & & 5/2 & & 1/2 \\
\ab_7: &  & & 7/2 & & 5/2 & & 3/2 &  \\
\ab_6: &  & & & 5/2 & & 5/2 & & 1/2 \\
\ab_5: &  & & & & 5/2 & & 3/2 &  \\
\ab_4: &  & & & & & 5/2 & & 3/2 \\
\ab_3: &  & & & & & & 3/2 &  \\
\ab_2: &  & & & & & & & 3/2 \\
	\end{tikzcd}
	\]
We do not display $\ab_0$ and $\ab_1$, which are both the empty tuple.
	\end{example}

\begin{rem} 
The reader familiar with the Okounkov--Vershik \cite{MR1443185} approach to the
representation theory of symmetric groups 
should think of $m$-spin-partitions as being analogous to partitions of $m$, 
and $m$-spin-tableaux as being analogous to standard Young tableaux, 
i.e.~interlacing sequences of partitions.

The reader familiar with Gelfand--Tsetlin bases for irreducible representations of $\glm$ will observe the 
clear parallel between $m$-spin-tableau and (type $A$) Gelfand--Tsetlin patterns \cite{GelfandCetlin}.
\end{rem}

We next state a refinement of the type I nonclassical case of Proposition \ref{prop-classification-thm}.

\begin{thm}\label{thm:reptheorytheorem}
For each $m$-spin-partition $\ab\in \SP_m$, there is a 
type I nonclassical irreducible representation $M_{\ab}$ of $\iqUsom$. 
Upon restriction to $\iqUsom[m-1]$, we have
\begin{equation}\label{eq:resM}
M_{\ab}\cong \bigoplus_{\substack{ \abb\in \SP_{m-1} \\\abb \ \text{interlaces} \ \ab}} M_{\abb} \, , 
\end{equation}
and the representation $M_{\ab}$ has a compatible Gelfand--Tsetlin basis 
which is indexed by all $m$-spin-tableaux 
$(\ab_m,\ab_{m-1}, \ldots, \ab_1, \ab_0)$
with $\ab_m = \ab$.
Every type I nonclassical irreducible representation of $\iqUsom$ is isomorphic to $M_{\ab}$ 
for some $\ab\in \SP_m$, 
and $M_{\ab} \cong M_{\ab'}$ implies $\ab = \ab'$.
\end{thm}

\begin{proof}
In \cite[Section 3]{MR2143754}, Iorgov--Klimyk construct a
nonclassical irreducible representation of $\iqUsom$ associated to each $\ab\in \SP_m$, 
and describe their Gelfand--Tsetlin basis.
Lemma \ref{lem:typeIexist} then shows that some (possibly trivial) 
sign twist of their representation is of type I. We denote this one by $M_{\ab}$.
The Gelfand--Tsetlin basis of $M_{\ab}$ is compatible with restriction maps, 
as $m$ varies in \eqref{eqn:chain},
determining explicitly the decomposition in \eqref{eq:resM}; 
see \cite[Remark in Section 3]{MR2143754}.
The final statement follows from \cite{MR1798769} and \cite[Theorem 4]{MR2143754}. 
\end{proof}

Using that the Gelfand--Tsetlin basis is compatible with restriction maps between representations of $\iqUsom$ 
as $m$ varies in \eqref{eqn:chain}, we can explicitly understand the restriction to 
${U}^{\iota}_{-q^2}(\som[2])$, which is generated by the element $b_1$. 

\begin{lem}\label{lem:b1restriction}
Let $v$ be the Gelfand--Tsetlin basis vector in $M_{\ab}$ 
indexed by the $m$-spin-tableau 
$(\ab_m,\ab_{m-1}, \ldots, \ab_1, \ab_0)$.
If $\lead(\ab_2)=\frac{2j+1}{2}$ for $j \in \Z_{\ge 0}$, then
\[
b_1 v = (-1)^j \frac{[2j+1]}{[2]} v \, .
\]
\end{lem}
\begin{proof}
By construction, the Gelfand--Tsetlin basis vector $v$ corresponding to 
$(\ab=\ab_m,\ab_{m-1}, \ldots, \ab_0)$
spans a $\iqUsom[2]$-submodule of $M_{\ab}$ 
which is isomorphic to $M_{\ab_2}$. 
If follows from \cite[Equations (14) and (15)]{MR2143754} 
that if $\lead(\ab_2) = \frac{2j+1}{2}$, then $b_1$ acts by $(-1)^j \frac{[2j+1]}{[2]}$.
\end{proof}

\begin{cor}\label{cor:b1-spectrum}
Let $\ab\in \SP_m$ be an $m$-spin-partition, and let $k$ be such that $\lead(\ab) = \frac{2k+1}{2}$. 
The set of eigenvalues (spectrum) of $b_1 \in \iqUsom$ acting on $M_{\ab}$ is equal to the set $\{ (-1)^j \frac{[2j+1]}{[2]} \}_{j=0}^{k}$. 
\end{cor}
	
\begin{proof}
In any $m$-spin-tableau 
$(\ab=\ab_m,\ab_{m-1}, \ldots, \ab_0)$,
the value of $\lead(\ab_2)$ is some positive half-integer between $\frac{1}{2}$ and $\lead(\ab)$. 
Conversely, given any such half-integer $\frac{2j+1}{2}$, 
it is easy to construct an $m$-spin-tableau such that $\lead(\ab_2) = \frac{2j+1}{2}$ and $\ab_m =\ab$.
For $3 \leq i \leq m$ we simply let $\ab_i$ consist of the first $\lfloor i/2 \rfloor$ entries of $\ab$.
\end{proof}

\begin{example}
As an instance of the final construction in the preceding proof, 
the $9$-spin partition $\ab=\ab_9=(9/2, 5/2, 5/2, 1/2)$, 
is such that $\frac{1}{2}< \frac{7}{2} < \frac{9}{2} = \lead(\ab)$. 
In the notation of Example \ref{ex:spintable},
the construction gives the following $9$-spin-tableau 
\[
\begin{tikzcd}[column sep = 0.1 em, row sep =0.1 em]
\ab_9: & 9/2 & & 5/2 & & 5/2 & & 1/2 &  \\
\ab_8: &  & 9/2 & & 5/2 & & 5/2 & & 1/2 \\
\ab_7: &  & & 9/2 & & 5/2 & & 5/2 &  \\
\ab_6: &  & & & 9/2 & & 5/2 & & 5/2 \\
\ab_5: &  & & & & 9/2 & & 5/2 &  \\
\ab_4: &  & & & & & 9/2 & & 5/2 \\
\ab_3: &  & & & & & & 9/2 &  \\
\ab_2: &  & & & & & & & 7/2 \\
	\end{tikzcd}
	\]
with $\lead(\ab_2)=\frac{7}{2}$.
\end{example}

\begin{rem} 
The action of the elements $b_i$ for $i > 1$ on the vectors $v$ from Lemma \ref{lem:b1restriction} 
is more complicated; again, see \cite[equations (14) and (15)]{MR2143754}. 
Indeed, $v$ is a simultaneous eigenvector for 
some commutative Gelfand--Tsetlin subalgebra of $\iqUsom$, 
but not all elements $b_i$ are in that commutative subalgebra. 
E.g.~$b_2$ does not commute with $b_1$, 
so they are not simultaneously diagonalizable.
\end{rem}

\begin{prop}\label{prop:repsofiQn}
For $m \ge 2$, the representations $M_{\ab}$ with $\ab \in \SP_m^{\le n}$ 
give a complete, irredundant collection of irreducible representations of $\iqUsom^{\le n}$. 
\end{prop}
\begin{proof} 
By Lemma \ref{lem:Wenzlso3}, if $M$ is a type I nonclassical (finite-dimensional irreducible) representation, 
then the action of $b_1$ on $M$ is conjugate to the action of $b_i$ on $M$ for $2\leq i \leq m-1$. 
Corollary \ref{cor:b1-spectrum} then implies that the representations $M_{\ab}$ with $\ab\in \SP_m^{\le n}$ 
are precisely the irreducible representations of $\iqUsom$ for which the spectrum of all $b_i$ are 
contained in $\{ (-1)^j \frac{[2j+1]}{[2]} \}_{j=0}^{n}$.
By Definition \ref{def:iQn}, these are exactly the irreducible representations of $\iqUsom^{\leq n}$.
\end{proof}

Having gained control over the irreducible representations of $\iqUsom^{\le n}$, 
we now argue that these irreducible representations are in natural bijection 
with the isomorphism classes of irreducible summands of the $U_q(\son)$-representation $S^{\ot m}$.
Once this is established, an appeal to the following basic lemma gives the desired injectivity of \eqref{eq:phipsin}.

\begin{lem}\label{lem:fullforssisfaithful}
If $A$ is a finite-dimensional semisimple $\K$-algebra 
and $V$ is a representation of $A$ such that each irreducible representation of $A$ appears in $V$, 
then $A\rightarrow \End_{\K}(V)$ is injective.
\end{lem}
\begin{proof}
Suppose $S_1, \dots, S_r\subset V$ are all the irreducible representations of $A$ (up to isomorphism). 
By complete reducibility, $V$ contains $S_1\oplus \dots \oplus S_r$ as a summand. 
Thus, ${}_AA \cong S_1^{\oplus \dim S_1}\oplus \dots \oplus S_r^{\oplus \dim S_r}$ is isomorphic to a summand of $V^{\oplus N}$, 
for large enough $N$. 
Since ${}_AA$ is faithful, it follows that $V^{\oplus N}$ is faithful, 
and therefore $V$ is faithful (if $a \in A$ acted by $0$ on $V$, then $a$ would act by $0$ on $V^{\oplus N}$). 
\end{proof}

To motivate the desired bijection, recall the \emph{double centralizer theorem}, 
which says that if $N$ is a finite-dimensional $\K$-vector space and $A \subset \End_{\K}(N)$ is a semisimple subalgebra, 
then $\End_A(N) \subset \End_{\K}(N)$ is a semisimple subalgebra 
with isomorphism classes of irreducible representations in bijection with irreducible representations of $A$. 
Explicitly, for an irreducible $A$-module $V$, the bijection is realized as $V\mapsto \Hom_A(V, N)$, 
where $\End_A(N)$ acts on the latter via postcomposition.

\begin{prop}\label{prop:u'-doublecantralizer}
The set
\begin{equation}\label{eq:spinsummands}
\left\{\Hom_{U_q(\son)}(V_{\lambda}, S^{\otimes m}) \mid 
\text{$\lambda \in X_+(\son)$ and $V_{\lambda}$ is a $U_q(\son)$ summand of $S^{\otimes m}$}\right\}
\end{equation}
is a complete and irredundant set of irreducible $\iqUsom^{\le n}$-representations appearing 
as summands of $S^{\otimes m}$.
\end{prop}
\begin{proof}
Proposition \ref{prop:alg-hom-surj} gives that the 
algebra homomorphism $\iqUsom^{\le n}\rightarrow \End_{U_q(\son)}(S^{\otimes m})$ is surjective,
thus, the claim follows from the double centralizer theorem.  
\end{proof}

By Proposition \ref{prop:repsofiQn}, each representation appearing in \eqref{eq:spinsummands}
is isomorphic to a type I classical representation $M_{\ab}$ for some $\ab \in \SP_m^{\le n}$.

\begin{defn}\label{defn:2xcentralizer-map-on-weights}
Set
\[
X_+(\son)^{S^{\otimes m}} := \{\lambda\in X_+(\son) \mid \Hom_{U_q(\son)}(V_{\lambda}, S^{\otimes m})\ne 0\} \, .
\] 
For $\lambda \in X_+(\son)^{S^{\otimes m}}$, 
denote by $\ab_{m,n}(\lambda) \in \SP_m^{\le n}$ 
the $m$-spin-partition such that 
\[
M_{\ab_{m,n}(\lambda)} \cong \Hom_{U_q(\son)}(V_{\lambda}, S^{\otimes m})\, .
\]
\end{defn}

By the above recollection on the double centralizer theorem, the assignment 
\begin{equation}\label{eq:ourbijection}
X_+(\son)^{S^{\otimes m}} \xrightarrow{\ab_{m,n}(-)} \SP_m^{\le n} 
\end{equation}
is injective. In fact, it is a bijection: 
this follows from \cite[Proposition 5.5]{Wenzl-Spin}, but we now work to re\"{e}stablish it here. 
By the injectivity already noted, 
it suffices to construct some bijection between $X_+(\son)^{S^{\otimes m}}$ and $\SP_m^{\le n}$.
Before giving the general construction, we consider some easy examples 
(which establish the base case for an inductive argument).

\begin{example}\label{ex:m2discussionformatching}
When $m \le 1$, both $S^{\ot 0}$ and $S^{\ot 1}$ are irreducible, 
so there is exactly one $U_q(\son)$-irreducible summand. 
Meanwhile, there is exactly one $m$-spin-partition for $m=0, 1$: 
the empty tuple.

When $m=2$, Proposition \ref{prop:decomp} shows that the
irreducible summands of $S^{\ot 2}$ are $V_i$ for $0 \le i \le n-1$ and $V_{2 \varpi_n}$, 
each with multiplicity one.
There are $n+1$ summands, and their highest weights are
\[
X_+(\son)^{S^{\otimes 2}}=\{(0,\ldots, 0), (1,0,\ldots, 0), (1,1,0,\ldots, 0), \ldots, (1,1,\ldots, 1)\}
	= \{0,\varpi_1,\ldots,\varpi_{n-1},2\varpi_n\} \, .
\]
Meanwhile, we have
\[
\SP_2^{\le n} = \left\{\frac{1}{2}, \frac{3}{2}, \dots, \frac{2n+1}{2}\right\} 
= \left\{ \frac{2j-1}{2} \right\}_{j=1}^{n+1} \, .
\] 
Recall from (the $m=2$ case of) Lemma \ref{lem:b1restriction} that $b_1$ acts on 
$M_{\frac{2k+1}{2}}$ by $(-1)^k \frac{[2k+1]}{[2]}$.
Comparing with the relation \eqref{eq:blackrungtriangle}, 
we see that the map $\lambda\mapsto \ab_{2,n}(\lambda)$ 
from Definition \ref{defn:2xcentralizer-map-on-weights} is such that 
$\ab_{2,n}(0) = \frac{2n+1}{2}$, $\ab_{2,n}(\varpi_1) = \frac{2n-1}{2}$, $\ab_{2,n}(\varpi_2) = \frac{2n-3}{2}$, etc., 
and $\ab_{2,n}(2 \varpi_n) = \frac{1}{2}$.
Hence, $\ab_{2,n}(-)$ is a bijection.
\end{example}

This matching of $m$-spin-partitions and highest weights in $X_+(\son)^{S^{\otimes m}}$ is generalized as follows.

\begin{defn}\label{def:xi}
Fix $n \ge 1$ and $m \ge 0$. 
Let $\ab \in\SP_m^{\le n}$
(so $\ab= (a_1 \ge a_2 \ge \cdots \ge a_k)$ and $\lead(\ab) = a_1 \le \frac{2n+1}{2}$). 
For each $1 \le j \le n$, let 
\begin{equation}\label{eq:defofcs}
c_j = \# \left\{i \ \bigg| \ a_i \le \frac{2(n-j)+1}{2} \right\} \, .
\end{equation}
Set
\[
\Xi_{m,n}(\ab) = (c_1, c_2, \ldots, c_n) :=
\begin{cases}
(c_1, c_2, \ldots, c_n) & m \text{ is even} \\
(c_1 + \frac{1}{2}, c_2 + \frac{1}{2}, \ldots, c_n + \frac{1}{2}) & m \text{ is odd}
	\end{cases}
\]
\end{defn}

\begin{lem}\label{lem:xi-in-dominant}
Definition \ref{def:xi} gives a map 
\[
\Xi_{m,n} \colon \SP_m^{\le n}\rightarrow X_+(\son) \, ,
\]
if we identify $X_+(\son)$ with nonnegative decreasing sequences of either all integers or all half-integers.
\end{lem}
\begin{proof}
It follows from \eqref{eq:defofcs} that $c_1 \ge c_2 \ge \ldots \ge c_n \ge 0$ and the $c_i$ are all integers. 
Thus, $\Xi_{m,n}(\ab)$ is a decreasing sequence consisting of either all integers or all half-integers.
\end{proof}

We sometimes write $\Xi$ when the numbers $m$ and $n$ are understood, 
but it is important to note that the definition of $c_j$ does depend on $n$, 
and therefore the map $\Xi_{m,n}$ also depends on $n$.

\begin{example}\label{ex:basecase}
When $m=0$, 
$\Xi(\emptyset) = (0,0,\ldots, 0)$, the highest weight of $S^{\otimes 0}$. 
When $m=1$, $\Xi(\emptyset) = (\frac{1}{2}, \frac{1}{2}, \ldots, \frac{1}{2})
= \varpi_n$, the highest weight of $S^{\otimes 1}$. 

When $m=2$, $\Xi(\frac{2n+1}{2}) = (0,0,\ldots,0)$, $\Xi(\frac{2n-1}{2}) = (1,0,0\ldots,0)= \varpi_1$, 
$\Xi(\frac{2n-3}{2}) = (1,1,0\ldots,0)= \varpi_2$, and $\Xi(\frac{1}{2}) = (1,1,1\ldots,1)= 2 \varpi_n$, 
just as in Example \ref{ex:m2discussionformatching}. 
\end{example}

As noted in Lemma \ref{lem:xi-in-dominant}, the image of $\Xi_{m,n}$ consists of dominant $\son$-weights, 
but, a priori, it is not clear that these weights are all contained in $X_+(\son)^{S^{\otimes m}}$.
Our next result is crucial for the inductive argument that establishes this.

\begin{lem}\label{lem:bijection}
Let $m \geq 1$.
Given any $\ab_{m-1}\in \SP_{m-1}^{\le n}$, 
the irreducible summands of $V_{\Xi_{m-1,n}(\ab_{m-1})} \otimes S$ 
are precisely the representations $V_{\Xi_{m,n}(\ab_m)}$ for 
$\ab_m\in \SP_m^{\le n}$ which interlace $\ab_{m-1}$. 
\end{lem}

\begin{proof}
To begin, since $S$ is a minuscule $U_q(\son)$-representation, 
\cite[Corollary 30.7]{EtingofLie} gives that
\begin{equation}\label{eq:tensor-product}
V_{\lambda}\otimes S \cong 
\bigoplus_{\substack{\mu\in  \operatorname{wt}(S) \\ \lambda + \mu \in X_+(\son)}} 
	V_{\lambda+ \mu} 
\end{equation}
for any $\lambda\in X_+(\son)$.
Observe that in \eqref{eq:tensor-product} 
the sum is indexed by all weights $\mu$ of $S$ such that $\lambda + \mu$ is dominant, 
and that each summand appears with multiplicity one. 
Each weight of $S$ has the form
\[ 
\left(\pm \frac{1}{2}, \pm \frac{1}{2}, \ldots, \pm \frac{1}{2}\right),
\]
and we will encode such weights as $(\ep_1, \ldots, \ep_n) \cdot \frac{1}{2}$ for $\ep_i \in \{+, -\}$.

Suppose that $m = 2k+1$ is odd and that $\ab_{m-1} = (a_1 \ge a_2 \ge \ldots \ge a_k)$ 
is an $(m-1)$-spin-partition with $\Xi(\ab_{m-1}) = (c_1, c_2, \ldots, c_n) =: \lambda$. 
Set $a_0 = \frac{2n+1}{2}$, $a_{k+1} = \frac{1}{2}$, and, for $0 \le i \le k$, set $q_i = a_i - a_{i+1}$.
It follows that $\lambda$ contains a sequence of $q_i$ consecutive entries $c_j$, where $c_j = k - i$.

We now characterize the weights $\mu = (\ep_1, \ldots, \ep_n) \cdot \frac{1}{2} $ in $S$ such that $\lambda + \mu$ is dominant. 
For each $i$, the corresponding $q_i$ consecutive entries in $\mu$ must be a \emph{plus-to-minus sequence}, 
meaning that in those entries we see all instances of $+\frac{1}{2}$ followed by all instances of $-\frac{1}{2}$.
This is required since, in order for $\lambda + \mu$ to be dominant, it must be a (weakly) decreasing sequence. 
Observe that there are $q_i + 1$ choices of plus-to-minus sequence of length $q_i$:  
namely, one with $q_i - p_i$ pluses followed by $p_i$ minuses, for each $0 \le p_i \le q_i$. 
In addition to this, for the $q_k$ entries with $c_j = 0$, we must have $\mu_j = +\frac{1}{2}$, 
since in order for $\lambda + \mu$ to be dominant, it also must have no negative entries.

We now show that the set of such weights $\lambda + \mu$ is exactly the image under $\Xi_{m,n}$ 
of the $m$-spin-partitions $\ab_m=(b_1 \ge b_2 \ge \ldots \ge b_k)$ that interlace with $\ab_{m-1}$, 
i.e.~such that $b_1 \ge a_1 \ge b_2 \ge a_2 \ge \ldots \ge b_k \ge a_k$.
The choice of each $b_i$ is independent of the others, and there are $q_i + 1$ half-integers between $a_i$ and $a_{i+1}$ (inclusive): 
those of the form $a_{i+1} + p_i$ for $0 \le p_i \le q_i$. 
Setting $b_i := a_i + p_{i-1}$, 
it is straightforward to then verify that 
$\Xi_{m,n}(\ab_m) = \lambda + \mu$,
where $\mu = (\ep_1, \ldots, \ep_n)$ is the weight corresponding to the $p_i$'s as described above.
(This is most easily checked by noting that $b_i - b_{i+1} = a_i + p_{i-1} - a_{i+1} - p_i = (q_i - p_i) + p_{i-1}$, 
which records the number of $(k-i+\tfrac{1}{2})$'s that appear in $\Xi_{m,n}(\ab_m)$.)
Note that $p_k$ (and hence $q_k$) plays no role in the construction of $\ab_m$; 
this corresponds to the fact that, in the corresponding entries of $\mu$, 
we were required to take the ``all-plus'' sequence.

By \eqref{eq:tensor-product}, this concludes the proof when $m$ is odd, 
as we have shown that $\Xi_{m,n}$ provides a bijection between 
the $\prod_{i=0}^{k-1} (q_i + 1)$ choices of interlacing $m$-spin-partition and 
the choices of weight $\mu$ such that $\lambda + \mu$ is dominant.

The proof when $m = 2k+2$ is similar. 
Now, there are $\prod_{i=0}^{k} (q_i+1)$ choices of interlacing $m$-spin-partition (now, a product of $k+1$ terms), 
since there is also a choice of $b_{k+1}$ satisfying $a_k \ge b_{k+1} \ge \frac{1}{2}$. 
Correspondingly, $\lambda$ is a string of half-integers with $q_k$ copies of $\frac{1}{2}$, 
so the final plus-to-minus sequence is permitted to have minuses since $\frac{1}{2} - \frac{1}{2} = 0$ is still non-negative.
\end{proof}

Before proceeding, 
we give explicit instances of the correspondence
just described.

\begin{example}
Let $m=9$ and consider $\ab_8=(9/2, 5/2, 5/2, 1/2) \in \SP_{8}^{\le 4}$. 
In the language of the proof of Lemma \ref{lem:bijection} we have that 
$q_0=0, q_1=2, q_2=0, q_3=2, q_4=0$.
Correspondingly, we see that
$\Xi_{8,4}(\ab_8) = (3,3,1,1)$, 
since here there are $q_1=2$ threes, followed by $q_3=2$ ones, and there 
are no fours, twos, or zeros since $q_0=q_2=q_4=0$.
The choices of $\ab_9$ that interlace $\ab_8$ are as follows:
\[
\begin{tikzcd}[column sep = 0.1 em, row sep =0.1 em]
\ab_9: & 9/2 & & b_2 & & 5/2 & & b_4 &  \\
\ab_8: &  & 9/2 & & 5/2 & & 5/2 & & 1/2
	\end{tikzcd}
	\]
with $b_2 \in \{9/2,7/2,5/2\}$ and $b_4 \in \{5/2,3/2,1/2\}$.
The three choices of $b_2$ correspond to the plus-to-minus sequences $(+,+), (+,-), (-,-)$ that 
could appear in the first two entries of $\mu = (\ep_1, \ep_2, \ep_3, \ep_4) \cdot \frac{1}{2}$ 
and the three choices of $b_4$ correspond to the same plus-to-minus sequences that could appear as 
the last two entries of $\mu$.
Taking $\mu= (+,-,-,-) \cdot \frac{1}{2}$ corresponds to taking $p_1=1$ and $p_3=2$, 
which yields $\ab_9 = (9/2,7/2,5/2,5/2)$ and we see
\[
\Xi_{9,4}(\ab_9) = (3+\tfrac{1}{2},2+\tfrac{1}{2}, \tfrac{1}{2}, \tfrac{1}{2}) 
	= (3,3,1,1) + (+\tfrac{1}{2},-\tfrac{1}{2}, -\tfrac{1}{2}, -\tfrac{1}{2}) \, .
	\]
	\end{example}

\begin{example}
Again, let $m=9$ but consider $\ab_8=(7/2, 5/2, 3/2, 3/2) \in \SP_{8}^{\le 4}$. 
Now, we have $q_0=1, q_1=1, q_2=1, q_3=0, q_4=1$, and therefore
$\Xi_{8,4}(\ab_8) = (4,3,2,0)$. 
The $\ab_9$ that interlace $\ab_8$ are:
\[
\begin{tikzcd}[column sep = 0.1 em, row sep =0.1 em]
\ab_9: & b_1 & & b_2 & & b_3 & & 3/2 &  \\
\ab_8: &  & 7/2 & & 5/2 & & 3/2 & & 3/2
	\end{tikzcd}
	\]
with $b_1 \in \{9/2,7/2\}$, $b_2 \in \{7/2,5/2\}$ and $b_3 \in \{5/2,3/2\}$. 
The two possible choices in each case corresponds to the fact that $q_0=q_1=q_2=1$.
Although $q_4 \neq 0$, this does not play a role in the possible choices 
(since there is no entry $b_5$ in $\ab_9$).
	\end{example}

\begin{example} 
Let $m=10$ and consider 
$\ab_9 = (41/2, 31/2, 21/2, 9/2)\in \SP_9^{\le 24}$. 
We now have $q_0=4, q_1=5, q_2=5, q_3=6$, and $q_4=4$, 
which, given Definition \ref{def:xi}, are now the number of $4+\tfrac{1}{2}$'s, $3+\tfrac{1}{2}$'s, 
etc.~appearing in $\Xi(\ab_9)$, i.e.
\[
\lambda := \Xi(\ab_9) = (4,4,4,4,3,3,3,3,3,2,2,2,2,2,1,1,1,1,1,1,0,0,0,0) + (\tfrac{1}{2},\ldots,\tfrac{1}{2}) \, .
	\]
The choices of interlacing $\ab_{10}$ are indicated as follows:
\[
\begin{tikzcd}[column sep = 0.1 em, row sep =0.1 em]
\ab_{10}: & b_1 & & b_2 & & b_3 & & b_4 && b_5  \\
\ab_9: &  & 41/2 & & 31/2 & & 21/2 & & 9/2 &
	\end{tikzcd}
	\]
where 
$b_1 \in \{49/2, \ldots, 41/2\}$, 
$b_2 \in \{41/2, \ldots, 31/2\}$, 
$b_3 \in \{31/2, \ldots, 21/2\}$, 
$b_4 \in \{21/2, \ldots, 9/2\}$,
and $b_5 \in \{9/2,\ldots,1/2\}$.
As noted above, the number of choices for each $b_i=a_{i}+p_{i-1}$ 
is exactly equal to $q_{i-1}+1$. 
Observe here that, since $m$ is even, 
we do have five possible choices $b_5$, 
corresponding to having $q_4=4$.
Taking $p_0=1, p_1=3, p_2=2, p_3=1$, and $p_4=3$, we have
\[
\mu = (+,+,+,-,+,+,-,-,-,+,+,+,-,-,+,+,+,+,+,-,+,-,-,-) \cdot \tfrac{1}{2}
\]
so
$\lambda + \mu = (5,5,5,4,4,4,3,3,3,3,3,3,2,2,2,2,2,2,2,1,1,0,0,0)$.
One can easily verify that this equals 
$\Xi_{10,24}\big( (43/2, 37/2, 25/2,11/2, 7/2) \big)$.
\end{example}

\begin{thm}\label{thm:bijection}
The map $\lambda \mapsto \ab_{m,n}(\lambda)$ 
is a bijection $X_+(\son)^{S^{\otimes m}}\longrightarrow \SP_m^{\le n}$.
\end{thm}

\begin{proof}
As already noted above, the map $\ab_{m,n}(-)$ in \eqref{eq:ourbijection} is injective, 
so it suffices to find some bijection between 
the finite sets $X_+(\son)^{S^{\otimes m}}$ and $\SP_m^{\le n}$.
By Lemma \ref{lem:bijection}, $\Xi_{m,n} \colon \SP_m^{\le n} \to X_+(\son)^{S^{\otimes m}}$ 
is a bijection if $\Xi_{m-1,n} \colon \SP_{m-1}^{\le n} \to X_+(\son)^{S^{\otimes m-1}}$ is a bijection.
The result then follows by induction, with the base case established as part of Example \ref{ex:basecase}.
	\end{proof}
	
\begin{cor}\label{cor:inj}
The algebra homomorphism 
$\varphi \circ \psi^{\leq n} \colon \iqUsom^{\le n} \rightarrow \End_{U_q(\son)}(S^{\otimes m})$
from \eqref{eq:phipsin} is an isomorphism.
	\end{cor}
\begin{proof}
Proposition \ref{prop:alg-hom-surj} implies that \eqref{eq:phipsin} is surjective, 
so it suffices to show \eqref{eq:phipsin} is injective. 
By Corollary \ref{cor:complete-reducibility}, $\iqUsom^{\le n}$ is a finite-dimensional semisimple algebra. 
By Lemma \ref{lem:fullforssisfaithful} and Proposition \ref{prop:repsofiQn}, 
it suffices to show that for all $\ab\in \SP_{m}^{\le n}$, $M_{\ab}$ is isomorphic to a direct summand of $S^{\otimes m}$. 
This is forced by Theorem \ref{thm:bijection}, and Proposition \ref{prop:u'-doublecantralizer}.
	\end{proof}
		
\section{Putting it all together}

To conclude, we restate Theorem \ref{thm:main} and piece together the results that give its proof.

\begin{thm}
There is an equivalence of $\C(q)$-linear pivotal categories 
\[
\varphi \colon \Web(\son) \to \FRep(U_q(\son))
\]
sending $k \mapsto V_{\varpi_k}$ (for $1 \le k \le n-1$) and $\gS \mapsto S$.
\end{thm}

\begin{proof}
The functor $\varphi$ is defined in Theorem \ref{thm:functor}, and, by definition, 
it is essentially surjective.
Next, Proposition \ref{prop:FF} shows that this functor is fully faithful if and only if 
all of the corresponding algebra homomorphisms
\begin{equation}\label{eq:alghom}
\varphi \colon \End_{\Web(\son)}(\gS^{\otimes m}) \to \End_{U_q(\son)}(S^{\otimes m})
	\end{equation}
are isomorphisms. 
Corollary \ref{cor:inj} implies that the composition $\varphi \circ \psi^{\leq n}$ is an isomorphism, 
where $\psi^{\leq n} \colon \iqUsom^{\le n} \rightarrow \End_{\Web(\son)}(\gS^{\otimes m})$ 
is the surjective homomorphism induced by Proposition \ref{prop:iQtoLad}.
This implies that \eqref{eq:alghom} is an isomorphism.
\end{proof}

Having established Theorem \ref{thm:main}, 
we can deduce the tangle and tangled web relations involving $c_{\gS,\gS}^{\pm 1}$ 
that were promised in Remark \ref{rem:promisedStanglerlns}.

\begin{cor}\label{cor:Stangledweb}
We have
\begin{equation}\label{eq:Stangle}
\begin{tikzpicture}[scale=.4, anchorbase]
	\draw[very thick, gray] (2,2) to (2,0) to [out=270,in=0] (1.5,-.5) to [out=180,in=270] (1,0) 
		to [out=90,in=270] (0,1.5) to [out=90,in=0] (-.5,2) to [out=180,in=90] (-1,1.5) to (-1,-.5);
	\draw[overcross] (0,0) to [out=90,in=270] (1,1.5); 
	\draw[very thick,gray] (0,-.5) to (0,0) to [out=90,in=270] (1,1.5) to (1,2);
\end{tikzpicture}
=
\begin{tikzpicture}[scale=.4, anchorbase]
	\draw[very thick, gray] (0,0) to [out=90,in=270] (1,1.5);
	\draw[overcross] (1,0) to [out=90,in=270] (0,1.5);
	\draw[very thick,gray] (1,0) to [out=90,in=270] (0,1.5);
\end{tikzpicture}
=
\begin{tikzpicture}[scale=.4, anchorbase]
	\draw[very thick, gray] (1,-.5) to (1,0) to [out=90,in=270] (0,1.5) to (0,2);
	\draw[overcross] (0,0) to [out=90,in=270] (1,1.5); 
	\draw[very thick,gray] (-1,2) to (-1,0) to [out=270,in=180] (-.5,-.5) to [out=0,in=270] (0,0) 
		to [out=90,in=270] (1,1.5) to [out=90,in=180] (1.5,2) to [out=0,in=90] (2,1.5) to (2,-.5);
\end{tikzpicture} \qquad , \qquad
\begin{tikzpicture}[scale=.375, tinynodes,anchorbase]
	\draw[very thick,gray] (-.5,0) to [out=90,in=210] (0,.75);
	\draw[very thick,gray] (.5,0) to [out=90,in=330] (0,.75);
	\draw[very thick] (0,.75) to [out=90,in=270] (0,2.5);
	\draw[overcross] (-1,2.5) to [out=270,in=180] (0,1.5) to [out=0,in=270] (1,2.5);
	\draw[very thick, gray] (-1,2.5) to [out=270,in=180] (0,1.5) to [out=0,in=270] (1,2.5);
\end{tikzpicture}
=
\begin{tikzpicture}[scale=.375, tinynodes,anchorbase]
	\draw[very thick,gray] (-.5,0) to [out=90,in=210] (0,1.75);
	\draw[very thick,gray] (.5,0) to [out=90,in=330] (0,1.75);
	\draw[very thick] (0,1.75) to [out=90,in=270] (0,2.5);
	\draw[overcross] (-1,2.5) to [out=270,in=180] (0,1) to [out=0,in=270] (1,2.5);
	\draw[very thick, gray] (-1,2.5) to [out=270,in=180] (0,1) to [out=0,in=270] (1,2.5);
\end{tikzpicture} \qquad , \qquad
\begin{tikzpicture}[scale=.45,anchorbase,tinynodes]
	\draw[very thick, gray] (.8,-1) to [out=90,in=270] (0,.5);
	\draw[very thick, gray] (1.6,-1) to [out=90,in=270] (0,2);
	\draw[overcross] (0,.5) to [out=90,in=270] (.8,2);
	\draw[very thick, gray] (0,.5) to [out=90,in=270] (.8,2);
	\draw[overcross] (0,-1) to [out=90,in=270] (1.6,2);
	\draw[very thick,gray] (0,-1) to [out=90,in=270] (1.6,2);
\end{tikzpicture}
=
\begin{tikzpicture}[scale=.45,anchorbase,tinynodes]
	\draw[very thick, gray] (1.6,-1) to [out=90,in=270] (0,2);
	\draw[very thick, gray] (1.6,.5) to [out=90,in=270] (.8,2);
	\draw[overcross] (.8,-1) to [out=90,in=270] (1.6,.5);
	\draw[very thick, gray] (.8,-1) to [out=90,in=270] (1.6,.5);
	\draw[overcross] (0,-1) to [out=90,in=270] (1.6,2);
	\draw[very thick,gray] (0,-1) to [out=90,in=270] (1.6,2);
\end{tikzpicture} \, .
	\end{equation}
\end{cor}
\begin{proof}
Since $\FRep(U_q(\son))$ is a ribbon category (braided, pivotal, and with twist satisfying a certain equation), 
Proposition \ref{prop:SpinBraiding} and Lemma \ref{lem:otherR2} imply that the image under $\varphi$ of 
each of the relations in \eqref{eq:Stangle}
are satisfied in $\FRep(U_q(\son))$. 
By Theorem \ref{thm:main}, these relations hold in $\Web(\son)$. 
\end{proof}

\begin{rem}\label{rem:braiding}
The equivalence $\varphi$ can be used to define a braiding on $\Web(\son)$, 
which promotes Theorem \ref{thm:main} to an equivalence of ribbon categories.
The linear combinations of webs $c_{\gS, \gS}^{\pm 1}$ give the braiding and its inverse in $\End_{\Web(\son)}(\gS\otimes \gS)$. 
Then, since each other generating object in $\Web(\son)$ is a direct summand of $\gS\otimes \gS$, 
naturality of the braiding in $\Rep(U_q(\son))$ shows that we can define the braiding 
on arbitrary objects in $\Web(\son)$ in terms of $c_{\gS,\gS}$ and the webs
$\begin{tikzpicture}[scale =.25, tinynodes,anchorbase,rotate=180]
	\draw[very thick, gray] (0,0) to [out=90,in=210] (.5,.75);
	\draw[very thick, gray] (1,0) to [out=90,in=330] (.5,.75);
	\draw[very thick] (.5,.75) to (.5,1.5) node[below=-2pt]{$k$};
\end{tikzpicture}$
and
$\begin{tikzpicture}[scale =.25, tinynodes,anchorbase]
	\draw[very thick, gray] (0,0) to [out=90,in=210] (.5,.75);
	\draw[very thick, gray] (1,0) to [out=90,in=330] (.5,.75);
	\draw[very thick] (.5,.75) to (.5,1.5) node[above=-4pt]{$k$};
\end{tikzpicture}$.
\end{rem}

Finally, we prove Theorem \ref{thm:mainthm2} on relative braid group symmetries of type I nonclassical representations.

\begin{thm}\label{thm:iquantumWeylR3}
The relations
\begin{equation}\label{eq:iquantumWeylR3}
\iQW_i\iQW_{i\pm1}\iQW_i = \iQW_{i\pm 1}\iQW_i\iQW_{i\pm1} 
\quad \text{and} \quad 
\iQW_i^{-1}\iQW_{i\pm1}^{-1}\iQW_i^{-1} = \iQW_{i\pm 1}^{-1}\iQW_i^{-1}\iQW_{i\pm1}^{-1}
\end{equation}
hold in $\iqUsom^{\le n}$, 
and thus in any finite-dimensional type I nonclassical representation.
\end{thm}
\begin{proof}
This follows from Corollary \ref{cor:inj} and Remark \ref{rem:TisR}, 
since $R_{S,S}^{\pm 1}$ satisfy the Reidemeister III relations.
\end{proof}

\appendix

\allowdisplaybreaks

%
\section{Proof of Lemma \ref{lem:techcomm}}\label{Appendix:proof-of-commutator-lemma}
%

\begin{lem}
For $k\geq 1$, the following relation holds in $\iqUsom$:
\begin{equation}
q^{-k}b_{i,i+1}\xxb_{i}^{(k)} = (-1)^kq^{k}\xxb_{i}^{(k)}b_{i,i+1} +(-1)^kq^{k-1}\xxb_i^{(k-1)}b_{i,i+1} + q^{-k+1}\xxb_i^{(k-1)}b_{i+1}.
\end{equation}
\end{lem}

\begin{proof}
We induct on $k$. 
Using \eqref{eq:longb}, 
equation \eqref{eqn:iSerre-alg-version}
implies
\[
b_{i,i+1} b_{i} = -q^2b_{i}b_{i,i+1} + qb_{i+1} \, .
\]
This gives
\begin{align*}
q^{-1}b_{i,i+1}\xxb_i &= q^{-1}b_{i,i+1} b_i - \frac{q^{-1}}{[2]}b_{i,i+1} = -qb_{i}b_{i,i+1}  + b_{i+1} - \frac{q^{-1}}{[2]}b_{i,i+1} \\
&= -qb_{i}b_{i,i+1} + b_{i+1} - \frac{q^{-1}}{[2]}b_{i,i+1} - \frac{-q}{[2]}b_{i,i+1} + \frac{-q}{[2]}b_{i,i+1}
=-q\xxb_{i}b_{i,i+1} - b_{i,i+1} + b_{i+1} \, ,
\end{align*}
which establishes the base case $k=1$. 
We record this equality as:
\[
b_{i,i+1}\xxb_i = -q^2\xxb_{i}b_{i,i+1} -qb_{i,i+1} + qb_{i+1} \, .
\]
Also, a short calculation using \eqref{eq:longb} implies that
\[
b_{i+1}\xxb_{i} = -q^{-2}\xxb_{i}b_{i+1} + q^{-1}b_{i,i+1} - q^{-1}b_{i+1} \, .
\]

Now, assume the Lemma is true for $k\ge 1$. 
Recall that \eqref{eq:blackx} says
\begin{equation}\label{eqn:alpha-beta-dividedpowerrecursion}
\xxb_i^{(k+1)} = \alpha_k \xxb_i^{(k)}\xxb_i + \beta_k \xxb_i^{(k)} 
\quad \text{where} \qquad \alpha_k = \frac{(-1)^k}{``[k+1]^2"} \quad \text{and} \quad \beta_k = -\frac{``[k+1][k]"}{``[k+1]^2"} \, .
\end{equation}
Note that 
\begin{equation}\label{eqn:2-times-beta-over-alpha}
\frac{\beta_k}{\alpha_k} = (-1)^{k+1}``[k+1][k]" \qquad \text{and} \qquad [2]``[k+1][k]" = [2k+1] + (-1)^{k+1} \, .
\end{equation}
We then compute
\begin{align*}
q^{-(k+1)}b_{i,i+1}\xxb_{i}^{(k+1)} &= q^{-1}\left(q^{-k}b_{i,i+1} \xxb_{i}^{(k)}\right)(\alpha_k\xxb_{i} + \beta_k) \\
&=q^{-1} \left((-1)^kq^{k}\xxb_{i}^{(k)}b_{i,i+1} + (-1)^kq^{k-1}\xxb_i^{(k-1)}b_{i,i+1} + q^{-k+1}\xxb_i^{(k-1)}b_{i+1}\right)(\alpha_k\xxb_{i} + \beta_k) \\
&= (-1)^kq^{k-1}\alpha_k\cdot\xxb_{i}^{(k)}b_{i,i+1}\xxb_{i} + (-1)^kq^{k-2}\alpha_k\cdot\xxb_i^{(k-1)}b_{i,i+1}\xxb_{i} + q^{-k}\alpha_k\cdot\xxb_i^{(k-1)}b_{i+1}\xxb_{i} \\
	&\quad+ (-1)^kq^{k-1}\beta_k\cdot\xxb_{i}^{(k)}b_{i,i+1} + (-1)^kq^{k-2}\beta_k\cdot\xxb_i^{(k-1)}b_{i,i+1} + q^{-k}\beta_k\cdot \xxb_i^{(k-1)}b_{i+1} \\
&= (-1)^kq^{k-1}\alpha_k\cdot\xxb_{i}^{(k)}\left(-q^2\xxb_{i}b_{i,i+1} -qb_{i,i+1} + qb_{i+1}\right) \\
	&\quad + (-1)^kq^{k-2}\alpha_k\cdot\xxb_i^{(k-1)}\left(-q^2\xxb_{i}b_{i,i+1} -qb_{i,i+1} + qb_{i+1}\right) \\
	&\qquad + q^{-k}\alpha_k\cdot\xxb_i^{(k-1)}\left(-q^{-2}\xxb_{i}b_{i+1} + q^{-1}b_{i,i+1} - q^{-1}b_{i+1}\right) \\
	&\qquad \quad + (-1)^kq^{k-1}\beta_k\cdot\xxb_{i}^{(k)}b_{i,i+1} + (-1)^kq^{k-2}\beta_k\cdot\xxb_i^{(k-1)}b_{i,i+1} + q^{-k}\beta_k\cdot \xxb_i^{(k-1)}b_{i+1} \\
&= \underline{(-1)^{k+1}q^{k+1}\alpha_k\cdot\xxb_{i}^{(k)}\xxb_{i}b_{i,i+1}} + (-1)^{k+1}q^{k}\alpha_k\cdot\xxb_{i}^{(k)}b_{i,i+1} 
		+ (-1)^kq^{k}\alpha_k\cdot\xxb_{i}^{(k)}b_{i+1} \\
	&\quad + \underline{(-1)^{k+1}q^{k}\alpha_k\cdot\xxb_i^{(k-1)}\xxb_{i}b_{i,i+1}} + (-1)^{k+1}q^{k-1}\alpha_k\cdot\xxb_i^{(k-1)}b_{i,i+1} \\
	&\qquad + (-1)^kq^{k-1}\alpha_k\cdot\xxb_i^{(k-1)}b_{i+1} - \underline{ q^{-k-2}\alpha_k\cdot\xxb_i^{(k-1)}\xxb_{i}b_{i+1} } + q^{-k-1}\alpha_k\cdot\xxb_i^{(k-1)}b_{i,i+1} \\
	&\qquad \quad - q^{-k-1}\alpha_k\cdot\xxb_i^{(k-1)}b_{i+1} + (-1)^kq^{k-1}\beta_k\cdot\xxb_{i}^{(k)}b_{i,i+1} \\
	&\qquad \qquad + (-1)^kq^{k-2}\beta_k\cdot\xxb_i^{(k-1)}b_{i,i+1} + q^{-k}\beta_k\cdot \xxb_i^{(k-1)}b_{i+1} \, .
\end{align*}
Using \eqref{eqn:alpha-beta-dividedpowerrecursion}, the \underline{underlined} terms become
\begin{multline*}
(-1)^{k+1}q^{k+1}\alpha_k\cdot\left(\frac{1}{\alpha_k}\xxb_{i}^{(k+1)} - \frac{\beta_k}{\alpha_k}\xxb_{i}^{(k)}\right)b_{i,i+1}
+(-1)^{k+1}q^{k}\alpha_k\cdot\left(\frac{1}{\alpha_{k-1}}\xxb_{i}^{(k)} - \frac{\beta_{k-1}}{\alpha_{k-1}}\xxb_{i}^{(k-1)}\right)b_{i,i+1} \\
-q^{-k-2}\alpha_k\cdot\left(\frac{1}{\alpha_{k-1}}\xxb_{i}^{(k)} - \frac{\beta_{k-1}}{\alpha_{k-1}}\xxb_{i}^{(k-1)}\right)b_{i+1}
\end{multline*}
or
\begin{multline*}
(-1)^{k+1}q^{k+1}\cdot\xxb_{i}^{(k+1)}b_{i,i+1}  
	+(-1)^{k}q^{k+1}\beta_k\cdot \xxb_{i}^{(k)}b_{i,i+1}
	+(-1)^{k+1}q^{k}\frac{\alpha_k}{\alpha_{k-1}}\cdot\xxb_{i}^{(k)} b_{i,i+1} \\
	+(-1)^{k}q^{k}\frac{\alpha_k\beta_{k-1}}{\alpha_{k-1}}\cdot\xxb_{i}^{(k-1)}b_{i,i+1}
	-q^{-k-2}\frac{\alpha_k}{\alpha_{k-1}}\cdot\xxb_{i}^{(k)}b_{i+1} 
	+q^{-k-2}\frac{\alpha_k\beta_{k-1}}{\alpha_{k-1}}\cdot\xxb_{i}^{(k-1)}b_{i+1} \, .
\end{multline*}
Combining like terms we arrive at
\begin{align*}
q^{-(k+1)}b_{i,i+1}\xxb_{i}^{(k+1)} &= (-1)^{k+1}q^{k+1}\cdot\uwave{\xxb_{i}^{(k+1)}b_{i,i+1}} \\
	&\quad +\left( (-1)^{k+1}q^{k}\alpha_k + (-1)^k (q^{k-1}+q^{k+1})\beta_k 
		+(-1)^{k+1}q^{k}\frac{\alpha_k}{\alpha_{k-1}} \right)\cdot\uwave{ \xxb_{i}^{(k)}b_{i,i+1} } \\
	&\quad + \left((-1)^kq^{k}\alpha_k -q^{-k-2}\frac{\alpha_k}{\alpha_{k-1}}\right)\cdot\uwave{\xxb_{i}^{(k)}b_{i+1}} \\
	&\quad + \left( ((-1)^{k+1}q^{k-1} + q^{-k-1})\alpha_k + (-1)^kq^{k-2}\beta_k
		+(-1)^{k}q^{k}\frac{\alpha_k\beta_{k-1}}{\alpha_{k-1}}\right)\cdot\uwave{\xxb_{i}^{(k-1)}b_{i,i+1}} \\
	&\quad +\left(((-1)^kq^{k-1} - q^{-k-1})\alpha_k + q^{-k}\beta_k
		+q^{-k-2}\frac{\alpha_k\beta_{k-1}}{\alpha_{k-1}}\right)\cdot\uwave{\xxb_{i}^{(k-1)}b_{i+1}} \, .
\end{align*}
To finish the proof, we analyze the coefficients on the \uwave{wavy underlined} terms:
we must show that the second and third are $(-1)^{k+1}q^{k}$ and $q^{-k}$ respectively, 
and that the fourth and fifth are zero. 

To show that the fourth coefficient is zero, we first divide by $\alpha_k$ and then multiply by $[2]$. 
Using \eqref{eqn:2-times-beta-over-alpha}, 
this results in 
\begin{align*}
(-1)^{k+1}q^{k-1}[2] + q^{-k-1}[2] &-q^{k-2}\left([2k+1] + (-1)^{k+1}\right) +q^{k}\left([2(k-1)+1] + (-1)^k\right) \\
	&=(-1)^{k+1}q^{k} + (-1)^{k+1}q^{k-2} + q^{-k} + q^{-k-2} + (-1)^kq^{k-2} + (-1)^kq^{k} \\
	& \qquad + \left(-q^{k-2}[2k+1] +q^{k}[2(k-1)+1] \right) \\
	&= q^{-k} + q^{-k-2} +\left(-q^{3k-2} - \dots - q^{-k-2} + q^{3k-2} + \dots + q^{-k+2}\right) = 0 \, .
\end{align*}
A similar argument shows that the fifth coefficient is zero.

To handle the second and third coefficients, we first observe that 
\[
\alpha_k =  \frac{(-1)^k}{``[k+1]^2"}\frac{[2]}{[2]} = (-1)^k\frac{[2]}{[2k+2]} 
\qquad \text{and} \qquad \beta_k = \frac{-``[k+1][k]"}{``[k+1]^2"}\frac{[2]}{[2]} = \frac{-[2k+1] + (-1)^k}{[2k+2]} \, .
\]
Thus, the second coefficient becomes
\begin{align*}
(-1)^{k+1}q^{k}\alpha_k + (-1)^kq^{k-1}\beta_k &+(-1)^{k}q^{k+1}\beta_k +(-1)^{k+1}q^{k}\frac{\alpha_k}{\alpha_{k-1}} \\
	&= (-1)^{k+1}q^k\left(\alpha_k -(q^{-1}+q)\beta_k +\frac{\alpha_k}{\alpha_{k-1}} \right) \\
	&= \frac{(-1)^{k+1}q^k}{[2k+2]}\left((-1)^k[2]-[2](-[2k+1] + (-1)^k)-[2k] \right)\\
	&= \frac{(-1)^{k+1}q^k}{[2k+2]}\left(-[2](-[2k+1])-[2k] \right)
	= \frac{(-1)^{k+1}q^k[2k+2]}{[2k+2]} = (-1)^{k+1}q^k \, .
\end{align*}
Similarly, the third coefficient becomes
\begin{align*}
(-1)^kq^{k}\alpha_k - q^{-k-2}\frac{\alpha_k}{\alpha_{k-1}} &= q^{-k}\left( (-1)^kq^{2k}\alpha_k - q^{-2}\frac{\alpha_k}{\alpha_{k-1}} \right) \\
	&= q^{-k}\left( q^{2k}\frac{[2]}{[2k+2]} + q^{-2}\frac{[2k]}{[2k+2]} \right)
	= q^{-k}\frac{q^{2k}[2] + q^{-2}[2k]}{[2k+2]}
	=q^{-k} \, . \qedhere
\end{align*}
\end{proof}

%
\section{Comparing cups and caps}\label{Appendix:comparing-cups-and-caps}
%

Fix $k\in \{1, \dots, n-1\}$. We will show the following.
\begin{lem}\label{lem-spin-vs-O-cups-and-caps}
\begin{equation}\label{eq:BW-vs-BER}
\hat{\varphi}_{O}\bigg(
\begin{tikzpicture}[scale=.4, anchorbase,tinynodes]
	\draw[very thick] (0,1.25) to [out=90,in=270] 
	(0,1.5) to [out=90,in=180] (.5,2.25) node[above=-2pt]{$k$}
		to [out=0,in=90] (1,1.5) to (1,1.25);
\end{tikzpicture}  
\bigg)
= (-1)^{\binom{k}{2}}\hat{\varphi}_{BER}\bigg(
\begin{tikzpicture}[scale=.4, anchorbase,tinynodes]
	\draw[very thick] (0,1.25) to [out=90,in=270] (0,1.5) to [out=90,in=180] (.5,2.25) node[above=-2pt]{$k$}
		to [out=0,in=90] (1,1.5) to (1,1.25);
\end{tikzpicture}  
\bigg)
\end{equation}
\end{lem}

Before we prove Lemma \ref{lem-spin-vs-O-cups-and-caps}, 
we review how each intertwiner in \eqref{eq:BW-vs-BER} is defined. 
To understand the definitions, 
recall that there are two descriptions of the type $B_n$ fundamental representations $V(\varpi_k)$.

The first description $V_k$ is the finite-dimensional irreducible (type I) representation 
generated by a highest weight vector $v_k^+$ of weight $\varpi_k$, 
e.g.~$V_1$ is the unique finite-dimensional irreducible generated by a highest weight vector $v_1^+$ with weight $\varpi_1 = \epsilon_1$.
The image of the cup and cap morphism under $\hat{\varphi}_{BER}$ are described using $V_k$.

The second description $\Lambda_q^k$ is the (quantum) $k$-th exterior power of the defining representation. 
For example, $\Lambda_q^1$ is the defining representation itself which has basis $\{a_1, \dots, a_n, u, b_n, \dots, b_1\}$ 
where $\mathrm{wt}(a_i) = \epsilon_i$, $\mathrm{wt}(u) = 0$, and $\mathrm{wt}(b_i)= - \epsilon_i$. 
The action of the generators of $U_q(\son)$ on $\Lambda_q^1$ can be read from \cite[Proposition 4.7]{BER}. 
The image of the cups and caps under $\hat{\varphi}_{O}$ are described using $\Lambda_q^k$, 
which, in general, is defined to be the degree $k$ component of the following algebra.

\begin{defn}\label{def:q-exterior-algebra}
The type $B_n$ (quantum) exterior algebra $\Lambda^*_q$ is the associative $\C(q)$-algebra 
with generators $a_1, \dots, a_n, u, b_n, \dots, b_1$ subject to the following relations.
\begin{gather*}
a_i^2 = 0 \, , \quad b_i^2 = 0 \, , \quad u^2 = q^{-1}(q^2-q^{-2})\sum_{\ell=1}^n (-q^2)^{\ell-1}a_{n-\ell+1}b_{n-\ell+1} \\
a_ja_i = -q^{-2}a_ia_j \, , \quad b_ib_j = -q^{-2}b_jb_i \quad i< j \\
ua_i = -q^{-2}a_iu \, , \quad b_iu= -q^{-2}ub_i \, , \quad b_ia_i = -a_ib_i + (q^2-q^{-2})\sum_{\ell=1}^{i-1}(-q^2)^{\ell-1}a_{i-\ell}b_{i-\ell} \\
b_ja_i = -q^{-2}a_ib_j \quad i\ne j
\end{gather*}
The algebra $\Lambda_q^*$ is graded, with generators in degree $1$.
\end{defn}

\begin{rem}
The definition of $\Lambda_q^*$ is borrowed from \cite[Definition 3.43]{BodWu}. 
More precisely, Definition \ref{def:q-exterior-algebra} uses the simplified presentation found in \cite[Corollary 3.45]{BodWu}, 
modified for our setting (since \cite{BodWu} uses the opposite coproduct from our Definition \ref{def:quantumgroup}).
\end{rem}

The simplified presentation is deduced from the presentation in \cite[Definition 3.43]{BodWu}, 
which describes $\Lambda_q^*$ as the quotient of the tensor algebra 
by a $q$-analogue of the symmetric square. 
From this perspective, it follows that the algebra $U_q(\son)$ acts on $\Lambda_q^*$,
preserving each $\Lambda_q^r$ (the action is induced from the action on $(\Lambda_q^1)^{\otimes r}$ 
via the quotient map $(\Lambda_q^1)^{\otimes r}\rightarrow \Lambda_q^r$).
On the other hand, the simplified presentation makes it routine to show that $\Lambda_q^*$ has a basis 
\[
\{a_1^{x_1}\cdots a_n^{x_n}u^{z}b^{y_n}\cdots b_1^{y_1} \mid x_i, z, y_j\in \{0,1\}\} \, .
\]
(For details, see \cite[Theorem 3.47]{BodWu}.) Note that degree $r$ monomials are a basis of $\Lambda_q^r$.
In particular, the weight of a basis element is the sum of the weights of each term in the monomial. 
Moreover, the algebra structure maps $\Lambda_q^{r} \otimes \Lambda_q^{\ell}\rightarrow \Lambda_q^{r+\ell}$ 
commute with the action of $U_q(\son)$.

One may readily check that $\Lambda_q^k$ has highest weight $\varpi_k$,
so abstract theory implies that $V_k\cong \Lambda_q^k$. 
In particular, there is a choice of isomorphism $V_k\longrightarrow \Lambda_q^k$ such that $v_k^+\mapsto a_1\cdots a_k$. 
Recall that $v_k^-$ is the lowest weight vector of $V(\varpi_k)$ characterized by the condition $T_{w_0}\cdot v_k^-= v_k^+$, 
where $T_{w_0}$ is the longest element in the quantum Weyl group using conventions fixed in \cite[Section 3.2]{BER}. 

\begin{lem}
The $U_q(\son)$ isomorphism $V_k\longrightarrow \Lambda_q^k$ such that $v_k^+\mapsto a_1\cdots a_k$ sends $v_k^-\mapsto b_k\cdots b_1$. 
\end{lem}
\begin{proof}
Similar to proof of \cite[Lemma 4.5]{BER}. For example, if $k=2$, then using 
\[
w_0= (s_2\dots s_{n-1}s_ns_{n-1}\dots s_2)(s_1)(s_2\dots s_{n-1}s_ns_{n-1}\dots s_2)w_0^{1,3, \dots, n}
\]
where $w_0^{1,3,\dots, n}$ is the longest element in parabolic $\langle s_1,s_3,\dots, s_n\rangle\subset W_{B_n}$, it is routine to check that 
\begin{align*}
T_{w_0}\cdot b_2b_1 &= T_{s_2\dots s_{n-1}s_ns_{n-1}\dots s_2}T_{s_1}T_{s_2\dots s_{n-1}s_ns_{n-1}\dots s_2}T_{w_0^{1,3, \dots, n}}\cdot b_2b_1\\
&= T_{s_2\dots s_{n-1}s_ns_{n-1}\dots s_2}T_{s_1}T_{s_2\dots s_{n-1}s_ns_{n-1}\dots s_2}\cdot b_2b_1 \\
&= T_{s_2\dots s_{n-1}s_ns_{n-1}\dots s_2}T_{s_1}\cdot a_2b_1
=T_{s_2\dots s_{n-1}s_ns_{n-1}\dots s_2}\cdot a_1b_2
= a_1a_2 \, . \qedhere
\end{align*}
\end{proof}

In \cite[Proposition 4.14]{BER}, the morphism 
$\hat{\varphi}_{BER}\bigg(
\begin{tikzpicture}[scale=.4, anchorbase,tinynodes]
	\draw[very thick] (0,1.25) to [out=90,in=270] (0,1.5) to [out=90,in=180] (.5,2.25) node[above=-2pt]{$k$}
		to [out=0,in=90] (1,1.5) to (1,1.25);
\end{tikzpicture}  
\bigg)$ is \emph{defined} by the condition that $v_k^+\otimes v_k^-\mapsto 1$. 
Similarly, the image of the $k$-labelled black cup under $\hat{\varphi}_{BER}$ is \emph{defined} 
to be the intertwiner $\C(q)\rightarrow V_k\otimes V_k$ sending $1$ to $v_k^-\otimes v_k^+  + \sum x\otimes y$, 
where $x\otimes y\in V_k[< \varpi_k]\otimes V_k[> - \varpi_k]$. 

On the other hand, 
\cite[Definitions 3.69 and 3.71]{BodWu} imply that $
\hat{\varphi}_{O}\bigg(
\begin{tikzpicture}[scale=.4, anchorbase,tinynodes]
	\draw[very thick] (0,1.25) to [out=90,in=270] (0,1.5) to [out=90,in=180] (.5,2.25) node[above=-2pt]{$k$}
		to [out=0,in=90] (1,1.5) to (1,1.25);
\end{tikzpicture}  
\bigg)$ 
is \emph{defined} by
\[
\hat{\varphi}_{O}\bigg(
\begin{tikzpicture}[scale=.4, anchorbase,tinynodes]
	\draw[very thick] (0,1.25) to [out=90,in=270] (0,1.5) to [out=90,in=180] (.5,2.25) node[above=-2pt]{$1$}
		to [out=0,in=90] (1,1.5) to (1,1.25);
\end{tikzpicture}  
\bigg)
:= \hat{\varphi}_{BER}\bigg(
\begin{tikzpicture}[scale=.4, anchorbase,tinynodes]
	\draw[very thick] (0,1.25) to [out=90,in=270] (0,1.5) to [out=90,in=180] (.5,2.25) node[above=-2pt]{$1$}
		to [out=0,in=90] (1,1.5) to (1,1.25);
\end{tikzpicture}  
\bigg)
\qquad 
\text{and} 
\qquad 
\hat{\varphi}_{O}\bigg(
\begin{tikzpicture}[scale=.4, anchorbase,tinynodes]
	\draw[very thick] (0,1.25) to [out=90,in=270] (0,1.5) to [out=90,in=180] (.5,2.25) node[above=-2pt]{$k$}
		to [out=0,in=90] (1,1.5) to (1,1.25);
\end{tikzpicture}  
\bigg)
:= \frac{[2]}{[2k]}
\hat{\varphi}_{O} \Bigg(
\begin{tikzpicture}[scale =.35, tinynodes,anchorbase,rotate=180]
	\draw[very thick, black] (2.5,.75) to [out=330,in=90] (3,0) to [out=270,in=0] (1.5,-1.5) 
		node[above=-2pt]{$k{-}1$} to [out=180,in=270] 
			(0,0) to [out=90,in=210] (.5,.75);
	\draw[very thick, black] (2.5,.75) to [out=210,in=90] (2,0) to [out=270,in=0] (1.5,-.5) 
		node[above=-2pt]{$1$} to [out=180,in=270]
			(1,0) to [out=90,in=330] (.5,.75);
	\draw[very thick, black] (.5,.75) to (.5,1.5) node[below=-2pt]{$k$};
	\draw[very thick, black] (2.5,.75) to (2.5,1.5) node[below=-2pt]{$k$};
\end{tikzpicture}
\Bigg) \, .
\]
Since $\hat{\varphi}_{O}$ is pivotal, 
the analogous equation \emph{defines} the image of $k$ labelled black cup under $\hat{\varphi}_{O}$. 
For the latter, one also needs the image of trivalent vertices under $\hat{\varphi}_{O}$,
which are defined in \cite{BodWu} using the multiplication on $\Lambda_q^*$. 
In particular, the images of trivalent vertices under $\hat{\varphi}_{O}$ automatically satisfy associativity \eqref{eq:assoc}. 

Having established how the intertwiners in \eqref{eq:BW-vs-BER} are defined, we now prove Lemma \ref{lem-spin-vs-O-cups-and-caps}. 

\begin{proof}[Proof of Lemma \ref{lem-spin-vs-O-cups-and-caps}]
Since the functors $\hat{\varphi}_{O}$ and $\hat{\varphi}_{BER}$ are pivotal, 
it suffices to check that the image of black $k$-labelled cups differ by $(-1)^{\binom{k}{2}}$.
Comparing \cite[Definition 3.69]{BodWu} and \cite[Proposition 4.14]{BER}, 
it follows that the image of the $1$-labelled cup under $\hat{\varphi}_{O}$ and $\hat{\varphi}_{BER}$ \emph{both} agree 
with the intertwiner $\C(q)\rightarrow \Lambda_q^1\otimes \Lambda_q^1$ defined by
\[
1\mapsto \sum_{i=1}^{n} -(-q^2)^{2n-i}a_i\otimes b_i + \frac{(-q^2)^n}{[2]}u\otimes u + \sum_{i=1}^{n} (-q^2)^{i-1}b_i\otimes a_i \, .
\]

Since the $1$-labelled cups in either reference agree, 
we are reduced to computing the projection onto $\Lambda_q^k[\varpi_k]\otimes \Lambda_q^k[-\varpi_k]$ 
of the vector which is the image of $1$ under $\frac{1}{``[k]^2"!}$ times the composition
\[
\C(q)\rightarrow \Lambda_q^1\otimes \Lambda_q^1 \rightarrow 
	\Lambda_q^1\otimes \Lambda_q^1 \otimes \Lambda_q^1 \otimes \Lambda_q^1 \rightarrow \dots 
		\rightarrow (\Lambda_q^1)^{\otimes k}\otimes (\Lambda_q^1)^{\otimes k} \rightarrow \Lambda_q^k\otimes \Lambda_q^k \, .
\]
We do so by taking the coefficient of $b_k\dots b_1\otimes a_1\dots a_k$ (the avatar of $v_k^-\otimes v_k^+$ in $\Lambda_q^k\otimes \Lambda_q^k$) 
which one easily checks to be $(-1)^{0+1+2 + \dots + (k-1)} = (-1)^{\binom{k}{2}}$.

For example, when $k=3$, 
the relevant tensor is
\begin{align*}
\frac{1}{``[3]^2"!}&(-q^2)^{0+1+2}
\Big(b_1b_2b_3\otimes a_3a_2a_1 + b_2b_1b_3\otimes a_3a_1a_2 + b_1b_3b_2\otimes a_2a_3a_1 + b_2b_3b_1\otimes a_1a_3a_2 \\
		& \qquad \qquad \qquad \qquad \qquad \qquad \qquad \qquad 
			\qquad \qquad \qquad + b_3b_1b_2\otimes a_2a_1a_3 + b_3b_2b_1\otimes a_1a_2a_3\Big) \\
	=&\frac{1}{``[3]^2"!}(-q^2)^{0+1+2}\Big((-q^{-2})^6 + (-q^{-2})^4 +(-q^{-2})^4 + (-q^{-2})^2 + (-q^{-2})^2 + 1 \Big)
		b_3b_2b_1\otimes a_1a_2a_3 \\
	=&(-1)^{0+1+2}b_3b_2b_1\otimes a_1a_2a_3 \, . \qedhere
\end{align*}

\end{proof}

%
\section{Comparing bijections}\label{Appendix:doublecentralizerbijection}
%

Although unnecessary for the proof of our main theorems, 
for the sake of completeness we now record Wenzl's refinement of Theorem \ref{thm:bijection}, 
which was alluded to in Remark \ref{rem:Wenzldoesmore}.  
The following arguments, which are adapted from the proof of \cite[Proposition 5.5]{Wenzl-Spin}, 
culminate in Proposition \ref{prop:Xiisinversetoa}, 
which verifies that the bijection $\Xi_{m,n}$ from Definition \ref{def:xi} 
is in fact the inverse to $\ab_{m,n}(-)$.

\begin{lem}\label{lem:reslem1}
Let $\ab, \ab'\in \SP_{m}$. 
If $M_{\ab}$ and $M_{\ab'}$ are isomorphic upon restriction to $\iqUsom[m-1]$, 
then $M_{\ab}\cong M_{\ab'}$ as $\iqUsom$-modules. 
\end{lem}
\begin{proof}
This follows from Theorem \ref{thm:reptheorytheorem}, 
using the combinatorial fact that each $\ab\in \SP_m$ is uniquely determined 
by the set $\{\abb\in \SP_{m-1} \mid  \ab \ \text{and} \ \abb \ \text{interlace}\}$.
\end{proof}

\begin{lem}\label{lem:reslem2}
Let $\lambda\in X_+(\son)^{S^{\otimes m}}$. Upon restriction to $\iqUsom[m-1]$, 
we have an isomorphism
\begin{equation}\label{eq:reslem2}
\Hom_{U_q(\son)}(V_{\lambda}, S^{\otimes m}) \cong \Hom_{U_q(\son)}(V_{\lambda}\otimes S, S^{\otimes m-1})
\end{equation}
of $\iqUsom[m-1]$-modules.
\end{lem}
\begin{proof}
Set $\mathrm{cap}_{S,S}:=\hat{\varphi}_{BER}
\Big(
\begin{tikzpicture}[scale=.4, anchorbase,tinynodes]
	\draw[very thick, gray] (0,1.25) to [out=90,in=270] (0,1.5) to [out=90,in=180] (.5,2.25)
		to [out=0,in=90] (1,1.5) to (1,1.25);
\end{tikzpicture}  
\Big)$. 
The duality isomorphism in \eqref{eq:reslem2}, explicitly given by 
\[
f\mapsto (\id_{S^{\otimes m-1}}\otimes \mathrm{cap}_{S,S})\circ (f\ot \mathrm{id}_S) \, ,
\]
commutes with the actions of $\iqUsom[m-1]$.
\end{proof}

\begin{lem}\label{lem:bijectionbackwards}
Let $\ab\in \SP_m^{\le n}$. 
As $U_q(\son)$-modules,
there is a multiplicity free decomposition
\begin{equation}\label{eq:bijectionbackwards}
V_{\Xi_{m,n}(\ab)}\otimes S \cong 
\bigoplus_{\substack{ \abb\in \SP_{m-1}^{\le n} \\\abb \ \text{interlaces} \ \ab}} V_{\Xi_{m-1,n}(\abb)} \, .
\end{equation}
\end{lem}
\begin{proof}
By complete reducibility, it suffices to show that
$\dim\Hom_{U_q(\son)}(V_{\mu}, V_{\Xi_{m,n}(\ab)}\otimes S) \ne 0$ 
if and only if
$\mu \in \Xi_{m-1, n}\left( \{\abb\in \SP_{m-1} \mid  \ab \ \text{and} \ \abb \ \text{interlace}\}\right)$,
and that $\dim\Hom_{U_q(\son)}(V_{\mu}, V_{\Xi_{m,n}(\ab)}\otimes S) =1$ for such $\mu$. 
Using the isomorphisms
\begin{align*}
\Hom_{U_q(\son)}(V_{\mu}, V_{\Xi_{m,n}(\ab)}\otimes S) &\cong \Hom_{U_q(\son)}(V_{\mu}\otimes S, V_{\Xi_{m,n}(\ab)}) \\
 f&\mapsto (\id_{V_{\Xi_{m,n}(\ab)}}\otimes \mathrm{cap}_{S,S})\circ (f\otimes \mathrm{id}_S) \, ,
\end{align*}
the claim then follows from Lemma \ref{lem:bijection} (read backwards) and Schur's Lemma.
\end{proof}

\begin{prop}\label{prop:Xiisinversetoa}
If $\ab\in \SP_{m}^{\le n}$, then $\ab_{m,n}(\Xi_{m,n}(\ab)) = \ab$. 
\end{prop}
\begin{proof}
By Theorem \ref{thm:reptheorytheorem}, it suffices to show that $M_{\ab_{m,n}(\Xi_{m,n}(\ab))}\cong M_{\ab}$. 
We proceed by induction on $m$, the base case already being established in Example \ref{ex:m2discussionformatching} and Example \ref{ex:basecase}. 
Let $m>2$ and suppose that $\ab_{m-1, n}$ and $\Xi_{m-1,n}$ are inverse to one another. 

By Lemma \ref{lem:reslem1}, it suffices to show that $M_{\ab_{m,n}(\Xi_{m,n}(\ab))}\cong M_{\ab}$ as $\iqUsom[m-1]$-modules. 
Upon restriction to $\iqUsom[m-1]$, we have
\begin{align*}
M_{\ab_{m,n}(\Xi_{m,n}(\ab))} &\cong \Hom_{U_q(\son)}(V_{\Xi_{m,n}(\ab)}, S^{\otimes m}) \\
&\!\!\!\stackrel{\eqref{eq:reslem2}}{\cong} \Hom_{U_q(\son)}(V_{\Xi_{m,n}(\ab)}\otimes S, S^{\otimes m-1}) \\
&\!\!\!\stackrel{\eqref{eq:bijectionbackwards}}{\cong} \bigoplus_{\substack{ \abb\in \SP_{m-1}^{\le n} \\\abb \ \text{interlaces} \ \ab}} \Hom_{U_q(\son)}(V_{\Xi_{m-1,n}(\abb)}, S^{\otimes m-1}) \\
&\cong \bigoplus_{\substack{ \abb\in \SP_{m-1}^{\le n} \\\abb \ \text{interlaces} \ \ab}} M_{\ab_{m-1, n}(\Xi_{m-1, n}(\abb))}
\cong \bigoplus_{\substack{ \abb\in \SP_{m-1}^{\le n} \\\abb \ \text{interlaces} \ \ab}} M_{\abb}
\cong M_{\ab} \, ,
\end{align*}
where the second to last isomorphism follows from the induction hypothesis, 
and the last isomorphism follows from \eqref{eq:resM} and Remark \ref{rem:interlacesaturated}.
\end{proof}


\end{document}